\documentclass[11pt, reqno]{amsart}
\usepackage{amsmath, amsthm, amscd, amsfonts, amssymb, graphicx, xcolor}
\usepackage[bookmarksnumbered, colorlinks, plainpages]{hyperref}
\usepackage{appendix}
\usepackage{graphicx}
\graphicspath{ {./draft/} } 
\usepackage{caption}

\newtheorem{theorem}{Theorem}[section]
\newtheorem{lemma}[theorem]{Lemma}
\newtheorem{proposition}[theorem]{Proposition}
\newtheorem{corollary}[theorem]{Corollary}
\theoremstyle{definition}

\theoremstyle{remark}
\newtheorem{remark}[theorem]{Remark}

\DeclareMathOperator*{\diam}{\textup{diam}}

\numberwithin{equation}{section}

\def\Xint#1{\mathchoice
	{\XXint\displaystyle\textstyle{#1}}%
	{\XXint\textstyle\scriptstyle{#1}}%
	{\XXint\scriptstyle\scriptscriptstyle{#1}}%
	{\XXint\scriptscriptstyle\scriptscriptstyle{#1}}%
	\!\int}
\def\XXint#1#2#3{{\setbox0=\hbox{$#1{#2#3}{\int}$ }
		\vcenter{\hbox{$#2#3$ }}\kern-.6\wd0}}

\def\dashint{\Xint-}        % Assignment title   % Student name
\newcommand{\bs}{\symbol{92}}
\newcommand{\D}{ \left(\!\!\! 
	\begin{array}{c}
		N^T \nabla_{\theta} \\
		\partial_t
	\end{array}
	\!\!\!\right)}
\DeclareMathOperator*{\osc}{osc}
\DeclareMathOperator*{\Supp}{\textup{Supp}}
\DeclareMathOperator*{\dist}{\textup{dist}}

\newcommand{\Le}{\mathcal{L}_\varepsilon}
\usepackage{graphicx}

\begin{document}
	\setcounter{page}{1}

	\title[]{Asymptotic analysis of boundary layers\\ for Stokes systems in periodic homogenization}
	
	\author[]{Moustapha Agne$^1$ }
	
	\address{$^1$ AGM--Department of Mathematics, 2 av. Adolphe Chauvin,
		95302 CERGY-PONTOISE, CY Cergy Paris University, Cergy-Pontoise, France.}
	\email{\textcolor[rgb]{0.00,0.00,0.84}{moustapha.agne@cyu.fr}}

	%\dedicatory{This paper is dedicated to Professor ABCD}
	

	\begin{abstract}
		We investigate the asymptotics of boundary layers in periodic homogenization. The analysis is focused on a  Stokes system with periodic coefficients and periodic Dirichlet data posed in the half-space $\{y\in \mathbb{R}^d: y\cdot n -s>0\}$. In particular, we establish the convergence of the velocity as $y\cdot n \rightarrow \infty$. We obtain this convergence for arbitrary normals $n\in \mathbb{S}^{d-1}$\!. Moreover, we build an asymptotic expansion of Poisson's kernel for the periodically oscillating Stokes operator in the half-space. 	The presence of the pressure and the incompressibility condition impose certain innovations. In particular, we provide a framework for the analysis of the boundary layers in homogenization that relies only on physical space techniques and not on techniques that rely on the quasiperiodic structure of the problem.
	\end{abstract}

	\maketitle
	
	\section{Introduction}
	
	\noindent In this paper we study the asymptotics of $u^s=u^s(y)\in\mathbb{R}^d$, $y\in\mathbb{R}^d,$ where $(u^s,p^s)$ is a solution to the Stokes system
	\begin{equation}\label{hsStokes}
		\left\{
		\begin{array}{rcll}
			-\nabla\cdot A (y) \nabla u^s +  \nabla p^s&=& 0 &   \mbox{ in } \{ y\cdot n-s>0 \},\\
			\nabla\cdot u^s&=& 0, \; \\
			u^s&=& g (y)  & \mbox{ on }  \{y\cdot n-s=0\}
		\end{array}
		\right. 
	\end{equation}
	set in the half-space $\{ y\cdot n-s>0 \}$  and with periodically oscillating coefficients $A=A(y)$ and boundary data $g=g(y)$.\\

	\noindent 	Homogenization can be defined as a process of averaging some micro-scale features  to obtain proper-\\ties at a  level that represent the overall behavior of an object of study. It is used  throughout the sciences and in particular in physics and mechanics with the study of composite media. In  mathematics and, in particular, in the theory of Partial Differential Equations (PDE) the topic of homogenization comes in the form  of equations as in (\ref{basic_Stokes}). In these equations or PDE the parameter $\varepsilon$ usually represents a small number accounting for a micro-scale.
	
	\subsection{Motivations}
	\noindent The solution to problem (\ref{hsStokes}) is a half-space boundary layer corrector. Indeed, it is used for refining asymptotic expansions in homogenization theory for problems involving the presence of a boundary. Let us describe briefly how this application works. A well-known method of analyzing the solution  	$(u^\varepsilon,p^\varepsilon)$ to the problem
		\begin{equation}\label{basic_Stokes}
		\left\{
		\begin{array}{rcll}
			-\nabla\cdot A (x/\varepsilon ) \nabla u^\varepsilon +  \nabla p^\varepsilon&=& f\;  & \mbox{ in } \{ x\cdot n-s>0 \},\\
			\nabla\cdot u^\varepsilon&=& h,\; &\\
			u^\varepsilon&=&g   &\mbox{ on }  \{x\cdot n-s=0\}
		\end{array}
		\right. 
	\end{equation}
	 with highly oscillating coefficients but non-oscillating boundary data	is by introducing a multi-scale expansion  
	of the form 
	\begin{align*}
		u^\varepsilon(x)&=u^0(x,x/\varepsilon)+ \varepsilon
		u^1(x,x/\varepsilon)+...\\
		p^\varepsilon(x)&=  
		p^0(x,x/\varepsilon)+\varepsilon p^1(x,x/\varepsilon)+ ...
	\end{align*}
	where $u^k=u^k(x,y), p^k=p^k(x,y)$, $k=0,1$, are periodic in the variable $y$ and the pair $(u^0,p^0)$ satisfies the so-called homogenized problem
	\begin{equation*}
		\left\{
		\begin{array}{rcll}
			-\nabla\cdot A^0\nabla u^0 +  \nabla p^0&=& f\;  &\mbox{ in } \{ x\cdot n-s>0 \},\\
			\nabla\cdot u^0&=& h,&\\
			u^0&=&g  & \mbox{ on }  \{x\cdot n-s=0\}.
		\end{array}
		\right. 
	\end{equation*}
	See subsection \ref{ansatz} for more details on the multi-scale expansion.\\
	
	\noindent 	Such an Ansatz provides a satisfactory approximation of the solution $u^\varepsilon$ in the interior of the domain as it  captures the periodic micro-structure of the domain. Nonetheless, in general, the boundary breaks the periodic structure : the periodicity of the term $u^1=u^1(x,x/\varepsilon)$ in the variable $x/\varepsilon$  is incompatible with the Dirichlet conditions
	\begin{align*}
		u^0(x,x/\varepsilon)+ \varepsilon
		u^1(x,x/\varepsilon)+...=g\quad  \mbox{and}\quad u^0=g \quad \mbox{ on } \{x\cdot n-s=0\}.
	\end{align*}
	Hence we add in the Ansatz a boundary layer term $u_{bl}^\varepsilon$,
	\begin{align*}
		u^\varepsilon(x)&=u^0(x,x/\varepsilon)+ \varepsilon
		u^1(x,x/\varepsilon)+  \varepsilon u_{bl}^\varepsilon + ...,
	\end{align*}
	that acts as to compensate for this incompatibility. In fact, the boundary layer term $u_{bl}^\varepsilon$ is chosen to satisfy $u_{bl}^\varepsilon + u^1=0$ on the boundary $ \{x\cdot n-s=0\}$.  In other words the boundary layer corrector is introduced to account for the effects of the boundary on the Ansatz at order $\varepsilon$. More precisely, the quantity $u^1(x,x/\varepsilon)$ is taken to equal $	\chi (x/\varepsilon)\cdot  \nabla u^0(x)$ which is highly oscillating and the boundary layer term 	$u_{bl}^\varepsilon$ satifies 
	\begin{equation*}
		\left\{
		\begin{array}{rcll}
			-\nabla\cdot A (x/\varepsilon ) \nabla u_{bl}^\varepsilon +  \nabla p_{bl}^\varepsilon&=& 0\; & \mbox{ in } \{ x\cdot n-s>0 \},\\
			\nabla\cdot u_{bl}^\varepsilon&=& 0,\; &\\
			u_{bl}^\varepsilon&=&-\chi (x/\varepsilon)\cdot  \nabla u^0(x)  & \mbox{ on }  \{x\cdot n-s=0\}.
		\end{array}
		\right. 
	\end{equation*} 
	As a consequence, the addition of the boundary layer term to the expansion  leads to refined estimates on the error $u^\varepsilon- u^0$ as shown for instance in \cite{AA,homogen_polygon} for elliptic equations and in \cite{Slip_cond} for Stokes in the slightly different context of oscillating boundaries.
	
	The study in the present article focuses on the geometry of the half-space. It is a simple case geometry but it retains the essential difficulties and it is also a basic element in the case of most general smooth curved domain where the half-space case is used as a building block (see \cite{AKMP,homogen_bl} and \textsc{Figure} \ref{domain_approximate}).
	
	Among the interests of problem (\ref{basic_Stokes}) is that it models  some physical and mechanical objects with a composite micro-structure. It is also used in the  modeling of flows in a domain with a rough boundary as in \cite{Slip_cond, KLS}. Indeed, let us outline the connection between the flow over a rough boundary and system (\ref{basic_Stokes}). Via a change of variables one flattens the boundary which transfers the oscillations from the boundary into the coefficients of the Stokes operator. Moreover, the system (\ref{basic_Stokes}) is related to the first-order expansion of eigenvalue problems for the Stokes operator with oscillating coefficients as in the elliptic case (\cite{Vogelius,Prange_eigen,hom_eigen}).

	\subsection{State of the art}
	There is an important body of works dealing with the topic of boundary layers in homogenization and related topics such as the oscillating boundary problem as well as the oscillating data problem (\cite{ aleyksan,AA,fully_nonlin,homogen_bl,Slip_cond,Prange,num_bound,regul_hmgnized_data,hom_finite_dom}). Let us examine this literature with the angle of the equations and geometry.	\subsubsection{A variety of equations}	The  oscillating boundary value problems have drawn the attention of several authors. The common feature is the fact that the boundary data is oscillating. In particular, elliptic systems with an oscillating boundary condition  were extensively investigated by many authors \cite{aleyksan,AA,Feldman,homogen_bl,homogen_polygon,KLS,Prange,num_bound,
	hom_finite_dom}. Moreover, Dirichlet \cite{aleyksan,AKMP} and  Neumann boundary conditions were studied in \cite{regul_hmgnized_data}. Fully nonlinear equations with oscillating Dirichlet boundary conditions were also studied in \cite{Feldman,fully_nonlin}. Besides, the Stokes system  and the stationary Navier-Stokes equations were examined in the context of rough  boundaries \cite{Slip_cond,HigPrZhuge}.
	\subsubsection{The geometry of the domain}	
	The boundary of the domain plays an important role in many PDE problems. This is particularly the case when the problem presents oscillations phenomena  near the boundary such as the oscillating boundary data problems.
	A fundamental feature of the problem that is present in the elliptic as well as in the Stokes case is the way the boundary interacts with the periodic structure. Locally, this interaction is the way the tangent hyperplane intersects the lattice. The number of hyperplanes needed to determine or approximate the domain may be one or an arbitrary large number.

 \noindent(1) \textbf{Half-space :} This corresponds to the case where the domain is determined by a unique hyperplane. This hyperplane being defined by its unit normal vector $n\in \mathbb{S}^{d-1}$  one can distinguish two situations: $n $ is rational and $n$ is irrational. When $n$ is rational i.e. $n\in \mathbb{R}\mathbb{Z}^{d}$ one is in fact in a periodic setting. Indeed, in this case one has $n=\lambda v, \lambda\in\mathbb{R}, v\in\mathbb{Z}^d$, which is a definite intersection of the line $\{\lambda n, \lambda\in\mathbb{R} \}$ with the lattice $\mathbb{Z}^d$. Then up to a normalization, considering $|v|$ as the reference length, and by the periodicity of $A$ and $g$ we get $A(y+hv)=A(y), \ g(y+hv)=g(v)$, for all $h\in\mathbb{Z}$. This is the periodicity in the normal direction which, by the orthogonality of the periodic structure, yields the periodicity in the tangential direction. In the case  $n$ is irrational the line $\{\lambda n, \lambda\in\mathbb{R} \}$ has no intersection with the lattice $\mathbb{Z}^d$. Yet there are points of the lattice that are arbitrarily close to the line. However, there are irrational $n$ for which  the distance  $\dist(h, \{\lambda n, \lambda\in\mathbb{R} \} ), h\in \mathbb{Z}^d,$ is, in some sense, bounded below by a positive number. More precisely, such irrational direction $n$ satisfy a condition of the type : there exist $A(n),\tau>0$ such that
		\begin{align}\label{Dioph}
			| P_{n^\perp} (h)| \geq \frac{A(n)}{|h|^{d+\tau}},\quad  \mbox{ for all } h\in \mathbb{Z}^d\bs \{0\},
		\end{align}
	 where $P_{n^\perp}$ denotes the orthogonal projection onto the hyperplane $n^\perp =\{y\in\mathbb{R}^d:y\cdot n= s\}.$ 
		This so-called Diophantine condition is a non-resonance condition. It implies a quantitative ergodic property  of the half-space boundary layer which yields a convergence to a boundary layer tail faster than any polynomial rate of convergence  (see \cite{homogen_polygon,Slip_cond}).

		The case of general irrational normals  $n$ for elliptic systems is treated in the article \cite{Prange}. It is shown in that article that for an arbitrary $n\notin \mathbb{R}\mathbb{Z}^{d}$ the convergence of the boundary layer can be arbitrarily slow.
		
		Studies concerning the half-space domain are found in \cite{AKMP,fully_nonlin,homogen_bl}. Some authors take the polygon to be the intersection of half-spaces and thus consider the half-space problem as a part in their analysis of  equations posed in a polygonal domain. Others approximate the smooth and uniformly convex domain by a polygon and thus approximate the boundary by a sequence of hyperplanes (see \textsc{Figure} \ref{domain_approximate}).

	\noindent(2) \textbf{Polygons :}  A polygon can be determined by a finite sequence of hyperplanes. Polygons make a part of the family of bounded domains which are extensively considered in the literature. The article \cite{AA} treats the problem in a rectangular domain which is a special case of a polygon. A study considering more general polygon is found in \cite{aleyksan,homogen_polygon}.
	 For  polygons, the non-smoothness at the vertices is  a source of difficulty.
	
\noindent(3)	\textbf{Smooth domain :}  For curved domains $\Omega$ we require the domain  to have a certain regularity e.g. to be smooth and uniformly convex (\cite{regul_hmgnized_data,AKMP,homogen_bl}). The assumption of smoothness of the boundary  implies that the unit normal $n(x),  \ x\in \partial \Omega,$ is smooth in $x$, which plays a part in the  regularity of the homogenized boundary data in \cite{AKMP}. This is important because it allows applications of boundary regularity estimates. The uniform convexity allows to get a control on the constant $A(n)$ in the Diophantine condition (\ref{Dioph}) for instance in \cite{AKMP} we have that
 $A^{-1}(n(x)), \, x\in\partial\Omega$, as a function of $x\in\partial\Omega$, belongs to the weak Lebesgue space
 $L^{d-1,\infty}(\partial\Omega)$ .

		\begin{figure}
		\centering
		\includegraphics[scale=0.6]{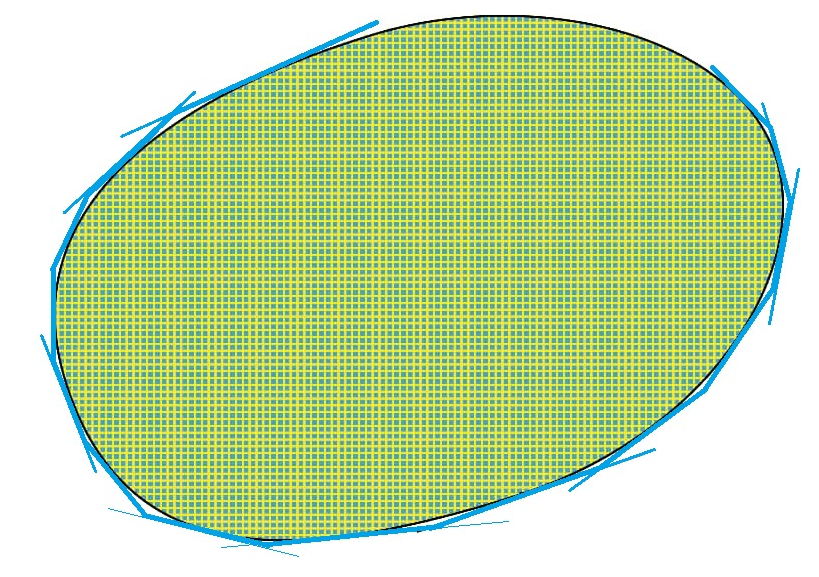}
		\caption{Approximating a smooth domain by a set of hyperplanes.}
		\label{domain_approximate}
	\end{figure}
	
\noindent Let us also mention the articles \cite{AKMP, homogen_bl, num_bound, hom_finite_dom} that deal with the homogenization of an elliptic system in bounded smooth domains. In these articles, the authors use a technique by which the boundary is broken into pieces and each piece is  approximated by a hyperplane, thus, making the half-space problem an entire step in their study. The article \cite{AKMP}
treats important issues  such as the thickness of the boundary layer. Knowing this thickness allows to estimate how far from the boundary one can start to use correctly the interior Ansatz without the boundary layer corrector.

One of the goals of the study of the asymptotics of the boundary layer in the half-space is to find the term that must be added to the interior Ansatz so that it remains valid next to  the boundary.  	Moreover, the question of the convergence far from the boundary is related to the problem of homogenization of the oscillating boundary data. In fact,  the homogenized boundary data is constructed from the boundary layer tail as shown in \cite{regul_hmgnized_data}.  In the present article, one intends to adapt the methods of \cite{Prange} to the Stokes system. 
	
	\subsection{Difficulties specific to the Stokes system}
	\noindent The questions of existence and uniqueness of a solution to system 
		\begin{equation*}
		\left\{
		\begin{array}{rcll}
			-\nabla\cdot A (y ) \nabla u^s &=& 0  & \mbox{ in } \{ y\cdot n>s \},\\
			u^s&=&g  &\mbox{ on }  \{y\cdot n=s\}.
		\end{array}
		\right. 
	\end{equation*}
	have been handled in many articles by using the quasiperiodicity of the data in the hyperplane $\{ y\cdot n=s\}$ (see \cite{homogen_bl, homogen_polygon,regul_hmgnized_data}),
	\newcommand{\bigB}{\mathcal{B}}
	\begin{align*}
		g( M(z',z_d) )=  V_0( N z',z_d), \quad	  M^TA( M (z',z_d))M=  \mathcal{B} (N z',z_d), \quad
	\end{align*}
	where $ V_0=V_0(\theta, t), \,\mathcal{B}=\bigB(\theta, t)$ are periodic in $\theta$,
		and by turning the problem into one posed in the cylinder $\mathbb{T}^d\times \mathbb{R}_+$
	\begin{equation*}
		\left\{
		\begin{array}{rcll}
			\vspace{.1in}
			-\D\cdot \bigB(\theta, t) \D V &=& 0, &(\theta,t) \mbox{ in }\mathbb{T}^d\times \mathbb{R}_+,\\
			V&=& V_0 (\theta,0),   &\theta\mbox{ in }\mathbb{T}^d.
		\end{array}
		\right. 
	\end{equation*}
	 There is a cost in the move from the half-space to the higher dimension cylinder which is the loss of the ellipticity of the differential operator. Indeed, through this move we get a degenerate gradient. This degeneracy is handled by a regularization method. For instance, the existence of a solution is obtained by some authors through this regularization argument. It consists in using the compactness of the sequence $(V_\delta)_{\delta>0}$ of solutions e.g. to  elliptic systems
	\begin{equation*}
		\left\{
		\begin{array}{rcll}
			\vspace{.1in}
			-\delta \Delta_\theta V_\delta	-\D\cdot \bigB(\theta, t) \D V_\delta &=& 0,  &(\theta,t) \mbox{ in }\mathbb{T}^d\times \mathbb{R}_+,\\
			V_\delta&=& V_0 (\theta,0),   &\theta\mbox{ in }\mathbb{T}^d.
		\end{array}
		\right. 
	\end{equation*}
	 The article \cite{Slip_cond} follows the same lines as for the well-posedness issue yet it studies a Stokes system. Nonetheless the Stokes problem treated therein,  posed in a half-space, is divergence-free and has a homogeneous boundary data
	 	\begin{equation*}
	 	\left\{
	 	\begin{array}{rcll}
	 		-\nabla\cdot A (y ) \nabla u +  \nabla p&=& f & \mbox{ in } \{ y\cdot n-s>0 \},\\
	 		\nabla\cdot u&=&0 ,&\\
	 		u&=&0   &\mbox{ on }  \{y\cdot n-s=0\}.
	 	\end{array}
	 	\right. 
	 \end{equation*}
	  Therefore the operation that turns this problem posed in a half-space into one posed in the cylinder preserve the solenoidal aspect of the problem and the homogenous Dirichlet condition
	\begin{equation}\label{big_Stokes}
		\left\{
		\begin{array}{rcll}
			\vspace{.1in}
			-\D\cdot \bigB(\theta, t) \D V +\;  M\D R&=& G , &(\theta,t) \mbox{ in }\mathbb{T}^d\times \mathbb{R}_+,\\
			\vspace{.1in}
			\D\cdot\ M^T V&=& 0 ,\\
			V&=& 0,   &\theta\mbox{ in }\mathbb{T}^d.
		\end{array}
		\right. 
	\end{equation} 
	In this case one successfully carries out energy estimates that lead to uniqueness for instance. In contrast, in the case of a Stokes problem with a non-homogenous boundary condition such as g in (\ref{hsStokes}), lifting the boundary data leads to technical difficulties. Indeed, there is no obvious divergence-free lifting of the boundary data. 
	The switching  to the higher dimension cylinder  $\mathbb{T}^d\times \mathbb{R}_+$ is also used in the article \cite{AKMP} for comparing two boundary layers $v_1$ and $v_2$ associated respectively to normal vectors $n_1$ and $n_2$. The authors study the difference $W=V_1-V_2$ where  $V_1, V_2$ are  defined on the same domain $\mathbb{T}^d\times \mathbb{R}_+$ and are respective representations of $u_1, u_2.$ 
	
\noindent	Moreover, there are  formulae given in \cite{regul_hmgnized_data} (adapted in the present orientation of the normal $n$) :
	when  a solution $u^s$   to problem (\ref{hsStokes})
	is in the form
		\begin{align}
			u^s(y)&=V(y-(y\cdot n- s)n,y\cdot n- s)\label{def_u^s},
		\end{align} 
		 then $V=V(\theta,t)$ has satisfy
 	\begin{equation}\label{prob_V}
		\left\{
		\begin{array}{rcll}
			\vspace{.1in}
			-\D\cdot \bigB(\theta, t) \D V +\;  M\D R&=& 0, &(\theta,t) \mbox{ in }\mathbb{R}^d\times \mathbb{R}_+,\\
			\vspace{.1in}
			\D\cdot M^T V&=& 0 ,\\
			V&=& g (\theta),   &\theta\mbox{ in }\mathbb{R}^d.
		\end{array}       
		\right. 
	\end{equation} 
	Conversely, the identity (\ref{def_u^s}) yields
		\begin{align}
		V(\theta,t)&=u^{\theta\cdot n}(\theta+t n),\quad (\theta,t)\in \mathbb{R}^d\times \mathbb{R}_+.\label{def_Vtheta,t}
	\end{align} 
		By these formulae one can readily switch from the system (\ref{prob_V}) 
	to the  problem (\ref{hsStokes}) set in the physical space.
	In particular, one can perform computations directly on $u^s$ to derive   estimates on $V$.
	
	\subsection{Statement of the main Result}
	\noindent	In this article we aim at showing the following :
	\begin{theorem}\label{intro_ThConverge}
	There exists a unique solution $u^s\in L^\infty(\mathbb{H}_n(s) )$ to problem \textup{(\ref{hsStokes})} such that 
  \begin{equation}\label{class}
      \sup_{y\in\mathbb{H}_n(s)} \;  (y\cdot n - s) |\nabla u^s (y)| <\infty.
  \end{equation}
Moreover, as  $y\cdot n \rightarrow \infty$	this solution  converges to a constant vector fields $U^{s,\infty}(n)$ \textup{(}see formula \textup{
(\ref{formul_bltail})} in  \textup{Section  \ref{section_converge}}\textup{)} and  this limit does not depend on $s$ when $n$ is irrational.
	\end{theorem}
	
	\subsection{Strategy of the proof}
	
	\noindent This article aims at establishing the proof of Theorem \ref{intro_ThConverge}. It focuses on the convergence to a boundary layer tail. Besides, a proof of the uniqueness of the solution $u^s$ under hypothesis (\ref{class}) is presented in the Appendix. 	Our method uses an integral representation of the solution to problem (\ref{hsStokes}). To put it in more details, we use a representation via the Poisson kernel associated to the Stokes operator and the domain $ \{ y\cdot n>0 \}$. In fact, for most of the paper we will consider the case $s=0$ in problem (\ref{hsStokes}).  One of the reasons for such a choice is that one can make re-scalings without changing the domain. Most of the  results and estimates are uniform in $s$ and continue to hold in the case of an arbitrary $s\in\mathbb{R}$ up to a translation in the formulae. A pivotal point of our method is obtaining  an asymptotic expansion  of the Poisson kernel which subsequently leads to the asymptotic behavior of the boundary layer. We use  uniform estimates to control this expansion.
	These are important tools for the study of  problems involving oscillating coefficients. In this context they correspond to  estimates of $u^\varepsilon, p^\varepsilon$ uniformly in $\varepsilon$. For instance, in the elliptic case, Avellaneda and Lin established uniform Lipschitz estimates. In the case of a Stokes system, the article \cite{GZ} presents a uniform Lipschitz estimate for the velocity $u^\varepsilon$ and a uniform oscillation estimate for the pressure, see also \cite{hom_stokes, optimBoundEstimate, HigPrZhuge}.
	\subsection{Organization of the paper}
	\noindent	This article is organized as follows. It starts with an introduction that gives an overview of the question. In Section \ref{section_prepa} we recall the construction of the interior correctors and give  an Ansatz for the heterogeneous quantities $(u^\varepsilon, \, p^\varepsilon)$ satisfying a Stokes system. Estimates on Green's functions are also mentioned in this section. Then we move on to Section \ref{section_hom}, where a homogenization result is established for a Stokes system posed in the half-space. This homogenization result is applied in Section \ref{section_hom_G}, where we study the homogenization of the Green function. In Section \ref{section_expans} we get  asymptotic expansions of Green's functions and of Poisson's kernel.
	Section \ref{section_converge} is where we give the proof of the convergence towards a boundary layer tail.
	\subsection{Notation and main assumptions}

	\noindent In this paper we use the small case letter $d$ to denote the dimension of the domain and we shall consider $d\geq3$ unless stated otherwise.	For $n\in\mathbb{S}^{d-1}$ and $s\in \mathbb{R}$, we use  $\mathbb{H}_n(s)$ to denote the set $\{ y\in\mathbb{R}^d: y\cdot n 
	-s>0 \}$. Further, a disk in $\mathbb{H}_n(s)$ is denoted by $D_r(y)=B_r(y)\cap \mathbb{H}_n(s)$ where $B_r(y)$ is the ball in $\mathbb{R}^d$ with center $y$ and radius $r$ and $\Gamma_r(y):=B_r(y)\cap \partial\mathbb{H}_n(s)$.	
		We occasionally use  $Y$ to denote the fundamental  cell : $Y=\mathbb{T}^d\simeq(0,1]^d.$\\
	{\bf The case $s=0$} : the half-space  $\mathbb{H}_n(0)$ is simply written  $\mathbb{H}_n$ and in this case the solutions $(u^s, p^s)$ to (\ref{hsStokes}) are simply denoted by $(u, p)$. Our analysis of problem (\ref{hsStokes}) will mostly  be focused on the case $s=0$ and occasionally when referring to problem (\ref{hsStokes}) we may implicitely consider the case $s=0$.  	The vector $e_j$ is the $j^{th}$ element of the canonical basis of $\mathbb{R}^d$.
	\noindent We frequently use $M$ to denote a rotation sending $e_d=(0, ...,0,1)$ to $n$	and we shall employ $N$ to denote the matrix made of the $d-1$ first columns of the matrix $M$.
	The oscillation of a function $f$ over a domain $D\subset \mathbb{R}^d$ will be denoted 
	$\underset{D}{\textup{osc}} f$ and it is defined by $$\osc_D f := \sup_{w,z \in D} |f(w) -f(z)|.$$
	
	\noindent For $A(y)=(A^{\alpha\beta}_{ij}(y)),$ we set $\mathcal{L}_\varepsilon := -\nabla\cdot A (x/\varepsilon) \nabla\cdot  $  to be a divergence form differential operator that  reads
	$$( \mathcal{L}_\varepsilon )^i=-\partial_{x_{\alpha}}\left( A^{\alpha\beta}_{ij}(x/\varepsilon) \partial_{x_{\beta}}  \right).$$
	We define the matrix $A^*$ to be the adjoint matrix of $A$ i.e. $(A^*)^{\alpha\beta}_{ij}=A^{\beta\alpha}_{ji}$ for all $\alpha,\beta\in\{ 1,..., d\} $ and for all $i,j\in\{ 1,..., d\} $ and the adjoint operator is  $\mathcal{L}^*_\varepsilon := -\nabla\cdot A^* (x/\varepsilon) \nabla $. We shall frequently use the H\"older semi-norm $[A]_{C^{0,\eta}(Y)}$, $\eta>0$,	of $A$ which we write  $[A]_{C^{0,\eta}}$ in short.
	The term heterogeneous refers to quantities $u^\varepsilon, p^\varepsilon, G^\varepsilon, P^\varepsilon ...$ that depend on $\varepsilon$.\\

	\noindent We assume that the  matrix of coefficients is real and satisfies the following conditions:
	\begin{enumerate}
		\item ellipticity: there exists $\mu>0,$ such that 
		\begin{align}
			\mu |w|^2 \leq  A^{\alpha\beta}_{ij} (y)w^{\alpha}_{i} w^{\beta}_{j} \leq  \frac{1}{\mu }|w|^2 \quad \mbox{  for all } w\in \mathbb{R}^{d\times d},  \mbox{  and for all } y\in \mathbb{R}^d;\label{ellipticity}
		\end{align}
		\item  periodicity:
		\begin{align}
			A (y+h)=A(y)\quad \mbox{  for all } y\in \mathbb{R}^{d} \mbox{  and for all } h\in \mathbb{Z}^d;\label{periodicity}
		\end{align}
		\item regularity:
		\begin{equation}
			\mbox{ A belongs to }C^\infty (\mathbb{R}^{d} ). \label{regular_A}\\
		\end{equation}
	\end{enumerate}
	We also assume that the boundary data $g$ satisfies
	\begin{enumerate}
		\item  periodicity : the boundary data is assumed to be periodic, i.e.
		\begin{align}
			g (y+h)=g(y)\quad \mbox{  for all } y\in \mathbb{R}^{d} \mbox{  and for all } h\in \mathbb{Z}^d.
		\end{align}
		\item regularity : 
		\begin{equation}
			\mbox{the boundary data g  belongs to }C^\infty (\mathbb{R}^{d} ). \label{regular_g}
		\end{equation}
		
	\end{enumerate}
	
	\section{Preliminaries}\label{section_prepa}

\noindent	This section is, for the most part, dedicated to introducing some  preparatory  elements and tools that will be used repeatedly. These preliminaries include among others the construction of an Ansatz for the Stokes equations. It is also the occasion to present some useful estimates. Let us start with  a result that gives a control of the pressure field by the derivative of the velocity.

\begin{lemma} [\cite{choi}, Lemma 3.4]
	Let $r>0$. For a pair $(u^s,p^s)$ satisfying the two first equations of system \textup{(\ref{hsStokes})} we have \begin{align}
		\|p^s-(p^s)_{D_r(0)} \|_{L^2(D_r(0))}\leq C \|\nabla u^s \|_{L^2(D_r(0))},\label{estim_Bogov}
	\end{align}
		where $C$ depends $d$ and $\mu$.
\end{lemma}
\noindent This Lemma has a generalization in the following sense.
\begin{lemma} \label{Bogov_p}
	Let $r>0$. For $q\in(1,\infty)$ and for a pair $(u^s,p^s)$ satisfying the first two equations of system \textup{(\ref{hsStokes})} we have \begin{equation}
		\|p^s-(p^s)_{D_r(0)} \|_{L^q(D_r(0))}\leq C \|\nabla u^s \|_{L^q(D_r(0))}, \label{estim_Bogov_p}\nonumber
	\end{equation}
	where $C$ depends $d$ and $\mu$.
\end{lemma}
\noindent A proof of Lemma \ref{Bogov_p} is given in Appendix \ref{append_estim}.\\

	\noindent We will repeatedly use asymptotic expansions of heterogeneous quantities. These are multi-scale expansions that are constructed with the aid of the following correctors
	\subsection{Interior correctors}\label{ansatz}
	\subsubsection{First-order correctors.}
	Let $\chi^{\beta}_{j}=(\chi^{\beta}_{kj})_k$ and   $\pi^{\beta}_{j}$ be 1-periodic functions that satisfy the following system of equations
	\begin{equation}\label{chipi}
		\left\{
		\begin{array}{rcl}
			-\nabla\cdot A(y) \: \nabla (\chi^{\beta}_{j}+ P^{\beta}_{j} ) +  \nabla \pi^{\beta}_{j} &=& 0  \;  \quad \mbox{ in } Y,\\
			\nabla\cdot \chi^{\beta}_{j} &=& 0 \;  \quad   \mbox{ in } Y,\\
			\int_{\mathbb{T}^d} \chi^{\beta}_{j}(y) dy&=& 0.
		\end{array}
		\right. 
	\end{equation}
	where $ P^{\beta}_{j}$ is the function defined by $ P^{\beta}_{j}(y)=y_j e_\beta$. These are the first-order interior correctors for the Stokes system. \\
	\subsubsection{Second-order correctors}
	Let  $\Gamma^{\alpha\beta}_{j}=(\Gamma^{\alpha\beta}_{kj})_k$ and   $Q^{\alpha\beta}_{j}$ be 1-periodic functions satisfying 
	\begin{equation}\label{GamQ}
		\left\{
		\begin{array}{rcll}
			-\partial_{y_{\gamma}}A^{\gamma\delta}_{ik}(y) \partial_{ y_{\delta}}\Gamma^{\alpha\beta}_{kj}	+\partial_{y_i}Q^{\alpha\beta}_{j}
			&=& \, B^{\alpha\beta}_{ij}(y)\!-\!\overline{B^{\alpha\beta}_{ij}}  - \delta_{i\alpha} \, \pi^{\beta}_{j}(y)& \mbox{ in } Y,\\
			\nabla_y \cdot\Gamma^{\alpha\beta}_{j}(y)	&=&-  \chi^{\beta}_{\alpha j}(y) &  \mbox{ in } Y, \\
			\int_{\mathbb{T}^d} \Gamma^{\alpha\beta}_{j}(y) dy&=& 0,
		\end{array}
		\right. \\
	\end{equation}
	where
	\begin{align*}
		B^{\alpha\beta}_{ij}(y)&:= A^{\alpha\beta}_{ij}(y)\!+\!\partial_{y_{\gamma}}[A^{\gamma\alpha}_{ik}(y) \chi^{\beta}_{kj}(y)]\!+\! A^{\alpha\gamma}_{ik}(y) \partial_{y_{\gamma}} \chi^{\beta}_{kj}(y)\\
		\mbox{and } \delta_{i\alpha} & \mbox{  is the Kronecker delta}.
	\end{align*}

	\subsubsection{Estimates on the correctors} Many quantities that we are going to study are built from the above correctors. We shall particularly need to control these quantities and for this reason we would like the correctors to be bounded. 
	\begin{lemma}\label{bound_on_correctors}
		Seeing that the coefficients $A$ satisfy assumption \textup{(\ref{regular_A})} we have the following estimates on the above  correctors :
		\begin{align*}
		\|\chi \|_{L^\infty(Y)}&\leq C,\quad   \|\pi \|_{L^\infty(Y)}\leq C, \\
			\| \Gamma \|_{L^\infty(Y)}&\leq C,\quad\| Q\|_{L^\infty(Y)}\leq C.
		\end{align*}
		The constant $C$ in these estimates depends on $[A]_{C^{0,\eta}}$. Furthermore, the derivatives of these functions are similarly bounded:
			\begin{align*}
			\| \partial^\lambda_y \,\chi \|_{L^\infty(Y)}&\leq C,\quad   \|  \partial^\lambda_y \,\pi \|_{L^\infty(Y)}\leq C, \\
			\|\partial^\lambda_y \, \Gamma \|_{L^\infty(Y)}&\leq C,\quad\|\partial^\lambda_y \, Q\|_{L^\infty(Y)}\leq C,
		\end{align*}
		for any multi-index $\lambda\in\mathbb{N}^d$.
		
	\end{lemma}

	\subsection{Asymptotic expansion}
	With these  correctors one builds an interior multi-scale expansion of the solution $(u^\varepsilon, p^\varepsilon)$ to problem (\ref{basic_Stokes}) :
	\begin{align}
		u^\varepsilon(x)&=u^0(x)+ \varepsilon	u^1(x,x/\varepsilon)  + \varepsilon^2 u^2(x,x/\varepsilon)   +...\label{expansu}\\
		p^\varepsilon(x)&=  p^0(x,x/\varepsilon) +\varepsilon p^1(x,x/\varepsilon)  + ...\label{expansp}
	\end{align} 
	where $(u^0,p^0)$ satisfies the homogenized problem associated to (\ref{basic_Stokes})
		\begin{equation*}
		\left\{
		\begin{array}{rcll}
			-\nabla\cdot A^0\nabla u^0 +  \nabla p^0&=& f& \mbox{ in } \{ x\cdot n-s>0 \},\\
			\nabla\cdot u^0&=& h,&\\
			u^0&=& g
			&\mbox{ on } \{ x\cdot n-s=0 \}
		\end{array}
		\right. 
	\end{equation*}
	and where
		\begin{align*}
	(A^0)_{ij}^{\alpha\beta}:= \frac{1}{|Y|}\int_Y [A^{\alpha\beta}_{ij}(y) + A^{\alpha\gamma}_{ik}(y)\partial_{y_{\gamma}} \chi^{\beta}_{kj}(y)] dy,
	\end{align*}
and
	\begin{align}
		u^1(x,x/\varepsilon) =	\chi^{\beta}_{kj}(x/\varepsilon) \partial _{x_{\beta}} u^0_j  (x)	+ \tilde{u}_k^1(x),\quad u^2(x,x/\varepsilon)=& \Gamma^{\alpha\beta}_{kj}(x/\varepsilon)\partial_{x_{\alpha}}\partial_{x_{\beta}} u^0_j  \\
		p^0(x,x/\varepsilon)= \tilde{p}^0(x)+ \pi^{\beta}_{j}(x/\varepsilon) \partial_{x_{\beta}} u^0_j(x),\quad p^1(x,x/\varepsilon) =& Q^{\alpha\beta}_{j}(x/\varepsilon)\partial_{x_{\alpha}}\partial_{x_{\beta}} u^0_j 
	\end{align} 
	with 
	$ \tilde{u}_k^1$ and $ \tilde{p}^0$ being functions independent of $y$ 
and  $(\chi^{\beta}_{j},\pi^{\beta}_{j})$, $(\Gamma^{\alpha\beta}_{j}, Q^{\alpha\beta}_{j})$ are the correctors given in (\ref{chipi}) and (\ref{GamQ}) respectively.

	\begin{remark}\label{Rmk_starcorrectors} We occasionally use the homogenized equation written into coordinates
	\begin{align}
			- (A^0)_{ij}^{\alpha\beta} \partial_{x_{\alpha}}\partial_{x_{\beta}}u_j^0 (x)  +   \partial_{x_i} p^0(x)	=& f^i(x).\label{eq_homgn}
		\end{align}
	We moreover define the correctors $(\chi^*, \pi^*)$ and $(\Gamma^*, Q^*)$ associated with the adjoint operator $\mathcal{L}^*_1$ in a similar way. These correctors have the above properties of  $(\chi, \pi)$ and $(\Gamma, Q)$ respectively. In particular, they satisfy the estimates in \textup{Lemma \ref{bound_on_correctors}}.
	
	\end{remark}

	\subsection{Some uniform regularity estimates}
	
	\noindent In our study, we take advantage of the following results to derive estimates uniform in $\varepsilon$ for the heterogeneous Green function.\\
	
	\noindent First, we present a result of a uniform H\"older regularity estimate.
	\begin{lemma}\label{Holder_estim}
		Let $l>0$. Suppose the pair $(u^\varepsilon, p^\varepsilon)$ satisfies 
		\begin{equation}
			\left\{
			\begin{array}{rcll}
				-\nabla\cdot A(x/\varepsilon) \: \nabla u^\varepsilon  +  \nabla p^\varepsilon &=&  f \quad& \mbox{in } D_1(0),\\
				\nabla\cdot u^\varepsilon&=& 0   \quad & \mbox{in } D_1(0),\\
				u^\varepsilon &=& g\quad & \mbox{on } \Gamma_1(0),
			\end{array}
			\right. 
		\end{equation}
		with \begin{align*}
			\| u^\varepsilon \|_{L^2(D_{1}(0))}< \infty ,\qquad     \| f\|_{L^{d+l}(D_{1}(0))}< \infty ,\qquad     [g]_{C^{0,1}(\Gamma _{1}(0))}< \infty .
		\end{align*}
		Then there exists a constant $C>0$ depending only on $d, \mu, \|A\|_{C^{0,\eta}}$ and $l,$ such that for $\lambda=1- d/(d+l)$, 
		\begin{align*}
			[u^\varepsilon]_{C^{0,\lambda}(D_{1/2}(0))} \leq C \left(  \| u^\varepsilon \|_{L^2(D_{1}(0))}+    \| f\|_{L^{d+l}(D_{1}(0))}+  [g]_{C^{0,1}(\Gamma _{1}(0))}\right) .
		\end{align*}
	\end{lemma}
	
	\noindent Besides the H\"older estimate there is a uniform Lipschitz estimate given by the following Theorem.
	\begin{theorem}\label{unif_estim} \textup{  (\cite{GZ}, Theorem 3.1.)}.
		Suppose that the matrix $A$ of the coefficients satisfies the ellipticity condition \textup{(\ref{ellipticity})} and periodicity \textup{(\ref{periodicity})} and $A\in C^{0,\eta}(D(0,1))$  for some $\eta\in(0,1]$. For $\varepsilon>0,$ let $(u^\varepsilon,  p^\varepsilon)$ be a weak solution of 
		\begin{equation}\label{Stokes_system_D1}
			\left\{
			\begin{array}{rcll}
				-\nabla\cdot A (x/\varepsilon) \: \nabla u^\varepsilon +  \nabla p^\varepsilon&=& f+ \nabla\cdot F, & x\in D(0,1),\\
				\nabla\cdot u^\varepsilon &=& h,&   x\in D(0,1),\\
				u^\varepsilon &=& g , & 
				x\in  \Gamma(0,1),
			\end{array}
			\right. 
		\end{equation}
		where $f\in L^{d+l}(D(0,1)), $ for some $l>0$, $F\in C^{0,\eta}(D(0,1)), \; h\in C^{0,\eta}(D(0,1))$ and $ g\in C^{1,\eta}(\Gamma(0,1) )$, for some $\eta \in (0,1)$. Then there exists a constant $C>0$ such that 
		\begin{align}\label{unif_estim_thGZ}
			\|\nabla u^\varepsilon \|_{L^\infty(D(0,1/2))} + \
			\underset{D(0,1/2)}{\textup{osc}} [p^\varepsilon ]\leq   C \left\{ \| u^\varepsilon \|_{L^\infty(D(0,1))} \right. &+ \| f \|_{L^{d+l}(D(0,1))} +  [F]_{C^{0,\eta}(D(0,1))}\\
			&+     \| h \|_{L^\infty(D(0,1))} +  [h]_{C^{0,\eta}(D(0,1))}  \nonumber\\
			+\| g \|_{L^\infty(\Gamma(0,1))} &+ \| \nabla_{\!\!\textup{ tan}} \, g \|_{L^\infty(\Gamma(0,1))} + \left.  [\nabla_{\!\!\textup{ tan}}\; g]_{C^{0,\eta}(\Gamma(0,1))} \right\} \nonumber
		\end{align}
		where $C$ depends on $d, l, \eta, \mu$ and $[A]_{C^{0,\eta}(D(0,1))}$  but $C$ is uniform in $\varepsilon $.
	\end{theorem}
	\noindent \begin{remark}
	The original estimate in \textup{\cite{GZ}, Theorem 3.1,} is given in a rescaled version. Using the uniformity of $C$ in $\varepsilon$ one can  go from this version to the rescaled one and conversely.	
		By the inequality
	\begin{align} \left(  \dashint_{D(0,1)} | u^\varepsilon |^2 \right)^{1/2} \leq   \| u^\varepsilon \|_{L^\infty(D(0,1))}  \nonumber
	\end{align}
	 the original estimate implies estimate (\ref{unif_estim_thGZ}) because in the original estimate we have the term $\left(  \dashint_{D(0,1)} | u^\varepsilon |^2 \right)^{1/2}$ in the place of $ \| u^\varepsilon \|_{L^\infty(D(0,1))}  $ in estimate (\ref{unif_estim_thGZ}).\\
		\end{remark}

	\subsubsection{Green's function and Poisson's kernel for the Stokes system}
	
	\noindent 	Our strategy is based on a representation formula via the Poisson kernel. 
Because of this representation the study of the behaviour of $u$ far away from the boundary  is transferred to that of Poisson's kernel. Yet, the Poisson kernel is constructed from Green's functions. Therefore these functions are going to be frequently used in our study.\\
	
	\noindent  Green's functions for Stokes system (\ref{hsStokes}) in the half-space ${\mathbb{H}}_n$ can be regarded as a family $(G,\Pi)=(G_j,\Pi_j)_{j=1,...,d}$ where for each $j=1,...,d,$ and for $y\neq \tilde{y}$, $G_j(y,\tilde{y})\in\mathbb{R}^d $ is a vector and $\Pi_j(y,\tilde{y})\in\mathbb{R}$ is a scalar, satisfying 
	\begin{equation*}
		\left\{
		\begin{array}{rcll}
			-\nabla_y\cdot A(y) \: \nabla_y G_j(\cdot ,\tilde{y}) +  \nabla_y \Pi_j(\cdot ,\tilde{y})&=& \delta(\cdot -\tilde{y}) e_j   &\mbox{in }{\mathbb{H}}_n ,\\
			\nabla_y\cdot G_j(\cdot,\tilde{y})&=& 0 , &\\
			G_j(\cdot,\tilde{y}) &=& 0   & \mbox{on } \partial{\mathbb{H}}_n\\
		\end{array}
		\right. 
	\end{equation*}  	
	where $\delta$ is the Dirac distribution.

	\noindent 	The $i^{th}$ component of vector $G_j(y,\tilde{y})$ is denoted 	 $G_{ij}(y,\tilde{y})$.\\

	\noindent Occasionally we write the above system of equations in the following more condensed form
	\begin{equation}\label{G}
		\left\{
		\begin{array}{rcll}
			-\nabla_y\cdot A(y) \: \nabla_y G(\cdot, \tilde{y}) +  \nabla_y \Pi(\cdot,\tilde{y})&=& \delta(\cdot - \tilde{y}) I & \mbox{in } {\mathbb{H}}_n,\\
			\nabla_y\cdot G(\cdot,\tilde{y})&=& 0,    &\\
			G(\cdot ,\tilde{y}) &=& 0  \quad & \mbox{on } \partial{\mathbb{H}}_n.
		\end{array}
		\right. 
	\end{equation}

	\noindent  Green's functions  $(G^*,\Pi^*)=(G_j^*,\Pi_j^*)_{j=1,...,d}$ associated to the adjoint operator and to the half-space $\mathbb{H}_n$ are defined to be the solution to problem 
	\begin{equation}\label{GPj}
		\left\{
		\begin{array}{rcll}
			-\nabla_{\tilde{y}} \cdot A^*(\tilde{y}) \: \nabla_{\tilde{y}} G^*(\cdot,y) +  \nabla_{\tilde{y}} \Pi^*(\cdot,y )&=& \delta(\cdot -y) I  &\mbox{in }\mathbb{H}_n ,\\
			\nabla_{\tilde{y}}\cdot G^*(\cdot,y)&=& 0  ,  &\\
			G^*(\cdot ,y) &=& 0  \quad & \mbox{on } \partial\mathbb{H}_n.\\
		\end{array}
		\right. 
	\end{equation} 
	
	\noindent The Green function in (\ref{G}) and its adjoint counterpart are also related by the symmetry property 
	\begin{equation}\label{sym}
		G^T(y,\tilde{y})\\
		=    \,  G^*(\tilde{y},y) \quad \mbox{ i.e. } \,  G_{ij}(y,\tilde{y})=	G_{ji}^*(\tilde{y},y).
	\end{equation}
		Further, for $\varepsilon>0$ we shall use  Green's functions $(G^{*,\varepsilon},\,\Pi^{*,\varepsilon})$ associated to the coefficient $A^*(\cdot/\varepsilon)$ in the domain $\mathbb{H}_n$
	\begin{equation}\label{system_Greeneps}
		\left\{	\begin{array}{rcll}
			-\nabla_{\tilde{x}}\cdot A^* (\tilde{x}/\varepsilon) \, \nabla_{\tilde{x}} G^{*,\varepsilon}(\cdot,x) +  \nabla_{\tilde{x}} \Pi^{*,\varepsilon}(\cdot,x)&=&  \delta(\cdot -x) I  & \mbox{in }\mathbb{H}_n,\\
			\nabla_{\tilde{x}} \cdot  G^{*,\varepsilon}(\cdot,x)&=& 0 , \\
			 G^{*,\varepsilon}(\cdot,x)&=& 0  \ & \mbox{on } \partial\mathbb{H}_n.\\
		\end{array}
		\right. 
	\end{equation} 
	
	\noindent We shall also use Green's functions associated to the homogenized operated $\mathcal{L}_0:= -\nabla\cdot A^0 \: \nabla $ in the half-space $\mathbb{H}_n$, precisely, the pair  $(G^0,\Pi^0)=(G^0_j,\Pi_j)$ solution to problem 
	\begin{equation}\label{GP0}
		\left\{
		\begin{array}{rcll}
			-\nabla\cdot A^0 \: \nabla G^0(\cdot,\tilde{y}) +  \nabla \Pi^0(\cdot,\tilde{y})&=& \delta(\cdot -\tilde{y}) I   &\mbox{in }\mathbb{H}_n,\\
			\nabla\cdot G^0(\cdot,\tilde{y})&=& 0  ,    &\\
			G^0(\cdot,\tilde{y}) &=& 0 & \mbox{on } \partial\mathbb{H}_n.\\
		\end{array}
		\right. 
	\end{equation}

	\noindent  The existence 	and uniqueness of Green's functions are established in \cite{choi}, Theorem 2.7,  under the Assumption 2.5 therein that requires  the weak solutions $(u,p)$ of (\ref{hsStokes}) to be locally H\"older continuous. Since  $A$ is assumed to be smooth the local H\"older continuity is automatically satisfied by $(u,p)$. 
	Moreover we have the following pointwise estimates on these Green functions $(G,\Pi)$, solution to problem (\ref{G}) that shall be frequently used in our work.
	
	\begin{proposition}\label{Propo_estim_Green_y}
		There exists a constant $C>0$ depending on $d,\mu, [A]_{C^{0,\eta}}$  such that for all~ $y,\tilde{y} \in~\mathbb{H}_n$ we have : 
		\begin{itemize}
			\item 	estimates on the velocity component $G(y,\tilde{y})$
			\begin{equation}\label{estim_G_y}
				| G(y,\tilde{y})| \leq C \min \left\{  \frac{1 }{|y-\tilde{y}|^{d-2}};	\frac{y\cdot n}{|y-\tilde{y}|^{d-1}}; 
				\frac{\tilde{y}\cdot n }{|y-\tilde{y}|^{d-1}}; \frac{(y\cdot n)(\tilde{y}\cdot n) }{|y-\tilde{y}|^d}  \right\}\!,
			\end{equation}
		\item 	estimates  on the first-order derivatives 
			\begin{equation}\label{estim_Dy_G}
				|\nabla_y G(y,\tilde{y})| \leq  C \min \left\{  \frac{1 }{|y-\tilde{y}|^{d-1}}; 
				\frac{\tilde{y}\cdot n }{|y-\tilde{y}|^{d}}  \right\}
			\end{equation}		
			and
			\begin{equation}\label{estim_dertildG}
				|\nabla_{\tilde{y}} G(y,\tilde{y})| \leq  C \min \left\{  \frac{1 }{|y-\tilde{y}|^{d-1}}; 
				\frac{y\cdot n }{|y-\tilde{y}|^{d}}\right\}\!,
			\end{equation}
		\item 	an estimate on the mixed derivative
			\begin{equation}\label{estim_mixder}
				| \nabla_y\!\nabla_{\tilde{y}} G(y,\tilde{y})| \leq  \; \frac{C }{|y-\tilde{y}|^d} ,
			\end{equation}
		\item	estimates  on the pressure component $\Pi(y,\tilde{y}) $ :   for all $z\in \mathbb{H}_n$ satisfying $|y-\tilde{y}|\leq |z-\tilde{y}|$ we have
			\begin{equation}\label{estim_diff_Pi_y}
				| \Pi(y,\tilde{y}) -  \Pi(z,\tilde{y})| \leq C \min \;\left\{\frac{1 }{|y-\tilde{y}|^{d-1}}  ,\frac{\tilde{y}\cdot n }{|y-\tilde{y}|^{d}} \right\} 
			\end{equation}
			and 
			\begin{equation}\label{estimPi_y}
				| \Pi(y,\tilde{y})| \leq C \min \;\left\{\frac{1 }{|y-\tilde{y}|^{d-1}}  ,\frac{\tilde{y}\cdot n }{|y-\tilde{y}|^{d}}   \right\} ,
			\end{equation}		
		\item 	an estimate on the derivative $\nabla_{\tilde{y}} \Pi$  :   for all $z\in \mathbb{H}_n$ satisfying $|y-\tilde{y}|\leq |z-\tilde{y}|$ we have
			\begin{equation}\label{estim_diff_derPi_y}
				|\nabla_{\tilde{y}} \Pi(y,\tilde{y}) - \nabla_{\tilde{y}} \Pi(z,\tilde{y})|\leq \;C \min \;\left\{\frac{ 1}{|y-\tilde{y}|^{d}}; \frac{\tilde{y}\cdot n }{|y-\tilde{y}|^{d+1}}   \right\} 
			\end{equation}	
			and 	\begin{equation}\label{estimderPi_y}
				|\nabla_{\tilde{y}} \Pi(y,\tilde{y})|\leq \;\frac{C }{|y-\tilde{y}|^{d}}  .\\
			\end{equation}				
			
		\end{itemize}

	\end{proposition}	
	
\noindent	A proof of  Proposition \ref{Propo_estim_Green_y} is given in Appendix \ref{append_Gfunction}.

	\begin{remark}
	The Green functions  $(G^*, \Pi^*)$ associated to the adjoint operator $A^*$ and defined in (\ref{GPj}) satisfy the estimates in Proposition \ref{Propo_estim_Green_y} with the interchanging of $y$ and $\tilde{y}$. For instance, we have a "star-version" of estimate (\ref{estim_mixder})
		\begin{equation}\label{estim_mixderG*}
		|\nabla_{\tilde{y}} \nabla_y G^*(\tilde{y},y)| \leq  \; \frac{C }{|y-\tilde{y}|^d} ,
	\end{equation}
and estimate	(\ref{estimderPi_y})
		\begin{equation}\label{estimderPi*_y}
		|\nabla_{y} \Pi^*(\tilde{y},y)|\leq \;\frac{C }{|y-\tilde{y}|^{d}}  .
	\end{equation}
		The homogenized Green functions $(G^0, \Pi^0)$ satisfy  estimates \textup{(\ref{estim_G_y})} through \textup{(\ref{estimderPi_y})}. Moreover, we have, for all $\lambda\in \mathbb{N}^d$, 
		\begin{equation}\label{estim_der_G0}
			| \partial^\lambda_y  G^0(y,\tilde{y})| \leq  \; \frac{C }{|y-\tilde{y}|^{d-2+|\lambda|}} .
		\end{equation}
		
	\end{remark}

	\begin{lemma}\label{rescaled_GP}
		The funtions $(G,\Pi)$ and $(G^{\varepsilon},\Pi^{\varepsilon} )$ given respectively by problems \textup{(\ref{G})} and \textup{(\ref{system_Greeneps})} are related as follows
		\begin{align*}
			G^{\varepsilon} (x,\tilde{x})&= \frac{1}{\varepsilon^{d-2}} G\left( \frac{x}{\varepsilon}, \frac{\tilde{x}}{\varepsilon}\right)\quad \mbox{and }\quad 	\Pi^{\varepsilon} (x,\tilde{x})=  \frac{1}{\varepsilon^{d-1}} \Pi\left( \frac{x}{\varepsilon}, \frac{\tilde{x}}{\varepsilon}\right) \mbox{ for all } x,\tilde{x}\in\mathbb{H}_n .
		\end{align*}
	The functions $(G^0,\Pi^0)$ given by \textup{(\ref{GP0})} satisfy, for all $x,\tilde{x}\in\mathbb{H}_n$,
		\begin{align*}
			G^{0}(x,\tilde{x})&=  \frac{1}{\varepsilon^{d-2}}G^{0}\left( \frac{x}{\varepsilon}, \frac{\tilde{x}}{\varepsilon}\right),\quad
			\partial^\lambda_{\!x} G^{0}(x,\tilde{x})=  \frac{1}{\varepsilon^{d-2+|\lambda|}} \, \partial^\lambda_{\!x} G^{0}\left( \frac{x}{\varepsilon}, \frac{\tilde{x}}{\varepsilon}\right),\quad 
			\Pi^{0}(x,\tilde{x})= \frac{1}{\varepsilon^{d-1}} \Pi^{0}\left( \frac{x}{\varepsilon}, \frac{\tilde{x}}{\varepsilon}\right).
		\end{align*}
		
	\end{lemma}
	
	\begin{remark} The above estimates remain valid when the Green functions  $(G,\Pi)$ resp. $(G^{\varepsilon},\Pi^{\varepsilon})$ are replaced with their adjoint counterparts $(G^*,\Pi^*)$ resp. $(G^{*,\varepsilon},\Pi^{*,\varepsilon})$.
	\end{remark}
	
	\subsubsection{Representation of the solution to the system \textup{(\ref{hsStokes})}}
	
	\noindent	When a solution $u$ to problem (\ref{hsStokes}) satisfies \textup{(\ref{class})} then it is unique (proved in  Appendix \ref{uniqueness}) and it  has the following integral representation
	\begin{align*}
		u(y)= \int_{\partial\mathbb{H}_n} \left[-A^*(\tilde{y}) \nabla_{\tilde{y}} G^*(\tilde{y},y)\cdot n\,g (\tilde{y})+ \Pi^*(\tilde{y},y)g(\tilde{y}) \cdot n\right] d\tilde{y}
	\end{align*}
	that can be written into coordinates as follows
	\begin{align}\label{represintgral}
		u_i(y)&= \int_{\partial\mathbb{H}_n} P_{ij}(\tilde{y},y)g_j(\tilde{y})  d\tilde{y}\\\nonumber\\
		\mbox{ where }\quad 	P_{ij}(y,\tilde{y})
		&=-(A^*)_{jk}^{\alpha\beta}(\tilde{y})\, \partial_{\tilde{y}_\beta}G_{ki}^*(\tilde{y},y) n_\alpha (\tilde{y})+  \Pi_i^*(\tilde{y},y) n_j(\tilde{y})\\
		&=  -n_\alpha A^{\beta\alpha}_{kj}(\tilde{y})\partial_{\tilde{y}_\beta} G_{ik} (y,\tilde{y}) +\Pi^*_i (\tilde{y},y) n_j	\nonumber.
	\end{align}
	This function $	P(y,\tilde{y})$ is called  Poisson's kernel. In a condensed form we have
	\begin{align*}
		P(y,\tilde{y})
		&=-A^*(\tilde{y}) \nabla_{\tilde{y}} G^*(\tilde{y},y)\cdot n (\tilde{y})+ \Pi^*(\tilde{y},y)\otimes n.
	\end{align*}

		\section{Homogenization in the half-space}\label{section_hom}
	\noindent The aim of this section is to establish a  homogenization   result  which is going to be applied later in Section \ref{section_hom_G}. More precisely, we get an estimate of the difference $u^\varepsilon -u^0$ between a heterogenous quantity and its associate homogenized counterpart. In the present context, the quantities $u^\varepsilon, u^0$ represent velocities respectively associated to some Stokes problems. More precisely, they are defined as follows.
	Let $f$ be a vector valued-function defined in $\overline{\mathbb{H}}_n$ and consider the following system of equations
	\begin{equation}\label{hpstoks}
		\left\{
		\begin{array}{rcll}
			-\nabla\cdot A (x/\varepsilon)\nabla u^\varepsilon +  \nabla p^\varepsilon&=& f  \quad &\mbox{in } {\mathbb{H}}_n,\\
			\nabla\cdot u^\varepsilon&=& 0 ,&\\
			u^\varepsilon &=& 0\quad  &\mbox{on } \partial{\mathbb{H}}_n.
		\end{array}
		\right. 
	\end{equation}
	
	\noindent The solution  $(u^\varepsilon$, $p^\varepsilon)$  can be formally expanded as in (\ref{expansu}), (\ref{expansp}). We do so with the addition of  boundary layer correctors so as to take into account the effects of the boundary :
	\begin{align}
		u_k^\varepsilon(x)&=u_k^0(x)+ \varepsilon	\chi^{\beta}_{kj}(x/\varepsilon)\partial_{x_{\beta}}u_j^0 (x)	 + \varepsilon u^\varepsilon_{bl,k}  (x) +...\label{expansubl}\\
		p^\varepsilon(x)&=  p^0(x)+ \pi^{\beta}_{j}(x/\varepsilon)\partial_{x_{\beta}}u^0_j(x)+\varepsilon  p^\varepsilon_{bl}  (x) + ...\label{expanspbl}
	\end{align}
	where $(u^0,p^0)$ is the solution of the homogenized system of equations
	\begin{equation}\label{hphom0bv}
		\left\{
		\begin{array}{rcll}
			-\nabla\cdot A^0 \: \nabla u^0 +  \nabla p^0&=& f \quad &\mbox{in } {\mathbb{H}}_n,\\
			\nabla\cdot u^0&=& 0 ,\\
			u^0 &=& 0\quad  &\mbox{on } \partial{\mathbb{H}}_n
		\end{array}
		\right. 
	\end{equation}
	and $(u^\varepsilon_{bl} , p^\varepsilon_{bl} )$ being the solution of the problem 
	\begin{equation}\label{eq_uepsbl}
		\left\{
		\begin{array}{rcll}
			-\partial_{x_\alpha}( A^{\alpha\beta}_{ik}(x/\varepsilon)  \partial_{x_\beta}u^\varepsilon_{bl,k})  + \partial_{x_i} p^\varepsilon_{bl} &=& 0  &\mbox{in } {\mathbb{H}}_n,\\
			\partial_{x_k} u^\varepsilon_{bl,k} &=& 0 , &\\
			u^\varepsilon_{bl,k} (x) &=& -\chi^{\beta}_{kj}(x/\varepsilon)\partial_{ x_{\beta}}  u^j_0   (x)\quad  &\mbox{on } \partial{\mathbb{H}}_n.
		\end{array}
		\right. 
	\end{equation}
	Now we intend to estimate the error terms in the approximations (\ref{expansubl}) and (\ref{expanspbl}). To that end, we introduce $r^\varepsilon=(r^\varepsilon_k)_k$ and $q^\varepsilon$ defined as follows
	\begin{align}\label{errors_bl}
		r_k^\varepsilon(x)&:= u_k^\varepsilon(x) -u_k^0(x) - \varepsilon\chi^{\beta}_{kj}(x/\varepsilon)\partial_{ x_{\beta}}  u^0_j   (x) - \varepsilon u^\varepsilon_{bl,k}(x) ,\\
		q^\varepsilon(x)&:=p^\varepsilon(x)- p^0(x)- \pi^{\beta}_{j}(x/\varepsilon)\partial_{ x_{\beta}}  u^0_j   (x)   - \varepsilon p^\varepsilon_{bl}(x).
	\end{align}
	\noindent Obtaining a bound on  $r^\varepsilon$ allows to get a control on the error $ u^\varepsilon -u^0.$\\
	
	\noindent Moreover we assume that the source term $f$ in (\ref{hpstoks}) and (\ref{hphom0bv}) belongs to $C_c^1 (\overline{\mathbb{H}}_n)$. This property of the source term will turn out to be relevant in an application later in Section  \ref{section_hom_G}. Indeed, the proof of the main result in that section is based on a duality argument.

	\begin{proposition}\label{propolinfestim}
		Let $\delta\in(0,d/2)$ and $f\in C^1_c(B(0,1/2))$.	There exists $C>0$ depending  on $d,\mu,[A]_{C^{0,\eta}}$ and $\delta$
		such that for all $\varepsilon>0$ we have the following estimates 
		\begin{align}
			\| r^\varepsilon \|_{L^\infty(\mathbb{H}_n)}&\leq C\varepsilon\| f\|_{W^{1,d/2+\delta }(\mathbb{H}_n)},\nonumber\\
			\| u^\varepsilon_{bl}\|_{L^\infty(\mathbb{H}_n)}&\leq C\| f\|_{W^{1,d/2+ \delta }(\mathbb{H}_n)},\nonumber\\
			\| u^\varepsilon - u^0 \|_{L^\infty(\mathbb{H}_n)}&\leq C\varepsilon\| f \|_{W^{1,d/2+\delta }(\mathbb{H}_n)}. \label{estimL8uepsu0}  
		\end{align}
	\end{proposition}

	\begin{remark}
		The pair $(r^\varepsilon,q^\varepsilon)$ satisfies
		\begin{equation}\label{rqeps}
			\left\{
			\begin{array}{rcll}
				-\partial_{x_\alpha} ( A^{\alpha\beta}_{ik}(x/\varepsilon)\partial_{x_\beta}  r_k^\varepsilon) + \partial_{x_i} q^\varepsilon&=&  f_i^\varepsilon  , \quad i=1,...,d, &\mbox{in } \mathbb{H}_n,\,\\
				\partial_{x_k} r_k^\varepsilon&=& - \varepsilon \chi^{\beta}_{kj}(x/\varepsilon)\partial_{x_k}\partial_{ x_{\beta}}  u^j_0   (x)   &\mbox{in } {\mathbb{H}}_n,\\
				r^\varepsilon &=& 0\quad & \mbox{on } \partial{\mathbb{H}}_n.
			\end{array}
			\right. 
		\end{equation}
		where
		\begin{align}\label{feps}
			f_i^\varepsilon
			= f_i&	+\partial_{x_\alpha} \left( A^{\alpha\beta}_{ik}(x/\varepsilon)\partial_{x_\beta}	\! u_k^0   \right)- \partial_{x_i}  p_0 \nonumber\\
			& 	+\partial_{x_\alpha} \left( A^{\alpha\gamma}_{ik}(x/\varepsilon)\partial_{x_\gamma}	\! ( \varepsilon\chi^{\beta}_{kj}(x/\varepsilon)\partial_{ x_{\beta}}  u^j_0  ) \right)- \partial_{x_i} \left( \pi^{\beta}_{j}(x/\varepsilon)\partial_{ x_{\beta}}  u^j_0 \right).
		\end{align}
		Here we have used 	the first equations in  systems \textup{(\ref{eq_uepsbl})} and \textup{(\ref{hpstoks})}:
		\begin{align*}
			f_i 
			= & 	-\partial_{x_\alpha} ( A^{\alpha\beta}_{ik}(x/\varepsilon)\partial_{x_\beta}	u_k^\varepsilon ) + \partial_{x_i} p^\varepsilon \quad
			\mbox{  and }\quad
			0
			= 	 -\partial_{x_\alpha} (A^{\alpha\beta}_{ik}(x/\varepsilon)\partial_{x_\beta}	\!  u^\varepsilon_{bl,k})  +  \partial_{x_i} p^\varepsilon_{bl}.\nonumber\\
		\end{align*}		
		
	\end{remark}
	\noindent	The quantity $f^\varepsilon$, given in the expression (\ref{feps}), apparently contains terms of order zero in $\varepsilon$. Yet 
	we would like to factorize it by $\varepsilon$. To that end, we make some transformation of $f^\varepsilon$ given in (\ref{feps}).

	\begin{lemma}\label{fepsordereps} The source term $ f^\varepsilon$ in \textup{(\ref{rqeps})} can be factorized in the following way
		\begin{align}
			f_i^\varepsilon(x)
			=& \,\varepsilon \left[ \partial_{x_\gamma} \{ 	(h^\varepsilon_1)_i^\gamma (x) \}	+ \partial_{x_i} \{ 	h^\varepsilon_2(x) \}
			+  \partial_{x_\alpha} \{ 	(h^\varepsilon_3)^\alpha_i(x) \}+ \partial_{x_\alpha}\{	(h_4^\varepsilon)_i^\alpha(x) \} \right] \nonumber \\
			&-\varepsilon \left[ 	(g^\varepsilon_1)_i(x)+	(g^\varepsilon_2)_i(x) \right] \label{Lemm_feps}
		\end{align}
		where we have set
		\begin{align}
			(h^\varepsilon_1)_i^\gamma (x)	:=& \phi^{\gamma\alpha\beta}_{ij}(x/\varepsilon)\partial_{x_\alpha}\partial_{x_\beta} u_j^0 (x), \quad
			h^\varepsilon_2(x):= q^{\alpha\beta}_{j}(x/\varepsilon)\partial_{x_\alpha}\partial_{x_\beta} u_j^0(x), \nonumber\\
			(h^\varepsilon_3)^\alpha_i(x):=&q^{\alpha\beta}_{j}(x/\varepsilon)\partial_{x_i}\partial_{x_\beta} u_j^0(x),
			\label{def_h_i_eps}\\
			(h_4^\varepsilon)_i^\alpha(x):= & A^{\alpha\gamma}_{ik} (x/\varepsilon) \chi^{\beta}_{kj}(x/\varepsilon)\partial_{x_\gamma}	\partial_{ x_{\beta}}  u_j^0 (x)\nonumber
		\end{align}
	and
		\begin{align}
			(g^\varepsilon_1)_i(x):=\phi^{\gamma\alpha\beta}_{ij}(x/\varepsilon)\partial_{x_\alpha}\partial_{x_\beta} \partial_{x_\gamma} u_j^0(x), \qquad 
			(g^\varepsilon_2)_i(x) := 2 q^{\alpha\beta}_{j}(x/\varepsilon)\partial_{x_i} \partial_{x_\alpha}\partial_{x_\beta} u_j^0(x)\label{def_geps}
		\end{align}
	 with the functions $ \phi^{\gamma\alpha\beta}_{ij}, \, q^{\alpha\beta}_{j}$  being given by the following Lemma. 
	\end{lemma}

		\noindent	We now state the following classical lemma
	\begin{lemma}\label{divphi} \textup{ (\cite{GZ}, Lemma 2.2.)}
		Let us set 
		\begin{equation*}
			b^{\alpha\beta}_{ij}(y):= A^{\alpha\beta}_{ij} (y) +  A^{\alpha\gamma}_{ik} (y)\partial_{y_\gamma}\chi ^{\beta}_{kj}(y) -  (A^ 0)^{\alpha\beta}_{ij}.
		\end{equation*}
		Then 
		for fixed $\alpha,\beta,j \in \{1,...,d\}$  there exists $(\phi^{\gamma\alpha\beta}_{ij} , q^{\alpha\beta}_{j}) \in H^1 (Y)^d\times H^1(Y)$ such that 
		\begin{equation*}
			b^{\alpha\beta}_{ij}= \frac{\partial}{\partial y_\gamma}\phi^{\gamma\alpha\beta}_{ij} + \frac{\partial}{\partial y_i}q^{\alpha\beta}_{j},\qquad 	\pi^{\beta}_{j}=\frac{\partial}{\partial y_\alpha}q^{\alpha\beta}_{j}.
		\end{equation*}
		and
		\begin{align*}
			\| \phi^{\gamma\alpha\beta}_{ij} \|_{L^\infty(Y)} + \| q^{\alpha\beta}_{j} \|_{L^\infty(Y)} \leq C,
		\end{align*}
		where $C>0$ depends only on $d, \mu$ and $[A]_{C^{0,\eta}}$.
	\end{lemma}
	\begin{proof}[Proof  of \textup{Lemma \ref{fepsordereps}}]
		By the first equation  in (\ref{chipi} )  $f_i^\varepsilon$ can be rewritten
		\begin{align}
			f_i^\varepsilon=f_i^0 + f_i^1 
		\end{align}			
		with	\begin{align*}
			f_i^0:
			=&f_i- \partial_{x_i}  p^0 +\partial_{x_\alpha}\left( A^{\alpha\beta}_{ik}(y)\partial_{x_\beta}	\! u_k^0 \right)+  \partial_{x_\alpha} \left(A^{\alpha\gamma}_{ik}(y)\partial_{y_\gamma}	\chi^{\beta}_{kj}(y)\partial_{ x_{\beta}}  u^0_j \right)-  \pi^{\beta}_{j}(y)\partial_{x_i} \partial_{ x_{\beta}}  u^0_j
		\end{align*} 		
		and 
		\begin{align*}
			f_i^1:&=	\partial_{y_\alpha} \left(A^{\alpha\gamma}_{ik}(y)  \chi^{\beta}_{kj}(y) \right)\partial_{x_\gamma}	\partial_{ x_{\beta}}  u^0_j  +\varepsilon \partial_{x_\alpha}\left(  A^{\alpha\gamma}_{ik} (y) \chi^{\beta}_{kj}(y) \partial_{x_\gamma}	\partial_{ x_{\beta}}  u^0_j \right) \\ &= \varepsilon \partial_{x_\alpha} \left(A^{\alpha\gamma}_{ik} (x/\varepsilon) \chi^{\beta}_{kj}(x/\varepsilon) \partial_{x_\gamma}	\partial_{ x_{\beta}}  u^0_j  \right) .
		\end{align*}
		
		\noindent 	Using the homogenized equation,  $	-\partial_{x_\alpha}(A^0)^{\alpha\beta}_{ij}\partial_{x_\beta}	\! u_j^0   +  \partial_{x_i} p^0 =f_i,\ $    we rewrite  $f^0$ as
		\begin{align*}
			f_i^0=&-\partial_{x_\alpha}(A^0)^{\alpha\beta}_{ij}\partial_{x_\beta}	\! u_j^0 +\partial_{x_\alpha}\left(A^{\alpha\beta}_{ij}(y)\partial_{x_\beta}	\! u_j^0 \right) + [ \partial_{x_\alpha} A^{\alpha\gamma}_{ik}(y)\partial_{y_\gamma}	\chi^{\beta}_{kj}(y)\partial_{ x_{\beta}}  u^0_j]-  \pi^{\beta}_{j}(y)\partial_{x_i} \partial_{ x_{\beta}}  u^0_j\\
			=&\left(  - (A^0)^{\alpha\beta}_{ij} +A^{\alpha\beta}_{ij}(y) +   A^{\alpha\gamma}_{ik}(y)\partial_{y_\gamma}	\chi^{\beta}_{kj}(y)  \right) \partial_{x_\alpha} \partial_{ x_{\beta}}  u^0_j-  \pi^{\beta}_{j}(y)\partial_{x_i} \partial_{ x_{\beta}}  u^0_j.
		\end{align*}  
		Now we use Lemma \ref{divphi} to obtain
		\begin{align*}
			f_i^0=& \left(\partial_{y_\gamma} \phi^{\gamma\alpha\beta}_{ij}(y) + \partial_{y_i} q^{\alpha\beta}_{j} (y)  \right) \partial_{x_\alpha}  \partial_{x_\beta} u_j^0  -\partial_{y_\alpha}q^{\alpha\beta}_{j}(y) \partial_{ x_i}\partial_{ x_{\beta}}  u^0_j\\
			=	& \,\partial_{y_\gamma} \phi^{\gamma\alpha\beta}_{ij}(y)\partial_{x_\alpha}\partial_{x_\beta} u_j^0  + \partial_{y_i} q^{\alpha\beta}_{j} (y)\partial_{x_\alpha}\partial_{x_\beta} u_j^0  -\partial_{y_\alpha}q^{\alpha\beta}_{j}(y) \partial_{ x_i}\partial_{ x_{\beta}}  u^0_j.
		\end{align*} 
		The terms in the right hand side of the above equation can be factorized as follows. From the identity 
		\begin{align*}
		\partial_{x_\gamma}\!\left( \phi^{\gamma\alpha\beta}_{ij}\!(x/\varepsilon)\partial_{x_\alpha}\!\partial_{x_\beta} u_j^0(x) \right)\! =&\: (1/\varepsilon)[  \partial_{y_\gamma}  \phi^{\gamma\alpha\beta}_{ij}(y)\partial_{x_\alpha}\!\partial_{x_\beta} u_j^0] (x,x/\varepsilon) +  [  \partial_{x_\gamma}  \phi^{\gamma\alpha\beta}_{ij}(y)\partial_{x_\alpha}\!\partial_{x_\beta} u_j^0] (x,x/\varepsilon)
		\end{align*}
		we deduce
		\begin{align*}
			[ \partial_{y_\gamma} \phi^{\gamma\alpha\beta}_{ij}(y)\partial_{x_\alpha}\partial_{x_\beta} u_j^0] (x,x/\varepsilon) =&\; \varepsilon \partial_{x_\gamma} \{ \phi^{\gamma\alpha\beta}_{ij}(x/\varepsilon)\partial_{x_\alpha}\partial_{x_\beta} u_j^0(x) \} - \varepsilon [  \partial_{x_\gamma}  \phi^{\gamma\alpha\beta}_{ij}(y)\partial_{x_\alpha}\partial_{x_\beta} u_j^0] (x,x/\varepsilon).
		\end{align*}
		Similarly we have 
		\begin{align*}
			[ \partial_{y_i} q^{\alpha\beta}_{j} (y)\partial_{x_\alpha}\partial_{x_\beta} u_j^0 ] (x,x/\varepsilon) =&\, \varepsilon \partial_{x_i} \{ q^{\alpha\beta}_{j}(x/\varepsilon)\partial_{x_\alpha}\partial_{x_\beta} u_j^0(x) \} - \varepsilon [  \partial_{x_i}  q^{\alpha\beta}_{j}(y)\partial_{x_\alpha}\partial_{x_\beta} u_j^0] (x,x/\varepsilon)
		\end{align*}
		and 
		\begin{align*}
			[\partial_{y_\alpha}q^{\alpha\beta}_{j}(y) \partial_{ x_i}\partial_{ x_{\beta}}  u^0_j] (x,x/\varepsilon) 		=& \,  \varepsilon \partial_{x_\alpha} \{ q^{\alpha\beta}_{j}(x/\varepsilon)\partial_{x_i}\partial_{x_\beta} u_j^0(x) \} - \varepsilon [  \partial_{x_\alpha}  q^{\alpha\beta}_{j}(y)\partial_{x_i}\partial_{x_\beta} u_j^0] (x,x/\varepsilon).
		\end{align*}
	Thus, one obtains a factorization of  $f^0_i$ in $\varepsilon$ 
		\begin{align*}
			f_i^0 (x)=&\; \varepsilon \partial_{x_\gamma} \{ \phi^{\gamma\alpha\beta}_{ij}(x/\varepsilon)\partial_{x_\alpha}\partial_{x_\beta} u_j^0 \} - \varepsilon [  \partial_{x_\gamma}  \phi^{\gamma\alpha\beta}_{ij}(y)\partial_{x_\alpha}\partial_{x_\beta} u_j^0] (x,x/\varepsilon)\\
			&	+ \varepsilon \partial_{x_i} \{ q^{\alpha\beta}_{j}(x/\varepsilon)\partial_{x_\alpha}\partial_{x_\beta} u_j^0 \} - \varepsilon [  \partial_{x_i}  q^{\alpha\beta}_{j}(y)\partial_{x_\alpha}\partial_{x_\beta} u_j^0] (x,x/\varepsilon)\\
			&	+  \varepsilon \partial_{x_\alpha} \{ q^{\alpha\beta}_{j}(x/\varepsilon)\partial_{x_i}\partial_{x_\beta} u_j^0 \} - \varepsilon [  \partial_{x_\alpha}  q^{\alpha\beta}_{j}(y)\partial_{x_i}\partial_{x_\beta} u_j^0] (x,x/\varepsilon).
		\end{align*}
	Consequently we get
		\begin{align*}
			f_i^\varepsilon
			=& \,\varepsilon \left( \partial_{x_\gamma} \{ \phi^{\gamma\alpha\beta}_{ij}(x/\varepsilon)\partial_{x_\alpha}\partial_{x_\beta} u_j^0 \}	+ \partial_{x_i} \{ q^{\alpha\beta}_{j}(x/\varepsilon)\partial_{x_\alpha}\partial_{x_\beta} u_j^0 \}\right. \\
			&\qquad \left.	+  \partial_{x_\alpha} \{ q^{\alpha\beta}_{j}(x/\varepsilon)\partial_{x_i}\partial_{x_\beta} u_j^0 \}+ \partial_{x_\alpha}\{ A^{\alpha\gamma}_{ik} (x/\varepsilon) \chi^{\beta}_{kj}(x/\varepsilon)\partial_{x_\gamma}	\partial_{ x_{\beta}}  u_j^0 \} \right)  \\
			&-\varepsilon \left( \phi^{\gamma\alpha\beta}_{ij}(y)\partial_{x_\alpha}\partial_{x_\beta} \partial_{x_\gamma} u_j^0+   2 q^{\alpha\beta}_{j}(y)\partial_{x_i} \partial_{x_\alpha}\partial_{x_\beta} u_j^0\right) (x,x/\varepsilon).
		\end{align*}
		Now, by setting, on the one hand
		\begin{align*}
			(h^\varepsilon_1)_i^\gamma (x)	:=& \phi^{\gamma\alpha\beta}_{ij}(x/\varepsilon)\partial_{x_\alpha}\partial_{x_\beta} u_j^0(x) , \qquad
			h^\varepsilon_2(x):= q^{\alpha\beta}_{j}(x/\varepsilon)\partial_{x_\alpha}\partial_{x_\beta} u_j^0(x), \\
			(h^\varepsilon_3)^\alpha_i(x):=&q^{\alpha\beta}_{j}(x/\varepsilon)\partial_{x_i}\partial_{x_\beta} u_j^0(x),\,\quad
			(h_4^\varepsilon)_i^\alpha(x):=  A^{\alpha\gamma}_{ik} (x/\varepsilon) \chi^{\beta}_{kj}(x/\varepsilon)\partial_{x_\gamma}	\partial_{ x_{\beta}}  u_j^0 (x)
		\end{align*}
		 and on the other hand,	\begin{align*}
			(g^\varepsilon_1)_i(x):=\phi^{\gamma\alpha\beta}_{ij}(y)\partial_{x_\alpha}\partial_{x_\beta} \partial_{x_\gamma} u_j^0(x), \qquad 
			(g^\varepsilon_2)_i(x) := 2q^{\alpha\beta}_{j}(y)\partial_{x_i} \partial_{x_\alpha}\partial_{x_\beta} u_j^0(x)
		\end{align*}  
		one gets 
		\begin{align*}
			f_i^\varepsilon(x)
			=& \,\varepsilon \left[ \partial_{x_\gamma} \{ 	(h^\varepsilon_1)_i^\gamma (x) \}	+ \partial_{x_i} \{ 	h^\varepsilon_2(x) \}
			+  \partial_{x_\alpha} \{ 	(h^\varepsilon_3)^\alpha_i(x) \}+ \partial_{x_\alpha}\{	(h_4^\varepsilon)_i^\alpha(x) \} \right]  \\
			&-\varepsilon \left[ 	(g^\varepsilon_1)_i(x)+	(g^\varepsilon_2)_i(x) \right] .\qquad\qquad\qedhere
		\end{align*}
		
	\end{proof}
	\begin{lemma}
		Assume that the source term $f$ in \textup{(\ref{hpstoks})} and \textup{(\ref{hphom0bv})} satisfies 
		$\Supp(f)\subset B(0, 1/2)$.	Then there exists a constant $C>0$ depending on $d,\mu, \delta$ and $[A]_{C^{0,\eta}}$ 
		such that, for $x\in \mathbb{H}_n$ and $|x|\geq 1$,   we have
		\begin{align}
			|h^\varepsilon_j (x) |&\leq \frac{C}{|x|^{d}}\, 
			\| f\|_{L^{d/2+\delta}(\mathbb{H}_n)}, \quad j=1,2,3,4,\label{boundh}\\
			|	g^\varepsilon_j (x) |&\leq \frac{C}{|x|^{d+1}}\| f\|_{L^{d/2+\delta}(\mathbb{H}_n)}, \quad j=1,2. \label{boundg}
		\end{align}
		and
\begin{align}
	\|\nabla^2 u^0 \|_{L^{d+\delta}( \mathbb{H}_n)}&\leq  C \|f \|_{W^{1,d/2+\delta}( \mathbb{H}_n)}\\
	\|\nabla^2 u^0  \|_{L^{d/2+\delta}( \mathbb{H}_n)}&\leq  C \|f \|_{L^{d/2+\delta}( \mathbb{H}_n)}
\end{align}		
	\end{lemma}
	\begin{proof}[Proof]
		Recall that $u^0$ is the solution to problem (\ref{hphom0bv}) and as such it can be represented through the Green kernel as follows$$u^0(x)= \int_{\mathbb{H}_n} G^{0} (x,\tilde{x}) f(\tilde{x})d\tilde{x}.$$
		Hence, for a multi-index $\lambda\in \mathbb{N}^d,$ we have 
		$$\partial^\lambda_x u^0(x)= \int_{\mathbb{H}_n} \partial^\lambda_x G^{0} (x,\tilde{x}) f(\tilde{x})d\tilde{x}.$$
		which implies 
		\begin{align*}
			|\partial^\lambda_x u^0(x)|&\leq \int_{\mathbb{H}_n}| \partial^\lambda_x G^{0} (x,\tilde{x})| |f(\tilde{x})|d\tilde{x}.
		\end{align*}	 
		Then estimate (\ref{estim_der_G0}) implies 
		\begin{align*}
			|	\partial^\lambda_x u^0(x)| &\leq  \int_{\mathbb{H}_n} \frac{C}{|x-\tilde{x}|^{d-2+|\lambda|}}\, | f(\tilde{x})|d\tilde{x}\\
			&\leq  \int_{B(0,1/2)} \frac{C}{|x-\tilde{x}|^{d-2+|\lambda|}}\, | f(\tilde{x})|d\tilde{x}.
		\end{align*}
		Since for $|x|\geq 1$ and  $\tilde{x}\in B(0,1/2)$ we have $|x-\tilde{x}|\geq |x| -1/2$ the inequality right above leads to 
		\begin{align*}
			|\partial^\lambda_x u^0(x)| &\leq  \int_{B(0,1/2)} \frac{C}{\left(|x|-\frac{1}{2}\right)^{d-2+|\lambda|}}\, | f(\tilde{x})|d\tilde{x}.
		\end{align*}
		H\"older inequality then yields 	
		\begin{align*}
			|\partial^\lambda_x u^0(x)| &\leq  \frac{C}{(|x|-\frac{1}{2})^{d-2+|\lambda|}}\, 
			\left(\int_{B(0,1/2)} 1 d\tilde{x}\right)^{1/p'} \left(
			\int_{B(0,1/2)} | f(\tilde{x})|^pd\tilde{x}\right)^{1/p}\\
			&\leq  \frac{C}{(|x|-\frac{1}{2})^{d-2+|\lambda|}}\, 
			C(d,p) \| f\|_{L^{p}(\mathbb{H}_n)}
		\end{align*} with $p=d/2+\delta$ and $\frac{1}{p}+\frac{1}{p'}=1  $.
		Then by noticing that
		\begin{align*}
			\frac{1}{(|x|-\frac{1}{2})^{d-2+|\lambda|}}= \frac{1}{|x|^{d-2+|\lambda|}\left(1-\frac{1}{2|x|}\right)^{d-2+|\lambda|}}
			\qquad \mbox{	and } \quad
			1-\frac{1}{2|x|}\geq \frac{1}{2}
		\end{align*} one obtains 
		\begin{align*}
			|	\partial^\lambda_x u^0(x)|
			&\leq  \frac{C}{|x|^{d-2+|\lambda|}}\, 
			\| f\|_{L^{d/2+\delta}(\mathbb{H}_n)}, \quad \mbox{ for all } |x|\geq 1.
		\end{align*}
		The constant $C$ in the above lines depends on $d,\mu, \delta$ and $[A]_{C^{0,\eta}}$.
		This last estimate  with $|\lambda|=2,3$ together with the boundedness of the functions  $\,A^{\alpha\beta}_{ij}, \chi^{\beta}_{kj},\, q^{\alpha\beta}_{j} $ given in Lemma \ref{bound_on_correctors} and Lemma \ref{divphi}    leads to
		\begin{align*}
			|h^\varepsilon_j (x) |&\leq \frac{C}{|x|^{d}}\, 
			\| f\|_{L^{d/2+\delta}(\mathbb{H}_n)}, \quad j=1,2,3,4 
		\end{align*}
		and to 
		\begin{align*}
			|	g^\varepsilon_j (x) |&\leq \frac{C}{|x|^{d+1}}\| f\|_{L^{d/2+\delta}(\mathbb{H}_n)}, \quad j=1,2. 
		\end{align*}  where $C$ depends on $d,\mu, \delta$ and $[A]_{C^{0,\eta}}$.
			\end{proof}
	\noindent Now we continue the proof of Proposition \ref{propolinfestim}.	\begin{proof}[Proof of Proposition \textup{\ref{propolinfestim}}] Recall that the aim is to establish a bound on $r^\varepsilon$ solution to problem (\ref{rqeps}).
		Let us start with an integral representation of $r^\varepsilon$ :
		\begin{align*}
			r_i^\varepsilon(x)= &	\int_{\mathbb{H}_n} f_k^\varepsilon(\tilde{x}) G_{ki}^{*,\varepsilon} (\tilde{x}, x)\, d\tilde{x}-\int_{\mathbb{H}_n}  \Pi_i^{*,\varepsilon}(\tilde{x} ,x)  \left(- \varepsilon \chi^{\beta}_{kj}(\tilde{x}/\varepsilon)\partial_{x_k}\partial_{ x_{\beta}}  u^0_j   (\tilde{x})\right)d\tilde{x} .
		\end{align*}
		We  replace $f^\varepsilon_i$ in the first integral  by its expression given in (\ref{Lemm_feps})
		\begin{align*}
			r_i^\varepsilon(x)
			= &	\quad \varepsilon \int_{\mathbb{H}_n} \left[ \partial_{\tilde{x}_\gamma} \{ 	(h^\varepsilon_1)_k^\gamma (\tilde{x}) \}	+ \partial_{\tilde{x}_k} \{ 	h^\varepsilon_2(\tilde{x}) \}
			+  \partial_{\tilde{x}_\alpha} \{ 	(h^\varepsilon_3)^\alpha_k(\tilde{x}) \}+ \partial_{\tilde{x}_\alpha}\{	(h_4^\varepsilon)_k^\alpha(\tilde{x}) \} \right]  G_{ki}^{*,\varepsilon} (\tilde{x}, x)\, d\tilde{x} \\
			&-\varepsilon \int_{\mathbb{H}_n} \left[ 	(g^\varepsilon_1)_k(\tilde{x})+	(g^\varepsilon_2)_k(\tilde{x}) \right]  G_{ki}^{*,\varepsilon} (\tilde{x}, x)\, d\tilde{x}+\varepsilon\int_{\mathbb{H}_n} \Pi_i^{*,\varepsilon}(\tilde{x} ,x)  \chi^{\beta}_{kj}(\tilde{x}/\varepsilon)\partial_{\tilde{x}_k}\partial_{\tilde{x}_{\beta}}  u^0_j   (\tilde{x}) \,d\tilde{x} .
		\end{align*}
		Then one performs integration by parts for some of the above terms		
		\begin{align*}
			r_i^\varepsilon(x)
			= &	\quad -\varepsilon \int_{\mathbb{H}_n} \left[  	(h^\varepsilon_1)_k^\alpha (\tilde{x}) 
			+  	(h^\varepsilon_3)^\alpha_k(\tilde{x}) + 	(h_4^\varepsilon)_k^\alpha(\tilde{x})  \right] \partial_{\tilde{x}_\alpha} G_{ki}^{*,\varepsilon} (\tilde{x}, x)\, d\tilde{x}  -\varepsilon \int_{\mathbb{H}_n}   	h^\varepsilon_2(\tilde{x}) 
			\partial_{\tilde{x}_k} G_{ki}^{*,\varepsilon} (\tilde{x}, x)\, d\tilde{x}\\
			&	\quad -\varepsilon \int_{\mathbb{H}_n} \left[ 	(g^\varepsilon_1)_k(\tilde{x})+	(g^\varepsilon_2)_k(\tilde{x}) \right]  G_{ki}^{*,\varepsilon} (\tilde{x}, x)\, d\tilde{x}+\varepsilon\int_{\mathbb{H}_n} \Pi_i^{*,\varepsilon}(\tilde{x} ,x)  \chi^{\beta}_{kj}(\tilde{x}/\varepsilon) \partial_{\tilde{x}_k}\partial_{\tilde{x}_{\beta}}  u^0_j   (\tilde{x}) \,d\tilde{x} 
		\end{align*}
		which implies 
		\begin{align}
			|r_i^\varepsilon(x)|\leq  \, \varepsilon \!\int_{\mathbb{H}_n} &   \left|(h^\varepsilon)^\alpha_k(\tilde{x})\right| \left|\partial_{\tilde{x}_\alpha} G_{ki}^{*,\varepsilon} (x,\tilde{x})\right|    d\tilde{x}+  \varepsilon \!\int_ \mathbb{H} \left|(h_2^\varepsilon)(\tilde{x})\right|\left|\partial_{\tilde{x}_k} G_{ki}^{*,\varepsilon} (x,\tilde{x})\right|  d\tilde{x}  \nonumber\\
			&+  \varepsilon \int_{\mathbb{H}_n}|g_k^\varepsilon(\tilde{x})|  \left|G_{ki}^{*,\varepsilon} (x,\tilde{x})\right| 
			d\tilde{x}+\varepsilon\int_{\mathbb{H}_n} \left|\Pi_i^{*,\varepsilon}(\tilde{x} ,x)\right|\, \left|\chi^{\beta}_{kj}(\tilde{x}/\varepsilon) \partial_{\tilde{x}_k}\partial_{\tilde{x}_{\beta}} u^0_j   (\tilde{x})\right| \,d\tilde{x}\label{1boundreps}
		\end{align}
		with 	
		\begin{align}
			|(h^\varepsilon)^\alpha_k|:=& |(h^\varepsilon_1)_k^\alpha|
			+  |(h^\varepsilon_3)^\alpha_k| + 	|(h_4^\varepsilon)_k^\alpha|, \label{def_heps}\\
			|g^\varepsilon|:=& |g^\varepsilon_1|+	|g^\varepsilon_2|\nonumber.
		\end{align}	
		
		\noindent{\bf Step 1 : an estimate of $r^\varepsilon$ in the  $L^\infty$ norm }\\
		We rewrite the bound (\ref{1boundreps}) as follows:
		\begin{align}
			|r^\varepsilon(x)|\leq 
			& \: \varepsilon \,(I_1 +I_2+ I_3+I_4) \label{repsIk}
		\end{align}
		with
		\begin{align*}
			&	I_1:= \!\int_{\mathbb{H}_n}  |(h^\varepsilon)^\alpha_k(\tilde{x})|\left|\partial_{\tilde{x}_\alpha}G_{ki}^{*,\varepsilon} (x,\tilde{x})\right|    d\tilde{x},\\ &I_2 :=\!\int_ \mathbb{H} |(h_2^\varepsilon)(\tilde{x})|\left|\partial_{\tilde{x}_k} G_{ki}^{*,\varepsilon} (x,\tilde{x})\right|  d\tilde{x},  \nonumber\\
			&I_3:=\int_{\mathbb{H}_n}|g_k^\varepsilon(\tilde{x})|  \left|G_{ki}^{*,\varepsilon} (x,\tilde{x})\right| 
			d\tilde{x},\\&I_4:=\int_{\mathbb{H}_n} \left|\Pi_i^{*,\varepsilon}(\tilde{x} ,x)\right| \left|\chi^{\beta}_{kj}(\tilde{x}/\varepsilon)\partial_{\tilde{x}_k}\partial_{\tilde{x}_\beta} u^0_j   (\tilde{x})\right| \,d\tilde{x}.
		\end{align*}
		
		\noindent Now we seek to bound the quantities $I_k$,  $k=1,2,3,4.$
		
		\noindent	We  extend  $  (h^\varepsilon)^\alpha_k,\, (h_2^\varepsilon),\, g^\varepsilon,
		\,    u^0  $ to the whole space $\mathbb{R}^d$ by setting 
		\begin{align*}
			(h^\varepsilon)^\alpha_k(\tilde{x})= (h_2^\varepsilon)(\tilde{x})= g^\varepsilon(\tilde{x})=  u^0 (\tilde{x}) =0 \quad \mbox{for all } \tilde{x}\in\mathbb{R}^d \bs \overline{\mathbb{H}}_n.
		\end{align*}

		\noindent 	\underline{Step 1-a :	Control of  $I_1$}\\
		The bound (\ref{estim_dertildG}) on  Green's function $\partial_{\tilde{x}_\alpha} G_{ki}^{*,\varepsilon} (x,\tilde{x})$ gives
		\begin{align*}
			I_1&\leq  \,\!\int_{\mathbb{R}^d} \frac{C}{|x- \tilde{x}|^{d-1}} \mathbf{1}_{B(0,1)} (x- \tilde{x}) |h^\varepsilon(\tilde{x})| d\tilde{x}
			+   \int_{\mathbb{R}^d} \,  \frac{C}{|x- \tilde{x}|^{d-1}}   \mathbf{1}_{B(0,1)^c} (x- \tilde{x})|h^\varepsilon(\tilde{x}) |  d\tilde{x}\\
			&=: \, I_{1,1} + I_{1,2}
		\end{align*}
		where the constant $C$ depends on $d,\mu, [A]_{C^{0,\eta}}$.
		Let us first control $I_{1,1}$. We take $p,p'\in(1,+\infty)$ satisfying $\frac{1}{p}+\frac{1}{p'}=1$ and $p<\frac{d}{d-1}$ so that
		\begin{align*}
			\int_{| \tilde{x}|<1}\, \left( \frac{1}{|\tilde{x}|^{d-1}}  \right)^{\!\!^p} \! d\tilde{x}			=&  \int_{| \tilde{x}|<1} \,  \frac{1}{| \tilde{x}|^{(d-1)p}}  d\tilde{x}\\
			=&\, C(d) \int_0^1 \,  \frac{r^{d-1}}{r^{(d-1)p}}  dr\\	
			=&  \, \frac{C(d)}{d-(d-1)p} =C(d,p)<\infty.
		\end{align*}	
		So we have
		$$\tilde{x}\mapsto  \frac{1}{|\tilde{x}|^{d-1}}\mathbf{1}_{B(0,1)} (\tilde{x})\in L^p(\mathbb{R}^d).$$ 
		
		\noindent	The condition $\frac{1}{p}+\frac{1}{p'}=1$ combined with $p<\frac{d}{d-1}$  implies  $p'>d$ i.e. $p'=d+\delta$ for some $\delta>0$.
		From the definitions (\ref{def_h_i_eps}) and (\ref{def_heps}) and from Lemma \ref{bound_on_correctors} and Lemma \ref{divphi} we derive the bound
		\begin{align}\label{bound_h_nab2u}
			|h^\varepsilon(\tilde{x})|&\leq C ([A]_{C^{0,\eta}},d,\mu)|\nabla^2 u^0(\tilde{x})|, \quad \mbox{ for all } \tilde{x}\in \mathbb{H}_n,
		\end{align}
		which implies
		\begin{align*}
			\|h^\varepsilon \|_{L^{d+\delta}(\mathbb{R}^d)}&=\|h^\varepsilon \|_{L^{d+\delta}(\mathbb{H}_n)}\\&\leq C([A]_{C^{0,\eta}},d,\mu)	\|\nabla^2 u^0 \|_{L^{d+\delta}(\mathbb{H}_n)}.
		\end{align*}
		Moreover, Sobolev's embedding theorem yields 
		\begin{align*}
			\|\nabla^2 u^0 \|_{L^{d+\delta}(\mathbb{H}_n)}
			& \leq C(d, \delta) 	 \|\nabla^2 u^0\|_{W^{1,d/2+\delta}(\mathbb{H}_n)}.
		\end{align*}
		This Sobolev's embedding holds for $\delta\in(0,d/2)$.
		By the regularity theory of the Stokes problem (see \cite{Gld}, Theorem IV.3.2), we get
		\begin{align*}
			\|u^0\|_{W^{3,d/2+\delta}(\mathbb{H}_n)}\leq C \|f\|_{W^{1,	d/2+\delta}(\mathbb{H}_n)}.
		\end{align*}
		Hence, we obtain
		\begin{align*}
			\|h^\varepsilon \|_{L^{p'}(\mathbb{H}_n)}&\leq C([A]_{C^{0,\eta}},d,\mu, \delta)  \|f\|_{W^{1,d/2+\delta}(\mathbb{H}_n)}. 
		\end{align*}
		Finally we apply Hölder's inequality to $I_{1,1}$
		\begin{align*}
			I_{1,1}	=\int_{\mathbb{R}^d} \frac{C}{|x- \tilde{x}|^{d-1}} \mathbf{1}_{B(0,1)} (x- \tilde{x}) |h^\varepsilon(\tilde{x})| d\tilde{x} & \leq \left\| \frac{1}{| \cdot - x|^{d-1}}\mathbf{1}_{B(x,1)}  \right\|_{L^p(\mathbb{R}^d)}  \|h^\varepsilon\|_{L^{p'}(\mathbb{R}^d)}	\\
			&\leq  C([A]_{C^{0,\eta}},d,\mu, \delta)  \|f\|_{W^{1,d/2+\delta}(\mathbb{H}_n)}.
		\end{align*}
		
		We now control the term $I_{1,2}$ :
		\begin{align}
			I_{1,2}=& \int_{\mathbb{R}^d} \,  \frac{C}{|x- \tilde{x}|^{d-1}}   \mathbf{1}_{B(0,1)^c} (x- \tilde{x})|h^\varepsilon(\tilde{x}) |  d\tilde{x}\nonumber\\
			\leq& \int_{\mathbb{R}^d} \,  \frac{C}{|x- \tilde{x}|^{d-1}}   \mathbf{1}_{B(0,1)^c} (x- \tilde{x})|h^\varepsilon(\tilde{x}) |\mathbf{1}_{B(0,1)} (\tilde{x})  d\tilde{x}\nonumber\\
			&+  \int_{\mathbb{R}^d} \,  \frac{C}{|x- \tilde{x}|^{d-1}}   \mathbf{1}_{B(0,1)^c} (x- \tilde{x})|h^\varepsilon(\tilde{x}) | \mathbf{1}_{B(0,1)^c} (\tilde{x}) d\tilde{x}\label{A+B}
		\end{align}
		If  $q=d/2+\delta$ and $q'$ is defined by $\frac{1}{q}+\frac{1}{q'}=1$ then we have
		\begin{align*}
			\int_{\mathbb{R}^d} \,  \frac{C}{|x- \tilde{x}|^{d-1}}   \mathbf{1}_{B(0,1)^c} (x- \tilde{x})|h^\varepsilon(\tilde{x}) |\mathbf{1}_{B(0,1)} (\tilde{x})  d\tilde{x}\leq C\left\| \frac{1}{|\cdot|^{d-1}}\mathbf{1}_{B(0,1)^c} \right\|_{L^{q'}(\mathbb{R}^d)} \, \left\| |h^\varepsilon |\mathbf{1}_{B(0,1)} \right\|_{L^{q}(\mathbb{H}_n)}
		\end{align*}
		because with such a $q'$, one has
		\begin{align*}
			\int_{{B(0,1)^c}} \left|  \frac{1}{|\tilde{x}|^{d-1}}   \right|^{q'} d\tilde{x}
			=&	\int_{|\tilde{x}|>1} \frac{1}{|\tilde{x}|^{(d-1)q'}} d\tilde{x}<\infty.
		\end{align*}
		So by recalling the fact that
		\begin{align*}
			\left\| h^\varepsilon \right\|_{L^{q}(\mathbb{H}_n)}\leq C \left\| \nabla^2 u^0\right\|_{L^{d/2+\delta}(\mathbb{H}_n)}
			\leq\; C \left\|f\right\|_{L^{d/2+\delta}(\mathbb{H}_n)}
			\leq\; C\left\|f\right\|_{W^{1,d/2+\delta}(\mathbb{H}_n)}
		\end{align*}
		we obtain 
		\begin{align}\label{bound_of_A}
			\int_{\mathbb{R}^d} \,  \frac{C}{|x- \tilde{x}|^{d-1}}   \mathbf{1}_{B(0,1)^c} (x- \tilde{x})|h^\varepsilon(\tilde{x}) |\mathbf{1}_{B(0,1)} (\tilde{x}) \, d\tilde{x}
			\leq&\; C\left\|f\right\|_{W^{1,d/2+\delta}(\mathbb{H}_n)}
		\end{align}
		with $C=	 C(d,\mu,[A]_{C^{0,\eta}},\delta)$.

		The term 
		\begin{align*}
			\int_{\mathbb{R}^d} \,  \frac{C}{|x- \tilde{x}|^{d-1}}   \mathbf{1}_{B(0,1)^c} (x- \tilde{x})|h^\varepsilon(\tilde{x}) | \mathbf{1}_{B(0,1)^c} (\tilde{x}) d\tilde{x}
		\end{align*} 
		is estimated as follows. On the one hand, for $p>\frac{d}{d-1}$,  we have
		\begin{align*}
			\int_{\mathbb{R}^d} \,  \frac{1}{|\tilde{x}|^{(d-1)p}}  \mathbf{1}_{B(0,1)^c} (\tilde{x}) d\tilde{x}&=  \int_{| \tilde{x}|>1} \,  \frac{1}{|\tilde{x}|^{(d-1)p}}  d\tilde{x}\\
			&= C(d,p)<\infty, 
		\end{align*} 
		and on the other hand,  by (\ref{boundh}), we have 
		\begin{align*}
			\int_{\mathbb{R}^d} \,  |h^\varepsilon(\tilde{x}) |^{p'} \mathbf{1}_{B(0,1)^c} (\tilde{x}) d\tilde{x} & \leq  C\left\|f\right\|^{p'}_{L^{d/2+\delta}(\mathbb{H}_n)}	\int_{|\tilde{x}|>1} \frac{1}{|\tilde{x}|^{dp'}} d\tilde{x} \\
			& \leq  C\left\|f\right\|^{p'}_{L^{d/2+\delta}(\mathbb{H}_n)}	
		\end{align*}
		whenever the inequality $p'>1$ holds. Which is indeed te case when $p'$ is given by $\frac{1}{p}+\frac{1}{p'}=1$. We then obtain 
		\begin{align}
			\int_{\mathbb{R}^d} \,  \frac{C}{|x- \tilde{x}|^{d-1}}   \mathbf{1}_{B(0,1)^c} (x- \tilde{x})|h^\varepsilon(\tilde{x}) | \mathbf{1}_{B(0,1)^c} (\tilde{x}) d\tilde{x}\leq  C(d,\mu,\delta,[A]_{C^{0,\eta}})\left\|f\right\|^{}_{L^{d/2+\delta}(\mathbb{H}_n)}.\label{bound_of_B}
		\end{align} 
		Putting together (\ref{A+B}), (\ref{bound_of_A}) and (\ref{bound_of_B}) we get 
		\begin{align}
			I_{1,2}\leq  C(d,\mu,\delta,[A]_{C^{0,\eta}})\left\|f\right\|^{}_{L^{d/2+\delta}(\mathbb{H}_n)}.
		\end{align}
		We thus have obtained 
		\begin{align*}
			I_1 =I_{1,1}+I_{1,2}\leq C \|f\|_{W^{1,d/2+\delta}(\mathbb{H}_n)}
		\end{align*}
		where 	$C= C(d, \mu,[A]_{C^{0,\eta}},\delta)>0$	 and $\delta\in(0,d/2)$.\\

		\noindent \underline{Step 1-b : Control of the term $I_2$} \\The reasoning that has lead to the above bound of $I_1$ remains valid when $I_1$ is replaced with the integral $I_2$. Therefore, we also have
		\begin{align*}
			I_2\leq C(d,\mu,[A]_{C^{0,\eta}},\delta)\, \|f\|_{W^{1,d/2+\delta}(\mathbb{H}_n)}.
		\end{align*}
		
		\noindent \underline{Step 1-c :	Control of the term $I_3$}\\
		By the estimate (\ref{estim_G_y}) we have
		\begin{align*}
			I_3&\leq 	\int_{\mathbb{R}^d} \frac{C}{|x- \tilde{x}|^{d-2}}  \mathbf{1}_{B(0,1)} (x-\tilde{x}) |g^\varepsilon(\tilde{x})| d\tilde{x}	+ 	\int_{\mathbb{R}^d} \frac{C}{|x- \tilde{x}|^{d-2}}  \mathbf{1}_{B(0,1)^c} (x-\tilde{x})   |g^\varepsilon(\tilde{x})| d\tilde{x}\\
			&=	I_{3,1}+ I_{3,2}.
		\end{align*}
	First, we  bound the term $I_{3,1}$ as follows. Let  $p,p'\in(1,+\infty)$ with $\frac{1}{p}+\frac{1}{p'}=1$ and $p<\frac{d}{d-2}$ so that
		\begin{align*}
			\int_{\mathbb{R}^d\cap B(0,1)}    \left| \frac{1}{|\tilde{x}|^{d-2}} \right|^{p}  d\tilde{x}=& 	\int_{\mathbb{R}^d\cap B(0,1) }    \frac{1}{|\tilde{x}|^{(d-2)p}}  d\tilde{x} \\
			=& C(d,p)<\infty.	\end{align*} 
		Under the above condition on $p$ the exponent $p'$     satisfies $p'>d/2$   i.e. $p'=d/2+\delta$ for some $\delta>0$. From the definitions of $g^\varepsilon_j$ in (\ref{def_geps}) we get 
		\begin{align*}
			|g^\varepsilon_j(\tilde{x})|&\leq C (d,\mu)|\nabla^3 u^0(\tilde{x})|, \quad \mbox{ for all } \tilde{x}\in \mathbb{H}_n,
		\end{align*}
		which implies
		\begin{align}
			\|g^\varepsilon \|_{L^{d/2+\delta}(\mathbb{H}_n)}&\leq C (d,\mu) \|\nabla^3 u^0\|_{L^{d/2+\delta}(\mathbb{H}_n)}\nonumber\\
			&\leq C (d, \mu, \delta) \|f\|_{W^{1,d/2+\delta}(\mathbb{H}_n)}.\label{g_ineq_f}
		\end{align}
		Now we apply Hausdorff-Young inequality to obtain
		\begin{align}
			I_{3,1}=\int_{\mathbb{R}^d} \frac{1}{|x- \tilde{x}|^{d-2}} \mathbf{1}_{B(0,1)}(x-\tilde{x}) |g^\varepsilon(\tilde{x})| d\tilde{x} & \leq \left\| \frac{1}{|\cdot |^{d-2}} \mathbf{1}_{B(0,1)} \right\|_{L^p(\mathbb{R}^d)} \!\!\! \|g^\varepsilon\|_{L^{p'}(\mathbb{R}^d)}\nonumber	\\
			&\leq  C  \|f\|_{W^{1,d/2+\delta}(\mathbb{H}_n)}.\label{L1}
		\end{align}
		with $C=C(d,\mu,[A]_{C^{0,\eta}},\delta)$.\\

	\noindent	We now control  the integral 
		\begin{align}
			I_{3,2}=  &\int_{\mathbb{R}^d} \,  \frac{C}{|x- \tilde{x}|^{d-2}}   \mathbf{1}_{B(0,1)^c} (x- \tilde{x})|g^\varepsilon(\tilde{x}) |  d\tilde{x}\nonumber\\
			= & \int_{\mathbb{R}^d} \,  \frac{C}{|x- \tilde{x}|^{d-2}}   \mathbf{1}_{B(0,1)^c} (x- \tilde{x})|g^\varepsilon(\tilde{x})| \mathbf{1}_{B(0,1)} (\tilde{x}) d\tilde{x} \nonumber\\ &+ \int_{\mathbb{R}^d} \,  \frac{C}{|x- \tilde{x}|^{d-2}}   \mathbf{1}_{B(0,1)^c} (x- \tilde{x})|g^\varepsilon(\tilde{x}) |  \mathbf{1}_{B(0,1)^c} (\tilde{x})   d\tilde{x}.\label{I32_split}
		\end{align}
	One handles the term 
	\begin{align*}
	\int_{\mathbb{R}^d} \,  \frac{C}{|x- \tilde{x}|^{d-2}}   \mathbf{1}_{B(0,1)^c} (x- \tilde{x})|g^\varepsilon(\tilde{x})| \mathbf{1}_{B(0,1)} (\tilde{x}) d\tilde{x} 
	\end{align*}
	through an application of Hausdorff-Young inequality. 
		To do so, we set $q=d/2+\delta$ and define $q'$ by $\frac{1}{q}+\frac{1}{q'}=1$.
		Then we have $q'>\frac{d}{d-2}$ and consequently we obtain
		\begin{align*}
			\int_{\mathbb{R}^d} \,  \left( \frac{1}{|\tilde{x}|^{d-2}}  \right)^{\!\! q'}\!\!\!\mathbf{1}_{B(0,1)^c} ( \tilde{x}) d\tilde{x} = &\int_{B(0,1)^c} \left|  \frac{1}{|\tilde{x}|^{d-2}}  \right|^{q'} d\tilde{x}\\
			= \, &C(d,\delta) <\infty.
		\end{align*}
		The Hausdorff-Young inequality yields 
			\begin{align*}
		\int_{\mathbb{R}^d} \,  \frac{C}{|x- \tilde{x}|^{d-2}}   \mathbf{1}_{B(0,1)^c} (x- \tilde{x})|g^\varepsilon(\tilde{x}) | \mathbf{1}_{B(0,1)} (\tilde{x}) d\tilde{x} \leq C (d,\mu,\delta,[A]_{C^{0,\eta}} ) 	\|g^\varepsilon\|_{L^{d/2+\delta}(\mathbb{R}^d)}
		\end{align*}
		and subsequently estimate (\ref{g_ineq_f}) implies
			\begin{align}
			\int_{\mathbb{R}^d} \,  \frac{C}{|x- \tilde{x}|^{d-2}}   \mathbf{1}_{B(0,1)^c} (x- \tilde{x})|g^\varepsilon(\tilde{x}) | \mathbf{1}_{B(0,1)} (\tilde{x}) d\tilde{x} 
				&\leq  C (d, \mu, \delta, ,[A]_{C^{0,\eta}}  ) \|f\|_{W^{1,d/2+\delta}(\mathbb{H}_n)}.\label{I321}
		\end{align}
			The term 
		\begin{align*}
		\int_{\mathbb{R}^d} \,	\frac{C}{|x- \tilde{x}|^{d-2}}   \mathbf{1}_{B(0,1)^c} (x- \tilde{x})|g^\varepsilon(\tilde{x}) |  \mathbf{1}_{B(0,1)^c} (\tilde{x})  	d\tilde{x}
		\end{align*} 
		is estimated as follows. On the one hand, for $p>\frac{d}{d-2}$,  we have
		\begin{align*}
			\int_{\mathbb{R}^d} \,  \frac{1}{|\tilde{x}|^{(d-2)p}}  \mathbf{1}_{B(0,1)^c} (\tilde{x}) d\tilde{x}&=  \int_{| \tilde{x}|>1} \,  \frac{1}{|\tilde{x}|^{(d-2)p}}  d\tilde{x}\\
			&= C(d,p)<\infty, 
		\end{align*} 
		and on the other hand,  by (\ref{boundg}), we have 
		\begin{align*}
			\int_{\mathbb{R}^d} \,  |g^\varepsilon(\tilde{x}) |^{p'} \mathbf{1}_{B(0,1)^c} (\tilde{x}) d\tilde{x} & \leq  C\left\|f\right\|^{p'}_{L^{d/2+\delta}(\mathbb{H}_n)}	\int_{|\tilde{x}|>1} \frac{1}{|\tilde{x}|^{(d+1)p'}} d\tilde{x} \\
			& \leq  C\left\|f\right\|^{p'}_{L^{d/2+\delta}(\mathbb{H}_n)}	
		\end{align*}
		whenever $p'>\frac{d}{d+1}$, which is indeed the case when $p'$ is given by $\frac{1}{p}+\frac{1}{p'}=1$. We then obtain 
		\begin{align}
			\int_{\mathbb{R}^d} \,  \frac{C}{|x- \tilde{x}|^{d-2}}   \mathbf{1}_{B(0,1)^c} (x- \tilde{x})|g^\varepsilon(\tilde{x}) | \mathbf{1}_{B(0,1)^c} (\tilde{x}) d\tilde{x}\leq  C(d,\mu,\delta,[A]_{C^{0,\eta}})\left\|f\right\|^{}_{L^{d/2+\delta}(\mathbb{H}_n)}\label{I322}
		\end{align} 
	where $C=	 C (d, \mu, \delta,[A]_{C^{0,\eta}})$. 
	Putting together (\ref{I32_split}), (\ref{I321}) and (\ref{I322}) we obtain
		\begin{align*}
			I_{3}= 	I_{3,1}+	I_{3,2}
			\leq  C\|f\|_{W^{1,d/2+\delta}(\mathbb{H}_n)}
		\end{align*}
		with	 $C=C\left(A,d, \mu, [A]_{C^{0,\eta}},\delta\right)$.\\
		
		\noindent \underline{Step 1-d :	Control of $I_4$}\\ The method carried out to bound integral $I_1$ still applies because on the one hand we have the pointwise estimate 
		(\ref{estimPi_y})
		for $\Pi_i^{*,\varepsilon}$ similar to that of $\partial_{\tilde{x}_\alpha} G_{ki}^{*,\varepsilon}$: for some $C$ depending on $d, \mu, [A]_{C^{0,\eta}}$ one has 
		\begin{align*}
			|\Pi_i^{*,\varepsilon}(\tilde{x},x)|\leq\frac{C}{|x- \tilde{x}|^{d-1}}, \quad \mbox{  for all } x,\tilde{x}\in \mathbb{H}_n,
		\end{align*}
		and on the other hand we have the bound
		\begin{align*}
			\left| \chi^\beta_{kj}(\tilde{x}/\varepsilon)\partial_{\tilde{x}_k} \partial_{\tilde{x}_\beta} u_j^0(\tilde{x}) \right|&\leq C (d,\mu)|\nabla^2 u^0(\tilde{x})|, \mbox{  for all } \tilde{x}\in \mathbb{H}_n.
		\end{align*}	
		Thus one has the estimate
		\begin{align*}
			I_4	&\leq   C \|f\|_{W^{1,\,d/2+\delta}(\mathbb{H}_n)}
		\end{align*}	
		where 	$C= C(d, \mu, [A]_{C^{0,\eta}},\delta)>0$	with $\delta\in(0,d/2)$.\\
		
		\noindent \emph{Conclusion.}	The inequality (\ref{repsIk}) together with the preceding bounds on $I_k$, $k=1,2,3,4,$	yields
		\begin{align}\label{boundrepseps}
			|r^\varepsilon (x)| \leq \, C  \varepsilon \| f\|_{W^{1,d/2+\delta}}, \quad \mbox{ for all } x\in \mathbb{H}_n,
		\end{align}
		where 	$C=C(A,d, \mu, [A]_{C^{0,\eta}},\delta)>0$	with $\delta\in(0,d/2)$.\\
		
		\noindent{\bf Step 2 : an estimate of $u_{bl}^\varepsilon$ in the  $L^\infty$-norm}\\
		Let us remind that the function $u_{bl}^\varepsilon$ is the solution to problem (\ref{eq_uepsbl}). Therefore it has the following representation   
		\begin{align*}
			u^\varepsilon_{bl,i} (x)&= \int_{\partial\mathbb{H}_n} -P_{ik}^{*,\varepsilon} (x,\tilde{x})\,\left(\chi (\tilde{x}/ \varepsilon) \cdot\nabla u^0 (\tilde{x})\right)_k  d\tilde{x},\qquad \mbox{for }x\in\mathbb{H}_n,\\
			&= \int_{\partial\mathbb{H}_n} -P_{ik}^{*,\varepsilon} (x,\tilde{x})\,\chi^\beta_{kj} (\tilde{x}/ \varepsilon) \,\partial_{\tilde{x}_\beta} u_j^0 (\tilde{x})  d\tilde{x}\\
			\mbox{ where }\quad 	P_{ik}^{*,\varepsilon} (x,\tilde{x})
			&=-(A^*)_{kr}^{\alpha\gamma}(\tilde{x}/\varepsilon)\, \partial_{\tilde{x}_\gamma}G_{ri}^{*,\varepsilon}(\tilde{x},x) n_\alpha (\tilde{x})+  \Pi_i^{*,\varepsilon}(\tilde{x},x) n_k(\tilde{x}).
		\end{align*}
		Then we get, using the boundedness of $\chi$ and on $A$,	given in Lemma \ref{bound_on_correctors},
		\begin{align*}
			|	u^\varepsilon_{bl,i} (x)|
			\leq& \int_{\partial\mathbb{H}_n} \left| (A^*)_{kr}^{\gamma\beta}(\tilde{x}/\varepsilon)\, \partial_{\tilde{x}_\gamma}G_{ri}^{*,\varepsilon}(\tilde{x},x) n_\alpha (\tilde{x})\right|\,| \partial_{\tilde{x}_\beta} u_j^0 (\tilde{x})| d\tilde{x}\\
			&\qquad\qquad  + \int_{\partial\mathbb{H}_n} \left| \Pi_i^{*,\varepsilon}(\tilde{x},x) n_k(\tilde{x})\right|\,|\partial_{\tilde{x}_\beta} u_j^0 (\tilde{x})| d\tilde{x}  \\
			\leq&\, C\int_{\partial\mathbb{H}_n} \left(\left|  \partial_{\tilde{x}_\beta}G_{ri}^{*,\varepsilon}(\tilde{x},x)\right|+\left| \Pi_i^{*,\varepsilon}(\tilde{x},x)\right|\right) | \partial_{\tilde{x}_\beta} u_j^0 (\tilde{x})| d\tilde{x} 
		\end{align*}	where $C$ depends on $A$.
		Next we apply the star-version of estimates (\ref{estim_dertildG}) and (\ref{estimPi_y}) on Green's functions to obtain
		\begin{align}
			|	u^\varepsilon_{bl,i} (x)|\leq& \, C\int_{\partial\mathbb{H}_n}  \frac{x\cdot n}{|x-\tilde{x}|^d} \, |\nabla u^0 (\tilde{x})| d\tilde{x}\nonumber\\
			\leq& \, C  \|\nabla u^0 \|_{L^\infty(\mathbb{H}_n)} \int_{\partial\mathbb{H}_n}  \frac{x\cdot n}{|x-\tilde{x}|^d} \, \,  d\tilde{x} \label{3bounduepsbl}
		\end{align}
		with $C$  depending on $ d,\mu,[A]_{C^{0,\eta}}$ and $A$.
		Further, since $2\cdot ((d/2 )+ \delta)= d+2\delta > d$, the Sobolev embedding theorem gives 
		\begin{align}
			\|\nabla u^0\|_{L^\infty(\mathbb{H}_n)} &\leq C(d,\delta) \|\nabla u^0\|_{W^{2,d/2+\delta}(\mathbb{H}_n)},\nonumber\\
			&\leq C(d,\delta) \|f\|_{W^{1,d/2+\delta}(\mathbb{H}_n)} \label{nabu0fw1}.
		\end{align}	
		Hence, taking into account the fact that there exists $C>0$ depending on $d, \mu$ and $[A]_{C^{0,\eta}}$ such that
		\begin{align}\label{bound_int|y-tildy|^d}
			\int_{\partial\mathbb{H}_n}  \frac{x\cdot n}{|x-\tilde{x}|^d} \,  d\tilde{x}\leq C <\infty ,
		\end{align}
		one obtains			\begin{align}
			|	u^\varepsilon_{bl,i} (x)|\leq& \,  C(d,\mu, [A]_{C^{0,\eta}},\delta) \|f\|_{W^{1,d/2+\delta}(\mathbb{H}_n)},\quad\mbox{ for all }x\in\mathbb{H}_n  .\\\nonumber
		\end{align}

		\noindent{\bf Step 3 : estimating $u^\varepsilon-u^0$ in the $L^\infty$-norm}\\
		From relation (\ref{errors_bl}) defining $r^\varepsilon$ we get 
		\begin{align*}
			u^\varepsilon(x) -u^0(x)= \varepsilon \chi(x/ \varepsilon) \nabla u^0(x)+\varepsilon u^\varepsilon_{bl}(x)+r^\varepsilon(x), \quad \mbox{ for all }x\in \mathbb{H}_n,
		\end{align*}
		which implies 
		\begin{align*}
			|u^\varepsilon(x) -u^0(x)| \leq  \varepsilon |\chi(x/ \varepsilon)|\,|\nabla u^0(x)|+\varepsilon |u^\varepsilon_{bl}(x)|+|r^\varepsilon(x)|, \quad \mbox{ for all }x\in \mathbb{H}_n.
		\end{align*}
		The estimates in $L^\infty$-norm obtained in the preceeding lines  allows then to write
		\begin{align}\label{order0error}
			|u^\varepsilon(x) -u^0(x)| \leq  C \varepsilon \,\|\nabla u^0\|_{L^\infty(\mathbb{H}_n)}+ C  \varepsilon\|f\|_{W^{1,d/2+\delta}} +C\varepsilon\|f\|_{W^{1,d/2+\delta}} .
		\end{align}
		Finally, we take (\ref{nabu0fw1}) into account in inequality (\ref{order0error}) to obtain
		\begin{align}
			\|u^\varepsilon -u^0 \|_{L^\infty(\mathbb{H}_n)} \leq  C  \, \varepsilon\|f\|_{W^{1,d/2+\delta}(\mathbb{H}_n)} 
		\end{align}
		where 	$C= C(d,\mu,[A]_{C^{0,\eta}},\delta )>0$	with  $\delta\in(0,d/2)$.
		Here ends the proof of Proposition \ref{propolinfestim}.\qedhere\\
	\end{proof}

	\section{Homogenization of  Green's function $G^\varepsilon$}\label{section_hom_G}

	\noindent	The purpose of this section is to establish a quantitative homogenization result on the  Green function $G^\varepsilon(x,\tilde{x})$. In other words, we show that the heterogenous kernel $G^\varepsilon$ converges to the homogenized one $G^0($ as $\varepsilon$ vanishes. This fact is precisely stated in the following Proposition whose proof is based on a duality argument.
	\begin{proposition}\label{propo_GepsG0}
		For arbitrary $x\in\mathbb{H}_n$ and for all $s\in \left(\frac{1}{d+\frac{1}{d}},  \frac{1}{d-1 +\frac{2}{d}}\right)$ there exists $C>0$  such that for all $\varepsilon> 0$ we have
		\begin{align*}
			|G^\varepsilon(x,\tilde{x}) - G^0(x,\tilde{x}) | \leq C \varepsilon^s, \quad \mbox{ for all } \tilde{x}\in \overline{\mathbb{H}}_n \quad \mbox{with }\quad 1/3\leq |x-\tilde{x}|\leq 5/3.
		\end{align*}
		The constant $C$ in this estimate depends on $d,\mu$ and $[A]_{C^{0,\eta}} $.  
	\end{proposition}

			\begin{proof}[Proof]
			Let $x\in\overline{\mathbb{H}}_n$ and $\varepsilon>0$ and hold them fixed. We set
			\begin{equation}\label{def_D_x}
				D_x:= \{ \tilde{x}\in\overline{\mathbb{H}}_n : 1/3<  |x-\tilde{x}|<5/3 \}. \qquad \mbox{  (see {\scshape Figure} \ref{domain_D_x})   }
			\end{equation}
			The function $\tilde{x}\mapsto |G^\varepsilon(x,\tilde{x}) - G^0(x,\tilde{x}) |$ is continuous on $ \overline{D}_x$ which is compact. Therefore there exists $\tilde{x}_0\in   \overline{D}_x$ satisfying $	|G^\varepsilon(x,\tilde{x}_0) - G^0(x,\tilde{x}_0) |=	S(x,\varepsilon)$
			with	\begin{align*}
				S(x,\varepsilon):=  \sup \ \{ |G^\varepsilon(x,\tilde{x}) - G^0(x,\tilde{x}) |: \tilde{x} \in  \overline{D}_x \}.
			\end{align*}
			From the estimates 
			\begin{align*}
				| \nabla G^\varepsilon(x,\tilde{x})| \leq \; \frac{ C}{|x-\tilde{x}|^{d-1}}, \quad  | \nabla G^0(x,\tilde{x})| \leq \; \frac{ C}{|x-\tilde{x}|^{d-1}}, \quad  \mbox{ for all } x,\tilde{x}\in\mathbb{H}_n ,
			\end{align*} 
			with $C=C(d,\mu,[A]_{C^{0,\eta}}),$
			we derive the existence of a constant $C_1=C_1(d,\mu,[A]_{C^{0,\eta}} )>0$ such that 
			\begin{align*}
				|\nabla G^\varepsilon(x,\tilde{x})| +  |\nabla G^0(x,\tilde{x})| \leq C_1, \quad  \mbox{ for all }   \tilde{x} \in A(x,1/6,2):=\{ z\in\overline{\mathbb{H}}_n : 1/6<  |x-\tilde{x}|<2 \}.
			\end{align*}
			Now let  $\rho := \frac{S(x,\varepsilon)}{2 d^2 C_1}$ and suppose $\rho\in(0,1/6).$  In fact one can always increase $C_1$ so that $\rho<1/6$. With such a $\rho$ one has
			\begin{align*}
				B(\tilde{x}_0,\rho) \subset B(x,2)\bs \overline{B(x,1/6)}.
			\end{align*}	
		
			\begin{figure}
			\centering
			\includegraphics[scale=0.3]{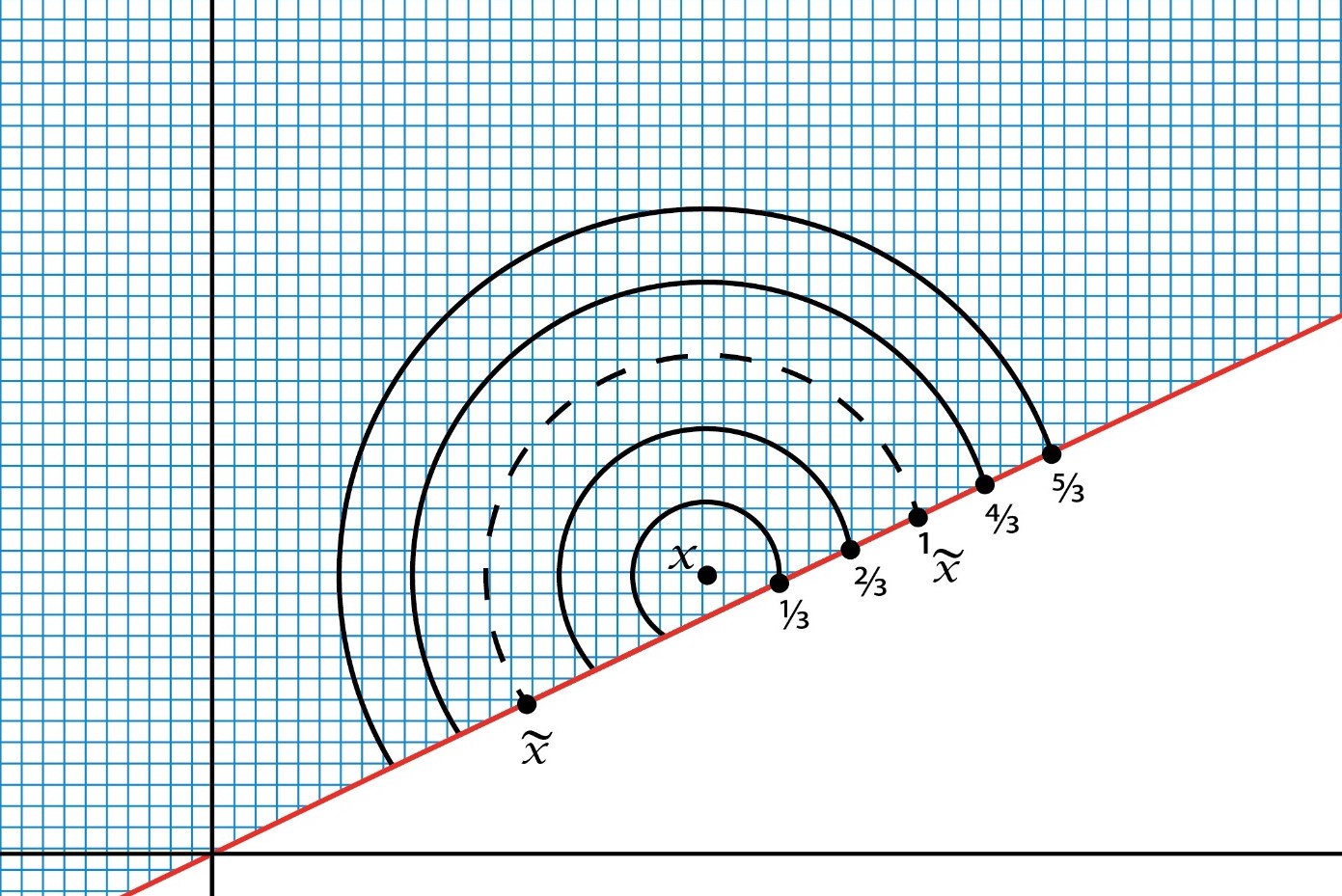}
			\caption{Domain $D_x$.}
			\label{domain_D_x}
		\end{figure}
			
			\noindent {\bfseries Step 1- lower bound : } \label{ij}	
			There exist $i,j\in \{ 1,...,d\}$ such that   for all   $\tilde{x} \in  D(\tilde{x}_0,\rho)$ we have
			\begin{align}\label{GijS}
				| G_{ij}^\varepsilon(x,\tilde{x})- G_{ij}^0(x,\tilde{x})| \geq \frac{S(x,\varepsilon)}{2d^2}.
			\end{align}
			
			\noindent The norm of matrices that we consider here is the maximum norm. From the equality
			\begin{align*}
				| G^\varepsilon(x,\tilde{x_0})- G^0(x,\tilde{x}_0)|=S(x,\varepsilon)
			\end{align*}
			we obtain the existence of $i,j\in\{ 1,...,d\}$ satisfying 
			\begin{align}
				| G_{ij}^\varepsilon(x,\tilde{x}_0)- G_{ij}^0(x,\tilde{x}_0)|\geq \frac{S(x,\varepsilon)}{d^2}.\label{low_bound_G}
			\end{align}
			Now let such integers $i,j$ be fixed. For all $\tilde{x}\in D(\tilde{x}_0, \rho)$ we have either 	$| G_{ij}^\varepsilon(x,\tilde{x})- G_{ij}^0(x,\tilde{x})|\geq \frac{S(x,\varepsilon)}{2d^2}$ or 	$| G_{ij}^\varepsilon(x,\tilde{x})- G_{ij}^0(x,\tilde{x})|< \frac{S(x,\varepsilon)}{2d^2}$. We show that the latter condition leads to a contradiction. Indeed, it implies 
			\begin{align*}
				0<	\frac{S(x,\varepsilon)}{d^2} -| G_{ij}^\varepsilon(x,\tilde{x})- G_{ij}^0(x,\tilde{x})|
			\end{align*}
			and by use of 
			inequality (\ref{low_bound_G}) we have
			\begin{align*}
				\frac{S(x,\varepsilon)}{d^2} -| G_{ij}^\varepsilon(x,\tilde{x})- G_{ij}^0(x,\tilde{x})|&\leq	| G_{ij}^\varepsilon(x,\tilde{x}_0)- G_{ij}^0(x,\tilde{x}_0)|-| G_{ij}^\varepsilon(x,\tilde{x})- G_{ij}^0(x,\tilde{x})|\\
				&\leq	| G_{ij}^\varepsilon(x,\tilde{x}_0)-G_{ij}^\varepsilon(x,\tilde{x}) |+|G_{ij}^0(x,\tilde{x}_0) - G_{ij}^0(x,\tilde{x})|\\
				&\leq	\left\{  \| \nabla G^\varepsilon(x,\cdot) \|_{L^\infty(A(x,1/6,2))} + \| \nabla G^0(x,\cdot) \|_{L^\infty(A(x,1/6,2))} \right\}|\tilde{x}_0-\tilde{x}|  \\
				&\leq	C_1 \rho =\frac{S(x,\varepsilon)}{2d^2} .
			\end{align*}
			So the condition 	$| G_{ij}^\varepsilon(x,\tilde{x})- G_{ij}^0(x,\tilde{x})|< \frac{S(x,\varepsilon)}{2d^2}$ yields \begin{align*}
				0<\frac{S(x,\varepsilon)}{d^2} -| G_{ij}^\varepsilon(x,\tilde{x})- G_{ij}^0(x,\tilde{x})|\leq	\frac{S(x,\varepsilon)}{2d^2} 
			\end{align*}
			which implies (\ref{GijS}).\\
			
			\noindent	{\bfseries Step 2 :} Let $\varphi\in C^1_c (B(0,1))	$ with $0\leq \varphi \leq 1$ and $\varphi(\tilde{x})=1$ for $|\tilde{x}|< 1/2$ and define $\varphi_{\tilde{x}_0, \rho}$ by $\varphi_{\tilde{x}_0, \rho} (\tilde{x})=\varphi(\frac{\tilde{x} - \tilde{x}_0}{\rho})$. Then we have 
			\begin{align}
				\varphi_{\tilde{x}_0, \rho}\in C^1_c (B(\tilde{x}_0,\rho))\quad \mbox{ and }	\quad  0 \leq \varphi_{\tilde{x}_0, \rho}\leq 1,\quad \| \nabla \varphi_{\tilde{x}_0, \rho}\|_{L^\infty(\mathbb{H}_n)} \leq \frac{C}{\rho} . \label{bound_phi}
			\end{align}
			Let $i,j$ be the integers given by Lemma \ref{ij}. The inequality (\ref{GijS}) together with intermediate value theorem imply that  $ G_{ij}^\varepsilon(x,\tilde{x})- G_{ij}^0(x,\tilde{x})$ has a constant sign on $D(\tilde{x}_0, \rho)$. This sign can be assumed to be positive, $ G_{ij}^\varepsilon(x,\tilde{x})- G_{ij}^0(x,\tilde{x})\geq 0$. Then (\ref{GijS}) implies
			\begin{align}
				G_{ij}^\varepsilon(x,\tilde{x})- G_{ij}^0(x,\tilde{x})\geq \frac{S(x,\varepsilon)}{2d^2} ,  \mbox{ for all  } \tilde{x} \in  D(\tilde{x}_0,\rho).
			\end{align}
			Now consider the solutions $u^\varepsilon$ and $u^0$ of the respective problems (\ref{hpstoks}) and (\ref{hphom0bv}) with source term $f$ defined by $f(\tilde{x})= \varphi_{\tilde{x}_0, \rho} (\tilde{x}) e_j$. Using Green representations of both solutions one obtains 
			\begin{align*}
				(u^\varepsilon (x) -u^0(x))_i &= \int_{\mathbb{H}_n} \{(G^\varepsilon(x,\tilde{x})- G^0(x,\tilde{x}) ) f(\tilde{x}) \}_i  d\tilde{x}\\
				&= \int_{D(\tilde{x}_0,\rho)} (G_{ij}^\varepsilon(x,\tilde{x})- G_{ij}^0(x,\tilde{x}))\varphi_{\tilde{x}_0, \rho} (\tilde{x})  d\tilde{x}
			\end{align*}
			and since $\varphi_{\tilde{x}_0, \rho}(\tilde{x})=1$, for $|\tilde{x}- \tilde{x}_0 | < \rho /2 $, one deduces
			\begin{align*}
				(u^\varepsilon (x) -u^0(x))_i 
				&\geq \int_{D(\tilde{x}_0,\rho/2)} G_{ij}^\varepsilon(x,\tilde{x})- G_{ij}^0(x,\tilde{x}) \;  d\tilde{x}\\
				&\geq  \int_{D(\tilde{x}_0,\rho/2)}  \frac{S(x,\varepsilon)}{2d^2} d\tilde{x}.
			\end{align*}
			Then by the definition of $\rho$	we get
			\begin{align*}
				(u^\varepsilon (x) -u^0(x))_i 
				&\geq  \int_{D(\tilde{x}_0,\rho/2)}  C_1 \rho \; d\tilde{x} = C_2 \rho (\rho/2)^{d} = C_2 \rho^{d+1}.
			\end{align*}	
			In summary, we have 
			\begin{align}
				(u^\varepsilon (x) -u^0(x))_i \geq C_2 \rho^{d+1}\label{boundrho1}
			\end{align}	with	$C_2=C_2(d,\mu,[A]_{C^{0,\eta}})>0$. 
			Rewriting  estimate (\ref{estimL8uepsu0}) with this specific $f$  yields 
			\begin{align*}
				\|u^{\varepsilon} - u^0\|_{L^\infty(\mathbb{H}_n)} &\leq C\, \varepsilon \| f \|_{W^{1,d/2+\delta}(\mathbb{H}_n)} \\
				&\leq C \, \varepsilon \left\{  \|  \varphi_{\tilde{x}_0, \rho}\|_{L^{d/2+\delta}(\mathbb{H}_n)} +\| \nabla \varphi_{\tilde{x}_0, \rho} \|_{L^{d/2+\delta}(\mathbb{H}_n)}   \right\}    \\
				&\leq C\, \varepsilon \left\{  \left(  \int_{D(\tilde{x}_0,\rho)}\|  \varphi_{\tilde{x}_0, \rho}\|^{d/2+\delta } _\infty d\tilde{x}  \right)^{\frac{1}{d/2+\delta  }}   +  \left(\int_{D(\tilde{x}_0,\rho)}    \| \nabla \varphi_{\tilde{x}_0, \rho} \|^{d/2+\delta }_\infty d\tilde{x} \right)^{\frac{1}{d/2+\delta  }} \right\}  
			\end{align*}
			for $\delta\in(0,d/2)$ and $C=C(d,\mu,[A]_{C^{0,\eta}}, \delta  )$.
			Then one employs the bounds of $\varphi_{\tilde{x}_0, \rho}$ and of its derivative given in (\ref{bound_phi})
			\begin{align*}
				\|u^{\varepsilon} - u^0\|_{L^\infty(\mathbb{H}_n)} 
				&\leq C \, \varepsilon \left\{  \left(  \int_{D(\tilde{x}_0,\rho)}\!\! 1\, d\tilde{x}  \right)^{\frac{1}{d/2+\delta}}   +  \left(\int_{D(\tilde{x}_0,\rho)}  \frac{C}{\rho} \, d\tilde{x} \right)^{\frac{1}{d/2+\delta }} \right\}    \\
				&\leq C \, \varepsilon \left\{  \left( C\rho ^d  \right)^{\frac{1}{d/2+\delta }}   +  \left( C\rho^{d-1} \right)^{\frac{1}{d/2+\delta }} \right\}    \\
				&\leq C \, \varepsilon \rho^{\frac{2d-2}{d+2\delta}}  .
			\end{align*}	
			The above lines are summarized in
			\begin{align}
				\|u^{\varepsilon} - u^0\|_{L^\infty(\mathbb{H}_n)} 
				&\leq C  \varepsilon \rho^{\frac{2d-2}{2\delta +d }} \label{boundrho2}
			\end{align}with $C=C (d,\mu,[A]_{C^{0,\eta}}, \delta  )>0$  	with  $\delta\in(0,d/2)$.
			
			\noindent 	Connecting inequality (\ref{boundrho1}) to (\ref{boundrho2}) we get 
			\begin{align}
				C_2 \rho^{d+1} \leq\|u^{\varepsilon} - u^0\|_{L^\infty(\mathbb{H}_n)} 
				&\leq C \varepsilon \rho^{\frac{2d-2}{2\delta +d }}  
			\end{align}
			from which  we extract 
			\begin{align}
				C_2 \rho^{d+1}   &\leq C  \varepsilon \rho^{\frac{2d-2}{2\delta +d }}  \nonumber
			\end{align}		
			which in turn yields
			\begin{align}\label{exposant_s}
				\rho &\leq C \varepsilon^s, \: \mbox{ with } s:= \frac{1}{d+\frac{2+2\delta -d}{2\delta +d }} 	 
			\end{align}		
			with the constant $C$ depending on $d,\mu,[A]_{C^{0,\eta}}$ and  $\delta$. Recalling the definition of 	$\rho$ one sees that
			\begin{align}
				S(x,\varepsilon)= 2d^2	C_1 \rho \leq C  \varepsilon^s\nonumber
			\end{align} with $C=C(d,\mu,[A]_{C^{0,\eta}},\delta)$.

			\noindent	By the definition of $	S(x,\varepsilon)$ this last inequality implies 
			\begin{align*}
				|G^{\varepsilon}(x,\tilde{x}) - G^0(x,\tilde{x}) | \leq C \varepsilon^s, \quad \mbox{ for all } \tilde{x}\in D_x
				\mbox{ and for all } s\in \left(\frac{1}{d+\frac{1}{d}},  \frac{1}{d-1 +\frac{2}{d}}\right)
			\end{align*}
			where 	 $C=C(d,\mu,[A]_{C^{0,\eta}} ) $ .
		\end{proof}
		\begin{remark}
			For $d\geq 3$ the range of $s$ in \textup{(\ref{exposant_s})} includes $1/d$ and the case $s=1/d$ corresponds to $\delta=d/2-1$. Moreover, Proposition \textup{\ref{propo_GepsG0}} continues to hold  for the adjoint Green functions i.e. we also have 
			\begin{align*}
				|G^{*,\varepsilon}(\tilde{x},x) - G^{*,0}(\tilde{x},x) | \leq C \varepsilon^s, \quad \mbox{ for all } (x,\tilde{x})\in     \overline{\mathbb{H}}_n  \times  \overline{\mathbb{H}}_n   \mbox{ with } 1/3<|x-\tilde{x}|<5/3\end{align*}where  $C=C(d,\mu,[A]_{C^{0,\eta}} ) $.
		\end{remark}

	\section{Asymptotic expansion of the Poisson kernel $P=P(y,\tilde{y})$}\label{section_expans}
	\noindent 	In this section we aim at establishing an asymptotic expansion of the Poisson kernel $P=P(y,\tilde{y}$)  as $|y-\tilde{y}|\rightarrow \infty$. Seeing that this kernel is expressed in terms of  Green's functions, its expansion can be obtained via those of Green's functions. In a first step we establish expansions of 	$(\nabla G^{*,\varepsilon},\Pi^{*,\varepsilon})$ in powers of $\varepsilon$. Then, in a next step we apply the uniform estimate (\ref{unif_estim_thGZ})  to control these expansions.\\

	\noindent By restricting the domain  in (\ref{system_Greeneps})  the adjoint Green's functions 
	$(G^{*,\varepsilon}(\cdot, x),\Pi^{*,\varepsilon}(\cdot, x))$  satisfy 
	\begin{equation}\label{probGreenD}
		\left\{	\begin{array}{rcl}
			-\nabla\cdot A^* (\tilde{x}/\varepsilon) \, \nabla_{\tilde{x}} G^{*,\varepsilon}(\tilde{x},x) +  \nabla_{\tilde{x}} \Pi^{*,\varepsilon}(\tilde{x},x)&=& 0,  \quad \tilde{x}\in D_x,\\
			\nabla_{\tilde{x}} \cdot G^*_\varepsilon(\tilde{x},x)&=& 0,  \quad \tilde{x}\in D_x,\\
			G^*_\varepsilon(\tilde{x},x)&=& 0  ,  \quad \tilde{x}\in \Gamma_x,\\
		\end{array}
		\right. 
	\end{equation}
	where $\Gamma_x :=  D_x \cap\partial\mathbb{H}_n$ and $D_x$ is defined in (\ref{def_D_x}).
	
	\subsection {Expansion of  Green's functions}
	
	\newcommand{\pbar}{	\overline{p}^{*,\beta}_j}
	
	\noindent We apply the expansions (\ref{expansu}), (\ref{expansp}) to Green's functions $(G^{*,\varepsilon}(\cdot, x),\Pi^{*,\varepsilon}(\cdot, x))$  satisfying (\ref{probGreenD}) and by doing so we include boundary layer correctors $(u_{bl}^*,p_{bl}^*):$
	\begin{align}
		G_{ki}^{*,\varepsilon} (\tilde{x},x)= G_{ki}^{*,0} (\tilde{x},x)&+ \varepsilon	\chi^{*,\beta}_{kj}(\tilde{x}/\varepsilon)\partial_{\tilde{x}_{\beta}} G_{ji}^{*,0}  (\tilde{x},x) + \varepsilon	(u_{bl}^*)^{\beta}_{kj}(\tilde{x}/\varepsilon)\partial_{\tilde{x}_{\beta}} G_{ji}^{*,0}  (\tilde{x},x)	\nonumber\\
		& + \varepsilon^2 \Gamma^{*,\alpha\beta}_{kj}(\tilde{x}/\varepsilon) \partial_{\tilde{x}_{\alpha}}\partial_{\tilde{x}_\beta} G_{ji}^{*,0}  (\tilde{x},x)\nonumber\\&+ \varepsilon^2	\chi^{*,\alpha}_{kl}(\tilde{x}/\varepsilon)(u_{bl}^*)^{\beta}_{lj}	(\tilde{x}/\varepsilon)  \partial_{\tilde{x}_{\alpha}}\partial_{\tilde{x}_\beta}G_{ji}^{*,0}  (\tilde{x},x)+	W_{ki}^{\varepsilon} (\tilde{x}),\label{expansG}	\\\nonumber\\
		\Pi_i^{*,\varepsilon} (\tilde{x},x)=  \Pi_i^{*,0} (\tilde{x},x)&+	\pi^{\beta}_{j}(\tilde{x}/\varepsilon)\partial_{\tilde{x}_\beta} G_{ji}^{*,0}   (\tilde{x},x)	 + (p_{bl}^*)^{\beta}_{j}(\tilde{x}/\varepsilon)  \partial_{\tilde{x}_\beta}G_{ji}^{*,0}  (\tilde{x},x)\nonumber\\
		&	+ \varepsilon Q^{*,\alpha\beta}_{j}(\tilde{x}/\varepsilon) \partial_{\tilde{x}_{\alpha}}\partial_{\tilde{x}_\beta}G_{ji}^{*,0}  (\tilde{x},x)\nonumber\\& +  \varepsilon	\pi^{*,\alpha}_{k}(\tilde{x}/\varepsilon)	(u^*_{bl})^{\beta}_{kj}(\tilde{x}/\varepsilon)\partial_{\tilde{x}_{\alpha}}\partial_{\tilde{x}_\beta} G_{ji}^{*,0} (\tilde{x},x)+	S_{i}^{\varepsilon} (\tilde{x})\label{expansPi}
	\end{align}	
	where functions $\chi^*, \pi^*, \Gamma^*$ and $Q^*$ are the interior correctors for the adjoint operator $\mathcal{L}_1^*$ given in Remark \ref{Rmk_starcorrectors}	and for fixed $\beta,j\in\{ 1,...,d\}$ the function  $(u^*_{bl})^{\beta}_{j}= ((u^*_{bl})^{\beta}_{kj})_k$  together with $    (p_{bl}^*)^{\beta}_{j} $ are boundary layer correctors solving 
	\begin{equation}\label{u*p*}
		\left\{
		\begin{array}{rcll}
			-\nabla\cdot A^*(\tilde{y}) \: \nabla (u^*_{bl})^{\beta}_{j}  +  \nabla(p_{bl}^*)^{\beta}_{j}  &=& 0   \quad &\mbox{ in } {\mathbb{H}}_n,\\
			\nabla\cdot(u^*_{bl})^{\beta}_{j} &=& 0  \quad  &\mbox{ in } {\mathbb{H}}_n,\\
			(u^*_{bl})^{\beta}_{j}  &=& -\chi_j^{*,\beta}\quad & \mbox{ on } \partial{\mathbb{H}}_n.
		\end{array}
		\right. 
	\end{equation}
	\begin{remark}
		By the regularity of $A$ and $\chi$ we have 
		\begin{align}\label{bound_Linf_u}
			\|u^*_{bl} \|_{L^\infty(\overline{\mathbb{H}}_n)} \leq C(d, \mu, [A]_{C^{0,\eta}})<\infty.
		\end{align}
	\end{remark}

	\noindent Notice that the remainder terms $W^\varepsilon$, $S^\varepsilon$ in the respective expansions are
	\begin{align}
		W_{ki}^{\varepsilon} (\tilde{x}):=	G_{ki}^{*,\varepsilon} (\tilde{x},x)-& G_{ki}^{*,0} (\tilde{x},x)- \varepsilon	\chi^{*,\beta}_{kj}(\tilde{x}/\varepsilon)\partial_{\tilde{x}_{\beta}} G_{ji}^{*,0}  (\tilde{x},x) - \varepsilon	(u^*_{bl})^{\beta}_{kj}(\tilde{x}/\varepsilon)\partial_{\tilde{x}_{\beta}} G_{ji}^{*,0}  (\tilde{x},x)	\nonumber\\
		-&  \varepsilon^2 \Gamma^{*,\alpha\beta}_{kj}(\tilde{x}/\varepsilon) \partial_{\tilde{x}_{\alpha}}\partial_{\tilde{x}_\beta} G_{ji}^{*,0}  (\tilde{x},x)\nonumber\\ -& \varepsilon^2	\chi^{*,\alpha}_{kl}(\tilde{x}/\varepsilon)(u^*_{bl})^{\beta}_{lj}	(\tilde{x}/\varepsilon)  \partial_{\tilde{x}_{\alpha}}\partial_{\tilde{x}_\beta}G_{ji}^{*,0}  (\tilde{x},x)\label{errorW} 
	\end{align} 
	and 
	\begin{align}
		S_i^{\varepsilon} (\tilde{x}):=	\Pi_i^{*,\varepsilon} (\tilde{x},x)-&  \Pi_i^{*,0} (\tilde{x},x)-	\pi^{\beta}_{j}(\tilde{x}/\varepsilon)\partial_{\tilde{x}_\beta} G_{ji}^{*,0}   (\tilde{x},x)- (p_{bl}^*)^{\beta}_{j}(\tilde{x}/\varepsilon)\partial_{\tilde{x}_\beta}G_{ji}^{*,0}  (\tilde{x},x)\nonumber\\
		-	&	 \varepsilon Q^{*,\alpha\beta}_{j}(\tilde{x}/\varepsilon) \partial_{\tilde{x}_{\alpha}}\partial_{\tilde{x}_\beta}G_{ji}^{*,0}  (\tilde{x},x)\nonumber \\ -&  \varepsilon	\pi^{*,\alpha}_{k}(\tilde{x}/\varepsilon)	(u^*_{bl})^{\beta}_{kj}(\tilde{x}/\varepsilon)\partial_{\tilde{x}_{\alpha}}\partial_{\tilde{x}_\beta} G_{ji}^{*,0} (\tilde{x},x)	\label{errorP}.
	\end{align}

	\subsubsection{Estimates on the error terms  $ (\nabla W^{\varepsilon}, S^{\varepsilon}) $ }
	
	\noindent Here
	we find out the equations satisfied by  the pair $(W^{\varepsilon}, S^{\varepsilon}) $.
Taking into account the fact that  	 $$[\Le^*( G_{i}^{*,\varepsilon} (\tilde{x},x)) + \nabla	\Pi_i^{*,\varepsilon} (\tilde{x},x)]_r =0$$ one obtains
	the following  equation satisfied by
	the pair $ (W^{\varepsilon}, S^{\varepsilon}) $
	\begin{equation*}
		[\Le^*(W_i^\varepsilon) +\nabla S_i^\varepsilon]_r	= F^\varepsilon +F_{bl}^{\varepsilon}, \quad   \tilde{x}\in D_x,
	\end{equation*}	
	where	 
	\begin{align*}
		F^\varepsilon:= \varepsilon^{-1} F
		^{(-1)} +F^{(0)}+\varepsilon F^{(1)} + \varepsilon^2 F^{(2)}\quad \mbox{and }\quad
		F^{\varepsilon}_{bl}:= \varepsilon^{-1} F^{(-1)}_{bl} +F^{(0)}_{bl}+\varepsilon F^{(1)}_{bl} + \varepsilon^2 F^{(2)}_{bl}.
	\end{align*}
Let us analyze each order separately. First we have
	\begin{align*}
		F^{(-1)}:=& [\{\partial_{\tilde{y}_\gamma} ( A^{*,\gamma\delta}_{rk}(\tilde{y})\partial_{\tilde{y}_\delta}\chi^{*,\beta}_{kj}(\tilde{y})) -\partial_{\tilde{y}_r}\pi^{*,\beta}_{j}(\tilde{y}) \} \partial_{\tilde{x}_\beta}  G_{ji}^{*,0}] (\tilde{x},x)\\
		&+[ \partial_{\tilde{y}_\gamma}  A^{*,\gamma\delta}_{rk}(\tilde{y})\partial_{\tilde{x}_\delta}  G_{ki}^{*,0} (\tilde{x},x) ]	(\tilde{x},\tilde{x}/\varepsilon)\\
		=& [\{\partial_{\tilde{y}_\gamma} ( A^{*,\gamma\delta}_{rk}(\tilde{y})\partial_{\tilde{y}_\delta}\chi^{*,\beta}_{kj}(\tilde{y})) -\partial_{\tilde{y}_r}\pi^{*,\beta}_{j}(\tilde{y})  
		+ \partial_{\tilde{y}_\gamma}  A^{*,\gamma\beta}_{rj}(\tilde{y}) \}\partial_{\tilde{x}_\beta}  G_{ji}^{*,0} (\tilde{x},x) ]	(\tilde{x},\tilde{x}/\varepsilon)
	\end{align*}
	and 
		\begin{align*}
		F^{(-1)}_{bl}:=& [\{\partial_{\tilde{y}_\gamma} ( A^{*,\gamma\delta}_{rk}(\tilde{y})\partial_{\tilde{y}_\delta}(u^*_{bl})^{\beta}_{kj}(\tilde{y})) -\partial_{\tilde{y}_r}(p_{bl}^*)^{\beta}_{j}(\tilde{y}) \} \partial_{\tilde{x}_\beta}  G_{ji}^{*,0} (\tilde{x},x)]	(\tilde{x},\tilde{x}/\varepsilon).
	\end{align*}
		These terms of order $\varepsilon^{-1}$ are both zero, $	F^{(-1)}=	F^{(-1)}_{bl}=0$, 	according to equations (\ref{chipi}) and (\ref{u*p*}).
	 Some of the terms of order 0 in $\varepsilon$ involve boundary layer quantities $ (u^*_{bl})^{\beta}_{kj}, (p_{bl}^*)^{\beta}_{j}$ whereas others do not. We shall gather the terms following this difference. The quantity $F^{(0)}$ gathers the terms of order $0$ in $\varepsilon$ without a boundary layer term:
	\begin{align*}
		F^{(0)}:=&	\partial_{\tilde{x}_\gamma} ( A^{*,\gamma\delta}_{rk}(\tilde{y})\partial_{\tilde{x}_\delta}  G_{ki}^{*,0} (\tilde{x},x))- \partial_{\tilde{x}_r} \Pi_i^{*,0} (\tilde{x},x)+	\partial_{\tilde{y}_\gamma} ( A^{*,\gamma\delta}_{rk}(\tilde{y}) \chi^{*,\beta}_{kj}(\tilde{y}))\partial_{\tilde{x}_\delta}\partial_{\tilde{x}_{\beta}} G_{ji}^{*,0} (\tilde{x},x)\\		
		&+	  A^{*,\gamma\delta}_{rk}(\tilde{y})\partial_{\tilde{y}_\delta} \chi^{*,\beta}_{kj}(\tilde{y})\partial_{\tilde{x}_\gamma}\partial_{\tilde{x}_{\beta}} G_{ji}^{*,0} (\tilde{x},x)- \pi^{*,\beta}_{j}(\tilde{y}) \partial_{\tilde{x}_r}\partial_{\tilde{x}_{\beta}} G_{ji}^{*,0}  (\tilde{x},x)\\
		&+\partial_{\tilde{y}_\gamma} ( A^{*,\gamma\delta}_{rk}(\tilde{y})\partial_{\tilde{y}_\delta}\Gamma^{*,\alpha\beta}_{kj}(\tilde{y}))\partial_{ \tilde{x}_{\alpha}} \partial_{\tilde{x}_{\beta}} G_{ji}^{*,0}  (\tilde{x},x)	-   \partial_{\tilde{y}_r} Q^{*,\alpha\beta}_{j}(\tilde{y})\partial_{\tilde{x}_{\alpha}} \partial_{\tilde{x}_{\beta}} G_{ji}^{*,0}  (\tilde{x},x).
	\end{align*}
	Using the fact that $ \nabla\cdot A^{*,0}\,\nabla G_i^{*,0}= \nabla \Pi_i^{*,0}$  in $D_x$ and some change of indices we get 
	\begin{align*}
		F^{(0)}=&	\partial_{\tilde{x}_\alpha} ( A^{*,\alpha\beta}_{rk}(\tilde{y})\partial_{\tilde{x}_\beta}  G_{ki}^{*,0} (\tilde{x},x))- 	\partial_{\tilde{x}_\alpha}  (A^{*,0})^{\alpha\beta}_{rk}\partial_{\tilde{x}_\beta}  G_{ki}^{*,0} (\tilde{x},x) +	\partial_{\tilde{y}_\gamma} ( A^{*,\gamma\alpha}_{rk}(\tilde{y}) \chi^{*,\beta}_{kj}(\tilde{y}))\partial_{\tilde{x}_\alpha}\partial_{\tilde{x}_{\beta}} G_{ji}^{*,0} (\tilde{x},x)\\		
		&+	 A^{*,\alpha\delta}_{rk}(\tilde{y})\partial_{\tilde{y}_\delta} \chi^{*,\beta}_{kj}(\tilde{y})\partial_{\tilde{x}_\alpha}\partial_{\tilde{x}_{\beta}} G_{ji}^{*,0} (\tilde{x},x)- \pi^{*,\beta}_{j}(\tilde{y}) \partial_{\tilde{x}_r}\partial_{\tilde{x}_{\beta}} G_{ji}^{*,0}  (\tilde{x},x)\\
		&+\partial_{\tilde{y}_\gamma} ( A^{*,\gamma\delta}_{rk}(\tilde{y})\partial_{\tilde{y}_\delta}\Gamma^{*,\alpha\beta}_{kj}(\tilde{y}))\partial_{ \tilde{x}_{\alpha}} \partial_{\tilde{x}_{\beta}} G_{ji}^{*,0}  (\tilde{x},x))	-   \partial_{\tilde{y}_r} Q^{*,\alpha\beta}_{j}(\tilde{y})\partial_{ \tilde{x}_{\alpha}} \partial_{\tilde{x}_{\beta}} G_{ji}^{*,0}  (\tilde{x},x).
	\end{align*}
	Then, factorizing the second-order derivative of the Green function yields
	\begin{align*}
		F^{(0)}
		=&	 \left( A^{*,\alpha\beta}_{rk}(\tilde{y})+	 A^{\alpha\delta}_{rk}(\tilde{y})\partial_{\tilde{y}_\delta} \chi^{*,\beta}_{kj}(\tilde{y})+\partial_{\tilde{y}_\gamma}(  A^{*,\gamma\alpha}_{rk}(\tilde{y}) \chi^{*,\beta}_{kj}(\tilde{y}))-  (A^{*,0})^{\alpha\beta}_{rk}\right)\partial_{\tilde{x}_\alpha}\partial_{\tilde{x}_{\beta}} G_{ji}^{*,0} (\tilde{x},x)\\
		&- \pi^{*,\beta}_{j}(\tilde{y}) \partial_{\tilde{x}_r}\partial_{\tilde{x}_{\beta}} G_{ji}^{*,0}  (\tilde{x},x)\\
		&+\partial_{\tilde{y}_\gamma} ( A^{*,\gamma\delta}_{rk}(\tilde{y})\partial_{\tilde{y}_\delta}\Gamma^{*,\alpha\beta}_{kj}(\tilde{y}))\partial_{ \tilde{x}_{\alpha}} \partial_{\tilde{x}_{\beta}} G_{ji}^{*,0}  (\tilde{x},x)	-   \partial_{\tilde{y}_r} Q^{*,\alpha\beta}_{j}(\tilde{y})\partial_{ \tilde{x}_{\alpha}} \partial_{\tilde{x}_{\beta}} G_{ji}^{*,0}  (\tilde{x},x).
	\end{align*}
	Now one can see that $F^{(0)}=0$ by the first equation of (\ref{GamQ}).\\
	
	\noindent The terms of order 0 in $\varepsilon$ involving the boundary layer correctors
	$u^*_{bl},p^*_{bl} $  are gathered in
	\begin{align*}
		F^{(0)}_{bl}:=	F^{(0)}_{bl,0} + F^{(0)}_{bl,1}+ F^{(0)}_{bl,2}+F^{(0),*}_{bl}.
	\end{align*}
	The term $F^{(0)}_{bl,0}$ corresponds to the quantities involving zeroth-order derivatives of  $u^*_{bl}$ and it is 
	\begin{align*}
		F^{(0)}_{bl,0}=	[	\partial_{\tilde{y}_\gamma} (& A^{*,\gamma\delta}_{rk}(\tilde{y})\partial_{\tilde{y}_\delta} \chi^{*,\alpha}_{kl}(\tilde{y}))](u^*_{bl})^{\beta}_{lj}	(\tilde{y})\partial_{ \tilde{x}_{\alpha}}\partial_{\tilde{x}_{\beta}} G_{ji}^{*,0}  (\tilde{x},x) -  \partial_{\tilde{y}_r}	\pi^{*,\alpha}_{l}(\tilde{y})(u^*_{bl})^{\beta}_{lj}	(\tilde{y})\partial_{\tilde{x}_\alpha}\partial_{\tilde{x}_{\beta}} G_{ji}^{*,0} (\tilde{x},x)   \\
		+	&	\partial_{\tilde{y}_\gamma}  A^{*,\gamma\delta}_{rk}(\tilde{y}) (u^*_{bl})^{\beta}_{kj}(\tilde{y})\partial_{\tilde{x}_\delta}\partial_{\tilde{x}_{\beta}} G_{ji}^{*,0} (\tilde{x},x).
	\end{align*}
	By some change of indices and factorization of $(u^*_{bl})^{\beta}_{lj}(\tilde{y})\partial_{\tilde{x}_\alpha}\partial_{\tilde{x}_{\beta}} G_{ji}^{*,0} (\tilde{x},x)$ we find
	\begin{align*}
		F^{(0)}_{bl,0}
		&= \{	\partial_{\tilde{y}_\gamma} ( A^{*,\gamma\delta}_{rk}(\tilde{y})\partial_{\tilde{y}_\delta} \chi^{*,\alpha}_{kl}(\tilde{y})) -  \partial_{\tilde{y}_r}	\pi^{*,\alpha}_{l}(\tilde{y})	
		+	\partial_{\tilde{y}_\gamma}  A^{*,\gamma\alpha}_{rl}(\tilde{y})\} (u^*_{bl})^{\beta}_{lj}(\tilde{y})\partial_{\tilde{x}_\alpha}\partial_{\tilde{x}_{\beta}} G_{ji}^{*,0} (\tilde{x},x)\\
		&=0
	\end{align*}
	by the first equation of (\ref{chipi}).
	\noindent Then the quantity $F^{(0)}_{bl}$ reduces to 
	\begin{align}\label{F0bl}
		F^{(0)}_{bl}:= F^{(0)}_{bl,1}+ F^{(0)}_{bl,2}+F^{(0),*}_{bl}	
	\end{align}
	where the term $F^{(0)}_{bl,1}$ relates to first-order derivatives of  $(u^*_{bl})^{\beta}_{j}$ and it reads
	\begin{align*}
		F^{(0)}_{bl,1}	:=	&	[  A^{*,\alpha\delta}_{rk}(\tilde{y})+	  A^{*,\delta\alpha}_{rk}(\tilde{y})]\partial_{\tilde{y}_\delta} (u^*_{bl})^{\beta}_{kj}(\tilde{y})\partial_{\tilde{x}_\alpha}\partial_{\tilde{x}_{\beta}} G_{ji}^{*,0} (\tilde{x},x)\nonumber\\
		& +[	\partial_{\tilde{y}_\gamma} ( A^{*,\gamma\delta}_{rk}(\tilde{y}) \chi^{*,\alpha}_{kl}(\tilde{y})) + A^{*,\delta\gamma}_{rk}(\tilde{y})\partial_{\tilde{y}_\gamma} \chi^{*,\alpha}_{kl}(\tilde{y})]	\partial_{\tilde{y}_\delta}(u^*_{bl})^{\beta}_{lj}	(\tilde{y})\partial_{ \tilde{x}_{\alpha}}\partial_{\tilde{x}_{\beta}} G_{ji}^{*,0}  (\tilde{x},x)\nonumber \\
		&-  	\pi^{*,\alpha}_{l}(\tilde{y})	\partial_{\tilde{y}_r} (u^*_{bl})^{\beta}_{lj}(\tilde{y}) \partial_{ \tilde{x}_{\alpha}}\partial_{\tilde{x}_{\beta}}G_{ji}^{*,0} (\tilde{x},x),
	\end{align*}
	the quantity $F^{(0)}_{bl,2}$ involves a second-order derivative of  $(u^*_{bl})^{\beta}_{lj}$ 
	\begin{align*}
		F^{(0)}_{bl,2}	:= A^{*,\gamma\delta}_{rk}(\tilde{y}) \chi^{*,\alpha}_{kl}(\tilde{y})	\partial_{\tilde{y}_\gamma}\partial_{\tilde{y}_\delta}(u^*_{bl})^{\beta}_{lj}	(\tilde{y})\partial_{ \tilde{x}_{\alpha}}\partial_{\tilde{x}_{\beta}} G_{ji}^{*,0}  (\tilde{x},x) 
	\end{align*}		
	and $F^{(0),*}_{bl}$ is  the term containing the pressure component of the boundary layer	
	\begin{align}
		F^{(0),*}_{bl}:=	& - (p_{bl}^*)^{\beta}_{j}(\tilde{x}/\varepsilon)  \partial_{\tilde{x}_r}\partial_{\tilde{x}_{\beta}}G_{ji}^{*,0}  (\tilde{x},x).\nonumber
	\end{align}
	The terms of  order 1 in $\varepsilon$ without boundary layer terms form
	\begin{align}
		F^{(1)}:=& \;  [A^{*,\gamma\alpha}_{rk}(\tilde{y})	\chi^{*,\beta}_{kj}(\tilde{y})+ \partial_{\tilde{y}_\delta} ( A^{*,\delta\gamma}_{rk}(\tilde{y})\Gamma^{*,\alpha\beta}_{kj}(\tilde{y}))+ A^{*,\gamma\delta}_{rk}(\tilde{y})\partial_{\tilde{y}_\delta}\Gamma^{*,\alpha\beta}_{kj}(\tilde{y})] \partial_{\tilde{x}_\gamma}\partial_{ \tilde{x}_{\alpha}} \partial_{\tilde{x}_{\beta}} G_{ji}^{*,0}  (\tilde{x},x)\nonumber\\&		-    Q^{*,\alpha\beta}_{j}(\tilde{y})\partial_{\tilde{x}_r}\partial_{ \tilde{x}_{\alpha}} \partial_{\tilde{x}_{\beta}} G_{ji}^{*,0}  (\tilde{x},x).\label{F1}
	\end{align}
	The quantities of  order 1 in $\varepsilon$ involving boundary layer terms are gathered in
	\begin{align}\label{F1bl}
		F^{(1)}_{bl}:=\;	& F^{(1)}_{bl,0}+ F^{(1)}_{bl,1}
	\end{align}
	with $	F^{(1)}_{bl,0}$ involving the quantity $(u^*_{bl})^{\beta}_{lj}$  without any of its derivatives
	\begin{align*}
		F^{(1)}_{bl,0}:=\;	&  [A^{*,\gamma\alpha}_{rk}(\tilde{y})(u^*_{bl})^{\beta}_{kj}	(\tilde{y})		
		+ A^{*,\gamma\delta}_{rk}(\tilde{y})\partial_{\tilde{y}_\delta}\chi^{*,\beta}_{kl}(\tilde{y})
		+ \partial_{\tilde{y}_\delta} (A^{*,\delta\gamma}_{rk}(\tilde{y})	\chi^{*,\beta}_{kl}(\tilde{y}))](u^*_{bl})^{\beta}_{lj}(\tilde{y})\partial_{\tilde{x}_\gamma}\partial_{\tilde{x}_\alpha}\partial_{\tilde{x}_{\beta}} G_{ji}^{*,0}  (\tilde{x},x)\nonumber\\	
		&	-\pi^{*,\beta}_{k}(\tilde{y})(u^*_{bl})^{\beta}_{kj}(\tilde{y})\partial_{\tilde{x}_r}\partial_{ \tilde{x}_{\alpha}} \partial_{\tilde{x}_{\beta}} G_{ji}^{*,0}  (\tilde{x},x)
	\end{align*}
	and  $	F^{(1)}_{bl,0}$ containing first-order derivatives of $  (u^*_{bl})^{\beta}_{lj}$
	\begin{align*}
		F^{(1)}_{bl,1}:=\;		
		& [A^{*,\delta\gamma}_{rk}(\tilde{y})	\chi^{*,\beta}_{kl}(\tilde{y}) + A^{*,\gamma\delta}_{rk}(\tilde{y})\chi^{*,\beta}_{kl}(\tilde{y})]\partial_{\tilde{y}_\delta}(u^*_{bl})^{\beta}_{lj}(\tilde{y}) \partial_{\tilde{x}_\gamma}\partial_{ \tilde{x}_{\alpha}} \partial_{\tilde{x}_{\beta}} G_{ji}^{*,0}  (\tilde{x},x).
	\end{align*}
	Finally, the terms of order 2 in $\varepsilon$ are 
	\begin{align}
		F^{(2)}:=&	  A^{*,\gamma\delta}_{rk}(\tilde{x}/\varepsilon)\Gamma^{*,\alpha\beta}_{kj}(\tilde{x}/\varepsilon)\partial_{\tilde{x}_\gamma}\partial_{\tilde{x}_\delta}\partial_{ \tilde{x}_{\alpha}} \partial_{\tilde{x}_{\beta}} G_{ji}^{*,0}  (\tilde{x},x)\label{F2}\\
		\mbox{and }\quad 	F^{(2)}_{bl}:=	&   A^{*,\gamma\delta}_{rk}(\tilde{x}/\varepsilon) \chi^{*,\alpha}_{kl}(\tilde{x}/\varepsilon)(u^*_{bl})^{\beta}_{lj}	(\tilde{x}/\varepsilon)	\partial_{\tilde{x}_\gamma}\partial_{\tilde{x}_\delta}\partial_{ \tilde{x}_{\alpha}}\partial_{\tilde{x}_{\beta}} G_{ji}^{*,0}  (\tilde{x},x).\label{F2bl}\\\nonumber
	\end{align}
	As for  the equation of the divergence, let us first remind that 
	\begin{align*}
		\nabla_{\tilde{x}} \cdot G^{*,\varepsilon}(\cdot,x)=\nabla_{\tilde{x}}  \cdot G^{*,0}(\cdot,x)=\nabla_{\tilde{y}} \cdot (u^*_{bl})^{\beta}_{j} =\nabla_{\tilde{y}} \cdot \chi^{*,\beta}_{j}= \partial_{y_k} \chi^{*,\beta}_{kj}= \partial_{y_k} (u^*_{bl})^{\beta}_{kj}=0.
	\end{align*}
	Then a straightforward computation consisting mainly of expanding the right-hand-side of
	\begin{align*}
		\partial_{\tilde{x}_k} W_{ki}^{\varepsilon}=& 	\partial_{\tilde{x}_k}  \left( G_{ki}^{*,\varepsilon} (\tilde{x},x) -  G_{ki}^{*,0} (\tilde{x},x)-\varepsilon	\chi^{*,\beta}_{kj}(\tilde{x}/\varepsilon)\partial_{\tilde{x}_{\beta}} G_{ji}^{*,0}  (\tilde{x},x) -\varepsilon	(u^*_{bl})^{\beta}_{kj}(\tilde{x}/\varepsilon)\partial_{\tilde{x}_{\beta}} G_{ji}^{*,0}  (\tilde{x},x)\right. \\
		&\qquad\qquad \left. - \varepsilon^2 \Gamma^{*,\alpha\beta}_{kj}(\tilde{x}/\varepsilon)\partial_{ \tilde{x}_{\alpha}}\partial_{\tilde{x}_{\beta}} G_{ji}^{*,0}  (\tilde{x},x) -\varepsilon^2	\chi^{*,\alpha}_{kl}(\tilde{x}/\varepsilon)(u^*_{bl})^{\beta}_{lj}	(\tilde{x}/\varepsilon)\partial_{ \tilde{x}_{\alpha}}\partial_{\tilde{x}_{\beta}} G_{ji}^{*,0}  (\tilde{x},x) \right)
	\end{align*}
	gives
	\begin{align*}
		\partial_{\tilde{x}_k} W_{ki}^{\varepsilon}
		=& - \varepsilon\left(	\chi^{*,\beta}_{kj}(\tilde{x}/\varepsilon)\partial_{\tilde{x}_k} \partial_{\tilde{x}_{\beta}} G_{ji}^{*,0}  (\tilde{x},x) +\partial_{\tilde{y}_k}\Gamma^{*,\alpha\beta}_{kj}(\tilde{x}/\varepsilon) \partial_{ \tilde{x}_{\alpha}}\partial_{\tilde{x}_{\beta}} G_{ji}^{*,0}  (\tilde{x},x)\right) \\
		&-\varepsilon \left(	(u^*_{bl})^{\beta}_{kj}(\tilde{x}/\varepsilon)\partial_{\tilde{x}_k}\partial_{\tilde{x}_{\beta}} G_{ji}^{*,0}  (\tilde{x},x)+  \chi^{*,\alpha}_{kl}(\tilde{x}/\varepsilon)\partial_{\tilde{y}_k}(u^*_{bl})^{\beta}_{lj}	(\tilde{x}/\varepsilon) \partial_{ \tilde{x}_{\alpha}}\partial_{\tilde{x}_{\beta}}G_{ji}^{*,0}  (\tilde{x},x) \right)  \\
		& - \varepsilon^2 \left(  \Gamma^{*,\alpha\beta}_{kj}(\tilde{x}/\varepsilon)\partial_{\tilde{x}_k}\partial_{ \tilde{x}_{\alpha}}\partial_{\tilde{x}_{\beta}} G_{ji}^{*,0}  (\tilde{x},x) + 	\chi^{*,\alpha}_{kl}(\tilde{x}/\varepsilon)(u^*_{bl})^{\beta}_{lj}	(\tilde{x}/\varepsilon)\partial_{\tilde{x}_k} \partial_{ \tilde{x}_{\alpha}}\partial_{\tilde{x}_{\beta}}G_{ji}^{*,0}  (\tilde{x},x) \right).
	\end{align*}
	The first line on the right hand side is zero according to the second equation of  (\ref{GamQ}). Then we obtain
	\begin{align}
		H^{\varepsilon}(\tilde{x}):=	&-\varepsilon \left(	(u^*_{bl})^{\beta}_{kj}(\tilde{x}/\varepsilon)\partial_{\tilde{x}_k}\partial_{\tilde{x}_{\beta}} G_{ji}^{*,0}  (\tilde{x},x)+  \chi^{*,\alpha}_{kl}(\tilde{x}/\varepsilon)\partial_{\tilde{y}_k}(u^*_{bl})^{\beta}_{lj}	(\tilde{x}/\varepsilon) \partial_{ \tilde{x}_{\alpha}}\partial_{\tilde{x}_{\beta}}G_{ji}^{*,0}  (\tilde{x},x) \right)\nonumber  \\
		& - \varepsilon^2 \left(  \Gamma^{*,\alpha\beta}_{kj}(\tilde{x}/\varepsilon)\partial_{\tilde{x}_k}\partial_{ \tilde{x}_{\alpha}}\partial_{\tilde{x}_{\beta}} G_{ji}^{*,0}  (\tilde{x},x) + 	\chi^{*,\alpha}_{kl}(\tilde{x}/\varepsilon)(u^*_{bl})^{\beta}_{lj}	(\tilde{x}/\varepsilon)\partial_{\tilde{x}_k}\partial_{ \tilde{x}_{\alpha}}\partial_{\tilde{x}_{\beta}} G_{ji}^{*,0}  (\tilde{x},x) \right). \label{h_divW}
	\end{align}
Using the cancellations \; $ G_{ki}^{*,\varepsilon} (\tilde{x},x)=	\chi^{*,\beta}_{kj}(\tilde{x}/\varepsilon) +(u^*_{bl})^{\beta}_{kj}(\tilde{x}/\varepsilon)=0 =  G_{ki}^{*,0} (\tilde{x},x)$,\; for $\tilde{x}\in \partial\mathbb{H}_n$,\;
	the boundary value is $\varphi^\varepsilon = (\varphi^\varepsilon_{k})_k$ defined by	\begin{align}
		\varphi_k^{\varepsilon}(\tilde{x}):=
		& - \varepsilon^2 \Gamma^{*,\alpha\beta}_{kj}(\tilde{x}/\varepsilon)\partial_{ \tilde{x}_{\alpha}}\partial_{\tilde{x}_{\beta}} G_{ji}^{*,0}  (\tilde{x},x) - \varepsilon^2	\chi^{*,\alpha}_{kl}(\tilde{x}/\varepsilon)(u^*_{bl})^{\beta}_{lj}	(\tilde{x}/\varepsilon)\partial_{ \tilde{x}_{\alpha}}\partial_{\tilde{x}_{\beta}}G_{ji}^{*,0}  (\tilde{x},x).\label{g_bv_of_W}
	\end{align}
	To summarize, 		the pair $ (W^{\varepsilon}\!, S^{\varepsilon}) $ satisfies the system of equations
	\begin{equation}
		\label{Pb_WS}	\left\{
		\begin{array}{rcll}
			-\nabla\cdot A^* (\tilde{x}/\varepsilon) \, \nabla W^{\varepsilon}  +  \nabla S^{\varepsilon}&=&F^{(0)}_{bl}+ \varepsilon\left( F^{(1)}+ F^{(1)}_{bl} \right)  +\varepsilon^2 \left(F^{(2)}+F^{(2)}_{bl}\right),   &\tilde{x}\in D_x,\\ 
			\nabla \cdot W^{\varepsilon} &=& H^\varepsilon (\tilde{x})  ,&\tilde{x}\in D_x,\\ \vspace{0.1in} 
			W^{\varepsilon}&=& \varphi^\varepsilon(\tilde{x}),&\tilde{x}\in \Gamma_x,
		\end{array}
		\right. 
	\end{equation}	
	with the functions $ F^{(0)}_{bl}, F^{(1)}, F^{(1)}_{bl} ,F^{(2)}, F^{(2)}_{bl},  h^\varepsilon$ and $g^\varepsilon$  defined respectively  as in (\ref{F0bl}), (\ref{F1}), (\ref{F1bl}), (\ref{F2}), (\ref{F2bl}), (\ref{h_divW}) and (\ref{g_bv_of_W}).\\

	\noindent 	Now we apply the uniform Lipschitz estimate of Theorem \ref{unif_estim} to the  system (\ref{Pb_WS}). To that end, we introduce the domain  $$D^0_x:= \{ \tilde{x}\in\overline{\mathbb{H}}_n : 2/3<  |x-\tilde{x}|<4/3 \} \subset D_x.$$
	\noindent Theorem  \ref{unif_estim}  yields the following inequality
	\begin{align}
		\|\nabla W^{\varepsilon}\|_{L^\infty(D_x^0)} + \osc_{D_x^0}  \;    S^{\varepsilon}
		\leq   C  \left\{ \|W^{\varepsilon} \|_{L^\infty(D_x)} \right. &+ \| F^{(0)}_{bl} \|_{L^{d+l}(D_x)}+     \| H^\varepsilon \|_{L^\infty(D_x)} +  [H^\varepsilon]_{C^{0,\eta}(D_x)} \nonumber \\
		&+ \varepsilon  \,(\| F^{(1)} \|_{L^{d+l}(D_x)}+ \| F^{(1)}_{bl} \|_{L^{d+l}(D_x)} )\label{unif_estim_applied}\\
		&+\varepsilon^2 (\| F^{(2)} \|_{L^{d+l}(D_x)}+ \| F^{(2)}_{bl} \|_{L^{d+l}(D_x)} )\nonumber\\ 
		&+\| \varphi^\varepsilon \|_{L^\infty(\Gamma_x)} + \| \nabla_{\!\!\textup{ tan}} \, \varphi^\varepsilon \|_{L^\infty(\Gamma_x)} + \left.  [\nabla_{\!\!\textup{ tan}}\; \varphi^\varepsilon]_{C^{0,\eta}(\Gamma_x)} \right\}\nonumber \nonumber
	\end{align} for $l>0$ and 	where $C$ depends on $d, l, \eta, \mu$ and $[A]_{C^{0,\eta}(D_x)}$ .
	We aim at showing that the right hand side of this inequality is of size $\varepsilon^\kappa$ for some $\kappa>0$. Therefore, in what follows we are going to establish a bound of the form $C \varepsilon^\kappa$ to each of the terms  in the right hand  side of (\ref{unif_estim_applied}). Let us begin with the term $	\| W^{\varepsilon}\|_{L^\infty(D_x)}$.
	\begin{lemma}
		We have the following estimate for $W^\varepsilon$:
		\begin{align*}
			\| W^{\varepsilon}\|_{L^\infty(D_x)}	 \leq   C  \varepsilon^{1/d}, 
		\end{align*}where the constant $C$ depends essentially on $A$ and on $d, \mu$ and $[A]_{C^{0,\eta}}$.
	\end{lemma}
	\begin{proof}[Proof]
		The term  $\| W^{\varepsilon}\|_{L^\infty(D_x)}$  is bounded as follows. By the definition of $W^{\varepsilon}$ and by the triangle inequality, we have
		\begin{align*}
			| W^{\varepsilon} (\tilde{x})|\leq& \;|G^{*,\varepsilon} (\tilde{x},x)-G^{*,0} (\tilde{x},x)|+ \varepsilon |\chi^*(\tilde{x}/\varepsilon) + u_{bl}^*(\tilde{x}/\varepsilon)|\cdot|\nabla G^{*,0}(\tilde{x},x)|\\
			&+\varepsilon^2 |\Gamma^*(\tilde{x}/\varepsilon) \nabla^2 G^{*,0} (\tilde{x},x)|+\varepsilon^2 |\chi^* (\tilde{x}/\varepsilon)u_{bl}^*(\tilde{x}/\varepsilon) \nabla^2 G^{*,0} (\tilde{x},x)|, \quad \mbox{ for all } \tilde{x}\in D_x.
		\end{align*}
		Hence 
		\begin{align*}
			\| W^{\varepsilon} \|_{L^\infty(D_x)}\leq \:&   \|G^{*,\varepsilon} (\cdot,x)-G^{*,0} (\cdot,x)\|_{L^\infty(D_x)}+ \varepsilon \, (\|\chi^*\|_\infty + \|u^*_{bl}\|_\infty) \|\nabla G^{*,0}(\cdot,x)\|_{L^\infty(D_x)} \\
			&\varepsilon^2 \left(\|\Gamma ^*\|_\infty   + \|\chi^*\|_\infty  \|u_{bl}^*\|_\infty\right) \|\nabla G^{*,0}(\cdot,x)\|_{L^\infty(D_x)}. 
		\end{align*}
		Therefore, by Proposition \ref{propo_GepsG0} together with the Remark that follows it, and by the boundednes of the correctors given in Lemma \ref{bound_on_correctors} along with the facts  that $ \nabla G^{*,0}(\cdot,x)$ is bounded in $D_x$ and  $ \|u_{bl}^*\|_\infty \leq C(d, \mu, [A]_{C^{0,\eta}})<\infty$ , we get
		\begin{align*}
			\| W^{\varepsilon} \|_{L^\infty(D_x)}&\leq C\varepsilon^{1/d} 
		\end{align*}
		with $C$ depending  on $d, \mu$ and $[A]_{C^{0,\eta}}$.
	\end{proof}
	
	\begin{remark}\label{bound_inf_derG0} Here we would like to clarify some tools used in the preceding proof and that we might use again.
		\begin{enumerate}
			\item We implicitely use the bound $\|\partial^\lambda G^{*,0}(\cdot,x) \|_{L^\infty(D_x)} \leq \sup_{\tilde{x}\in D_x} \frac{C}{|\tilde{x} -x|^{d-2+|\lambda|}} =C(1/3)^{d-2+|\lambda|}$ for a multi-index $\lambda\in \mathbb{N}^d,$ with $|\lambda|=1,2,3,4,$ throughout our study. The  constant $C $ in this estimate depends on $d, \mu$ and $[A]_{C^{0,\eta}}.$
			\item In the previous  proof and on other occasions we use abbreviations of the type $$ \chi^*\cdot \nabla G^{*,0},\quad    u_{bl}^* \cdot \nabla G^{*,0},\quad \Gamma^*:\nabla^2 G^{*,0},\quad \chi^* u_{bl}^*: \nabla^2 G^{*,0}.$$
			\noindent	Let us give the meanings	of these notations
			\begin{align*}
				\chi^*\cdot \nabla G^{*,0}&=\chi^{*,\beta}_{kj}\partial_{\tilde{x}_{\beta}} G_{ji}^{*,0} ,\qquad    u_{bl}^* \cdot \nabla G^{*,0}=(u^*_{bl})^{\beta}_{kj}\partial_{\tilde{x}_{\beta}} G_{ji}^{*,0} \\
				\Gamma^*:\nabla^2 G^{*,0}&=\Gamma^{*,\alpha\beta}_{kj}\partial_{ \tilde{x}_{\alpha}}\partial_{\tilde{x}_{\beta}} G_{ji}^{*,0}  , \qquad   Q^*:\nabla^2 G^{*,0}=Q^{*,\alpha\beta}_{j}\partial_{ \tilde{x}_{\alpha}}\partial_{\tilde{x}_{\beta}} G_{ji}^{*,0} \\ \chi^* u_{bl}^*: \nabla^2 G^{*,0}&=		\chi^{*,\alpha}_{kl}(u^*_{bl})^{\beta}_{lj}	\partial_{ \tilde{x}_{\alpha}}\partial_{\tilde{x}_{\beta}}G_{ji}^{*,0},\quad   \pi^{*}	u_{bl}^{*}:\nabla^2 G_{}^{*,0} =\pi^{*,\alpha}_{k}	(u^*_{bl})^{\beta}_{kj}	:\partial_{ \tilde{x}_{\alpha}}\partial_{\tilde{x}_{\beta}}G_{ji}^{*,0} .\\
			\end{align*}
			
		\end{enumerate} 
	\end{remark}
	\begin{lemma}
		We have the following estimate for $H^\varepsilon$ where $H^\varepsilon$ is defined in \textup{(\ref{h_divW})}
		\begin{align*}
			\| H^\varepsilon \|_{L^\infty(D_x)} +  [H^\varepsilon]_{C^{0,\eta}(D_x)} \leq   C  \varepsilon^{1-\eta}
		\end{align*}where $C $ depends on $d$, $\mu$ and $[A]_{C^{0,\eta}}$.
	\end{lemma}
	\begin{proof}[Proof] We have
		\begin{align*}
			H^{\varepsilon}(\tilde{x}):=	&-\varepsilon \left((u^*_{bl})^{\beta}_{kj}	(\tilde{x}/\varepsilon)\partial_{\tilde{x}_k}\partial_{\tilde{x}_{\beta}} G_{ji}^{*,0}  (\tilde{x},x)+  \chi^{*,\alpha}_{kl}(\tilde{x}/\varepsilon)\partial_{\tilde{y}_k}(u^*_{bl})^{\beta}_{lj}	(\tilde{x}/\varepsilon)\partial_{\tilde{x} _{\alpha}}\partial_{\tilde{x}_{\beta}} G_{ji}^{*,0}  (\tilde{x},x) \right)\nonumber  \\
			& - \varepsilon^2 \left(  \Gamma^{*,\alpha\beta}_{kj}(\tilde{x}/\varepsilon)\partial_{\tilde{x}_k}\partial_{\tilde{x} _{\alpha}}\partial_{\tilde{x}_{\beta}} G_{ji}^{*,0}  (\tilde{x},x) + 	\chi^{*,\alpha}_{kl}(\tilde{x}/\varepsilon)(u^*_{bl})^{\beta}_{lj}		(\tilde{x}/\varepsilon)\partial_{\tilde{x}_k}\partial_{\tilde{x} _{\alpha}}\partial_{\tilde{x}_{\beta}} G_{ji}^{*,0}  (\tilde{x},x) \right).
		\end{align*}	
	By  Lemma \ref{bound_on_correctors} 	and (\ref{bound_Linf_u}) this identity implies
		\begin{align*}
			\| H^\varepsilon \|_{L^\infty(D_x)}   \leq \: & \varepsilon  \| u_{bl}^{*}(\cdot /\varepsilon)\nabla^2 G_{ji}^{*,0}  (\cdot,x) \|_{L^\infty(D_x)} +   \varepsilon \| \chi^{*}(\cdot/\varepsilon)\nabla u_{bl}^{*}	(\cdot/\varepsilon)\nabla G_{ji}^{*,0}  (\tilde{x},x) \|_{L^\infty(D_x)}\\
			& +  \varepsilon^2 \|\Gamma^*(\cdot/\varepsilon) \nabla^3 G_{ji}^{*,0}  (\cdot,x) \|_{L^\infty(D_x)} + \varepsilon^2 \|\chi(\cdot/\varepsilon)u_{bl}^*(\cdot/\varepsilon) \nabla^3 G_{ji}^{*,0} (\cdot,x) \|_{L^\infty(D_x)} \\
			\leq \: & C\varepsilon.
		\end{align*}
		
		\noindent Now we estimate the H\"older semi-norm	of $H^\varepsilon$
		\begin{align*}		
			[H^\varepsilon]_{C^{0,\eta}(D_x)}	\leq &\: \varepsilon  [ u_{bl}^{*}(\cdot /\varepsilon)\nabla^2 G_{ji}^{*,0}  (\cdot,x) ]_{C^{0,\eta}(D_x)} +   \varepsilon [ \chi^{*}(\cdot/\varepsilon)\nabla u_{bl}^{*}	(\cdot/\varepsilon)\nabla G_{ji}^{*,0}  (\tilde{x},x) ]_{C^{0,\eta}(D_x)}\\
			& +  \varepsilon^2 [\Gamma^*(\cdot/\varepsilon) \nabla^3 G_{ji}^{*,0}  (\cdot,x) ]_{C^{0,\eta}(D_x)} + \varepsilon^2 [\chi^*(\cdot/\varepsilon)u_{bl}^*(\cdot/\varepsilon) \nabla^3 G_{ji}^{*,0} (\cdot,x) ]_{C^{0,\eta}(D_x)} .
		\end{align*}
		These semi-norms in the right hand side are estimated as follows : e.g. for $ [ u_{bl}^{*}(\cdot /\varepsilon)\nabla^2 G_{ji}^{*,0}  (\cdot,x) ]_{C^{0,\eta}(D_x)}$
		\begin{align*}		
			[ u_{bl}^{*}(\cdot /\varepsilon)\nabla^2 G_{ji}^{*,0}  (\cdot,x) ]_{C^{0,\eta}(D_x)} \leq& \: [ u_{bl}^{*}(\cdot /\varepsilon) ]_{C^{0,\eta}(D_x)} \| \nabla^2 G_{ji}^{*,0}  (\cdot,x)\|_{L^\infty(D_x)}\\&+ \|u_{bl}^{*}(\cdot /\varepsilon)\|_{L^\infty(D_x)}  [ \nabla^2 G_{ji}^{*,0}  (\cdot,x) ]_{C^{0,\eta}(D_x)}\\
			\leq&\: (1/\varepsilon)^\eta[ u_{bl}^{*} ]_{C^{0,\eta}(\frac{1}{\varepsilon}D_x)} \| \nabla^2 G_{ji}^{*,0}  (\cdot,x)\|_{L^\infty(D_x)}\\&+ \|u_{bl}^{*}\|_{L^\infty(\frac{1}{\varepsilon}D_x)}  [ \nabla^2 G_{ji}^{*,0}  (\cdot,x) ]_{C^{0,\eta}(D_x)}
		\end{align*}
		where $\frac{1}{\varepsilon}D_x:=\{\tilde{y}= \tilde{x}/\varepsilon : \tilde{x}\in D_x\}$. We then obtain
		\begin{align*}		
			[ u_{bl}^{*}(\cdot /\varepsilon)\nabla^2 G_{ji}^{*,0}  (\cdot,x) ]_{C^{0,\eta}(D_x)} 
			\leq &\: C\varepsilon^{-\eta}
		\end{align*}
		with $C$ depending on $A$ and $d$.	 Similarly, one has
		\begin{align*}		
			[\chi^{*}(\cdot /\varepsilon)  u_{bl}^{*}(\cdot /\varepsilon)\nabla^3 G_{ji}^{*,0}  (\cdot,x) ]_{C^{0,\eta}(D_x)} 
			\leq &\: C\varepsilon^{-\eta}\\
			[ \chi^{*}(\cdot /\varepsilon) \nabla u_{bl}^{*}(\cdot /\varepsilon)\nabla^2 G_{ji}^{*,0}  (\cdot,x) ]_{C^{0,\eta}(D_x)} 
			\leq &\: C\varepsilon^{-\eta}\\
			[ \Gamma^{*}(\cdot /\varepsilon)\nabla^3 G_{ji}^{*,0}  (\cdot,x) ]_{C^{0,\eta}(D_x)} 
			\leq &\: C\varepsilon^{-\eta}
		\end{align*}
		where $C$ depends on $A$ and on $d$.
		Thus we get 
		\begin{align*}		
			[H^\varepsilon]_{C^{0,\eta}(D_x)} 
			\leq &\: C(\varepsilon\varepsilon^{-\eta} + \varepsilon^2 \varepsilon^{-\eta})\\
			\leq 	&\: C \varepsilon^{1-\eta}.\qedhere
		\end{align*}
	\end{proof}
	
	\noindent	As for the boundary data in system (\ref{Pb_WS}) we have the following Lemma
	\begin{lemma}
		We have 
		\begin{align*}
			\| \varphi^\varepsilon \|_{L^\infty(\Gamma_x)}+
			\| \nabla_{\!\!\textup{ tan}} \, \varphi^\varepsilon \|_{L^\infty(\Gamma_x)} +
			[\nabla_{\!\!\textup{ tan}}\; \varphi^\varepsilon]_{C^{0,\eta}(\Gamma_x)}  
			&\leq C \varepsilon^{1-\eta}
		\end{align*}
		where the constant $C$ depends on $A, d,\mu$ and $[A]_{C^{0,\eta}}.$
	\end{lemma}
	
	\begin{proof}
		By Lemma \ref{bound_on_correctors} and Remark \ref{bound_inf_derG0} and estimate (\ref{bound_Linf_u}) we have a straightforward bound on $	g^{\varepsilon}$
		\begin{align*}
			\|	\varphi^{\varepsilon}\|_{L^\infty(\Gamma_x)}=
			& \left\|- \varepsilon^2 \Gamma^{*,\alpha\beta}_{kj}(\cdot/\varepsilon)\partial_{\tilde{x} _{\alpha}}\partial_{\tilde{x}_{\beta}} G_{ji}^{*,0}  (\cdot,x) - \varepsilon^2	\chi^{*,\alpha}_{kl}(\cdot/\varepsilon)(u^*_{bl})^{\beta}_{lj}	(\cdot/\varepsilon)\partial_{\tilde{x} _{\alpha}}\partial_{\tilde{x}_{\beta}} G_{ji}^{*,0}  (\cdot,x)\right\|_{L^\infty(\Gamma_x)}
			\leq  C\varepsilon^2
		\end{align*}	where $C$ depends on $d,\mu$ and $[A]_{C^{0,\eta}}.$	Each term in the derivative of $	\varphi^{\varepsilon}$ is either  of order 1 or of order 2 in $\varepsilon$. 	Thus we have
		\begin{align*}
			\| \nabla_{\!\!\textup{ tan}} \, \varphi^{\varepsilon} \|_{L^\infty(\Gamma_x)} &\leq \| \nabla \varphi^{\varepsilon}  \|_{L^\infty(\Gamma_x)}\
			\leq C\varepsilon
		\end{align*}with $C$ depending on $d,\mu$ and $[A]_{C^{0,\eta}}.$		Finally, the H\"older semi-norm related to the boundary value is bounded as follows
		\begin{align*}
			[\nabla_{\!\!\textup{ tan}}\; \varphi^\varepsilon]_{C^{0,\eta}(\Gamma_x)}  &\leq    \left[\nabla \left(-\varepsilon^2\Gamma^*(\cdot/\varepsilon): \nabla^2 G^{*,0}  (\cdot,x)-\varepsilon^2\chi^* (\cdot/\varepsilon)u^*_{bl}(\cdot/\varepsilon):\nabla^2 G^{*,0}  (\cdot,x) \right) \right]_{C^{0,\eta}(\Gamma_x)}\\
			&\leq  \varepsilon^2  \left[\nabla \left(\Gamma^*(\cdot/\varepsilon): \nabla^2 G^{*,0}  (\cdot,x)\right)\right]_{C^{0,\eta}(\Gamma_x)} \\
			&\qquad	+ \varepsilon^2\left[ \nabla \left(\chi^* (\cdot/\varepsilon)u^*_{bl}(\cdot/\varepsilon): \nabla^2 G^{*,0}  (\cdot,x)\right) \right]_{C^{0,\eta}(\Gamma_x)} .
		\end{align*}
		The terms in the right hand side of this above inequality are bounded as follows:
		\begin{align*}
			\left[\nabla \left(\Gamma^*(\cdot/\varepsilon): \nabla^2 G^{*,0}  (\cdot,x)\right)\right]_{C^{0,\eta}(\Gamma_x)}&
			= [  \varepsilon^{-1}\nabla_y \Gamma^*(\cdot/\varepsilon) :\nabla^2 G^{*,0} (\cdot,x)+ \Gamma^*(\cdot/\varepsilon): \nabla^3 G^{*,0} (\cdot,x)   ]_{C^{0,\eta}(\Gamma_x)}\\
			&\leq      \varepsilon^{-1-\eta}[\nabla_y \Gamma^* ]_{C^{0,\eta}(\frac{1}{\varepsilon}\Gamma_x)}\|\nabla^2 G^{*,0} (\cdot,x)\|_{L^\infty(\Gamma_x)}\\
			&\quad +    \varepsilon^{-1}\| \nabla_y \Gamma^*(\cdot/\varepsilon)\|_{L^\infty(\Gamma_x)} [\nabla^2 G^{*,0} (\cdot,x) ]_{C^{0,\eta}(\Gamma_x)}  \\
			& \quad+\varepsilon^{-\eta}[\Gamma^* ]_{C^{0,\eta}(\frac{1}{\varepsilon}\Gamma_x)}\|\nabla^3 G^{*,0} (\cdot,x)\|_{L^\infty(\Gamma_x)}\\
			&\quad + \| \Gamma^*(\cdot/\varepsilon)\|_{L^\infty(\Gamma_x)} \left[\nabla^3 G^{*,0} (\cdot,x)  \right ]_{C^{0,\eta}(\Gamma_x)}.
		\end{align*}
		So we obtain 
		\begin{align*}
			\varepsilon^2 \left[\nabla \left(\Gamma^*(\cdot/\varepsilon): \nabla^2 G^{*,0}  (\cdot,x)\right)\right]_{C^{0,\eta}(\Gamma_x)}  &\leq    C  \varepsilon^{1-\eta},
		\end{align*}where $C $ depends on $d,\mu$ and $[A]_{C^{0,\eta}}.$
		Similarly we have
		\begin{align*}
			\varepsilon^2	\left[\nabla \left(  \chi^* (\cdot/\varepsilon)u^*_{bl}(\cdot/\varepsilon)\!:\! \nabla^2 G^{*,0}  (\cdot,x)  \right)\right]_{C^{0,\eta}(\Gamma_x)}
			&\leq    C  \varepsilon^{1-\eta},
		\end{align*}with $C $ depending on $d,\mu$ and $[A]_{C^{0,\eta}}.$ 
	Consequently, we finally get 
		\begin{align*}
			[\nabla_{\!\!\textup{ tan}}\; \varphi^\varepsilon]_{C^{0,\eta}(\Gamma_x)}  &\leq    C   \varepsilon^{1-\eta} ,
		\end{align*}
		where $C$ depends on $d,\mu$ and $[A]_{C^{0,\eta}}.$ 	\end{proof}
	
	\noindent	Further, we have the following estimates for  source terms in the first equation of (\ref{Pb_WS}).
	\begin{lemma}
		Functions 	$ F^{(1)},  F^{(1)}_{bl},  F^{(2)},  F^{(2)}_{bl}$ satisfy
		\begin{align*}
			\varepsilon\left( \| F^{(1)}\|_{L^{d+l}(D_x)} +\| F^{(1)}_{bl}\|_{L^{d+l}(D_x)}\right) \leq C\varepsilon
		\end{align*}
		and 
		\begin{align*}
			\varepsilon^2\left( \| F^{(2)}\|_{L^{d+l}(D_x)} +\| F^{(2)}_{bl}\|_{L^{d+l}(D_x)}\right)\leq C\varepsilon^2,
		\end{align*}	where the constant $C$ depends on $ d,\mu$ and $[A]_{C^{0,\eta}}$.
	\end{lemma}	
	\begin{proof}[Proof]
		By Lemma \ref{bound_on_correctors} the functions $A^*$, $\chi^*$, $\pi^*$, $\Gamma^*$ and $Q^*$ and their derivatives are bounded by a constant $C$ depending on $\mu$ and $[A]_{C^{0,\eta}}$. The quantity $\partial_{\tilde{x}_r}\partial_{ \tilde{x}_{\alpha}} \partial_{\tilde{x}_{\beta}} G_{ji}^{*,0}  (\tilde{x},x)$ is bounded by a constant depending on $d,\mu$ and $[A]_{C^{0,\eta}}$ according to Remark \ref{bound_inf_derG0}.	Moreover the domain $D_x$ is bounded. Therefore each term in
		\begin{align}
			F^{(1)}=& \;  [A^{*,\gamma\alpha}_{rk}(\tilde{y})	\chi^{*,\beta}_{kj}(\tilde{y})+ \partial_{\tilde{y}_\delta} ( A^{*,\delta\gamma}_{rk}(\tilde{y})\Gamma^{*,\alpha\beta}_{kj}(\tilde{y}))+ A^{*,\gamma\delta}_{rk}(\tilde{y})\partial_{\tilde{y}_\delta}\Gamma^{*,\alpha\beta}_{kj}(\tilde{y})] \partial_{\tilde{x}_\gamma}\partial_{ \tilde{x}_{\alpha}} \partial_{\tilde{x}_{\beta}} G_{ji}^{*,0}  (\tilde{x},x)\nonumber\\&		-    Q^{*,\alpha\beta}_{j}(\tilde{y})\partial_{\tilde{x}_r}\partial_{ \tilde{x}_{\alpha}} \partial_{\tilde{x}_{\beta}} G_{ji}^{*,0}  (\tilde{x},x)
		\end{align}
		is in $L^{d+l}(D_x)$. For instance,  we have 
		\begin{align*}
			\int_{D_x} \left|A^{*,\gamma\delta}_{rk}(\tilde{x}/\varepsilon) 	\chi^{*,\beta}_{kj}(\tilde{x}/\varepsilon) \partial_{\tilde{x}_\gamma}\partial_{\tilde{x}_\delta}\partial_{\tilde{x}_{\beta}} G_{ji}^{*,0}  (\tilde{x},x)\right|^{d+l} d\tilde{x}\leq C^{d+l}  |D_x| <\infty
		\end{align*}	where the constant $C$ depends on  $d,\mu$ and $[A]_{C^{0,\eta}}$.	Similarly each term of  $F^{(2)}$ or $F^{(2)}_{bl}$ is in 	 $L^{d+l}(D_x)$. 
	\end{proof}	
	
	\noindent Now the remaining terms are those in $F^0_{bl}.$ Using estimates on the derivative $\nabla u_{bl}^*$ of the boundary layer corrector one can show that the $L^{d+l}$ norm of these terms are $O(\varepsilon^\kappa)$, for some $\kappa>0$. In fact, we have the following.
	
	\begin{lemma}\label{estimF0bl}For $l >0,$ we have
		\begin{align*}
			\left\|	F^{(0)}_{bl} \right\|_{L^{d+l}(D_x)} \leq C\varepsilon^{1/(2(d+l))}.
		\end{align*}
		where $F^{(0)}_{bl}$ is defined in \textup{(\ref{F0bl})} and 	where $C$ depends on $d, \mu, [A]_{C^{0,\eta}}$ and $M_2$ defined in \textup{(\ref{M1})}.
	\end{lemma}	
	
	\noindent  We apply the following fact	in the proof of the above Lemma.
	\begin{lemma}\label{LinfL2}
		For $k=1,2$, there exists a constant $C(d,\mu,\eta, [A]_{ C^{0,\eta} },k)>0$, $\eta\in(0,1)$,  such that 
		\begin{align*}
		 \left\|  \| \nabla^k u_{bl}^*(M(z',z_d))\|_{L_{z'}^\infty(\mathbb{R}^{d-1})}  \right\|_{L^2(0,+\infty)} \leq C< +\infty.
		\end{align*}
	\end{lemma}
	\noindent	The proof of Lemma \ref{LinfL2} is given in the Appendix as proofs of Lemma \ref{LinfLq} and Lemma \ref{deriv2_L2}.\\
	
		\noindent The bound in this Lemma can be rewritten more explicitly as follows: for $k=1,2,$ we have
	\begin{align*}
		\int_0^\infty  \| \nabla^k u_{bl}^*(M (z',z_d))\|^2_{L_{z'}^\infty(\mathbb{R}^{d-1})} dz_d =   \int_0^\infty  \sup_{z'\in \mathbb{R}^{d-1}} | \nabla^k u_{bl}^*(M(z',z_d))|^2 dz_d \leq C < +\infty.
	\end{align*}
	In particular, one has 
	\begin{align}\label{M1}
		M_2:= \left( \sup_{h\in\mathbb{Z}^{d-1}} \int_0^\infty \!\!\int_{T_h} |\nabla u^*_{bl}(Mz)|^2 dz \right)^{\!\!\!1/2} \!\!\! <\infty , \quad  T_h:= (0,1)^{d-1}+h, \ h\in\mathbb{Z}^{d-1}.
	\end{align}

	\begin{lemma}\label{estim_nabu_p}
		We have the following estimates on the pair $\left(u_{bl}^*(\cdot/\varepsilon),  p_{bl}^*(\cdot/\varepsilon)\right)$ 
		\begin{align}\label{nab_uespL2}
			\|\nabla u_{bl}^*(\cdot/\varepsilon) \|_{L^2(D_1(0))}\leq  M^{}_2  \,  \varepsilon^{1/2}
		\end{align}
		and
		\begin{align}\label{p_epsL2}
			\|p_{bl}^*(\cdot/\varepsilon) \|_{L^2(D_1(0))}\leq C\varepsilon^{1/2}
		\end{align}
		where $C$ depends on $d, \mu$ and $M_2$.
	\end{lemma}

	\begin{proof} This is an application of Lemma  \ref{estim_Lq_nabu_p}.\\
		\textbf{Step 1-Control of the velocity $ u_{bl}^*(\cdot/\varepsilon) $}\\
		Recall that   $D_{1}(0):= \{ \tilde{y}\in\mathbb{H}_n : |\tilde{y}|< 1 \}$ and  $D_{1/\varepsilon}(0):=  (1/\varepsilon) D_{1}(0)$. On the one hand, we have 
		\begin{align}
			\int_{D_1(0)} |\nabla u_{bl}^*(\tilde{x}/\varepsilon) |^2 d\tilde{x} &=   \varepsilon^d 	\int_{D_{1/\varepsilon}(0)} |\nabla u_{bl}^*(\tilde{y}|^2d\tilde{y}\nonumber\\
			\|\nabla u_{bl}^*(\cdot/\varepsilon) \|_{L^2(D_{1}(0))} & = \varepsilon^{d/2}	\|\nabla u_{bl}^* \|_{L^2(D_{1/\varepsilon}(0))},\label{nab_u_Dx_nab_u_Dy}
		\end{align}
		on the other hand an application of estimate (\ref{nab_uespL2}) with $r=1/\varepsilon$ and $q=2$  yields
			\begin{align*}
			\| \nabla u^{*}_{bl} \|_{L^2(D_{1/\varepsilon}(0))}&\leq M_2\,  \varepsilon^{1/2} \varepsilon^{-d/2}.
		\end{align*}
		Then by the equality (\ref{nab_u_Dx_nab_u_Dy}) we obtain 	
		\begin{align*}
			\|\nabla u_{bl}^*(\cdot/\varepsilon) \|_{L^2(D_1(0))}\leq M_2 \, \varepsilon^{1/2}.
		\end{align*}

		\noindent \textbf{Step 2-Control of the pressure  $p_{bl}^*(\cdot/\varepsilon) $} \\
		First, remark that 
		\begin{align}
			\int_{D_1(0)} |p_{bl}^*(\tilde{x}/\varepsilon) |^2 d\tilde{x} &=   \varepsilon^d 	\int_{D_{1/\varepsilon}(0)} | p_{bl}^*(\tilde{y}|^2d\tilde{y}\nonumber\\
			\|p_{bl}^*(\cdot/\varepsilon) \|_{L^2(D_{1}(0))} & = \varepsilon^{d/2}	\|p_{bl}^* \|_{L^2(D_{1/\varepsilon}(0))}.\label{p_Dx_p_Dy}
		\end{align}
	Then one applies estimate (\ref{estim_Lq_Dr})  with $r=1/\varepsilon$ and $q=2$  to get
		\begin{align*}
			\|p^{*}_{bl} \|_{L^2(D_{1/\varepsilon}(0))}&\leq C\,  \varepsilon^{1/2} \varepsilon^{-d/2}.
		\end{align*}
		Hence, the equality (\ref{p_Dx_p_Dy}) yields
		\begin{align}
			\|p^{*}_{bl}(\cdot/\varepsilon)\|_{L^2(D_1(0))} \leq   C  \varepsilon^{1/2} ,
		\end{align}where $C$ depends on $d,\mu$ and $M_2$.	
	\end{proof}
	
	\noindent In the above proof of Lemma \ref{estim_nabu_p} the size and the center of the domain $D_1(0)$ have the sole role of simplifying the notation. Therefore, by that Lemma we have
	\begin{align*}
		\|\nabla u_{bl}^*(\cdot/\varepsilon) \|_{L^2(D_2(x))}\leq C\varepsilon^{1/2}\quad \mbox{and }\: \|p_{bl}^*(\cdot/\varepsilon) \|_{L^2(D_2(x))}\leq C\varepsilon^{1/2}, \quad \mbox{for all }  x\in\mathbb{H}_n.
	\end{align*}
	Moreover, seeing that $D_x\subset D_2(x)$ (see (\ref{def_D_x})) we also have 
	\begin{align}\label{bound_on_nab_u_and_p_eps}
		\|\nabla u_{bl}^*(\cdot/\varepsilon) \|_{L^2(D_x)}\leq C\varepsilon^{1/2}\quad \mbox{and }\: \|p_{bl}^*(\cdot/\varepsilon) \|_{L^2(D_x)}\leq C\varepsilon^{1/2}.
	\end{align}

	\begin{proof}[Proof of \textup{Lemma 10}]
		\noindent Recall that  $	F^{(0)}_{bl}= F^{(0)}_{bl,1}+ F^{(0)}_{bl,2}+F^{(0)}_{bl,*}	$. The first term $ F^{(0)}_{bl,1}$ contains quantities involving first order derivatives of  $u_{bl}^*$ such as  $A^{*,\gamma\delta}_{rk}(\tilde{y})\partial_{\tilde{y}_\delta}(u^*_{bl})^{\beta}_{kj} (\tilde{y})\partial_{\tilde{x}_\gamma}\partial_{\tilde{x}_{\beta}} G_{ji}^{*,0} (\tilde{x},x)$. They    are estimated as follows. By Lemma \ref{bound_on_correctors} and by Remark \ref{bound_inf_derG0} there is a constant $C$ depending on $d,\mu, [A]_{C^{0,\eta}}$ such that
		\begin{align*}
			\int_{D_x} |A^{*,\gamma\delta}_{rk}(\tilde{x}/\varepsilon)\partial_{\tilde{y}_\delta} (u^*_{bl})^{\beta}_{kj}(\tilde{x}/\varepsilon)\partial_{\tilde{x}_\gamma}\partial_{\tilde{x}_{\beta}} G_{ji}^{*,0} (\tilde{x},x)|d\tilde{x} 
			&\leq C \int_{D_x} |\nabla_{\tilde{y}} u_{bl}^*(\tilde{x}/\varepsilon) | d\tilde{x}.
		\end{align*}
		Now by using Hölder inequality together with (\ref{bound_on_nab_u_and_p_eps}) we get
		\begin{align*}
			\int_{D_x} |\nabla_{\tilde{y}} u_{bl}^*(\tilde{x}/\varepsilon) | d\tilde{x} &\leq  |D_x|^{1/2} \left(	\int_{D_x} |\nabla_{\tilde{y}} u_{bl}^*(\tilde{x}/\varepsilon) |^2 d\tilde{x} \right)^{1/2}\\
			&\leq  |D_x|^{1/2} M_2\, \varepsilon^{1/2}.
		\end{align*}
		We then  have obtained
		\begin{align*}
			\|(A^* \nabla_{\tilde{y}} u_{bl}^*)(\cdot/\varepsilon) \nabla^2 G_{ji}^{*,0}(\cdot,x)\|_{L^1(D_x)} \leq C(d, M_2) \ \varepsilon ^{1/2} .
		\end{align*}
		Now one makes an interpolation of this last estimate with 
		\begin{align*}
			\|(A^* \nabla_{\tilde{y}} u_{bl}^*)(\cdot/\varepsilon) \nabla^2 G_{ji}^{*,0}(\cdot,x)\|_{L^\infty(D_x)} =O(1)
		\end{align*}
		to get
		\begin{align*}
			\|(A^* \nabla_{\tilde{y}} u_{bl}^*)(\cdot/\varepsilon) \nabla^2 G_{ji}^{*,0}(\cdot,x)\|_{L^{d+l}(D_x)} \leq C \varepsilon^{1/(2(d+l))} .
		\end{align*}
		\noindent The term  $F^0_{bl,2}$ is estimated similarly by using Lemma \ref{LinfL2} with $k=2$.\\
		\noindent The term  $F^{(0),*}_{bl}   =     - (p^*_{bl})^{\beta}_{j}(\tilde{x}/\varepsilon) \partial_{\tilde{x}_r}\partial_{\tilde{x}_{\beta}}G_{ji}^{*,0}  (\tilde{x},x)$  is bounded as follows. By Remark \ref{bound_inf_derG0} there exists a constant   $C$  depending on $d,\mu$ and $[A]_{C^{0,\eta}}$ such that 
		\begin{align*}
			\left\|(p^*_{bl})^{\beta}_{j}(\tilde{x}/\varepsilon) \partial_{\tilde{x}_r}\partial_{\tilde{x}_{\beta}}G_{ji}^{*,0}  (\tilde{x},x) \right\|_{L^1(D_x)}
			\leq& \; C \int_{D_x}  \left|p^*_{bl}(\tilde{x}/\varepsilon) \right| d\tilde{x}.
		\end{align*}
		By H\"older's inequality and by estimate (\ref{p_epsL2}) we get
		\begin{align*}
			\int_{D_x}   \left|p^*_{bl}	(\tilde{x}/\varepsilon) \right| d\tilde{x} &\leq  |D_x|^{1/2}\left(  \int_{D_x} \left|p^*_{bl}(\tilde{x}/\varepsilon)\right|^2  d\tilde{x} \right)^{1/2} \\
			&\leq C \varepsilon ^{1/2}	
		\end{align*}
		where $C$ depends on $d,\mu$ and $M_2$.
		Then one ends up with the estimate
		\begin{align*}
			\left\| p^*_{bl}  (\cdot/\varepsilon)\nabla^2 G^{*,0}(\cdot,x) \right\|_{L^1(D_x)}\leq C \varepsilon ^{1/2} 
		\end{align*}
		where $C$ depends on $d, \mu, [A]_{C^{0,\eta}}$ and $M_2$. Making an interpolation between the  above $L^1$ estimate and estimate
		\begin{align*}
			\left\| p^*_{bl}  (\cdot/\varepsilon) \nabla^2 G^{*,0} (\cdot,x) \right\|_{L^\infty(D_x)} =O(1)
		\end{align*}
		one obtains  \!\!\!
		\begin{align*}
			\left\|p^*_{bl}(\cdot/\varepsilon) \nabla^2 G^{*,0} (\cdot,x)\right\|_{L^{d+l}(D_x)} \leq C \varepsilon^{1/(2(d+l))}
		\end{align*}	where $C$ depends on $d, \mu, [A]_{C^{0,\eta}}$ and $M_2$.
	\end{proof}

	\noindent {\bf {Summary.} }\\
	At the beginning we had the  inequality (\ref{unif_estim_applied})
	then we have successively established the following estimates on the terms in the right hand side of (\ref{unif_estim_applied})
	\begin{align*}
		\| W^{\varepsilon}\|_{L^\infty(D_x)}&\leq C\varepsilon^{1/d},  \\
		\| h^\varepsilon \|_{L^\infty(D_x)} +  [h^\varepsilon]_{C^{0,\eta}(D_x)} &\leq C\varepsilon^{1-\eta}, \\
		\| g^\varepsilon \|_{L^\infty(\Gamma_x)} + \| \nabla_{\!\!\textup{ tan}} \, g^\varepsilon \|_{L^\infty(\Gamma_x)} + [\nabla_{\!\!\textup{ tan}}\; g^\varepsilon]_{C^{0,\eta}(\Gamma_x)} &\leq  C\varepsilon^{1-\eta},\\
		\| F^{(0)}_{bl} \|_{L^{d+l}(D_x)}&\leq  C\varepsilon^{1/(2(d+l))},  \\
		\varepsilon  \left(\| F^{(1)} \|_{L^{d+l}(D_x)}+ \| F^{(1)}_{bl} \|_{L^{d+l}(D_x)} \right)&\leq C \varepsilon,\nonumber\\
		\varepsilon^2 \left(\| F^{(2)} \|_{L^{d+l}(D_x)}+ \| F^{(2)}_{bl} \|_{L^{d+l}(D_x)} \right) &\leq C\varepsilon^2,
	\end{align*}
	where $C$ depends on $A, d, \mu, [A]_{C^{0,\eta}}$ and $M_2$.
	The number $\eta\in (0,1)$ can be taken small enough so that $1-\eta \geq \, 1/(2(d+l))$. We then arrive at the conclusion
	\begin{align}
		\|\nabla W^{\varepsilon}\|_{L^\infty(D_x^0)} + \underset{D_x^0}{osc} \;   [S^\varepsilon] \leq   C \varepsilon^\kappa, \mbox{ with } \kappa=1/(2(d+l)),\quad l>0,
	\end{align}
		with $C$ depending on $A, d, \mu, [A]_{C^{0,\eta}}$ and $M_2$.
	
	\subsection{Expansion of the Poisson kernel }
	\noindent We are looking for an asymptotic expansion of the Poisson kernel $P(y,\tilde{y})$ when $|y-\tilde{y}|\rightarrow \infty$ with  $y\in\mathbb{H}_n$, $\tilde{y}\in \partial\mathbb{H}_n$.	Recall that the Poisson kernel is defined as
	\begin{align*}
		P=(P_{ij})_{i,j}\quad \mbox{ with } \quad 	P_{ij}(y,\tilde{y})
		&=-A_{jk}^{*,\alpha\beta}(\tilde{y})\, \partial_{\tilde{y}_\beta}G_{ki}^*(\tilde{y},y) n_\alpha (\tilde{y})+ \Pi_i^*(\tilde{y},y ) n_j(\tilde{y})\\
		&=-A_{kj}^{\beta\alpha}(\tilde{y})\, \partial_{\tilde{y}_\beta}G_{ki}^*(\tilde{y},y) n_\alpha (\tilde{y})+ \Pi_i^*(\tilde{y},y ) n_j(\tilde{y}).
	\end{align*}
	
	\noindent	We may occasionally write it in the compact form
	\begin{align*}
		P(y,\tilde{y})
		&=-A^*(y)\nabla_{\!\tilde{y}}  G^*(\tilde{y},y) \cdot n (\tilde{y})+  \Pi^*(\tilde{y},y ) \otimes n(\tilde{y}).
	\end{align*}
	
	\noindent 	The expression of the Poisson kernel suggests to seek its expansion through the expansions of $\nabla_{\!\!y}  G^*$ and  $\Pi^*$. We are going to express the expansions  $\nabla 	W^{\varepsilon} (\tilde{x},x)\cdot n$ and  $S^{\varepsilon}(\tilde{x})$ given in (\ref{expansG}) and (\ref{expansPi}) in terms of the variables $y,\tilde{y}$. Here the parameter $\varepsilon$ determines the  scale separation of the two variables. More precisely we have : $y=x/\varepsilon,\; \tilde{y}= \tilde{x}/\varepsilon$ and  $\varepsilon=\frac{|x-\tilde{x}|}{|y-\tilde{y}|} \simeq \frac{1}{|y-\tilde{y}|}$ where the distance $|y-\tilde{y}|$ is supposed to be large.
	
	\subsubsection{Rescaling} 
	
	\noindent	We have the following estimate expressed in terms of the slow variables $x, \tilde{x}$
	\begin{align*}
		\|A(\cdot/\varepsilon)\nabla W^{\varepsilon}\|_{L^\infty(D_x^0)}  + \underset{D_x^0}{osc} \,[S^{\varepsilon}]\leq   C \varepsilon^\kappa, \mbox{ with } \kappa\in (0, 1/2d) \mbox{ and } C=C(A, d, \mu, [A]_{C^{0,\eta}},M_2).	\end{align*}
	\noindent Now we are looking for a large-scale version of the above estimate, i.e. we wish to express it in the variables $\tilde{y},y \in \overline{\mathbb{H}}_n$, satisfying $\tilde{y}\in D_{2R}(y)\bs \overline{D_R(y)}$ where $R:=\frac{2}{3}|y-\tilde{y}|.$

	\begin{remark} We have
		$D_{2R}(y) \bs \overline{D_R(y)}=\{\tilde{y}\in \overline{\mathbb{H}}_n:  \tilde{y} =\tilde{x}/ \varepsilon,  \tilde{x}\in D_x^0\}   $.
	\end{remark}

	\noindent	 In what follows we compute the rescaled  functions $W(\tilde{y})=W^\varepsilon(\varepsilon\tilde{y}),S(\tilde{y})=S^\varepsilon(\varepsilon\tilde{y})$.  We proceed by rewriting the expansion  (\ref{errorW})  and taking into account the scaling properties given in Lemma \ref{rescaled_GP}
	\begin{align*}
		W_{ki}^{\varepsilon} (\tilde{x})=&\frac{1}{\varepsilon^{d-2}}G^{*}_{ki}\!\!\left(\frac{\tilde{x}}{\varepsilon},\frac{x}{\varepsilon} \right)-\frac{1}{\varepsilon^{d-2}}G^{*,0}_{ki}\!\!\left(\frac{\tilde{x}}{\varepsilon},\frac{x}{\varepsilon} \right)-\varepsilon \left[\chi^{*,\beta}_{kj}\!\! \left(\frac{\tilde{x}}{\varepsilon} \right)+(u^*_{bl})^{\beta}_{kj} \!\!\left(\frac{\tilde{x}}{\varepsilon} \right)\right]\frac{1}{\varepsilon^{d-1}}\partial_{\tilde{y}_\beta} G^{*,0}_{ji}\!\!\left(\frac{\tilde{x}}{\varepsilon},\frac{x}{\varepsilon} \right)\\
		&-\varepsilon^2 \left[\Gamma^{*,\alpha\beta}_{kj}\left(\frac{\tilde{x}}{\varepsilon} \right)+ \chi^{*,\alpha}_{kl}\left(\frac{\tilde{x}}{\varepsilon} \right)(u^*_{bl})^{\beta}_{lj}\! \left(\frac{\tilde{x}}{\varepsilon} \right)\right] \frac{1}{\varepsilon^{d}} \partial_{\tilde{y}_\alpha} \partial_{\tilde{y}_\beta}G^{*,0}_{ji}\left(\frac{\tilde{x}}{\varepsilon},\frac{x}{\varepsilon} \right) \\
		=&\frac{1}{\varepsilon^{d-2}} \left\{ G_{ki}^* \left(\frac{\tilde{x}}{\varepsilon},\frac{x}{\varepsilon} \right)-G_{ki}^{*,0}\left(\frac{\tilde{x}}{\varepsilon},\frac{x}{\varepsilon} \right)\right.- \left[\chi^{*,\beta}_{kj} \left(\frac{\tilde{x}}{\varepsilon} \right)+(u^*_{bl})^{\beta}_{kj} \!\left(\frac{\tilde{x}}{\varepsilon} \right)\right]\partial_{\tilde{y}_\beta} G_{ji}^{*,0}\left(\frac{\tilde{x}}{\varepsilon}, \frac{x}{\varepsilon}\right)\\
		&\qquad\qquad\qquad\quad \qquad\qquad-\left. \left[\Gamma^{*,\alpha\beta}_{kj}\!\! \left(\frac{\tilde{x}}{\varepsilon} \right)+ \chi^{*,\alpha}_{kl}\!\!\left(\frac{\tilde{x}}{\varepsilon} \right)\! (u^*_{bl})^{\beta}_{lj}\!\left(\frac{\tilde{x}}{\varepsilon} \right)\right] \partial_{\tilde{y}_\alpha} \partial_{\tilde{y}_\beta}G^{*,0}_{ji}\!\!\left(\frac{\tilde{x}}{\varepsilon},\frac{x}{\varepsilon} \right) \right\}.
	\end{align*}
	We thus get the following formulation of $	W^{\varepsilon} (\tilde{x})$ :
	\begin{align*}
		W_{ki}^{\varepsilon}(\tilde{x}):=	W_{ki}^{\varepsilon}(\tilde{x},x)=\frac{1}{\varepsilon^{d-2}}W_{ki}\left(\frac{\tilde{x}}{\varepsilon},\frac{x}{\varepsilon} \right)
	\end{align*}
	with the function $W$ defined by
	\begin{align*}
		W_{ki}(\tilde{y},y) =&G_{ki}^* \left(\tilde{y},y\right)-G_{ki}^{*,0}\left(\tilde{y},y\right)- \left[\chi^{*,\beta}_{kj} \left(\tilde{y}\right)+(u^*_{bl})^{\beta}_{kj} \left(\tilde{y}\right)\right]\partial_{\tilde{y}_\beta} G_{ji}^{*,0}\left(\tilde{y},y\right)\\
		&\qquad\qquad\quad \qquad\qquad-\left[\Gamma^{*,\alpha\beta}_{kj}\!\! \left(\tilde{y}\right)+ \chi^{*,\alpha}_{kl}\!\!\left(\tilde{y}\right) (u^*_{bl})^{\beta}_{lj}\!\left( \tilde{y}\right)\right] \partial_{\tilde{y}_\alpha} \partial_{\tilde{y}_\beta}G^{*,0}_{ji}\left(\tilde{y},y\right).
	\end{align*}
	This identity yields 
	\begin{align*}
		\nabla_{\tilde{x}}	W^{\varepsilon} (\tilde{x})
		=&\frac{1}{\varepsilon^{d-1}}  \nabla_{\tilde{y}}W \left(\frac{\tilde{x}}{\varepsilon},\frac{x}{\varepsilon} \right)
	\end{align*}
	and then 
	\begin{align*}
		A\left(\frac{\tilde{x}}{\varepsilon}\right)	\nabla_{\tilde{x}}	W^{\varepsilon} (\tilde{x})
		=&\frac{1}{\varepsilon^{d-1}} A\left(\frac{\tilde{x}}{\varepsilon}\right) \nabla_{\tilde{y}}W \left(\frac{\tilde{x}}{\varepsilon},\frac{x}{\varepsilon} \right).
	\end{align*}
	Consequently, the inequality
	\begin{align*}
		\left|A\left(\frac{\tilde{x}}{\varepsilon}\right)	\nabla_{\tilde{x}}	W^{\varepsilon} (\tilde{x})\right|\leq C \varepsilon^\kappa 
	\end{align*}with $ C=C(A, d, \mu, [A]_{C^{0,\eta}},M_2)$,
	implies
	\begin{align}
		\left |  A\left(\frac{\tilde{x}}{\varepsilon}\right) \nabla_{\tilde{y}}W \left(\frac{\tilde{x}}{\varepsilon},\frac{x}{\varepsilon} \right)\right|   &\leq C \varepsilon^{d-1+\kappa} . \nonumber
	\end{align}
	Hence one obtains
	\begin{align}
		\left |  A\left(\tilde{y}\right) \nabla_{\tilde{y}}W \left(\tilde{y},y\right) \right|  &\leq \frac{C}{|y-\tilde{y}|^{d-1+\kappa}}\label{ineq_AnabW} 
	\end{align}
	where  $C=C(A, d, \mu, [A]_{C^{0,\eta}},M_2)$.\\
	
	\noindent We wish to have similar estimates for the pressure term.
	Taking into account the scaling properties of Green's functions and using (\ref{errorP}), we get 
	\begin{align}
		S_i^{\varepsilon} (\tilde{x})=	\frac{1}{\varepsilon^{d-1}} S_i\left(\frac{\tilde{x}}{\varepsilon},\frac{x}{\varepsilon} \right) \label{rescale_R}
	\end{align}
	with
	\begin{align*}
		S_i(\tilde{y},y):=	\Pi_i^{*,\varepsilon} (\tilde{y},y)&-  \Pi_i^{*,0} (\tilde{y},y)-[\pi^{*,\beta}_{j}(\tilde{y}) +(p^*_{bl})^{\beta}_{j} (\tilde{y})]\partial_{\tilde{y}_{\beta}} G_{ji}^{*,0}  (\tilde{y},y)\nonumber\\
		&	-[Q^{*,\alpha\beta}_{j}(\tilde{y})+ \pi^{*,\alpha}_{k}(\tilde{y})(u^*_{bl})^{\beta}_{kj}	(\tilde{y})]\partial_{\tilde{y}_\alpha} \partial_{\tilde{y}_\beta} G_{ji}^{*,0} (\tilde{x},x) .	\\	
	\end{align*}
\noindent We rescale the estimate $\underset{D_x^0}{\textup{osc}} [S^{\varepsilon}]\leq  C \varepsilon^\kappa$ with $ C=C(A, d, \mu, [A]_{C^{0,\eta}},M_2)$ :\\
	If $\tilde{x},\hat{x}\in D_x^0$ and satisfy 
	\begin{align*}
		|S^{\varepsilon} (\tilde{x})- S^{\varepsilon} (\hat{x}) |\leq C \varepsilon^\kappa,
	\end{align*}
	then	 we have, from (\ref{rescale_R}),
	\begin{align*}
		\left|S\left(\frac{\tilde{x}}{\varepsilon},\frac{x}{\varepsilon} \right)-  S \left(\frac{\hat{x}}{\varepsilon},\frac{x}{\varepsilon} \right)\right|  &\leq C \varepsilon^{d-1+\kappa} ,   
	\end{align*}
	hence
	\begin{align*}
		\left| S\left( \tilde{y},y\right)-  S\left( \hat{y},y\right) \right|  &\leq \frac{C}{|y-\tilde{y}|^{d-1+\kappa}}.  
	\end{align*}
	
	\noindent Thus the oscillation estimate  $\underset{D_x^0}{osc}\ [S^{\varepsilon}]\leq   \; C \varepsilon^\kappa$ with $ C=C(A, d, \mu, [A]_{C^{0,\eta}},M_2)$ rescales to 
	\begin{align*}
		\sup_{\tilde{y},\hat{y}\in D_{2R}(y) \bs \overline{D_R(y)}}  \left|  S\left( \tilde{y},y\right)-  S\left( \hat{y},y\right)\right|  &\leq  \sup_{\tilde{y}\in D_{2R}(y) \bs \overline{D_R(y)}}\frac{C}{|y-\tilde{y}|^{d-1+\kappa}} 
	\end{align*}
	yet the bound in the right hand side corresponds to the inverse of $\inf_{\tilde{y}\in D_{2R}(y) \bs \overline{D_R(y)}} |y-\tilde{y}|^{d-1+\kappa}$ which is $R^{d-1+\kappa}$. Therefore we have
	\begin{align*}
		\sup_{\tilde{y}\in D_{2R}(y) \bs \overline{D_R(y)}}\frac{C}{|y-\tilde{y}|^{d-1+\kappa}} = \frac{C}{R^{d-1+\kappa}} . 
	\end{align*}
	We then obtain the following oscillation estimate of the pressure term in the fast variable
	\begin{align}
		\sup_{\tilde{y},\hat{y}\in  D_{2R}(y) \bs \overline{D_R(y)} }  \left|  S\left( \tilde{y},y\right)-   S\left( \hat{y},y\right)\right|  &\leq \frac{C}{R^{d-1+\kappa}}   \nonumber \\
		\osc_{ D_{2R}(y) \bs \overline{D_R(y)}} [ S\left( \cdot,y\right)] &\leq \frac{C}{R^{d-1+\kappa}} \label{estim_osc_S}   \end{align}
	where $ C=C(A, d, \mu, [A]_{C^{0,\eta}},M_2)$.
	
	\noindent Nonetheless this estimate is  a local estimate as it is set in the domain $ D_{2R}(y) \bs \overline{D_R(y)}$ whereas we seek for an estimate valid for all $y\in \mathbb{H}_n$, $\tilde{y}\in \overline{\mathbb{H}}_n$ with $|y-\tilde{y}|$ large. So we extend (\ref{estim_osc_S}) to the entire domain  $\overline{\mathbb{H}}_n\bs \overline{D_R(y)}$.\\
	
	\noindent For any $m\in\mathbb{N}^*$, reasoning on $\tilde{y}_m,\hat{y}_{m}\in  D_{2^mR}(y) \bs \overline{D_{2^{m-1}R}(y) } $ similarly as above leads to 
	\begin{align*}
		\osc_{ D_{2^mR}(y) \bs \overline{D_{2^{m-1}R}(y) }} \:  [S\left( \cdot,y\right)] &\leq \frac{C}{(2^{m-1}R)^{d-1+\kappa}}  \end{align*}
		where $ C=C(A, d, \mu, [A]_{C^{0,\eta}},M_2)$.
		
	\begin{remark}
		The above estimate remains valid when $ D_{2^mR}(y) \bs \overline{D_{2^{m-1}R}(y) }$ is replaced with
		\begin{equation*}
		A_m:= \overline{D_{2^mR}(y)} \bs \overline{D_{2^{m-1}R}(y) }.
		\end{equation*} 
	\end{remark}
	
	\noindent Now decompose the punctured half-space $\mathbb{H}_n\bs \overline{D_R(y)}$ into annuli as follows
	\begin{align*}
		\mathbb{H}_n\bs \overline{D_R(y)} = \bigcup ^\infty_{m=1}  A_m.
	\end{align*}
	Then we get
	\begin{align*}
		\osc_{\mathbb{H}_n\bs \overline{D_R(y)}} \:	[S(\cdot,y)] &\leq \sum_{m=1}^\infty \: \osc_{A_m} \: [S\left( \cdot,y\right) ], \\
		&\leq \sum ^\infty_{m=1} \frac{C}{(2^{m-1}R)^{d-1+\kappa}}, \\
		&\leq  \frac{C}{R^{d-1+\kappa}}    \sum ^\infty_{m=0}  \frac{1}{(2^{d-1+\kappa})^m} ,\\
		&\leq  \frac{C}{R^{d-1+\kappa}}\quad  \mbox{  with }C=C(A, d, \mu, [A]_{C^{0,\eta}},M_2) .  
	\end{align*}
	This can be rephrased as : for any $\tilde{y},\hat{y}\in\mathbb{H}_n\bs \overline{D_R(y)}$ we have 
	\begin{align*}
		\left| S\!\left(\tilde{y},y\right)- S\!\left(\hat{y},y\right) \right|  &\leq \frac{C}{R^{d-1+\kappa}} \quad\mbox{ with } C=C(A, d, \mu, [A]_{C^{0,\eta}},M_2). 
	\end{align*}
	This estimate combined with the fact that 
	\begin{align*}
		\lim_{|\hat{y}|\rightarrow \infty}  S\!\left(\hat{y},y\right)=0
	\end{align*}
	yields 
	\begin{align}
		\left| S\!\left(\tilde{y},y\right) \right|  &\leq \frac{C}{R^{d-1+\kappa}}.\label{ineq_S}
	\end{align}
	Now we put together (\ref{ineq_AnabW}) and (\ref{ineq_S}) to obtain
	\begin{align*}
		\left| 	A(\tilde{y})\nabla W (\tilde{y},y) \right| +	\left| S(\tilde{y},y) \right|  &\leq \frac{C}{R^{d-1+\kappa}} , \\
		\left| 	A(\tilde{y})\nabla W (\tilde{y},y) \right| +	\left| S(\tilde{y},y) \right|  &\leq \frac{C}{|\tilde{y}-y|^{d-1+\kappa}} , \mbox{ for all }  y\in \mathbb{H}_n, \, \tilde{y}\in \partial \mathbb{H}_n,
	\end{align*}where  $C=C(A, d, \mu, [A]_{C^{0,\eta}},M_2)$.\\

	\noindent {\bf Large-scale expansion of the Poisson kernel $P=P(\tilde{y},y).$ }\\
	\noindent We aim at establishing an asymptotic expansion of $P= P(y,\tilde{y})$  when $|y-\tilde{y}| \longrightarrow +\infty$.\\
	\noindent Now we search for terms $P_0,$  $P_1$  and $P_2$ that satisfy
	\begin{align*}
		\left| P(y,\tilde{y})-P_0(y,\tilde{y})- P_1(y,\tilde{y})-P_2(y,\tilde{y}) \right| \leq \frac{C}{|\tilde{y}-y|^{d-1+\kappa}} , \mbox{ for all }  y\in \mathbb{H}_n, \, \tilde{y}\in \partial \mathbb{H}_n.
	\end{align*}
	We obtain the expansion for $P$ from the formula
	\begin{align*}
		P(y,\tilde{y})&=  -A^*(\tilde{y}) \nabla_{\tilde{y}} G^*(\tilde{y},y)\cdot n\, + \Pi^*(\tilde{y},y)\otimes n.
	\end{align*}
	Notice that
	\begin{align*}
		-A(\tilde{y}) \nabla_{\!\tilde{y}} W (\tilde{y},y)\cdot n+     S(\tilde{y},y) \otimes n&=	P(y,\tilde{y})-P_0(y,\tilde{y})- P_1(y,\tilde{y})-P_2(y,\tilde{y})
	\end{align*}
that
	\begin{align*}
		A^*(\tilde{y}) \nabla_{\!\tilde{y}} W (\tilde{y},y)
		=& A^*(\tilde{y}) \nabla_{\!\tilde{y}}G^* \left(\tilde{y},y \right)-A^*(\tilde{y}) \nabla_{\!\tilde{y}}G^{*,0} \left(\tilde{y},y \right)\\
		&- A^*(\tilde{y}) \nabla_{\!\tilde{y}}\left[\chi^* \left(\tilde{y} \right)+ u_{bl}^*\left(\tilde{y} \right)\right]\cdot \nabla G^{*,0}\left(\tilde{y},y\right)\\
		&- A^*(\tilde{y}) \left[\chi^* \left(\tilde{y} \right)+ u_{bl}^*\left(\tilde{y} \right)\right]: \nabla^2G^{*,0}\left(\tilde{y},y\right)\\
		&- A^*(\tilde{y}) \nabla_{\!\tilde{y}}\left[\Gamma^* \left(\tilde{y}\right)+ \chi^*\left(\tilde{y} \right) u_{bl}^*\left(\tilde{y} \right)\right]: \nabla^2 G^{*,0}\left(\tilde{y},y \right) \\
		&- A^*(\tilde{y})\left[\Gamma^* \left(\tilde{y}\right)+ \chi^*\left(\tilde{y} \right) u_{bl}^*\left(\tilde{y} \right)\right]: \nabla^3 G^{*,0}\left(\tilde{y},y \right),
	\end{align*}
		and that
	\begin{align*}
		S(\tilde{y},y)	=&\;\Pi^*( \tilde{y}, y)- \Pi^{*,0}( \tilde{y}, y)- [\pi^*(\tilde{y})+p_{bl}^*(\tilde{y})] \cdot\nabla G^{*,0}(\tilde{y}, y) \\
		&\qquad\qquad - [Q^*(\tilde{y})+\pi^* (\tilde{y})u_{bl}^*(\tilde{y})]:\nabla^2 G^{*,0}(\tilde{y}, y).
	\end{align*}

	\noindent The terms that contain a derivative of $G^{*,0}$ at the same order are grouped together. The terms of the first-order derivatives of $G^{*,0}$ are in
	\begin{align*}
		P_0	 (\tilde{y},y)=& \;-A^* (\tilde{y}) \nabla_{\tilde{y} } G^{*,0}(\tilde{y},y)\cdot n +\Pi^{*,0}(\tilde{y},y)\otimes n\\
		&\:- A^*(\tilde{y}) \nabla_{\!\tilde{y}}\left[\chi^* \left(\tilde{y} \right)+ u_{bl}^*\left(\tilde{y} \right)\right]\cdot \nabla G^{*,0}\left(\tilde{y},y\right)\\
		&\:+\{[\pi^*(\tilde{y})+p_{bl}^*(\tilde{y}) ] \cdot\nabla G^{*,0}(\tilde{y}, y)\}\otimes n.
	\end{align*}
	The term corresponding to the second-order derivatives of $G^{*,0}$ are in
	\begin{align*}
		P_1 (\tilde{y},y) =&\: -A^*(\tilde{y}) \left[\chi^* \left(\tilde{y} \right)+ u_{bl}^*\left(\tilde{y} \right)\right]: \nabla^2 G^{*,0}\left(\tilde{y},y\right)\\
		&\:-\{A^*  \nabla_{\tilde{y}} \Gamma^*(\tilde{y}) :\nabla^2 G^{*,0}(\tilde{y},y) \}\cdot n +\{ Q^*(\tilde{y}):\nabla^2 G^{*,0}(\tilde{y},y)\}\otimes n\\
		&\:-\{ A^*  \nabla_{\tilde{y} }  (\chi^*(\tilde{y}) u_{bl}^*(\tilde{y})): \nabla^2 G^{*,0}_0 (\tilde{y},y)\} \cdot n  +\{ \pi^* (\tilde{y})u_{bl}^*(\tilde{y}): \nabla^2 G^{*,0}(\tilde{y},y)\}\otimes n.
	\end{align*}
	and lastly we set
	\begin{align*}
		P_2(\tilde{y},y)	&= -A^*(\tilde{y})  \Gamma^*(\tilde{y}) \nabla^3 G^{*,0} (\tilde{y},y)\cdot n-A^*(\tilde{y})    \chi^* (\tilde{y})u_{bl}^*(\tilde{y}) \nabla^3 G^{*,0} (\tilde{y},y)\cdot n.
	\end{align*}
	By  summarizing the preceding lines we arrive at the following statement.
	\begin{proposition}\label{propoexpansp_poisson}
		For all $y\in \mathbb{H}_n$ and for all $\tilde{y}\in \partial\mathbb{H}_n$ we have the following expansion of the Poisson kernel
		\begin{align}
			P(y,\tilde{y})=&-A^*(\tilde{y}) \nabla_{\!\tilde{y}}G^{*,0} \left(\tilde{y},y \right)\cdot n+ \Pi^{*,0}( \tilde{y}, y)\otimes n\nonumber\\&-\{A^*(\tilde{y}) \nabla_{\!\tilde{y}}\left[\chi^* \left(\tilde{y} \right)+ u_{bl}^*\left(\tilde{y} \right)\right] \cdot \nabla G^{*,0}\left(\tilde{y},y\right)\}\cdot n\nonumber\\
			&\qquad \qquad + \{ [\pi^*(\tilde{y})+p_{bl}^*(\tilde{y})] \cdot \nabla G^{*,0}(\tilde{y}, y)\}\otimes n \nonumber\\
			&-A^*(\tilde{y}) \left\{ \left[\chi^* \left(\tilde{y} \right)+ u_{bl}^*\left(\tilde{y} \right)\right]:\nabla^2 G^{*,0}\left(\tilde{y},y\right) \right\}\cdot n\label{expans_P}\\
			&-A^*(\tilde{y})\left\{  \nabla_{\!\tilde{y}}[\Gamma^* (\tilde{y})+\chi^*\left(\tilde{y} \right) u_{bl}^*(\tilde{y} )]:\nabla^2 G^{*,0}\left(\tilde{y},y \right) 	\right\}\cdot n\nonumber\\
			&+ \{[Q^*(\tilde{y})+\pi^* (\tilde{y})u_{bl}^*(\tilde{y})]:\nabla^2 G^{*,0}(\tilde{y}, y)\}\otimes n\nonumber \\
			& -A^*(\tilde{y})  \Gamma^*(\tilde{y}) \nabla^3 G^{*,0} (\tilde{y},y)\cdot n-A^*(\tilde{y})    \chi^* (\tilde{y})u_{bl}^*(\tilde{y}) \nabla^3 G^{*,0} (\tilde{y},y)\cdot n\nonumber\\
			&+R(y,\tilde{y})\nonumber
		\end{align}
		with  the remainder $R(y,\tilde{y})$ being bounded as
		\begin{align}
			| R(y,\tilde{y})|\leq \frac{C}{|\tilde{y}-y|^{d-1+\kappa}},\quad \mbox{ for all } \kappa\in (0,1/2d)\label{bound_remainder}
		\end{align}and where the constant $C$ depends on $ d, \mu, [A]_{C^{0,\eta}}$ and $M_2$ defined in \textup{(\ref{M1})}.
	\end{proposition}
	
	\section{Convergence towards a boundary layer tail}\label{section_converge}
	\noindent In this section we aim at establishing the convergence of the solution $u^s$ to problem (\ref{hsStokes}) as $y\cdot n \rightarrow \infty$. First, we establish the convergence in the case $s=0$. Then we show that the parameter $s\in\mathbb{R}$ is not relevant in this convergence when $n$ is irrational.
		\subsection{Convergence of the boudary layer far away from the boundary in the case $s=0$}	\noindent The method that is going to be employed is based on the representation of the solution $u$ with the Poisson kernel
	\begin{align}
		u(y)=\int_{\partial\mathbb{H}_n} P(y,\tilde{y}) g(\tilde{y}) d\tilde{y}\label{represent_u}.
	\end{align}
	The asymptotic expansion (\ref{expans_P}) of this kernel will be key for studying the convergence properties. 
	
	\noindent We take advantage of the fact that the boundary data $g$    as well as functions $A^*,\chi^* ,\pi^* , \Gamma^*, Q^* $ are quasiperiodic on the hyperplane  $\partial\mathbb{H}_n$. We do so through an application of the following lemma.
	
	\begin{lemma}\textup{(\cite{thm_ergo}, Theorem S.3.)}\label{ergodicity}
		Let $f:\mathbb{R}^m\longmapsto \mathbb{R}$ be a quasiperiodic function then there exists $\mathcal{M}(f)\in \mathbb{R}$ such that for all $\varphi \in L^1(\mathbb{R}^m)$ one has
		\begin{align}\label{conv_ergo}
			\int_{\mathbb{R}^m} \varphi (y) f\left( \frac{y}{\varepsilon}\right)dy \xrightarrow{\varepsilon \rightarrow 0}\mathcal{M}(f)\int_{\mathbb{R}^m} \varphi (y)  dy.
		\end{align}
	\end{lemma}
	\noindent	This is a qualitative ergodic theorem for quasiperiodic functions. The hyperplane $\partial\mathbb{H}_n \sim \mathbb{R}^{d-1}$ will play the role of $\mathbb{R}^m$ in  Lemma \ref{ergodicity}. Notice that the convergence in (\ref{conv_ergo})   can be arbitrarily slow. By using  this Lemma we obtain the following theorem.
	\begin{theorem}[Theorem \ref{intro_ThConverge} in the case $s=0$]\label{Thm_s=0}
		As $y\cdot n \rightarrow \infty$	the solution $u=u(y)$ to problem \textup{(\ref{hsStokes})} with $s=0$ converges to a constant vector field $U^{\infty}(n)$ with components
		\begin{align}
			U_i^{\infty}(n)
			=&\:\:	 \mathcal{M} \left(\! - n_\alpha g_j A^{\beta\alpha}_{kj} |_{\partial\mathbb{H}_n} \! \right)\!\!\int_{\partial\mathbb{H}_n}\!\! \partial_{\tilde{y}_\beta} G^{*,0}_{ki} \left(\tilde{y},n\right) d\tilde{y}+  \mathcal{M} \left( n_j g_j|_{\partial\mathbb{H}_n} \right)\!\!\int_{\partial\mathbb{H}_n}\!\!\! \Pi^{*,0}_i\left(\tilde{y},n\right) d\tilde{y}\nonumber\\
			&+\left[-\mathcal{M}(n_\alpha g_l A^{\gamma\alpha}_{kl}\partial_{ \tilde{y}_\gamma}  \chi_{kj}^{*,\beta}|_{\partial\mathbb{H}_n}) + \mathcal{M}\left( n_r g_r \pi_{j}^{*,\beta} |_{\partial\mathbb{H}_n} \right) \right]\!
			\int_{\partial\mathbb{H}_n}\!\!\partial_{\tilde{y}_\beta} G^{*,0}_{ji}(\tilde{y},n) d\tilde{y} \nonumber\\
			&+ \left[ -\mathcal{M}(n_\alpha g_l A^{\gamma\alpha}_{kl}\partial_{ \tilde{y}_\gamma} (u^*_{bl})^{\beta}_{kj} |_{\partial\mathbb{H}_n})+
			\mathcal{M}\left( n_r g_r (p^*_{bl})^{\beta}_{j} |_{\partial\mathbb{H}_n} \right) \right]
			\int_{\partial\mathbb{H}_n}  \partial_{\tilde{y}_\beta}G^{*,0}_{ji}(\tilde{y},n) d\tilde{y}.\label{formul_bltail}
		\end{align}	
	\end{theorem}

	\begin{proof}[Proof]
		Injecting in (\ref{represent_u}) the expansion of $P(y,\tilde{y})$ given by Proposition \ref{propoexpansp_poisson}  yields
		\begin{align}
			u(y)=&\: U^{0,1}(y) +U^{0,2}(y) +U^{0,3}(y) +U^{0,4}(y) +U^{1,1}(y) +U^{1,2}(y) +U^{1,3}(y) \nonumber\\
			& +U^{2,1}(y) +U^{2,2}(y) +\int_{\partial\mathbb{H}_n} R(y,\tilde{y}) g(\tilde{y}) d\tilde{y}\label{expans_u}
		\end{align}
		where 	
		\begin{align*}
			U^{0,1}(y):=&\int_{\partial\mathbb{H}_n}  -A^*(\tilde{y}) \nabla_{\!\tilde{y}}G^{*,0}\left(\tilde{y},y \right)\cdot n g(\tilde{y})  d\tilde{y},   \qquad U^{0,2}:=\int_{\partial\mathbb{H}_n}\Pi^{*,0}( \tilde{y}, y)\otimes n g(\tilde{y})  d\tilde{y}  \nonumber \\
			U^{0,3}(y):=&\int_{\partial\mathbb{H}_n}-A^*(\tilde{y}) \{\nabla_{\!\tilde{y}}\left[\chi^* \left(\tilde{y} \right)+ u^*_{bl}\left(\tilde{y} \right)\right]\cdot  \nabla G^{*,0}\left(\tilde{y},y\right)\}\cdot n g(\tilde{y}) \   d\tilde{y}   \nonumber\\
			U^{0,4}(y):=& \int_{\partial\mathbb{H}_n}\{\pi^*(\tilde{y})+   p^*_{bl}(\tilde{y}) \} \cdot \nabla G^{*,0}(\tilde{y}, y)\otimes n  g(\tilde{y}) d\tilde{y} \nonumber
		\end{align*}
		are integrals that involve first-order derivatives of $G^{*,0}$ and $\Pi^{*,0}$. The other terms involve higher order derivatives of $G^{*,0}$
		\begin{align*}
			U^{1,1}(y):=&\int_{\partial\mathbb{H}_n}-A^*(\tilde{y}) \left\{\left[\chi^* \left(\tilde{y} \right)+ u_{bl}^*\left(\tilde{y} \right)\right]:\nabla^2 G^{*,0}\left(\tilde{y},y\right)\right\}\cdot n  g(\tilde{y}) \   d\tilde{y}  \nonumber \\
			U^{1,2}(y):=&\ \int_{\partial\mathbb{H}_n}\{[Q(\tilde{y})+\pi^* (\tilde{y})u_{bl}^*(\tilde{y})]:\nabla^2 G^{*,0}(\tilde{y}, y)\}\otimes n g(\tilde{y}) \ d\tilde{y} \nonumber\\
			U^{1,3}(y):=&\int_{\partial\mathbb{H}_n}  \left\{-A^*(\tilde{y}) \nabla_{\!\tilde{y}}\left[\Gamma^* \left(\tilde{y}\right)+ \chi^*\left(\tilde{y} \right) u_{bl}^*\left(\tilde{y} \right)\right]:\nabla^2 G^{*,0}\left(\tilde{y},y \right) 	\right\}\cdot n g(\tilde{y})   \  d\tilde{y}   \\
			U^{2,1}(y):=&\int_{\partial\mathbb{H}_n}  \left\{-A^*(\tilde{y}) \Gamma^* (\tilde{y}):\nabla^3 G^{*,0}\left(\tilde{y},y \right) 	\right\}\cdot n g(\tilde{y})   \  d\tilde{y} \nonumber  \\
			U^{2,2}(y):=&\int_{\partial\mathbb{H}_n}  \{-A^*(\tilde{y}) \chi^*\left(\tilde{y} \right) u_{bl}^*\left(\tilde{y} \right):\nabla^3 G^{*,0}\left(\tilde{y},y \right) \}\cdot n g(\tilde{y})   \  d\tilde{y}.   
		\end{align*}	
		The superscript $i$ in $U^{i,j}$ represents the order of the derivative of $\nabla G^{*,0}(\tilde{y}, y)$ or $\Pi^{*,0}( \tilde{y}, y)$.
	 Remark first that 
		the bound (\ref{bound_remainder}) on the remainder and the boundedness of $g$ imply 
		\begin{align*}
			\left| \int_{\partial\mathbb{H}_n} R(y,\tilde{y}) g(\tilde{y}) d\tilde{y} \right| \leq & \int_{\partial\mathbb{H}_n}  \frac{C}{|\tilde{y}-y|^{d-1+\kappa}} \\ \leq& \frac{C}{(y\cdot n)^\kappa} \stackrel{y\cdot n \rightarrow \infty}{ \longrightarrow } 0.
		\end{align*}
		Now we study the limit of each of these terms $U^{i,j}(y)$ as $y\cdot n \rightarrow \infty.$ \\

		\noindent {\bfseries Step 1 : terms $U^{0,1}, \,U^{0,2}, \,U^{0,3},\,  U^{0,4}$.}\\

		\noindent 
		\underline{Step 1-a : Convergence of the term $U^{0,1}$} \\
			\noindent 	First, we rewrite this quantity into coordinates by considering its $i^{th}$ component
		\begin{align*}
			U^{0,1}_i(y)= &\int_{\partial\mathbb{H}_n} \{[ -A^*(\tilde{y}) \nabla_{\!\tilde{y}}G^{*,0} \left(\tilde{y},y \right)\cdot n] \,g(\tilde{y})\}_i\,d\tilde{y} \\
			=& \int_{\partial\mathbb{H}_n} 
			-n_\alpha A^{\beta\alpha}_{kj}(\tilde{y})\partial_{\tilde{y}_\beta} G^{*,0}_{ki} \left(\tilde{y},y \right)g_j(\tilde{y})\,d\tilde{y} .
		\end{align*}
		For all $y\in \mathbb{H}_n$ and by using an estimate in (\ref{estim_dertildG}) on  Green's functions we get 
		\begin{align*}
			\int_{\partial\mathbb{H}_n}  |\partial_{ \tilde{y}_\beta} G^{*,0} _{ki} \left(\tilde{y},y \right) |\,d\tilde{y}\leq 	C \int_{\partial\mathbb{H}_n}  \frac{y\cdot n}{|\tilde{y}-y|^{d}}\,d\tilde{y} \leq C(d,\mu, [A]_{C^{0,\eta}})< \infty.
		\end{align*}		
		For arbitrary $R>0$,  consider $y\in\mathbb{H}_n $  such that its tangential component is bounded by $R$ 
		i.e. $y':=y-(y\cdot n)n \in B(0,R) \cap\partial\mathbb{H}_n$.
		Applying the rotation $M$, such that  $y=M z,$ $z\in  \mathbb{R}^d_+$, for all $y\in\mathbb{H}_n $, we obtain 
		\begin{align*}
			U^{0,1}_i(y) &= \int_{\mathbb{R}^{d-1}\times\{0\}} -n_\alpha A^{\beta\alpha}_{kj}(M(\tilde{z})) \partial_{ \tilde{y}_\beta}G^{*,0}_{ki} \left(M(\tilde{z}),M(z) \right)g_j(M(\tilde{z}))\,d\tilde{z}  \\
			&=\int_{\mathbb{R}^{d-1}} -n_\alpha A^{\beta\alpha}_{kj}(M(\tilde{z}',0)) \partial_{ \tilde{y}_\beta}G^{*,0} _{ki} (M(\tilde{z}',0), M(z',z_d)) g_j(M(\tilde{z}',0))\,d\tilde{z}'.  
		\end{align*}
		Now we set $\varepsilon := \frac{1}{y\cdot n}$
		and use  the change of variables $\tilde{z}'\rightarrow \tilde{z}'/\varepsilon $ (this change of variables in the integration does not concern the variable $z=(z',z_d)$)
		\begin{align*}
			U^{0,1}_i(y)
			&=\int_{\mathbb{R}^{d-1}} -\partial_{\tilde{y}_\beta} G^{*,0}_{ki}\left(M\left(\frac{\tilde{z}'}{\varepsilon },0\right), M\left( z',z_d  \right)\right) \\
			&\qquad \qquad\qquad \times n_\alpha A^{\beta\alpha}_{kj}\left(M\left(\frac{\tilde{z}'}{\varepsilon },0\right)\right)g_j\left(M\left(\frac{\tilde{z}'}{\varepsilon },0\right)\right)\,\frac{1}{\varepsilon^{d-1} } d\tilde{z}'.
		\end{align*}
		Note that $z_d = y\cdot n = 1/ \varepsilon$. Then the linearity of $M$ implies 
		\begin{align*}
			M(z',z_d)= M \left(\frac{\varepsilon z'}{\varepsilon },\frac{1}{\varepsilon}\right)= \frac{1}{\varepsilon} M (\varepsilon z',1).
		\end{align*}
		We use this remark to perform a slight transformation  in the integral
		\begin{align*}
			U^{0,1}_i(y)	&=\int_{\mathbb{R}^{d-1}} -\frac{1}{\varepsilon^{d-1} } \partial_{\tilde{y}_\beta} G^{*,0}_{ki} \left(\frac{M(\tilde{z}',0)}{\varepsilon },\frac{M (\varepsilon z',1)}{\varepsilon} \right) n_\alpha A^{\beta\alpha}_{kj} \left(\frac{M(\tilde{z}',0)}{\varepsilon }\right) g_j\left(\frac{M(\tilde{z}',0)}{\varepsilon }\right)d\tilde{z}'.
		\end{align*}
		Recall the following rescaling property
		\begin{align*}
			\frac{1}{\varepsilon^{d-1} } \nabla G^{*,0}\left(\frac{\tilde{x}}{\varepsilon },\frac{x}{\varepsilon} \right) =   \nabla G^{*,0}\left(\tilde{x},x\right)
		\end{align*}
		that we apply in the above expression to get 
		\begin{align*}
			U^{0,1}_i(y)
			&=\int_{\mathbb{R}^{d-1}} -\partial_{\tilde{y}_\beta} G^{*,0}_{ki}\left(M(\tilde{z}',0),M (\varepsilon z',1) \right) n_\alpha A^{\beta\alpha}_{kj} \left(\frac{M(\tilde{z}',0)}{\varepsilon }\right) g_j\left(\frac{M(\tilde{z}',0)}{\varepsilon }\right) \,d\tilde{z}'.  
		\end{align*}
		We  use the following trick to introduce the quantity $ G^{*,0}_{ki}\left(\tilde{y},n \right)$ that is independent of $\varepsilon$ and plays the role of $\varphi$ in Lemma \ref{ergodicity}
		\begin{align}\label{addrmove}
			U^{0,1}_i(y)=
			\int_{\mathbb{R}^{d-1}}\!\! &\left( \partial_{\tilde{y}_\beta} G^{*,0}_{ki}\left(M(\tilde{z}',0),M (0,1) \right) 
			- \partial_{\tilde{y}_\beta} G^{*,0}_{ki}\left(M(\tilde{z}',0),M (\varepsilon z',1) \right) \right)\\
			&\qquad \times n_\alpha A^{\beta\alpha}_{kj} \left(\frac{M(\tilde{z}',0)}{\varepsilon }\right) g_j\left(\frac{M(\tilde{z}',0)}{\varepsilon }\right)  \,d\tilde{z}' \nonumber \\
			&+\int_{\mathbb{R}^{d-1}} \!\! - \partial_{\tilde{y}_\beta} G^{*,0}_{ki}\left(M(\tilde{z}',0),M (0,1) \right) n_\alpha A^{\beta\alpha}_{kj} \left(\frac{M(\tilde{z}',0)}{\varepsilon }\right) g_j \! \left(\!\!\frac{M(\tilde{z}',0)}{\varepsilon }\!\right)\! d\tilde{z}'  .  \nonumber\\ \nonumber
		\end{align}
		
		\noindent We shall first show that 
		\begin{align*}
			\lim_{\varepsilon\rightarrow 0} 	\int_{\mathbb{R}^{d-1}}\!\! &\left(\partial_{\tilde{y}_\beta} G^{*,0}_{ki} \left(M(\tilde{z}',0),M (0,1) \right) 
			-\partial_{\tilde{y}_\beta} G^{*,0}_{ki}\left(M(\tilde{z}',0),M (\varepsilon z',1) \right) \right)\\
			&\qquad \times n_\alpha A^{\beta\alpha}_{kj} \left(\frac{M(\tilde{z}',0)}{\varepsilon }\right) g_j\left(\frac{M(\tilde{z}',0)}{\varepsilon }\right)  \,d\tilde{z}' =0
		\end{align*} uniformly in $z'\in B^{d-1}(0,R).$\\

		\noindent The mean value inequality gives 
		\begin{align*}
			\left|  \frac{\partial}{\partial \tilde{y}_\beta} G^{*,0}_{ki} (M(\tilde{z}',0),\right.&M (0,1) )- \left.\frac{\partial}{\partial \tilde{y}_\beta} G^{*,0}_{ki}(M(\tilde{z}',0),M (\varepsilon z',1) )\right|\\
			&\leq \sup_{w'\in B(0,\varepsilon R)} | \nabla_y \nabla_{\tilde{y}} G^{*,0}_{ki}\left(M(\tilde{z}',0),M (w',1) \right)| | \varepsilon z'|\\
			&\leq \varepsilon R \sup_{w'\in B(0,\varepsilon R)} |\nabla_y \nabla_{\tilde{y}} G^{*,0}_{ki}(M(\tilde{z}',0),M (w',1) )|.
		\end{align*}
		Using a version of  estimate  (\ref{estim_mixder}) 
		\begin{align*}
			|  \nabla_y \nabla_{\tilde{y}} G^{*,0}_{ki}(\tilde{y},y ) | &\leq\frac{C}{|y-\tilde{y}|^{d}}, \quad \mbox{ with } C=C(d,\mu,[A]_{C^{0,\eta}} ),
		\end{align*}
		we obtain 
		\begin{align*}
			\left|\nabla_y \nabla_{\tilde{y}} G^{*,0}_{ki}\left(M(\tilde{z}',0),\, M (w',1) \right)\right|
			&\leq \frac{C}{(1+|w'-\tilde{z}'|^2)^{d/2}}
		\end{align*}
		leading to
		\begin{align*}
			\left| \partial_{\tilde{y}_\beta} G^{*,0}_{ki}\left(M(\tilde{z}',0),M (\varepsilon z',1) \right) -  \partial_{\tilde{y}_\beta} G^{*,0}_{ki}(M(\tilde{z}',0),M (0,1) ) \right|
			&\leq \varepsilon R \!\!\sup_{w'\in B(0,\varepsilon R)}\frac{C}{(1+|w'-\tilde{z}'|^2)^{d/2}}.
		\end{align*}
		Taking $\varepsilon$ small enough, so that $\varepsilon R\leq 1,$ we get  
		\begin{align*}
			\left| \partial_{\tilde{y}_\beta} G^{*,0}_{ki}\left(M(\tilde{z}',0),M (\varepsilon z',1) \right) -  \partial_{\tilde{y}_\beta} G^{*,0}_{ki}(M(\tilde{z}',0),M (0,1) ) \right|&\leq \varepsilon R \!\!\! \sup_{w'\in B(0,1)}\frac{C}{(1+|w'-\tilde{z}'|^2)^{d/2}}.
		\end{align*}
		In the present context, $B(0,1)$ denotes the unit ball of $\mathbb{R}^{d-1}$.
		So we have the inequality
		\begin{align*}
			\int_{{\mathbb{R}}^{d-1}}&	\left| \partial_{\tilde{y}_\beta} G^{*,0}_{ki} \left(M(\tilde{z}',0),M (\varepsilon z',1) \right)
			-\partial_{\tilde{y}_\beta} G^{*,0}_{ki}(M(\tilde{z}',0),M (0,1) )  \right| \\
			&\qquad \times n_\alpha A^{\beta\alpha}_{kj} \left(\frac{M(\tilde{z}',0)}{\varepsilon }\right) g_j\left(\frac{M(\tilde{z}',0)}{\varepsilon }\right)  \,d\tilde{z}' \\
			&	\leq C\varepsilon R \int_{{\mathbb{R}}^{d-1}}\sup_{w'\in B(0,1)}\frac{1}{(1+|w'-\tilde{z}'|^2)^{d/2}}d\tilde{z}'
		\end{align*} where $C$ depends on $d,\mu$ and $[A]_{C^{0,\eta}}$.
	Moreover, one has 
		\begin{align*}
			\int_{{\mathbb{R}}^{d-1}} \sup_{w'\in B(0,1)}\frac{1}{(1+|w'-\tilde{z}'|^2)^{d/2}}  d\tilde{z}'= \int_{B(0,1)}& \sup_{w'\in B(0,1)}\frac{ 1}{(1+|w'-\tilde{z}'|^2)^{d/2}} d\tilde{z}'\\
			& + \int_{{\mathbb{R}}^{d-1}\bs B(0,1) } \sup_{w'\in B(0,1)}\frac{1}{(1+|w'-\tilde{z}'|^2)^{d/2}} d\tilde{z}'.
		\end{align*}
		If $ \tilde{z}'\in B(0,1)$ then $$\inf_{w'\in B(0,1)}  (1+|w'-\tilde{z}'|^2)=1$$ and if  $ \tilde{z}'\in{\mathbb{R}}^{d-1}\bs B(0,1)$ then $$\inf_{w'\in B(0,1)}  (1+|w'-\tilde{z}'|^2)=1+(|\tilde{z}'|- 1)^2.$$
		Thus  we obtain
		\begin{align*}
			\int_{{\mathbb{R}}^{d-1}} \sup_{u'\in B(0,1)}\frac{1}{(1+|u'-\tilde{z}'|^2)^{d/2}}  d\tilde{z}'=& \int_{B(0,1)}1\ d\tilde{z}' + \int_{{\mathbb{R}}^{d-1}\bs B(0,1) } \frac{1}{(1+(|\tilde{z}'|-1)^2)^{d/2}} d\tilde{z}'\\
			=& \ C(d) + C(d)\int_1^\infty \frac{r^{d-2}}{(1+(r-1)^2)^{d/2}} dr < +\infty.
		\end{align*}
		
		\noindent Finally we obtain 
		\begin{align*}
			\int_{{\mathbb{R}}^{d-1}}&	\left| \partial_{\tilde{y}_\beta} G^{*,0}_{ki}\left(M(\tilde{z}',0),M (\varepsilon z',1) \right)
			-\partial_{\tilde{y}_\beta} G^{*,0}_{ki}(M(\tilde{z}',0),M (0,1) )  \right| \\
			&\qquad \times n_\alpha A^{\beta\alpha}_{kj} \left(\frac{M(\tilde{z}',0)}{\varepsilon }\right) g_j\left(\frac{M(\tilde{z}',0)}{\varepsilon }\right)  \,d\tilde{z}' 	\leq C\varepsilon R
		\end{align*}	where $C$ depends on $d,\mu$ and $[A]_{C^{0,\eta}}$ and not on $\varepsilon$.	We let $\varepsilon\rightarrow 0$ and hold $R$ fixed. Then the above estimate directly gives
		\begin{align}\label{conv01}
			\int_{{\mathbb{R}}^{d-1}}&	\left| \partial_{\tilde{y}_\beta} G^{*,0}_{ki} \left(M(\tilde{z}',0),M (\varepsilon z',1) \right)
			-\partial_{\tilde{y}_\beta} G^{*,0}_{ki}(M(\tilde{z}',0),M (0,1) )  \right| \nonumber\\
			&\qquad \times n_\alpha A^{\beta\alpha}_{kj} \left(\frac{M(\tilde{z}',0)}{\varepsilon }\right) g_j\left(\frac{M(\tilde{z}',0)}{\varepsilon }\right)  \,d\tilde{z}' 	\longrightarrow 0\quad \mbox{ as } \varepsilon \rightarrow 0.
		\end{align}	
		Now one handles the term 
		\begin{align*}
			\int_{\mathbb{R}^{d-1}} &-\partial_{\tilde{y}_\beta} G^{*,0}_{ki} \left(M(\tilde{z}',0),M (0,1) \right)  n_\alpha A^{\beta\alpha}_{kj}\left(\frac{M(\tilde{z}',0)}{\varepsilon }\right) g_j\left(\frac{M(\tilde{z}',0)}{\varepsilon }\right) d\tilde{z}'
		\end{align*}
		through an application of Lemma \ref{ergodicity}. Recall that the function  $- n_\alpha A^{\beta\alpha}_{kj}g(M(\cdot,0))g_j(M(\cdot,0))$ is  quasiperiodic in $\mathbb{R}^{d-1}$.   In addition, one has  
		\begin{align*}
			\partial_{\tilde{y}_\beta} G^{*,0}_{ki}\left(M(\cdot,0),M (0,1) \right) \in L^1(\mathbb{R}^{d-1}).
		\end{align*}
		Now that all the elements required by Lemma \ref{ergodicity} are in hand one is ready to state that  
		\begin{align}
			\int_{\mathbb{R}^{d-1}}- \partial_{\tilde{y}_\beta} G^{*,0}_{ki}&\left(M(\tilde{z}',0),M (0,1) \right)  n_\alpha A^{\beta\alpha}_{kj}\left(\frac{M(\tilde{z}',0)}{\varepsilon }\right) g_j\left(\frac{M(\tilde{z}',0)}{\varepsilon }\right)\,d\tilde{z}' \nonumber\\
			&\xrightarrow{\varepsilon\rightarrow 0}\mathcal{M} \left(  n_\alpha A^{\beta\alpha}_{kj} g_j(M(\cdot ,0)\right)\int_{\mathbb{R}^{d-1}} - \partial_{\tilde{y}_\beta} G^{*,0}_{ki} \left(M(\tilde{z}',0),M (0,1) \right) \,d\tilde{z}'\nonumber\\
			&\xrightarrow{\varepsilon\rightarrow 0} \mathcal{M} \left( n_\alpha A^{\beta\alpha}_{kj} g_j(M(\cdot ,0))\right)\int_{\partial\mathbb{H}_n}- \partial_{ \tilde{y}_\beta} G^{*,0}_{ki} \left(\tilde{y},n\right) \,d\tilde{y}\label{convblderv}.\\ \nonumber
		\end{align}
		 Putting (\ref{addrmove}), (\ref{conv01}) and (\ref{convblderv}) together we obtain
		\begin{align}\label{convblp1}
			U_i^{0,1} ( y) 
			&\xrightarrow{y\cdot n\rightarrow \infty} \mathcal{M} \left(  n_\alpha A^{\beta\alpha}_{kj}(\cdot) g_j(\cdot)|_{\partial\mathbb{H}_n}  \right)\!\!\int_{\partial\mathbb{H}_n}\!\!\!\! - \frac{\partial}{\partial \tilde{y}_\beta} G^{*,0}_{ki} \left(\tilde{y},n\right) \,d\tilde{y}.\\ \nonumber
		\end{align}
		
	\noindent	\underline{Step 1-b :\, Convergence of the term $U^{0,2}$}\\
	\noindent One focuses on the $i^{th}$ component of $U^{0,2}$
		\begin{align*}
			U_i^{0,2}(y)
			=&  \int_{\partial\mathbb{H}_n}   \Pi_i^{*,0}(\tilde{y}, y) n_j g_j(\tilde{y})\,d\tilde{y} 
		\end{align*}
		and apply the method  carried out in the study of the convergence of $U^{0,1}$.
		
		\noindent Consider $y\in\mathbb{H}_n $  such that $y':=y-(y\cdot n)n \in B(0,R) \cap\partial\mathbb{H}_n$ for arbitrary $R>0$.
		The above-mentioned method  leads to 
		\begin{align}\label{addremove1}
			U_i^{0,2}(y)
			=\int_{\mathbb{R}^{d-1}} &\left( \Pi^{*,0}_i\left(M(\tilde{z}',0),M (\varepsilon z',1) \right) -  \Pi^{*,0}_i \left(M(\tilde{z}',0),M (0,1) \right) \right) n_j g_j\left(\frac{M(\tilde{z}',0)}{\varepsilon }\right)\,d\tilde{z}' \nonumber \\
			&	+\int_{\mathbb{R}^{d-1}} \Pi^{*,0}_i \left(M(\tilde{z}',0),M (0,1) \right) n_j g_j\left(\frac{M(\tilde{z}',0)}{\varepsilon }\right)\,d\tilde{z}'.\nonumber
		\end{align}
		where $M$ is the the rotation defined above and $\varepsilon := \frac{1}{y\cdot n}$.

		\noindent First, we are going to show that 
		\begin{align*}
			\lim_{\varepsilon\rightarrow 0} \int_{\mathbb{R}^{d-1}} \left( \Pi^{*,0}_i\left(M(\tilde{z}',0),M (\varepsilon z',1) \right) -  \Pi^{*,0}_i\left(M(\tilde{z}',0),M (0,1) \right)\right )  n_j g_j\left(\frac{M(\tilde{z}',0)}{\varepsilon }\right)\,d\tilde{z}' =0
		\end{align*} uniformly in $z'\in B^{d-1}(0,R).$\\
		
		\noindent Using the mean value inequality  
		\begin{align*}
			| \Pi^{*,0}_i\left(M(\tilde{z}',0),M (\varepsilon z',1) \right) -  \Pi^{*,0}_i(&M(\tilde{z}',0),M (0,1) ) |\\ &\leq \sup_{u'\in B(0,\varepsilon R)} | \nabla_y \Pi^{*,0}_i\left(M(\tilde{z}',0),M (u',1) \right)\!|\ |\varepsilon z'|\\
			&\leq \varepsilon R \sup_{u'\in B(0,\varepsilon R)} |\nabla_y  \Pi^{*,0}_i(M(\tilde{z}',0),M (u',1) )|.
		\end{align*}
		along with the  estimate (\ref{estimderPi*_y}) 
		\begin{align*}
			|  \nabla_y  \Pi^{*,0}_i(\tilde{y},y ) | &\leq\frac{C}{|y-\tilde{y}|^{d}}\\
			|\nabla_y  \Pi^{*,0}_i(M(\tilde{z}',0),M (u',1) )|
			&\leq \frac{C}{(1+|u'-\tilde{z}'|^2)^{d/2}}
		\end{align*}
		one gets
		\begin{align*}
			| \Pi^{*,0}_i\left(M(\tilde{z}',0),M (\varepsilon z',1) \right) -  \Pi^{*,0}_i(M(\tilde{z}',0),M (0,1) ) |
			&\leq \varepsilon R \sup_{u'\in B(0,1)}\frac{C}{(1+|u'-\tilde{z}'|^2)^{d/2}}
		\end{align*}
		with $\varepsilon$ small enough so as $\varepsilon R\leq 1$. Which leads to
		\begin{align*}
			\int_{{\mathbb{R}}^{d-1}}	\left| \Pi^{*,0}_i\left(M(\tilde{z}',0),M (\varepsilon z',1) \right)\right. - &\left. \Pi^{*,0}_i(M(\tilde{z}',0), M (0,1) ) \right|  \left| n_j g_j\left(\frac{M(\tilde{z}',0)}{\varepsilon }\right) \right|  d\tilde{z}'\\
			&\leq C\varepsilon R \int_{{\mathbb{R}}^{d-1}}\sup_{u'\in B(0,1)}\frac{1}{(1+|u'-\tilde{z}'|^2)^{d/2}}d\tilde{z}'\\
			&\leq C\varepsilon R. 
		\end{align*}
		Hence one obtains the convergence
		\begin{align}
			\int_{{\mathbb{R}}^{d-1}}	| \Pi^{*,0}_i\left(M(\tilde{z}',0),M (\varepsilon z',1) \right) - & \Pi^{*,0}_i(M(\tilde{z}',0), M (0,1) ) |  \left| n_j g_j\left(\frac{M(\tilde{z}',0)}{\varepsilon }\right) \right|  d\tilde{z}'\rightarrow 0
		\end{align}
		as $\varepsilon \rightarrow 0.$\\
		
		\noindent Now one applies Lemma \ref{ergodicity} with the fact that $g(M(\cdot,0))$ is quasiperiodic in $\mathbb{R}^{d-1}$ and
		\begin{align*}
			\Pi^{*,0}_i\left(M(\cdot,0),M (0,1) \right) \in L^1(\mathbb{R}^{d-1}).
		\end{align*}
		We then get
		\begin{align}
			\int_{\mathbb{R}^{d-1}} \Pi^{*,0}_i(M(\tilde{z}',0)&,M (0,1) ) n_j g_j\left(\!\frac{M(\tilde{z}',0)}{\varepsilon }\right)d\tilde{z}'\nonumber\\ &\xrightarrow{\varepsilon\rightarrow 0} \mathcal{M} \left( n_j g_j(M(\cdot ,0)\right)\int_{\mathbb{R}^{d-1}} \Pi^{*,0}_i\left(M(\tilde{z}',0),M (0,1) \right) d\tilde{z}'\nonumber\\
			&\xrightarrow{\varepsilon\rightarrow 0} \mathcal{M} \left( n_j g_j(\cdot )|_{\partial\mathbb{H}_n} \right)\int_{\partial\mathbb{H}_n} \Pi^{*,0}_i\left(\tilde{y},n\right) d\tilde{z}'. \nonumber
		\end{align}
		In summary, we have 
		\begin{align}
			U_i^{0,2}(y)\xrightarrow{y\cdot n\rightarrow \infty} \mathcal{M} \left( n_j g_j(\cdot )|_{\partial\mathbb{H}_n} \right)\int_{\partial\mathbb{H}_n} \Pi^{*,0}_i\left(\tilde{y},n\right) \,d\tilde{z}.'\label{convblp2}\\\nonumber
		\end{align}
	\noindent	\underline{Step 1-c : 	Convergence of the term $U^{0,3}$ }\\
		We study the convergence of the term 
		\begin{align*}
			U^{0,3}(y):=&\int_{\partial\mathbb{H}_n}-A(\tilde{y}) \{\nabla_{\!\tilde{y}}\left[\chi^* \left(\tilde{y} \right)+ u_{bl}^*\left(\tilde{y} \right)\right] \cdot \nabla G^{*,0}\left(\tilde{y},y\right)\}\cdot n g(\tilde{y}) \   d\tilde{y} 
		\end{align*}
		which we rewrite  into coordinates as
		\begin{align*}
			U_i^{0,3}(y)
			&= \int_{\partial\mathbb{H}_n}-n_\alpha A^{\gamma\alpha}_{kl}(\tilde{y})\left(\partial_{ \tilde{y}_\gamma} [\chi_{kj}^{*,\beta}\left(\tilde{y} \right) + (u^*_{bl})^{\beta}_{kj}\left(\tilde{y} \right)]\right) \partial_{ \tilde{y}_\beta} G^{*,0}_{ji}\left(\tilde{y},y\right)  g_l(\tilde{y}) d\tilde{y}\\
			&=\int_{\partial\mathbb{H}_n}-n_\alpha A^{\gamma\alpha}_{kl}(\tilde{y})\left(\partial_{ \tilde{y}_\gamma} \chi_{kj}^{*,\beta}\left(\tilde{y} \right)  \right) \partial_{ \tilde{y}_\beta} G^{*,0}_{ji}\left(\tilde{y},y\right)  g_l(\tilde{y}) d\tilde{y}\\&\qquad +\int_{\partial\mathbb{H}_n}-n_\alpha A^{\gamma\alpha}_{kl}(\tilde{y})\left(\partial_{ \tilde{y}_\gamma}  (u^*_{bl})^{\beta}_{kj}(\tilde{y})\right) \partial_{ \tilde{y}_\beta} G^{*,0}_{ji}\left(\tilde{y},y\right)  g_l(\tilde{y}) d\tilde{y}.
		\end{align*}
		Seeing that the quantity 
		\begin{align*}
			\int_{\partial\mathbb{H}_n}-n_\alpha A^{\gamma\alpha}_{kl}(\tilde{y})\left(\partial_{ \tilde{y}_\gamma} \chi_{kj}^{*,\beta}\left(\tilde{y} \right)  \right) \partial_{ \tilde{y}_\beta} G^{*,0}_{ji}\left(\tilde{y},y\right)  g_l(\tilde{y}) d\tilde{y}
		\end{align*}
		 is of the form 
		\begin{align*}
			\int_{\partial\mathbb{H}_n}-B_{j}^{\beta}(\tilde{y})\partial_{ \tilde{y}_\beta} G^{*,0}_{ji}(\tilde{y},y) d\tilde{y}
		\end{align*}
		where $B_j^{\beta}:\mathbb{R}^d \longmapsto \mathbb{R}$,\,  $B_j^{\beta}(\tilde{y}):=n_\alpha A^{\gamma\alpha}_{kl}(\tilde{y})\partial_{ \tilde{y}_\gamma} \chi_{kj}^{*,\beta}(\tilde{y}) g_l(\tilde{y}) $ is periodic, one can apply the same reasoning  that has led to (\ref{convblp1}). Then one obtains
		\begin{align*}
			\int_{\partial\mathbb{H}_n} -B_{j}^{\beta}(\tilde{y})\partial_{ \tilde{y}_\beta} G^{*,0}_{ji}(\tilde{y},y) d\tilde{y}  \xrightarrow{y\cdot n\rightarrow \infty}-\mathcal{M}(B_{j}^{\beta}(\cdot)|_{\partial\mathbb{H}_n})
			\int_{\partial\mathbb{H}_n} \partial_{ \tilde{y}_\beta} G^{*,0}_{ji}(\tilde{y},y) d\tilde{y} .
		\end{align*}
		
		\noindent 	As for the term 
		\begin{align*}
			\int_{\partial\mathbb{H}_n}\!\!-n_\alpha A^{\gamma\alpha}_{kl}(\tilde{y})\partial_{ \tilde{y}_\gamma}  (u^*_{bl})^{\beta}_{kj}\left(\tilde{y} \right) \partial_{ \tilde{y}_\beta} G^{*,0}_{ji}\left(\tilde{y},y\right)  g_l(\tilde{y}) d\tilde{y}
		\end{align*}
		we point out that the function $(u^*_{bl})^{\beta}_{kj}$ may not be  periodic and one can not directly apply the argument that has just been carried out on the term involving $ \chi_{kj}^{*,\beta}$. Nonetheless, it is quasiperiodic on the hyperplane  $\partial\mathbb{H}_n$ (by Appendix \ref{quasiperiod}) then it can be written 
		\begin{align}
			(u^*_{bl})^{\beta}_{kj}(\tilde{y})=  V_{kj}^{*,\beta}( N\tilde{z}',0 ),\quad \mbox{ for }\tilde{y}\in\partial\mathbb{H}_n,	\label{id_quasiperiod_u}
		\end{align}
		where $V_{kj}^{*,\beta}(\theta,t),\, \theta\in \mathbb{T}^d, \, t>0,$ is periodic in $\theta$.
		From this identity, we find that the derivative $\partial_{ \tilde{y}_\gamma} (u^*_{bl})^{\beta}_{kj}$ is quasiperiodic on $\partial\mathbb{H}_n$.	Then we obtain the convergence similarly to the preceding cases
		\begin{align*}
			\int_{\partial\mathbb{H}_n} -n_\alpha A^{\gamma\alpha}_{kl}(\tilde{y})\partial_{ \tilde{y}_\gamma}  (u^*_{bl})^{\beta}_{kj}(\tilde{y})&\partial_{\tilde{y}_\beta} G^{*,0}_{ji}(\tilde{y},y) g_l(\tilde{y}) d\tilde{y} \\
			& \xrightarrow{y\cdot n\rightarrow \infty}-\mathcal{M}(n_\alpha g_l A^{\gamma\alpha}_{kl}\partial_{ \tilde{y}_\gamma}  (u^*_{bl})^{\beta}_{kj}|_{\partial\mathbb{H}_n})
			\int_{\partial\mathbb{H}_n} \!\partial_{ \tilde{y}_\beta}G^{*,0}_{ji}(\tilde{y},y) d\tilde{y} .
		\end{align*}
		In conclusion, we have the following convergence for the term $U_i^{0,3}$
		\begin{align}
			U_i^{0,3}(y)
			& \xrightarrow{y\cdot n\rightarrow \infty}\left[-\mathcal{M}(n_\alpha g_l A^{\gamma\alpha}_{kl}\partial_{ \tilde{y}_\gamma}  \chi_{kj}^{*,\beta}|_{\partial\mathbb{H}_n})-\mathcal{M}(n_\alpha g_l A^{\gamma\alpha}_{kl}\partial_{ \tilde{y}_\gamma}  (u^*_{bl})^{\beta}_{kj}|_{\partial\mathbb{H}_n}) \right]\!
			\int_{\partial\mathbb{H}_n}\!\! \partial_{\tilde{y}_\beta}G^{*,0}_{ji}(\tilde{y},y) d\tilde{y} \label{convbl3}.
		\end{align}	
		
	\noindent	\underline{Step 1-d :	Convergence of $	U^{0,4}$}\\
	\noindent 	We treat this quantity
			\begin{align*}
			U_i^{0,4}(y):=& \int_{\partial\mathbb{H}_n}\{ \pi_{j}^{*,\beta}(\tilde{y})+(p^*_{bl})^{\beta}_{j}(\tilde{y}) \} \partial_{\tilde{y}_\beta} G_{ji}^{*,0}(\tilde{y}, y) n_r  g_r(\tilde{y}) d\tilde{y} .
		\end{align*}
		 the same way we have dealt with $	U^{0,3} $. 
Since the function $\pi_{j}^{*,\beta}$ is quasiperiodic on the hyperplane $\partial\mathbb{H}_n$, Lemma \ref{ergodicity} yields 
		\begin{align*}
			\int_{\partial\mathbb{H}_n}\pi_{j}^{*,\beta}(\tilde{y}) \partial_{\tilde{y}_\beta} G^{*,0}_{ji}(\tilde{y}, y) \,n_r g_r(\tilde{y})  d\tilde{y}\xrightarrow{y\cdot n\rightarrow \infty}\mathcal{M}\left( n_r g_r\pi_{j}^{*,\beta}(\cdot) |_{\partial\mathbb{H}_n} \right)
			\int_{\partial\mathbb{H}_n}  \partial_{\tilde{y}_\beta}G^{*,0}_{ji}(\tilde{y},n) d\tilde{y}.
		\end{align*}
		As for the  term 
		\begin{align*}
			\int_{\partial\mathbb{H}_n}(p^*_{bl})^{\beta}_{j}(\tilde{y})  \partial_{\tilde{y}_\beta} G^{*,0}_{ji} (\tilde{y}, y) \,n_r g_r(\tilde{y})  d\tilde{y}
		\end{align*}
		one uses the quasiperiodicity of    the pressure  $p^*_{bl}$ along 
		$\partial\mathbb{H}_n$, given in Appendix \ref{quasiperiod},
		$$
	(p^*_{bl})^{\beta}_{j}(\tilde{y})= (p^*_{bl})^{\beta}_{j}(N (\tilde{z}', 0))=R _{j}^{*,\beta}( N\tilde{z}',0 )$$
		where $ R _{j}^{*,\beta}=R_{j}^{*,\beta}( \theta,t )$ is periodic in $\theta$. Then by Lemma \ref{ergodicity} one obtains
		\begin{align*}
			\int_{\partial\mathbb{H}_n}   (p^*_{bl})^{\beta}_{j}(\tilde{y}) \partial_{\tilde{y}_\beta} G^{*,0}_{ji}(\tilde{y}, y) \,n_r g_r(\tilde{y})  d\tilde{y}\xrightarrow{y\cdot n\rightarrow \infty}\mathcal{M}\left( n_r g_r (p^*_{bl})^{\beta}_{j}(\cdot) |_{\partial\mathbb{H}_n} \right)
			\int_{\partial\mathbb{H}_n}  \partial_{\tilde{y}_\beta}G^{*,0}_{ji}(\tilde{y},n) d\tilde{y}.
		\end{align*}
		In conclusion, we have the convergence of $U^{0,4}$
		\begin{align}
			U_i^{0,4}(y)\xrightarrow{y\cdot n\rightarrow \infty}\left[ \mathcal{M}\left( n_r g_r \pi_{j}^{*,\beta}(\cdot) |_{\partial\mathbb{H}_n} \right) +
			\mathcal{M}\left( n_r g_r (p^*_{bl})^{\beta}_{j}(\cdot) |_{\partial\mathbb{H}_n} \right) \right]
			\int_{\partial\mathbb{H}_n}  \partial_{\tilde{y}_\beta}G^{*,0}_{ji}(\tilde{y},n) d\tilde{y}.\label{convbl4}
		\end{align}
		
		\noindent {\bfseries Step 2 :	Terms $U_i^{1,1}(y)$, $U_i^{1,2}(y)$, $U_i^{1,3 }(y)$
		$U_i^{2,1}(y)$, $U_i^{2,2}(y)$}.\\ We show that these terms converge to zero when $y\cdot n \rightarrow \infty.$ 	We use the estimate (\ref{estim_der_G0}) on the derivatives of Green's functions $ G^{*,0}$ associated with the constant coefficients operator 
		\begin{align*}
			|U^{1,1}(y)|\leq &\int_{\partial\mathbb{H}_n}\left|A(\tilde{y}) \left\{\left[\chi^* \left(\tilde{y} \right)+ u_{bl}^*\left(\tilde{y} \right)\right]:\nabla^2 G^{*,0}\left(\tilde{y},y\right)\right\}\cdot n  g(\tilde{y})   \right|   d\tilde{y} \\
			\leq & \, C\int_{\partial\mathbb{H}_n}\left|\nabla^2 G^{*,0}\left(\tilde{y},y\right)  \right|   d\tilde{y} \\
			\leq & \,  C\int_{\partial\mathbb{H}_n}\frac{C}{|y-\tilde{y}|^{d}}   d\tilde{y}.
		\end{align*}
		From estimate	(\ref{bound_int|y-tildy|^d}) one gets
		\begin{align*}
			\int_{\partial\mathbb{H}_n}\frac{1}{|y-\tilde{y}|^{d}}   d\tilde{y}\leq \frac{C}{y\cdot n} 
		\end{align*}
		from which we  draw  
		\begin{align}
			\lim_{y\cdot n \rightarrow \infty} 	|U^{1,1}(y)|\leq  \lim_{y\cdot n \rightarrow \infty}  \frac{C}{y\cdot n} = 0.\label{convbl11}
		\end{align}
		So are the limits of the remaining terms:
		\begin{align}
			\lim_{y\cdot n \rightarrow \infty} 	U^{1,2}(y)= \lim_{y\cdot n \rightarrow \infty} 	U^{1,3}(y)  = 	\lim_{y\cdot n \rightarrow \infty} 	U^{2,1}(y)=	\lim_{y\cdot n \rightarrow \infty} 	U^{2,2}(y)=0\label{convbl5}.
		\end{align}
		Now we arrive at a conclusion : from  (\ref{expans_u})  (\ref{convblp1}), (\ref{convblp2}), (\ref{convbl3}), (\ref{convbl4}), (\ref{convbl11}) and (\ref{convbl5}) we obtain
		\begin{align*}
			\lim_{y\cdot n \rightarrow \infty} 	u_i(y)=& \lim_{y\cdot n \rightarrow \infty} 	U_i^{0,1}(y)  +	\lim_{y\cdot n \rightarrow \infty} 	U_i^{0,2}(y)+	\lim_{y\cdot n \rightarrow \infty} 	U_i^{0,3}(y)+	\lim_{y\cdot n \rightarrow \infty} 	U_i^{0,4}(y)\\
			=&\:\:	 \mathcal{M} \left(\! - n_\alpha g_j A^{\beta\alpha}_{kj} |_{\partial\mathbb{H}_n} \! \right)\!\!\int_{\partial\mathbb{H}_n}\!\! \partial_{\tilde{y}_\beta} G^{*,0}_{ki} \left(\tilde{y},n\right) d\tilde{y}+  \mathcal{M} \left( n_j g_j|_{\partial\mathbb{H}_n} \right)\!\!\int_{\partial\mathbb{H}_n}\!\!\! \Pi^{*,0}_i\left(\tilde{y},n\right) d\tilde{y}\\
			&+\left[-\mathcal{M}(n_\alpha g_l A^{\gamma\alpha}_{kl}\partial_{ \tilde{y}_\gamma}  \chi_{kj}^{*,\beta}|_{\partial\mathbb{H}_n}) + \mathcal{M}\left( n_r g_r \pi_{j}^{*,\beta} |_{\partial\mathbb{H}_n} \right) \right]\!
			\int_{\partial\mathbb{H}_n}\!\!\partial_{\tilde{y}_\beta} G^{*,0}_{ji}(\tilde{y},n) d\tilde{y} \\
			&+ \left[ -\mathcal{M}(n_\alpha g_l A^{\gamma\alpha}_{kl}\partial_{ \tilde{y}_\gamma}  (u^*_{bl})^{\beta}_{kj}|_{\partial\mathbb{H}_n})+
			\mathcal{M}\left( n_r g_r (p^*_{bl})^{\beta}_{j}|_{\partial\mathbb{H}_n} \right) \right]
			\int_{\partial\mathbb{H}_n}  \partial_{\tilde{y}_\beta}G^{*,0}_{ji}(\tilde{y},n) d\tilde{y}.\qquad \qquad\qedhere\\
		\end{align*}
		
	\end{proof}
	
	\subsection{Convergence in the case of general $s\in\mathbb{R}$}
	\noindent The previous theorem gives the limit, as $y\cdot n \rightarrow \infty$, of the half-space boundary layer $u^s$ defined on the half-space ${\mathbb{H}}_n(s)$ in the particular case $s=0$. Next, we are going to investigate what the limit would be in the case of an arbitrary $s\in\mathbb{R}.$	\noindent Recall that for $s\in\mathbb{R}$ the pair $(u^s,p^s)$ denotes the solution to the system  
	\begin{equation*}
		\left\{
		\begin{array}{rcll}
			-\nabla_y\cdot A(y ) \: \nabla_y u^s +  \nabla_y p^s&=& 0& \mbox{in } {\mathbb{H}}_n(s),\\
			\nabla_y\cdot u^s&=& 0,    &\\
			u^s(y) &=& g(y)  \quad & \mbox{on } \partial{\mathbb{H}}_n(s).
		\end{array}
		\right. 
	\end{equation*}
	We shall use Theorem \ref{Thm_s=0} to determine the limit of $u^s(y)$ as $y\cdot n \rightarrow \infty$. To that end we introduce the pair 	$(\tilde{u}^s,\tilde{p}^s)$ solution of a similar problem set on the half-space $\mathbb{H}_n=\mathbb{H}_n(0)$. More precisely, we define
	$\tilde{u}^s(y)=u^s(y+sn),$ $\tilde{p}^s(y)=p^s(y+sn)$ and consequently the pair 	$(\tilde{u}^s,\tilde{p}^s)$ satisfies
	\begin{equation}\label{hspace_prob_tild}
		\left\{
		\begin{array}{rcll}
			-\nabla_y\cdot A^s(y ) \: \nabla_y \tilde{u}^s +  \nabla_y \tilde{p}^s&=& 0 & \mbox{in } {\mathbb{H}}_n,\\
			\nabla_y\cdot\tilde{u}^s &=& 0  ,    &\\
			\tilde{u}^s &=& g^s(y)  \quad & \mbox{on } \partial{\mathbb{H}}_n.
		\end{array}
		\right. 
	\end{equation}
	where 	 $A^s(y):=A(y+sn)$ and $g^s(y):=g(y+sn).$	\noindent More generally, for a function $B$ defined on $\mathbb{R}^d$ and for $s\in\mathbb{R}$, we will use the notation $B^s(y)$ to denote  $B(y+sn).$ We then have $\partial_{y_\beta} (B^s(y))=(\partial_{y_\beta} B)^s(y)$ for all $\beta=1,...,d,$ so one can simply denote it $\partial_{y_\beta} B^s(y)$ without confusion.
		\noindent 	An application of Proposition \ref{propoexpansp_poisson} yields :
		\begin{corollary}
		For all $s\in\mathbb{R}$, for all $y\in \mathbb{H}_n$ and for all $\tilde{y}\in \partial\mathbb{H}_n$ the following expansion of the Poisson kernel  $P^s(y,\tilde{y})$ associated to the system \textup{(\ref{hspace_prob_tild})} holds
		\begin{align*}
			P^s(y,\tilde{y})=&-(A^*)^s(\tilde{y}) \nabla_{\!\tilde{y}}G^{*,0} \left(\tilde{y},y \right)\cdot n+ \Pi^{*,0}( \tilde{y}, y)\otimes n\\
			&-\{(A^*)^s(\tilde{y}) \nabla_{\!\tilde{y}}\left[(\chi^*)^s \left(\tilde{y} \right)+ (u^*_{bl})^s\left(\tilde{y} \right)\right]\cdot \nabla G^{*,0}\left(\tilde{y},y\right)\}\cdot n\\
			&\qquad \qquad + \{(\pi^*)^s(\tilde{y})+(p_{bl}^*)^s(\tilde{y}) \} \cdot \nabla G^{*,0}(\tilde{y}, y)\otimes n \\
			&-(A^*)^s(\tilde{y}) \left\{ \left[(\chi^*)^s \left(\tilde{y} \right)+ (u_{bl}^*)^s \left(\tilde{y} \right)\right]:\nabla^2 G^{*,0}\left(\tilde{y},y\right) \right\}\cdot n\\
			&-(A^*)^s(\tilde{y})\left\{  \nabla_{\!\tilde{y}}[(\Gamma^*)^s (\tilde{y})+ (\chi^*)^s\left(\tilde{y} \right)  (u_{bl}^*)^s(\tilde{y} )]:\nabla^2 G^{*,0}\left(\tilde{y},y \right) 	\right\}\cdot n\\
			&+ \{[(Q^*)^s(\tilde{y})+(\pi^*)^s (\tilde{y}) (u_{bl}^*)^s(\tilde{y})]:\nabla^2 G^{*,0}(\tilde{y}, y)\}\otimes n \\
			& -(A^*)^s(\tilde{y}) ( \Gamma^*)^s(\tilde{y}) :\nabla^3 G^{*,0} (\tilde{y},y)\cdot n-(A^*)^s(\tilde{y})    (\chi^*)^s (\tilde{y})(u_ {bl}^*)^s(\tilde{y}): \nabla^3 G^{*,0} (\tilde{y},y)\cdot n\\
			&+R(y,\tilde{y})
		\end{align*}
		with  the remainder $R(y,\tilde{y})$ satisfying
		\begin{align}
			| R(y,\tilde{y})|\leq \frac{C}{|\tilde{y}-y|^{d-1+\kappa}},\quad \mbox{ for all } \kappa\in (0,1/2d),
		\end{align}  where the constant $C$ depends on $ d, \mu, [A]_{C^{0,\eta}}$ and $M_1$.  
	\end{corollary}

	\begin{corollary}
 For  arbitrary  $s\in\mathbb{R}$  	we have the convergence of $u^s$ towards a boundary layer tail :
	\begin{align}
		\lim_{y\cdot n \rightarrow \infty} u^s(y)=  U^{s,\infty}(n)
	\end{align}	
	with
	\begin{align*}
		U_i^{s,\infty}(n)
		=&\:\:	 \mathcal{M} \left(\! - n_\alpha g^s_j (A^s)^{\beta\alpha}_{kj} |_{\partial\mathbb{H}_n} \! \right)\!\!\int_{\partial\mathbb{H}_n}\!\! \partial_{\tilde{y}_\beta} G^{*,0}_{ki} \left(\tilde{y},n\right) d\tilde{y}+  \mathcal{M} \left( n_j g^s_j|_{\partial\mathbb{H}_n} \right)\!\!\int_{\partial\mathbb{H}_n}\!\!\! \Pi^{*,0}_i\left(\tilde{y},n\right) d\tilde{y}\\
		&+\left[-\mathcal{M}(n_\alpha g^s_l (A^s)^{\gamma\alpha}_{kl}\partial_{ \tilde{y}_\gamma}  (\chi^*)_{kj}^{s,\beta}|_{\partial\mathbb{H}_n}) + \mathcal{M}\left( n_r g^s_r (\pi^*)_{j}^{s,\beta} |_{\partial\mathbb{H}_n} \right) \right]\!
		\int_{\partial\mathbb{H}_n}\!\!\partial_{\tilde{y}_\beta} G^{*,0}_{ji}(\tilde{y},n) d\tilde{y} \\
		&+ \left[ -\mathcal{M}(n_\alpha g^s_l (A^s)^{\gamma\alpha}_{kl}\partial_{ \tilde{y}_\gamma}  (u_{bl}^*)_{kj}^{s,\beta}|_{\partial\mathbb{H}_n})+
		\mathcal{M}\left( n_r g^s_r( p_{bl}^*)_{j}^{s,\beta} |_{\partial\mathbb{H}_n} \right) \right]
		\int_{\partial\mathbb{H}_n}  \partial_{\tilde{y}_\beta}G^{*,0}_{ji}(\tilde{y},n) d\tilde{y}.\\
	\end{align*}	
	\end{corollary}
	
\begin{proof}
	\noindent Injecting the above expansion of $P^s(y,\tilde{y})$ in the representation of the solution $\tilde{u}^s$
\begin{align}
	\tilde{u}^s	(y)=\int_{\partial\mathbb{H}_n} P^s(y,\tilde{y}) g^s(\tilde{y}) d\tilde{y}
\end{align} we obtain an expansion of $\tilde{u}^s$ for all $s\in\mathbb{R}$
\begin{align}
	\tilde{u}^s	(y)=& \; U_s^{0,1}(y) +U_s^{0,2}(y) +U_s^{0,3}(y) +U_s^{0,4}(y) +U_s^{1,1}(y) +U_s^{1,2}(y) +U_s^{1,3}(y) \nonumber\\
	& +U_s^{2,1}(y) +U_s^{2,2}(y)+\int_{\partial\mathbb{H}_n} R(y,\tilde{y}) g^s(\tilde{y}) d\tilde{y} \label{expans_tild_u_s}
\end{align}
where 
\begin{align*}
	U_s^{0,1}(y):=&\int_{\partial\mathbb{H}_n}  -(A^*)^s(\tilde{y}) \nabla_{\!\tilde{y}}G^{*,0}\left(\tilde{y},y \right)\cdot n g^s(\tilde{y})  d\tilde{y},   \qquad U_s^{0,2}:=\int_{\partial\mathbb{H}_n}\Pi^{*,0}( \tilde{y}, y)\otimes n g^s(\tilde{y})  d\tilde{y}   \\
	U_s^{0,3}(y):=&\int_{\partial\mathbb{H}_n}-(A^*)^s(\tilde{y}) \{\nabla_{\!\tilde{y}}\left[(\chi^*)^s \left(\tilde{y} \right)+ (u_{bl}^*)^s\left(\tilde{y} \right)\right] \nabla G^{*,0}\left(\tilde{y},y\right)\}\cdot n g^s(\tilde{y}) \   d\tilde{y}   \\
	U_s^{0,4}(y):=& \int_{\partial\mathbb{H}_n}\{(\pi^*)^s(\tilde{y})+(p_{bl}^*)^s(\tilde{y}) \} \nabla G^{*,0}(\tilde{y}, y)\otimes n  g^s(\tilde{y}) d\tilde{y} 
\end{align*}
are the terms involving $\nabla G^{*,0},\,\Pi^{*,0}$. The remaining terms involve second and third order derivatives of $G^{*,0}$:
\begin{align*}
	U_s^{1,1}(y)=&\int_{\partial\mathbb{H}_n}-(A^*)^s(\tilde{y}) \left\{\left[(\chi^*)^s \left(\tilde{y} \right)+ (u_{bl}^*)^s\left(\tilde{y} \right)\right]:\nabla^2 G^{*,0}\left(\tilde{y},y\right)\right\}\cdot n  g^s(\tilde{y}) \   d\tilde{y}   \\
	U_s^{1,2}(y)=&\ \int_{\partial\mathbb{H}_n}\{[(Q^*)^s(\tilde{y})+(\pi^*)^s (\tilde{y}) (u_{bl}^*)^s(\tilde{y})]:\nabla^2 G^{*,0}(\tilde{y}, y)\}\otimes n g^s(\tilde{y}) \ d\tilde{y} \\
	U_s^{1,3}(y)=&\int_{\partial\mathbb{H}_n}  \left\{-(A^*)^s(\tilde{y}) \nabla_{\!\tilde{y}}\left[(\Gamma^*)^s \left(\tilde{y}\right)+ (\chi^*)^s\left(\tilde{y} \right) (u_{bl}^*)^s\left(\tilde{y} \right)\right]:\nabla^2 G^{*,0}\left(\tilde{y},y \right) 	\right\}\cdot n g^s(\tilde{y})   \  d\tilde{y}   \\
	U_s^{2,1}(y)=&\int_{\partial\mathbb{H}_n}  \left\{-(A^*)^s(\tilde{y}) (\Gamma^*)^s\left(\tilde{y}\right):\nabla^3 G^{*,0}\left(\tilde{y},y \right) 	\right\}\cdot n g^s(\tilde{y})   \  d\tilde{y}   \\
	U_s^{2,2}(y)=&\int_{\partial\mathbb{H}_n}  \{-(A^*)^s(\tilde{y}) (\chi^*)^s\left(\tilde{y} \right) (u_{bl}^*)^s\left(\tilde{y} \right):\nabla^3 G^{*,0}\left(\tilde{y},y \right) \}\cdot n g^s(\tilde{y})     d\tilde{y}.
\end{align*}
The superscript $i$ in $U_s^{i,j}$ represents the order of the derivative of $\nabla G^{*,0}(\tilde{y}, y)$ or $\Pi^{*,0}( \tilde{y}, y)$.\\

\noindent Remark that since $g^s$ is bounded the remaining term $ \int_{\partial\mathbb{H}_n} R(y,\tilde{y}) g^s(\tilde{y}) d\tilde{y}$
satisfies 
\begin{align*}
	\left| \int_{\partial\mathbb{H}_n} R(y,\tilde{y}) g^s(\tilde{y})d\tilde{y} \right|  \leq& \frac{C}{(y\cdot n)^\kappa} \stackrel{y\cdot n \rightarrow \infty}{ \longrightarrow } 0.
\end{align*}
Using the relation between $u^s$ and $\tilde{u}^s$ together with an application of Theorem \ref{Thm_s=0} to the solution of problem (\ref{hspace_prob_tild})  one obtains the convergence of $u^s$ for  arbitrary  $s\in\mathbb{R}$  :
\begin{align}
	\lim_{y\cdot n \rightarrow \infty} u^s(y)=\lim_{y\cdot n \rightarrow \infty} \tilde{u}^s(y)=  U^{s,\infty}(n)
\end{align}	
with
\begin{align*}
	U_i^{s,\infty}(n)
	=&\:	\: \mathcal{M} \left(\! - n_\alpha g^s_j (A^s)^{\beta\alpha}_{kj} |_{\partial\mathbb{H}_n} \! \right)\!\!\int_{\partial\mathbb{H}_n}\!\! \partial_{\tilde{y}_\beta} G^{*,0}_{ki} \left(\tilde{y},n\right) d\tilde{y}+  \mathcal{M} \left( n_j g^s_j|_{\partial\mathbb{H}_n} \right)\!\!\int_{\partial\mathbb{H}_n}\!\!\! \Pi^{*,0}_i\left(\tilde{y},n\right) d\tilde{y}\\
	&+\left[-\mathcal{M}(n_\alpha g^s_l (A^s)^{\gamma\alpha}_{kl}\partial_{ \tilde{y}_\gamma}  (\chi^*)_{kj}^{s,\beta}|_{\partial\mathbb{H}_n}) + \mathcal{M}\left( n_r g^s_r (\pi^*)_{j}^{s,\beta} |_{\partial\mathbb{H}_n} \right) \right]\!
	\int_{\partial\mathbb{H}_n}\!\!\partial_{\tilde{y}_\beta} G^{*,0}_{ji}(\tilde{y},n) d\tilde{y} \\
	&+ \left[ -\mathcal{M}(n_\alpha g^s_l (A^s)^{\gamma\alpha}_{kl}\partial_{ \tilde{y}_\gamma}  (u^*_{bl})_{kj}^{s,\beta}|_{\partial\mathbb{H}_n})+
	\mathcal{M}\left( n_r g^s_r( p_{bl}^*)_{j}^{s,\beta} |_{\partial\mathbb{H}_n} \right) \right]
	\int_{\partial\mathbb{H}_n}  \partial_{\tilde{y}_\beta}G^{*,0}_{ji}(\tilde{y},n) d\tilde{y}.\qquad\qedhere\\
\end{align*}	
\end{proof}

	\noindent  Now we compare two boundary layer tails $U^{s,\infty}(n)$ and  $U^{t,\infty}(n)$ associated respectively to arbitrary  $s,t\in\mathbb{R}.$ We  proceed through a study of the difference $U^{s,\infty}(n) -U^{t,\infty}(n)$.
	
	\begin{proposition}
		There exists $C>0$ such that, for all $s,t\in\mathbb{R}$,	we have 
		\begin{align*}
			|U^{s,\infty}(n) -U^{t,\infty}(n)| &\leq C |s-t|^{\eta}
		\end{align*}
			with  $C=C\left(\|A\|_{C^{0,\eta}(\mathbb{R}^d)},   \|g\|_{C^{0,\eta}(\mathbb{R}^d)}\right).$
	\end{proposition}
	\begin{proof}[Proof]
		
		The boundary layer tail $U^{s,\infty}(n)$  can be seen as the limit of the expansion (\ref{expans_tild_u_s}) as  $y\cdot n\rightarrow \infty$ and the difference  $U^{s,\infty}(n) -U^{t,\infty}(n)$ is then the limit of the difference of the correspponding expansions. More precisely, we have 
		\begin{align*}
			U^{s,\infty}(n) -U^{t,\infty}(n)= 	\lim_{y\cdot n \rightarrow \infty}	\tilde{u}^s(y)-\tilde{u}^t(y)
		\end{align*}
		and using the expansion (\ref{expans_tild_u_s}) this is written
		\begin{align*}
			U^{s,\infty}(n) -U^{t,\infty}(n)= &	\sum_{i,j} \lim_{y\cdot n \rightarrow \infty}  (U_s^{i,j}(y) -U_t^{i,j}(y)) +   \lim_{y\cdot n \rightarrow \infty} \int_{\partial\mathbb{H}_n} R(y,\tilde{y})(g^s(\tilde{y})-g^t(\tilde{y}))\ d\tilde{y}\\
			=&	\sum_{i,j} \lim_{y\cdot n \rightarrow \infty}  (U_s^{i,j}(y) -U_t^{i,j}(y)) 
		\end{align*}
		since
		\begin{align*}
			\lim_{y\cdot n \rightarrow \infty} \int_{\partial\mathbb{H}_n} R(y,\tilde{y})(g^s(\tilde{y})-g^t(\tilde{y}))\ d\tilde{y}=0.
		\end{align*}
		Now we aim at establishing that for $s,t\in\mathbb{R}$ we have 
		\begin{align*}
			|	U^{i,j}_{s}(y)- 	U^{i,j}_{t}(y) |
			\leq &\; C |s-t|, \quad \mbox{for some constant } C>0. \\
		\end{align*}

		\noindent {\bfseries Step 1 :	Bound of the difference $U_s^{0,1}- 	U^{0,1}_{t} $ }\\
		Consider the  $i^{th}$ component of this quantity
		\begin{align*}
			U^{0,1}_{s,i}(y)- 	U^{0,1}_{t,i}(y) 
			=&
			\int_{\partial\mathbb{H}_n} 
			n_\alpha (A^t)^{\beta\alpha}_{kj}(\tilde{y})g^t_j(\tilde{y})\frac{\partial}{\partial \tilde{y}_\beta} G^{*,0}_{ki} \left(\tilde{y},y \right)\,d\tilde{y} \\
			&	- \int_{\partial\mathbb{H}_n} 
			n_\alpha (A^s)^{\beta\alpha}_{kj}(\tilde{y})g^s_j(\tilde{y})\frac{\partial}{\partial \tilde{y}_\beta} G^{*,0}_{ki} \left(\tilde{y},y \right)\,d\tilde{y} \\
			=& \int_{\partial\mathbb{H}_n} 
			n_\alpha A^{\beta\alpha}_{kj}(\tilde{y}+tn)g_j(\tilde{y}+tn)\frac{\partial}{\partial \tilde{y}_\beta} G^{*,0}_{ki} \left(\tilde{y},y \right)\,d\tilde{y} \\
			&	- \int_{\partial\mathbb{H}_n} 
			n_\alpha A^{\beta\alpha}_{kj}(\tilde{y}+sn)g_j(\tilde{y}+sn)\frac{\partial}{\partial \tilde{y}_\beta} G^{*,0}_{ki} \left(\tilde{y},y \right)
			\,d\tilde{y} 
		\end{align*}
		The function $\tilde{y}\mapsto A^{\beta\alpha}_{kj}(\tilde{y})g_j(\tilde{y})$ is regular enough to satisfy \begin{align*}
			|A^{\beta\alpha}_{kj}g_j(\tilde{y}+sn)- A^{\beta\alpha}_{kj}g_j(\tilde{y}+tn)|\leq C |s-t|^\eta, \quad \mbox{ with } C=C(\|A\|_{C^{0,\eta}(\mathbb{R}^d)}, \|g\|_{C^{0,\eta}(\mathbb{R}^d)})>0.
		\end{align*}We then obtain
		\begin{align*}
			|	U^{0,1}_{s,i}(y)- 	U^{0,1}_{t,i}(y) |
			&\leq  C |s-t|^\eta\int_{\partial\mathbb{H}_n} 
			|	\partial_{\tilde{y}_\beta} G^{*,0}_{ki} \left(\tilde{y},y \right)|\,d\tilde{y}\\ &\leq C |s-t|^\eta,
		\end{align*}
	where $C$ depends on $d, \mu, \|A\|_{C^{0,\eta}(\mathbb{R}^d)}$ and $\|g\|_{C^{0,\eta}(\mathbb{R}^d)}.$\\
	
		\noindent {\bfseries Step 2 : Estimating the difference $U_s^{0,2}- U_t^{0,2}$}:\\
		The $i^{th}$ component of  this quantity reads
		\begin{align*}
			U_{s,i}^{0,2}(y)- U_{t,i}^{0,2}(y)	=&  \int_{\partial\mathbb{H}_n}   \Pi_i^{*,0}(\tilde{y}, y) n_j g_j(\tilde{y}+sn)\,d\tilde{y} - \int_{\partial\mathbb{H}_n}   \Pi_i^{*,0}(\tilde{y}, y) n_j g_j(\tilde{y}+tn)\,d\tilde{y}\\
			=&\int_{\partial\mathbb{H}_n}   \Pi_i^{*,0}(\tilde{y}, y) n_j [g_j(\tilde{y}+sn)- g_j(\tilde{y}+tn)]\,d\tilde{y}.
		\end{align*}
		Hence 		we have 
		\begin{align*}
			|	U_{s,i}^{0,2}(y)- U_{t,i}^{0,2}(y)|	\leq &  \, [g]_{C^{0,\eta}(\mathbb{R}^d)}|s-t|^\eta
			\int_{\partial\mathbb{H}_n}  | \Pi_i^{*,0}(\tilde{y}, y)|\,d\tilde{y}.
		\end{align*}		
		\noindent By  the fact that
		$	\left\|\Pi^{*,0}_i\left(\cdot,y \right)\right\|_{ L^1(\partial\mathbb{H}_n)}\leq C(d, \mu, [A]_{C^{0,\eta}})
		$ uniformly in $y$
		we get
		\begin{align}
			|	U_{s,i}^{0,2}(y)- U_{t,i}^{0,2}(y)|	\leq &  \, C|s-t|^\eta	\nonumber
		\end{align}with $C=C\left(d, \mu, [A]_{C^{0,\eta}(\mathbb{R}^d)}, [g]_{C^{0,\eta}(\mathbb{R}^d) }  \right).$\\
		
		\noindent {\bfseries Step 3 : We treat the term $U_s^{0,3}- U_t^{0,3}$}:\\
		\noindent	We write this difference into coordinates 
		\begin{align*}
			U_{s,i}^{0,3}(y)- 	U_{t,i}^{0,3}(y)
			=&\int_{\partial\mathbb{H}_n}-n_\alpha (A^s)^{\gamma\alpha}_{kl} (\tilde{y})\left(\partial_{ \tilde{y}_\gamma} (\chi^{*})^{s,\beta}_{kj}\left(\tilde{y} \right)  \right) \partial_{ \tilde{y}_\beta} G^{*,0}_{ji}\left(\tilde{y},y\right)  g^s_l(\tilde{y}) d\tilde{y}\\
			&+	\int_{\partial\mathbb{H}_n}n_\alpha (A^t)^{\gamma\alpha}_{kl} (\tilde{y})\left(\partial_{ \tilde{y}_\gamma} (\chi^{*})^{t,\beta}_{kj}\left(\tilde{y} \right)  \right) \partial_{ \tilde{y}_\beta} G^{*,0}_{ji}\left(\tilde{y},y\right)  g^t_l(\tilde{y}) d\tilde{y}\\
			& +\int_{\partial\mathbb{H}_n}-n_\alpha (A^s)^{\gamma\alpha}_{kl} (\tilde{y})\left(\partial_{ \tilde{y}_\gamma} (u_{bl}^{*})^{s,\beta}_{kj} \left(\tilde{y} \right)\right) \partial_{ \tilde{y}_\beta} G^{*,0}_{ji}\left(\tilde{y},y\right)  g^s_l(\tilde{y}) d\tilde{y}\\
			&+\int_{\partial\mathbb{H}_n}n_\alpha (A^t)^{\gamma\alpha}_{kl} (\tilde{y})\left(\partial_{ \tilde{y}_\gamma}(u_{bl}^{*})^{t,\beta}_{kj}  \left(\tilde{y} \right)\right) \partial_{ \tilde{y}_\beta} G^{*,0}_{ji}\left(\tilde{y},y\right)  g^t_l(\tilde{y}) d\tilde{y}
		\end{align*}
		which is also written
		\begin{align*}
			U_{s,i}^{0,3}(y)- 	U_{t,i}^{0,3}(y)
			=&\int_{\partial\mathbb{H}_n}- [B^{\beta}_{j}\left(\tilde{y}+sn \right) -B^{\beta}_{j}\left(\tilde{y}+tn \right)]    \partial_{ \tilde{y}_\beta} G^{*,0}_{ji}\left(\tilde{y},y\right) d\tilde{y}\\
			&+\int_{\partial\mathbb{H}_n}- [D^{\beta}_{j}\left(\tilde{y}+sn \right) -D^{\beta}_{j}\left(\tilde{y}+tn \right)]    \partial_{ \tilde{y}_\beta} G^{*,0}_{ji}\left(\tilde{y},y\right) d\tilde{y}
		\end{align*}
		where  $B_j^{\beta}(\tilde{y}):=n_\alpha A^{\gamma\alpha}_{kl}(\tilde{y})\partial_{ \tilde{y}_\gamma} \chi_{kj}^{*,\beta}(\tilde{y}) g_l(\tilde{y}) $ and   $D_j^{\beta}(\tilde{y}):=n_\alpha A^{\gamma\alpha}_{kl}(\tilde{y})\partial_{ \tilde{y}_\gamma}(u_{bl}^{*})^{\beta}_{kj} (\tilde{y}) g_l(\tilde{y}). $ 
		
		\noindent	These functions  $B_j^{\beta}$ and   $D_j^{\beta}(\tilde{y})$   satisfy   
		\begin{align*}
			|B^{\beta}_{j}\left(\tilde{y}+sn \right) -B^{\beta}_{j}\left(\tilde{y}+tn \right)| &\leq C |s-t|^\eta , \\
			|D^{\beta}_{j}\left(\tilde{y}+sn \right) -D^{\beta}_{j}\left(\tilde{y}+tn \right)|&\leq  C  |s-t|^\eta ,
		\end{align*}
		where $C>0$ is a constant depending on $\mu$, $\|A\|_{C^{0,\eta}(\mathbb{R}^d)}$ and on  $\|g\|_{C^{0,\eta}(\mathbb{R}^d)}$. Then we have 
		\begin{align*}
			|U_{s,i}^{0,3}(y)- 	U_{t,i}^{0,3}(y)|
			\leq &\: C  |s-t|^\eta \int_{\partial\mathbb{H}_n} |\partial_{ \tilde{y}_\beta} G^{*,0}_{ji}(\tilde{y},y)| d\tilde{y},\\
			\leq &\: C  |s-t|^\eta ,
		\end{align*}
		with $C=C(d,\mu,\|A\|_{C^{0,\eta}(\mathbb{R}^d)},\|g\|_{C^{0,\eta}(\mathbb{R}^d)})$ .\\
		
		\noindent {\bfseries  Step 4 : Term $U_s^{0,4}- U_t^{0,4}$}:\\
		Consider the $i^{th}$-component of the difference 	$U_{s}^{0,4}(y)-U_{t}^{0,4}(y)$
		\begin{align*}
			U_{s,i}^{0,4}(y)-U_{t,i}^{0,4}(y)=& \int_{\partial\mathbb{H}_n}\left( (\pi^{*})_{j}^{s,\beta}(\tilde{y})+(p_{bl}^{*})_{j}^{s,\beta}(\tilde{y}) \right) \partial_{\tilde{y}_\beta} G_{ji}^{*,0}(\tilde{y}, y) n_r  g^s_r(\tilde{y}) d\tilde{y} \\
			&  - \int_{\partial\mathbb{H}_n}    \left( (\pi^{*})_{j}^{t,\beta}(\tilde{y})+(p_{bl}^{*})_{j}^{t,\beta}(\tilde{y}) \right)  \partial_{\tilde{y}_\beta} G_{ji}^{*,0}(\tilde{y}, y) n_r  g^t_r(\tilde{y}) d\tilde{y} \\
			=& \int_{\partial\mathbb{H}_n}n_r \left( g^s_r (\tilde{y}) (\pi^{*})_{j}^{s,\beta}(\tilde{y})-g^t_r (\tilde{y}) (\pi^{*})_{j}^{t,\beta}(\tilde{y}) \right) \partial_{\tilde{y}_\beta} G_{ji}^{*,0}(\tilde{y}, y) (\tilde{y}) d\tilde{y} \\
			&  - \int_{\partial\mathbb{H}_n} n_r \left( g^s_r (\tilde{y}) (p_{bl}^{*})_{j}^{s,\beta}(\tilde{y})-g^t_r (\tilde{y}) (p_{bl}^{*})_{j}^{t,\beta}(\tilde{y}) \right) \partial_{\tilde{y}_\beta} G_{ji}^{*,0}(\tilde{y}, y) (\tilde{y})   d\tilde{y} .
		\end{align*}
		The quantity  $g^s_r (\tilde{y}) (\pi^{*})_{j}^{s,\beta}(\tilde{y})$  is Lipschitz continuous in $s$
		\begin{align*}
			|g^s_r (\tilde{y}) (\pi^{*})_{j}^{s,\beta}(\tilde{y})-g^t_r (\tilde{y}) (\pi^{*})_{j}^{t,\beta}(\tilde{y}) | &= |g_r (\tilde{y}+sn) \pi_{j}^{*,\beta}(\tilde{y}+sn )-g_r (\tilde{y}+tn) \pi_{j}^{*,\beta}(\tilde{y}+tn) | \\
			&\leq \: C |s-t|^\eta
		\end{align*}
	where 	$C=C(d,\mu,\|A\|_{C^{0,\eta}(\mathbb{R}^d)},\|g\|_{C^{0,\eta}(\mathbb{R}^d)}).$ \\
		
	\noindent The same inequality holds for  $ g^s_r (\tilde{y}) (p_{bl}^{*})_{j}^{s,\beta}(\tilde{y})$:
		\begin{align*}
			|g^s_r (\tilde{y}) (p_{bl}^{*})_{j}^{s,\beta}(\tilde{y})-g^t_r (\tilde{y}) (p_{bl}^{*})_{j}^{t,\beta}(\tilde{y})| 
			&\leq \: C |s-t|^\eta.
		\end{align*}  
		Hence we get 
		\begin{align*}
			|U_{s,i}^{0,4}(y)-U_{t,i}^{0,4}(y)| &\leq  C |s-t|^\eta\int_{\partial\mathbb{H}_n} | \partial_{\tilde{y}_\beta} G^{*,0}_{ji} (\tilde{y}, y)| d\tilde{y}\\
			&\leq  C |s-t|^\eta
		\end{align*}with $C$ depending on $d,\mu,\|A\|_{C^{0,\eta}(\mathbb{R}^d)}$ and $\|g\|_{C^{0,\eta}(\mathbb{R}^d)}$.\\
		
	\noindent	All the other terms in the expansion (\ref{expans_tild_u_s}) converge to $0$ as $y\cdot n \rightarrow \infty .$\\
		
		\noindent 	Putting together steps 1 through 4 one obtains
		\begin{align*}
			\sum_{k=1}^4|U_{s}^{0,k}(y)-U_{t}^{0,k}(y)| &\leq  C |s-t|^\eta, \quad \mbox{ for all } y\in \mathbb{H}_n,
		\end{align*}where $C$ depends on $d,\mu,\|A\|_{C^{0,\eta}(\mathbb{R}^d)}$ and $\|g\|_{C^{0,\eta}(\mathbb{R}^d)}$.
		This inequality implies 
		\begin{align*}
			\lim_{y\cdot n \rightarrow \infty}\sum_{k=1}^4|U_{s}^{0,k}(y)-U_{t}^{0,k}(y)| &\leq  C |s-t|^\eta, 
		\end{align*}
		and subsequently
		\begin{align*}
			| U^{s,\infty}(n) -U^{t,\infty}(n) | 
			\leq &	\:\sum_{i,j} \lim_{y\cdot n \rightarrow \infty}  |U_s^{i,j}(y) -U_t^{i,j}(y)| \\
			\leq & \: C |s-t|^\eta
		\end{align*}
		where  $C=C\left(d,\mu,\|A\|_{C^{0,\eta}(\mathbb{R}^d)},\|g\|_{C^{0,\eta}(\mathbb{R}^d)}\right).$
	\end{proof}
	\noindent A direct consequence of this Proposition is that the boundary layer tail $U^{\infty,s}(n)$ is uniformly 
	continuous in $s\in \mathbb{R}.$

	\begin{proposition}
		For $n\notin \mathbb{R}\mathbb{Z}^d$ and for all $s,t\in\mathbb{R}$,	we have 
		\begin{align*}
			U^{t,\infty}(n) =U^{s,\infty}(n).
		\end{align*}
		
	\end{proposition}
	
	\begin{proof}[Proof]
		
			Let $(u,p)$ be the solution to the problem
		\begin{equation*}
			\left\{
			\begin{array}{rcll}
				-\nabla_{\!y}\cdot A(y)\nabla_{\!y}  u +  \nabla_{\!y} p&=& 0 & \mbox{in } \{y\cdot n >0 \},\\
				\nabla_{\!y}	\cdot u&=& 0 ,    &\\
				u(y) &=& g(y) & \mbox{on }\{y\cdot n =0 \}.
			\end{array}
			\right. 
		\end{equation*}
	For arbitrary $h\in\mathbb{Z}^d$, we set
		$u^h:=u(\cdot-h),\;  p^h:=p(\cdot-h)$. Then, by the periodicity of $A=A(y)$ and $g=g(y)$,    the pair $(u^h,p^h)$ satisfies
		\begin{equation*}
			\left\{
			\begin{array}{rcll}
			-\nabla_{\!y}\cdot A(y)\nabla_{\!y} u^h +  \nabla_{\!y} p^h&=& 0& \mbox{in } \{y\cdot n - h\cdot n>0 \},\\
		\nabla_{\!y}	\cdot u^h&=& 0 ,    &\\
				u^h(y) &=& g(y)   & \mbox{on }\{y\cdot n - h\cdot n=0 \}.
			\end{array}
			\right. 
		\end{equation*}
		The pair $(u^h,p^h)$ coincides with the solution  $(u^s,p^s)$ to problem (\ref{hsStokes})  with $s=h\cdot n$. Hence, by the definition of $U^{s,\infty}(n),$ we have $\lim_{y\cdot n \rightarrow \infty} u^h (y)= U^{s,\infty}(n),\;  s=h\cdot n$.
		Then we also have $\lim_{y\cdot n \rightarrow \infty} u^h (y)=\lim_{y\cdot n \rightarrow \infty} u (y-h)= \lim_{y\cdot n \rightarrow \infty} u (y)= U^{\infty}(n):= U^{0,\infty}(n)$. So we obtain the following equality on the boundary layer tails 
		\begin{equation*}
			U^{s,\infty}(n):= U^{\infty}(n),\quad \mbox{ for all } s\in \{h\cdot n: h\in \mathbb{Z}^d \}.
		\end{equation*}
		Because $n\in\mathbb{S}^{d-1}$ is irrational the set  $\{h\cdot n: h\in \mathbb{Z}^d \}$ is a dense subset of $\mathbb{R}$. Now, the uniform continuity on $\mathbb{R}$  of $s\mapsto	U^{s,\infty}(n)$  allows to conclude that 	$U^{s,\infty}(n):= U^{\infty}(n)$ for all  $s\in  \mathbb{R}$.\qedhere\\
	\end{proof}

	\appendix

	\begin{appendices}
	\section{Proof of some estimates}\label{append_estim}
\noindent In what follows we provide some regularity estimates for solutions $u^s\in L^\infty$ of   the  system
\begin{equation}\label{StokesinHs}
	\left\{
	\begin{array}{rcll}
		-\nabla\cdot A (y) \: \nabla u^s +  \nabla p^s&=& 0  &\mbox{in } {\mathbb{H}}_n(s),\\
		\nabla\cdot u^s &=& 0, &\\
		u^s &=& g (y)   &\mbox{on } \partial{\mathbb{H}}_n(s),
	\end{array}
	\right. 
\end{equation}
 satisfying the condition (\ref{class}).\\

\noindent First, remark that if the vector  $u^s$ satisfies (\ref{class}) then, for $q\in(1,\infty)$, one has 
\begin{align}
	\int_{1}^\infty  \| \nabla u^s(M (\cdot,z_d))\|^q_{L_{z'}^\infty(\mathbb{R}^{d-1})} dz_d &=   \int_{1}^\infty  \sup_{z'\in \mathbb{R}^{d-1}} | \nabla u^s(M(z',z_d))|^q dz_d\nonumber \\
	&\leq C   \int_{1}^\infty \frac{1}{\delta^q} d\delta\nonumber \\
	&\leq C ,  \label{estim_deriv_u_delta0} 
\end{align}where the constant $C$ depends on $d,\mu$ and $[A]_{C^{0,\eta}}$.\\

\noindent The bound (\ref{class})  does not allow a control of the $L^q$-norm of $ \nabla_{\! y} u^s$ up to the boundary  $ \partial{\mathbb{H}}_n(s)$. However, one gains control of this derivative near the boundary by using the following local boundary Lipschitz estimate.

\begin{lemma}\label{localboundaryLipsch}
	For arbitrary $\bar{y} \in \partial{\mathbb{H}}_n(s)$ we define  $D_r (\bar{y})$ and  $\Gamma_r (\bar{y})$ as  $D_r (\bar{y} ):= B_r (\bar{y} )\cap {\mathbb{H}}_n(s)$ and  $\Gamma_r (\bar{y}):= B_r (\bar{y})\cap \partial{\mathbb{H}}_n(s)$. Then we have 
	\begin{equation}\label{localboundaryLipsch_estim}
		\| \nabla_{\! y} u^s \|_{L^{\infty}(D_1(\bar{y}))} \leq  C<\infty
	\end{equation}
	where $C$ depends on $d,\mu, [A]_{C^{0,\eta}}$ and $\| g \|_{ C^{1,\eta}(\Gamma_2 ( \bar{y} ) ) }$.
\end{lemma}

\begin{proof}
	
	\noindent
	
	\noindent A Lipschitz estimate given in   \cite{optimBoundEstimate}, Theorem 2.7, yields
	\begin{equation*}\label{estimLips}
		\| \nabla_y \: u^s \|_{L^{\infty}(D_1( \bar{y} ))} \leq  C\left\{ \left(  \dashint_{D_2 ( \bar{y} )} |u^s(y)|^2  dy  \right)^{\!\!\!1/2} + \| \nabla g \|_{ L^{\infty}(\Gamma_2 ( \bar{y} ) ) } +\| \nabla g \|_{ C^{0,\eta}(\Gamma_2 ( \bar{y} ) ) }    \right\}.
	\end{equation*}
	Moreover, for $u^s\in L^\infty({\mathbb{H}}_n(s))$  there is a constant $C$ with $0<C<\infty$ such  that
	\begin{equation*}
	|u^s(y)| \leq   C , \quad \mbox{ for all } y\in\mathbb{H}_n(s),
\end{equation*}
which implies
	\begin{equation*}
	\left(  \dashint_{D_2 ( \bar{y} )} |u^s(y)|^2  dy  \right)^{\!\!\!1/2} \leq C  .
\end{equation*}
Hence one has
	\begin{equation*}
		\| \nabla_y u^s \|_{L^{\infty}(D_1( \bar{y} ))} \leq  C
	\end{equation*}	where $C$ depends on $d,\mu, \eta$ and $[A]_{C^{0,\eta}}$ and $ \| g \|_{ C^{1,\eta}( \partial{\mathbb{H}}_n(s) ) }$.
\end{proof}
\noindent By Lemma \ref{localboundaryLipsch} we have
\begin{align*}
	\| \nabla_y u^s \|_{L^{\infty}(D_1( \bar{y} ))}&\leq  C, \qquad \mbox{ for all }\bar{y} \in \partial{\mathbb{H}}_n(s).
\end{align*}
 This fact implies 
\begin{equation}\label{bound_nab_u_boundry}
	\| \nabla_y u^s \|_{L^{\infty}(\{ 0<y\cdot n -s <1 \})} \leq  C,
\end{equation}
where $C$ depends on  $d,\mu, \eta$, $[A]_{C^{0,\eta}}$ and $ \| g \|_{ C^{1,\eta} (\partial{\mathbb{H}}_n(s)  ) }$.\\

\noindent	The following Lemma gives  the  $L^q$-boundedness of the derivative $\nabla u$ in the normal direction uniformally in the tangential variable. 
\begin{lemma}\label{LinfLq}
	For $q\in(1,\infty)$ we have 
	\begin{align*}
		\int_0^\infty  \sup_{z'\in \mathbb{R}^{d-1}} | \nabla u(M(z',z_d))|^q dz_d \leq C < \infty,
	\end{align*}
	where $C$ depends on  $d,\mu,\eta, [A]_{ C^{0,\eta} }$, $\eta\in(0,1)$ and $ \| g \|_{ C^{1,\eta} (\partial{\mathbb{H}}_n(s)  ) }$.
\end{lemma}
\noindent In particular, one has 
\begin{align}\label{Mq}
	M_q:= \left( \sup_{h\in\mathbb{Z}^{d-1}}\int_0^\infty  \!\!\!\! \int_{T_h} |\nabla u(Mz)|^q dz \right)^{\!\!\!1/q} <\infty , \quad  T_h:= (0,1)^{d-1}+h, \ h\in\mathbb{Z}^{d-1}.
\end{align}
\begin{proof}
	\noindent	Using estimates (\ref{estim_deriv_u_delta0}) and (\ref{bound_nab_u_boundry}) we find that there exists $C=C(d,\mu,\eta, [A]_{C^{0,\eta}},  \| g \|_{ C^{1,\eta} (\partial{\mathbb{H}}_n(s)  ) } )$ such that
	\begin{align}
		\int_{0}^\infty  \| \nabla u^s(M (\cdot,z_d))\|^q_{L_{z'}^\infty(\mathbb{R}^{d-1})} dz_d
		&\leq  C  <\infty . \label{L^2_nab_u}
	\end{align}
	Indeed, we have
	\begin{align*}
		\int_{0}^\infty  \| \nabla u^s(M (\cdot,z_d))\|^q_{L_{z'}^\infty(\mathbb{R}^{d-1})} dz_d &=  	\int_{0}^1 \| \nabla u^s(M (\cdot,z_d))\|^q_{L_{z'}^\infty(\mathbb{R}^{d-1})} dz_d\\ &\qquad\qquad+	\int_{1}^\infty  \| \nabla u^s(M (\cdot,z_d))\|^q_{L_{z'}^\infty(\mathbb{R}^{d-1})} dz_d \\
		&	\leq \int_{0}^1  C\left( \| g \|_{\infty} + \| g \|_{ C^{1,\eta}(\partial\mathbb{H}_n(s) ) } \right)^q dz_d+C \| g \|^q_{\infty} \\
		&\leq  C\left( \| g \|_{\infty}+ \| g \|_{ C^{1,\eta}(\partial\mathbb{H}_n(s) ) } \right)^q +C \| g \|^q_{\infty} <\infty . \qedhere\\
	\end{align*}
\end{proof}

\noindent	{\bfseries Control of the second order derivative}
\begin{lemma}\label{deriv2_L2}
	For all $i=1,...,d,$ one has 	\begin{align*}
		\int_{0}^\infty  \|\partial_{y_i}  \nabla u^s(M (\cdot,z_d))\|^2_{L_{z'}^\infty(\mathbb{R}^{d-1})} dz_d<\infty.
	\end{align*}
\end{lemma}

\begin{proof}
	We shall make use of the previous result through an application of the following Caccioppoli inequality.
	\begin{lemma}\label{Caccio}
		Let  $\Omega \subset \mathbb{R}^d$ be an open subset. Suppose $(u^s,p^s)$ satisfy
		\begin{equation*}
			\left\{
			\begin{array}{rcll}
				-\nabla\cdot A (y) \: \nabla u^s +  \nabla p^s&=& \nabla\cdot F& \mbox{in } \Omega,\\
				\nabla\cdot u^s&=& 0& \mbox{in } \Omega. \\
			\end{array}
			\right. 
		\end{equation*}
		Then for all $y_0\in\Omega $ such that $B(y_0,3)\subset \Omega$ one has
		\begin{equation*}\label{intlips}
			\| \nabla u^s \| _{L^2(B(y_0,2))} \leq  C \left( \| u^s\|_{L^2(B(y_0,3))} + \| F \|_{L^2(B(y_0,3))} \right).\\ 
		\end{equation*}
	\end{lemma}

	\noindent By differentiating the  first two equations of (\ref {StokesinHs})
	we see that the pair $(u^s,p^s)$  satisfies 
	\begin{equation}\label{partial_u}
		\left\{
		\begin{array}{rcll}
			-\nabla\cdot A (y) \: \nabla \partial_{y_i} u^s +  \nabla\partial_{y_i} p^s&=& \nabla\cdot( (\partial_{y_i} A(y) ) \nabla u^s) \;  \quad  &\mbox{in }  \mathbb{H}_n(s),\\
			\nabla\cdot\partial_{y_i} u^s&=& 0 \;  \qquad \qquad \qquad\qquad &\mbox{in }  \mathbb{H}_n(s) \\
		\end{array}
		\right. 
	\end{equation}
	for all $i=1,...,d$.\\
	\noindent Applying Lemma \ref{Caccio} to system (\ref{partial_u}) with $y_0=kn$, for $k\in \mathbb{N}$ and  $k\geq 3$,   yields
	\begin{align}
		\| \nabla \partial_{y_i} u^s \| _{L^2(B(y_0,2))}& \leq  C \left( \|   \partial_{y_i}u^s\|_{L^2(B(y_0,3))} + \|  \partial_{y_i} A(y)  \nabla u^s \|_{L^2(B(y_0,3))} \right)\nonumber\\ 
		\|\partial_{y_i} \nabla_y u^s\|_{L^2 (B(kn,2))}&\leq C (1+ \| \nabla A\|_{\infty} ) \|\nabla_y u^s\|_{L^2(B(kn,3)}\nonumber	\\
		&\leq C (A)\|\nabla_y u^s\|_{L^2(B(kn,3)}	\label{partlipsu}
	\end{align}
	Let $T:=\{ \bar{y}\in\partial\mathbb{H}_n: 0\leq | \bar{y}|\leq 1\}$. Then we have the following covering of the channel 
	$T\times (2,+\infty)$
	\begin{equation*}
		T\times (2,+\infty)	\subset \bigcup^\infty_{k=3} B(kn,2)
	\end{equation*}
	and consequently we obtain the inequality
	\begin{align*}
		\int_{T\times (2,+\infty)}|\partial_{y_i} \nabla_y u^s|^2	\leq  \sum^\infty_{k=3} \int_{B(kn,2)}|\partial_{y_i} \nabla_y u^s|^2
	\end{align*}
	then by (\ref{partlipsu}) and (\ref{L^2_nab_u}) we get 
	\begin{align*}
		\int_{T\times (2,+\infty)}|\partial_{y_i} \nabla_y u^s|^2
		&\leq C(A)\sum^\infty_{k=3}  \|\nabla_y u^s\|^2_{L^2(B(kn,3)}	\\
		&\leq C 	\int_{2}^\infty  \| \nabla u^s(M (\cdot,z_d))\|^2_{L_{z'}^\infty(\mathbb{R}^{d-1})} dz_d\leq C(A,g)<\infty.
	\end{align*}
	On the other hand one applies Lemma \ref{localboundaryLipsch} to show 
	\begin{equation}
		\| \nabla_y u^s \|_{L^{\infty}(T\times (0,2))} \leq  C(g),
	\end{equation}
	and subsequetly
	\begin{align*}
		\int_{T\times (0,+\infty)}|\partial_{y_i} \nabla_y u^s|^2
		&\leq C(A,g)<\infty.
	\end{align*}
	This estimate holds also when $T$ is replaced with $\tilde{y} + T$ for all  $\tilde{y} \in \partial\mathbb{H}_n(s)$. Then we arrive at 
	\begin{align*}
		\int_{0}^\infty  \|\partial_{y_i}  \nabla u^s(M (\cdot,z_d))\|^2_{L_{z'}^\infty(\mathbb{R}^{d-1})} dz_d\leq C(A,g)<\infty,\quad \mbox{for all } i=1,...,d.\qquad \qquad \qquad\qquad \qquad \qedhere
	\end{align*}
\end{proof}

\begin{lemma} 
	Let $r>0$. For $q\in(1,\infty)$ and for a pair $(u^s,p^s)$ satisfying the first two equations of system \textup{(\ref{hsStokes})} we have \begin{equation}
		\|p^s-(p^s)_{D_r} \|_{L^q(D_r)}\leq C\|\nabla u^s \|_{L^q(D_r)},\label{estim_Bogov_p_Appendix} 
	\end{equation}where $C$ depends $d$ and $\mu$
\end{lemma}
\begin{proof}
By a rescaling argument one can assume $r=1$ without loss of generality.
Let $q'\in(1,\infty)$ be such that $\frac{1}{q}+\frac{1}{q'}=1$ and let $\varphi\in L^{q'}(D_1)$. By the solvability of the divergence equation, see \cite{Bogovskii} (Theorem 4.1), there exists $\psi\in W_0^{1,q'}(D_1)$ satisfying 
	\begin{align}
		\nabla \cdot \psi =\varphi -(\varphi)_{D_1} ,\quad \mbox{ and } \quad  \| \psi \|_{W_0^{1,q'}(D_1)} \leq C \| \varphi -(\varphi)_{D_1} \|_{L^{q'}(D_1)}\label{Bogov_Lq}
	\end{align} where the constant $C$ depends on $d.$
	We then get 
	\begin{align*}
		\int_{D_1}(p^s-(p^s)_{D_1}) \varphi & = \int_{D_1}(p^s-(p^s)_{D_1}) (\varphi -(\varphi)_{D_1})\\
		&= \int_{D_1}(p^s-(p^s)_{D_1}) \nabla \cdot \psi
	\end{align*}
	and  integrate by parts to obtain 
	\begin{align*}
		\int_{D_1}(p^s-(p^s)_{D_1}) \varphi 
		&= \int_{D_1}-\nabla p^s\cdot \psi  . 
	\end{align*}
	Now one uses the first equation in (\ref{hsStokes}) and integrate by parts again
	\begin{align*}
		\int_{D_1}(p^s-(p^s)_{D_1}) \varphi 
		&= \int_{D_1}[- \nabla\cdot A (y)  \nabla u^s ]\cdot \psi  \\  
		&= \int_{D_1} A (y)  \nabla u^s \cdot \nabla  \psi  .  
	\end{align*}
	The H\"older  inequality then yields 
	\begin{align*}
		\int_{D_1}(p^s-(p^s)_{D_1}) \varphi 
		&\leq  C \int_{D_1}  | \nabla u^s |\, |\nabla  \psi | \\
		&\leq  C  \| \nabla u^s \|_{L^q(D_1)} \, \|\nabla  \psi \|_{L^{q'}(D_1)}.
	\end{align*}
	Now one combines the latter inequality with the estimate (\ref{Bogov_Lq}) to get
	\begin{align*}
		\int_{D_1}(p^s-(p^s)_{D_1}) \varphi 
		&\leq  C  \| \nabla u^s \|_{L^q(D_1)} \, \| \varphi -(\varphi)_{D_1} \|_{L^{q'}(D_1)}\\
		&\leq  C  \| \nabla u^s \|_{L^q(D_1)} \,  \| \varphi  \|_{L^{q'}(D_1)}.
	\end{align*}
	Seeing that $\varphi$ is an arbitrary element of $L^{q'}(D_1)$ one draws from the last estimate  that 
	\begin{align}
		\|p^s-(p^s)_{D_1} \|_{L^q(D_1)} 
		&\leq  C  \| \nabla u^s \|_{L^q(D_1)}.\label{Bogov_D_1}\qedhere
	\end{align}

\end{proof}

\noindent We also have the following  lemma that will be applied in the proof of the uniqueness of the solution to problem (\ref{hsStokes}).

\begin{lemma}\label{estim_Lq_nabu_p}
Let the pair $\left(u, p\right)$ satisfy
		\begin{equation*}
		\left\{
		\begin{array}{rcll}
			-\nabla\cdot A (y) \nabla u +  \nabla p&=& 0& \mbox{ in } \mathbb{H}_n,\\
			\nabla\cdot u &=& 0 .
		\end{array}
		\right. 
	\end{equation*}
	Then for $r>0$ and $q>1$ we have the following estimates :
	\begin{align}\label{nab_uespL2}
		\|\nabla u  \|^q_{L^q(D_r(0))}\leq  M_q^q\; r^{d-1},
	\end{align}
	and
	\begin{align}\label{estim_Lq_Dr}
		\|p \|^q_{L^q(D_r(0))}\leq C r^{d-1},
	\end{align}
	where $C$ depends on $d, \mu$, $M_q$ and $q$. The constant $M_q$ is defined in \textup{(\ref{Mq})}.
\end{lemma}

\begin{proof}[Proof of \textup{Lemma \ref{estim_Lq_nabu_p}}]

	\noindent \textbf{Step 1 : Control of the velocity $\nabla u$}\\
	We set $D_{1}(0):= \{ z\in\mathbb{H}_n : |z|< 1\}$ and  $D_{r}(0):=  r D_{1}(0)$. 
	We shall use the following fact :  for all $k\in\mathbb{N},$ one has
	\begin{align}\label{nab_u_NM}
		\| \nabla u\|^q_{L^q(2^{k}D_{r}(0))}\leq N_k\cdot M^q_q
	\end{align}
	where $N_k$ is the minimum number of cells $T_h$ that it takes  to cover $\mathbb{R}^{d-1}\cap \overline{2^{k}D_{r}(0)}$. Therefore, the number $N_k$ is at the order of the volume of $\mathbb{R}^{d-1}\cap \overline{2^{k}D_{r}(0)}$ since $|T_h|=1$. More precisely we have 
	\begin{align}\label{number_N}
		N_k =O \left((2^kr)^{d-1}\right).
	\end{align}
Thus, taking $k=0$ in estimates  (\ref{nab_u_NM}) and (\ref{number_N}) yields 
	\begin{align*}
		\| \nabla u \|^q_{L^q(D_{r}(0))}&\leq M^q_q  \;  r^{d-1}.
	\end{align*}

	\noindent \textbf{Step 2 : Control of the pressure  $p$}\, : \\
	We show that the limit  $\lim_{k\rightarrow \infty}(p)_{D(0,2^k r)}$  exists by showing that the series $$\lim_{N\rightarrow \infty}\sum^N_{k=0} (p)_{D(0,2^k r)} -  (p)_{D(0,2^{k+1} r)} = \sum^\infty_{k=0} (p)_{D(0,2^k r)} -  (p)_{D(0,2^{k+1} r)}  $$ 
	converges. From the identity 
	\begin{align}
		(p)_{D(0,2^k r)} -  (p)_{D(0,2^{k+1} r)}  	= \dashint_{D(0,2^k r)} \left(  p(y) -  (p)_{D(0,2^{k+1} r)}   \right) dy\nonumber
	\end{align}	one gets (using Jensen's inequality for instance)
	\begin{align}
		\left| (p)_{D(0,2^k r)} -  (p)_{D(0,2^{k+1} r)} \right|
		&\leq  \left( \dashint_{D(0,2^kr)} \left| p(y) -  (p)_{D(0,2^{k+1} r)}     \right|^q dy \right)^{\!\!1/p} \nonumber  \\
		&\leq  \left( \frac{1}{2^{dk}r^d|D_1(0)|} \int_{D(0,2^{k+1} r)} \left| p(y)-  (p)_{D(0,2^{k+1} r)} \right|^q dy  \right)^{\!\!1/q} \nonumber\\
		&\leq \frac{1}{|D_1(0)|^{1/q}} \frac{1}{2^{dk/q}\, r^{d/q}}\left(\int_{D(0,2^{k+1} r)} \left|   p(y)-  (p)_{D(0,2^{k+1} r)}\right|^q dy \right)^{\!\!1/q} .\nonumber
	\end{align}	
	Then by the estimate 
	\begin{align}
		\left\|  p(y)-  (p)_{D(0,2^{k+1} r)} \right\|_{L^q \left({D(0,2^{k+1} r)}  \right)} \leq C(d,\mu) \| \nabla u \|_{L^q\left(D(0,2^{k+1} r) \right)}\nonumber
	\end{align}
	we obtain the bound
	\begin{align}
		\left| (p)_{D(0,2^k r)} -  (p)_{D(0,2^{k+1} r)}    \right|
		&\leq \frac{C(d,\mu)}{|D_1(0)|^{1/q}} \frac{1}{2^{dk/q}} \frac{1}{r^{d/q}}\| \nabla u \|_{L^q\left(D(0,2^{k+1} r) \right)} .\nonumber
	\end{align}
	Using 	(\ref{nab_u_NM}) together with (\ref{number_N})  we control the quantity $\|\nabla u\|_{L^q(D(0,2^{k+1} r) )} $ as follows
	\begin{align*}
		\|\nabla u \|^q_{L^q(D(0,2^{k+1} r) )}
		\leq & \ M^q_q (2^{k+1}r)^{d-1}\\
		\leq & \ M^q_q \cdot 2^{(k+1)(d-1)}r^{d-1}.
	\end{align*}
	We integrate this fact in the preceding inequality which yields
	\begin{align}
		\left|   (p)_{D(0,2^k r)} -  (p)_{D(0,2^{k+1} r)} \right|
		&\leq M_q \frac{C(d,\mu)}{|D_1(0)|^{1/q}} \frac{1}{2^{dk/q}}\frac{1}{r^{d/q}}\,  2^{(k+1)(d-1)/q}
		\, r^{(d-1)/q} \nonumber \\ 
		&\leq 2^{(d-1)/q}M_q\frac{C(d,\mu)}{|D_1(0)|^{1/q}} \frac{  2^{(d-1)k/q} }{2^{dk/q}}
		\,r^{-1/q} \nonumber\\
		&\leq 2^{(d-1)/q}M_q\frac{C(d,\mu)}{|D_1(0)|^{1/q}} \frac{1}{2^{k/q}} 
		\,r^{-1/q}.\label{pdk}
	\end{align}
	The estimate (\ref{pdk}) yields
	\begin{align*}
		\sum^\infty_{k=0}\left|  (p)_{D(0,2^k r)} -  (p)_{D(0,2^{k+1} r)}  \right| &\leq 2^{(d-1)/q}M_q\frac{C(d,\mu)}{|D_1(0)|^{1/q}} \sum^\infty_{k=0}\left(\frac{1}{2^{1/q}} \right)^k  r^{-1/q}\\
		&\leq C r^{-1/q},
	\end{align*}	
	where $C$ depends on $d,\mu$ and $M_q$. The latter estimate shows that the series  $$\sum^\infty_{k=0}    (p)_{D(0,2^k r)} -  (p)_{D(0,2^{k+1} r)}  $$ is absolutely convergent which implies that the series itself converges. As a consequence the quantity $(p)_\infty :=\lim_{k\rightarrow \infty}(p)_{D(0,2^k r)}$ is well defined.  One can assume that $(p)_\infty =0$ since the pressure $p$ is defined up to a constant. Indeed, if we make the switch $p\rightarrow p- (p)_\infty$ then the limit becomes
	\begin{align*}
		\lim_{k\rightarrow \infty} \left( p-(p)_\infty \right)_{D(0,2^k r)}
		=&\,  \lim_{k\rightarrow \infty} (p)_{D(0,2^{k} r)} - (p)_\infty \\
		=& \: (p)_\infty - (p)_\infty \\
		=& \: 0.
	\end{align*}	
	
	\noindent  Therefore the quantity $p$ is also written 
	\begin{align}
		p
		=&\, p -(p)_{D(0,r)}   +\sum^\infty_{k=0}   (p)_{D(0,2^k r)} -  (p)_{D(0,2^{k+1} r)} \label{def_q}
	\end{align}	
	which implies
	\begin{align}
		\|p\|_{L^q(D(0,r))} 
		\leq&\, \|p -(p)_{D(0,r)} \|_{L^q(D(0,r))}  +\sum^\infty_{k=0}   \|(p)_{D(0,2^k r)} -  (p)_{D(0,2^{k+1} r)} \|_{L^q(D(0,r))}. \label{triangul_ineq}
	\end{align}	
	The estimate (\ref{estim_Bogov}) yields 
	\begin{align*}
		\|p-(p)_{D(0,r)} \|_{L^q(D(0,r))}& \leq C(d,\mu) \|\nabla u\|_{L^q(D(0,r))} \\
		& \leq C(d,\mu) M_q \,r^{d/q}r^{-1/q}.
	\end{align*}	
	
	\noindent Now  let us show that 
	\begin{align*}
		\sum^\infty_{k=0}    \left\| (p)_{D(0,2^k r)} -  (p)_{D(0,2^{k+1} r)}   \right\|_{L^q( D(0,r))}   \leq   C(d,\mu) M_q \,r^{d/q}r^{-1/q}
	\end{align*}
	with $C$ depending on $d,\mu$ and  $M_q$. The estimate (\ref{pdk})  implies 
	\begin{align*}
		\left\|   (p)_{D(0,2^k r)} -  (p)_{D(0,2^{k+1} r)}   \right\|_{L^q( D(0,r))}
		&\leq 2^{(d-1)/q}M_q\cdot \frac{C(d,\mu)}{| D(0, 1 )|^{1/q}} \frac{1}{2^{k/q}} 
		\,r^{-1/q}\cdot |D(0,r)|^{1/q}\\
		&\leq  M_q \cdot C(d,\mu,q)\ \frac{1}{2^{k/q}} 
		\,r^{-1/q} r^{d/q}. 
	\end{align*}
	Then, by summing over $k\in\mathbb{N}$, we obtain
	\begin{align*}
		\sum^\infty_{k=0} \left\| (p)_{D(0,2^k r)} -  (p)_{D(0,2^{k+1} r)}   \right\|_{L^q( D(0,r))}
		&\leq M_q\cdot C(d,\mu,q)  \left( \sum^\infty_{k=0}  \frac{1}{2^{k/q}} \right) 
		\, \,r^{-1/q} r^{d/q}  \\
		&\leq C(M_q, d,\mu,q) \, r^{d/q}r^{-1/q}.
	\end{align*}
	
	\noindent Hence by the inequality (\ref{triangul_ineq})	 we get 
	\begin{align*}
		\|	p\|_{L^q(D(0,r))}\
		\leq \;  &   C(M_q, d,\mu,q) r^{d/q}r^{-1/q}. \qedhere\\
\end{align*}	\end{proof}

	\section{Proof of Proposition \ref{Propo_estim_Green_y}}\label{append_Gfunction}
	
	\begin{proof}[Proof of \textup{Proposition \ref{Propo_estim_Green_y}}]
		The proof is divided into several steps.\\
		\noindent  \underline{Step 1 : Proof of the estimate}  $| G(y,\tilde{y})| \leq C  \frac{1 }{|y-\tilde{y}|^{d-2}}.$ \\ 
		\noindent  Let $y,\tilde{y}\in \mathbb{H}_n$. 
		Set $r:=|y-\tilde{y}|$ and let $(u,p)$ be the solution of 
		\begin{equation}\label{system_(u,p)}
			\left\{
			\begin{array}{rcll}
				-\nabla\cdot A(y) \: \nabla u  +  \nabla p&=&  f \;  \quad& \mbox{in } {\mathbb{H}}_n,\\
				\nabla\cdot u&=& 0, \;  \quad & \\
				u&=& 0\quad & \mbox{on } \partial{\mathbb{H}}_n,
			\end{array}
			\right. 
		\end{equation}
		where $f\in C^\infty_0(B(\tilde{y},r/3))$.

	\noindent Then  the velocity $u$ admits the following representation 
		\begin{align}
			u(y)&=  \int_{\mathbb{H}_n}  G(y,z)\cdot f(z) dz\\
			&=  \int_{\mathbb{H}_n\cap B(\tilde{y},r/3)}  G(y,z)\cdot f(z) dz.\label{u_equal_int Gfdz}
		\end{align}
		From the fact that $f$ vanishes on $D_{r/3}(y)$ one draws that 
		\begin{equation}
			\left\{
			\begin{array}{rcll}
				-\nabla_{\!z}\cdot A(z) \nabla_{\!z} u(z)  +  \nabla_{\!z}p(z)&=& 0 \;  \quad& \mbox{in } D_{r/3}(y),\\
				\nabla_{\!z}\cdot u(z)&=& 0 \;  \quad & \mbox{in } D_{r/3}(y),\\
				u(z)&=& 0\quad & \mbox{on }  \Gamma_{r/3}(y).
			\end{array}
			\right. 
		\end{equation}
		Then, by a rescaled version of Lemma \ref{Holder_estim}   there exists a constant $C$ depending on $d, \mu, \|A\|_{C^{0,\eta}}$, such that, for $\lambda\in(0,1)$, we have
		\begin{align}
			[u]_{C^{0,\lambda}(D_{r/6}(y))} &\leq C \frac{1}{r^{d/2+\lambda}}  \| u\|_{L^2(D_{r/3}(y))}.   \label{rescaled_Holder}
		\end{align}
		If $z_1,z_2\in  D_{r/6}(y)$ then $|z_1-z_2|\leq r/3$ and $3^\lambda\frac{|u(z_1)-u(z_2) |}{r^\lambda} \leq \frac{|u(z_1)-u(z_2) |}{|z_1-z_2|^\lambda}$. Hence, the estimate (\ref{rescaled_Holder})  leads to
		\begin{align*}
			|u(y)|\leq C	\left(\dashint_{D(y,r/3)} |u(z)|^2 dz\right)^{\!\!\!1/2}.
		\end{align*}
		Rewriting the latter inequality with the representation formula (\ref{represent_u})
		one obtains
		\begin{align*}
			\left| \int_{\mathbb{H}_n\cap B(\tilde{y},r/3)} G(y ,z)\cdot f(z) dz \right|&\leq  C	\left(\dashint_{\mathbb{H}_n\cap B(y,r/3)} |u(z)|^2 dz\right)^{\!\!\!1/2}\\
			&\leq  C	\left(\dashint_{\mathbb{H}_n\cap B(y,r/3)} |u(z)|^{\frac{2d}{d-2}} dz\right)^{\!\!\!\frac{d-2}{2d}} \mbox{ by H\"older inequality}\\
			&\leq  C	\left(\frac{1}{r^d}\int_{\mathbb{H}_n} |u(z)|^{\frac{2d}{d-2}} dz\right)^{\!\!\!\frac{d-2}{2d}}\\
			&\leq  C	\frac{1}{r^\frac{d-2}{2}}	\left(\int_{\mathbb{H}_n} |u(z)|^{\frac{2d}{d-2}} dz\right)^{\!\!\!\frac{d-2}{2d}}.
		\end{align*}
		We then use the Sobolev inequality to get
		\begin{align}
			\left| \int_{\mathbb{H}_n\cap B(\tilde{y},r/3)} G(y ,z)\cdot f(z) dz \right|
			&\leq  \frac{C}{r^\frac{d-2}{2}}	\left(\int_{\mathbb{H}_n} |\nabla u(z)|^{2} dz\right)^{\!\!\!\frac{1}{2}}\nonumber
		\end{align}
		which, combined with  the energy estimate \begin{align*}
			\|\nabla u \|_{L^{2}(\mathbb{H}_n)}
			&\leq  Cr\| f \|_{L^{2}(\mathbb{H}_n)},
		\end{align*}
		yields
		\begin{align*}
			\left| \int_{\mathbb{H}_n\cap B(\tilde{y},r/3)} G(y ,z)\cdot f(z) dz \right|
			&\leq  \frac{Cr}{r^{\frac{d}{2}-1}}	\; \| f \|_{L^{2}(\mathbb{H}_n)}
		\end{align*}
		and by duality this inequality gives
		\begin{align}
			\left( \int_{\mathbb{H}_n\cap B(\tilde{y},r/3)} |G(y ,z)|^2 dz \right)^{1/2}
			&\leq  \frac{Cr}{r^{\frac{d}{2}-1}}	\nonumber\\
			&\leq  \frac{Cr^2}{r^{d/2}}	\label{ineq_r^2_r^(d-2)}.
		\end{align}
		Now one multiplies both sides of (\ref{ineq_r^2_r^(d-2)}) by $\frac{1}{r^{d/2}}$ 
		\begin{align}
			\left( \dashint_{\mathbb{H}_n\cap B(\tilde{y},r/3)} |G(y ,z)|^2 dz \right)^{1/2}
			&\leq  C \frac{1}{r^{d-2}} \label{L2_norm_of_G}.
		\end{align}
		By applying Lemma \ref{Holder_estim} to the system
		\begin{equation*}
			\left\{
			\begin{array}{rcll}
				-\nabla_z\cdot A^* (z)  \nabla_z G^*(z,y) +  \nabla_z \Pi^*(z,y)&=&0 \ &\mbox{in } \mathbb{H}_n\cap B(\tilde{y},r/3),\\
				\nabla_z\cdot G^*(z,y)&=& 0 \;  &\mbox{in } \mathbb{H}_n\cap B(\tilde{y},r/3),\\
				G^*(z,y) &=& 0  & \mbox{on } \partial \mathbb{H}_n\cap B(\tilde{y},r/3) .
			\end{array}
			\right. 
		\end{equation*}
		we obtain 
		\begin{align}
			|  G^*(\tilde{y},y)   |\leq C	\left(\dashint_{\mathbb{H}_n\cap B(\tilde{y},r/3)} |   G^*(z,y)      |^2 dz\right)^{\!\!\!1/2}.\label{Gy_tildy_leq_L2norm_of_G}
		\end{align}
		Now the  estimates (\ref{L2_norm_of_G}) and  (\ref{Gy_tildy_leq_L2norm_of_G}) together give
		\begin{align*}
			|  G^*(\tilde{y},y)   |\leq C	\frac{1}{r^{d-2}} 
		\end{align*}
		Then by the symmetry property (\ref{sym}) we get
		\begin{align*}
			|  G(y,\tilde{y})   |\leq C	\frac{1}{r^{d-2}} .\\
		\end{align*}

		\noindent \underline{Step 2 : Proof of the estimates}  $| \nabla_y G(y,\tilde{y})| \leq C  \frac{1 }{|y-\tilde{y}|^{d-1}}$ and  $| \nabla_{\tilde{y}} G(y,\tilde{y})| \leq C  \frac{1 }{|y-\tilde{y}|^{d-1}}.$\\
		Let us set $r:=\frac{2}{3}|y-\tilde{y}|.$ 
		We apply a rescaled version of Theorem \ref{unif_estim} to the system of equations 
		\begin{equation}
			\left\{
			\begin{array}{rcll}
				-\nabla_z\cdot A (z)  \nabla_z G(z,\tilde{y}) +  \nabla_z \Pi(z,\tilde{y})&=& 0 &\mbox{in } D_r(y),\\
				\nabla_z\cdot G(\cdot,\tilde{y})&=& 0    &\mbox{in } D_r(y),\\
				G(z,\tilde{y}) &=& 0  & \mbox{on } \Gamma_r(y).
			\end{array}
			\right. 
		\end{equation}
		to obtain
		\begin{align}
			\|\nabla_y G(\cdot ,\tilde{y})\|_{L^{\infty}(D_{r/2}(y))} 
			&\leq \frac{C}{r} \|  G(\cdot ,\tilde{y})\|_{L^{\infty}(D_{r}(y))}.\label{estimLinf}
		\end{align}
		From  (\ref{estim_G_y}) one draws
		\begin{align*}
			\|  G(\cdot ,\tilde{y})\|_{L^{\infty}(D_{r}(y))} =& \sup_{z\in{D_{r}(y)}} |G(z ,\tilde{y})|\nonumber \\
			\leq& \sup_{z\in{D_{r}(y)}} \frac{C }{|z-\tilde{y}|^{d-2}}
		\end{align*}
		and, since $$r=\frac{2}{3}|y-\tilde{y}| \mbox{  implies  } |z-\tilde{y}| \geq r/2  \mbox{ for all } z \in{D_{r}(y)},$$ 
		we have
		\begin{align*}
			\|  G(\cdot ,\tilde{y})\|_{L^{\infty}(D_{r}(y))} 
			& \leq\frac{ C }{r^{d-2}}. 
		\end{align*}
		Then by the inequality (\ref{estimLinf}) one obtains
		\begin{align}
			|\nabla_y G(y,\tilde{y})|  &\leq\|\nabla_y G(\cdot ,\tilde{y})\|_{L^{\infty}(D_{r/2}(y))}\nonumber\\ 
			& \leq\frac{C}{r}  \frac{1}{r^{d-2}}\nonumber\\
			& \leq \frac{C}{r^{d-1}} \label{dminus1}
		\end{align}
		that is to say
		\begin{align*}
			|\nabla_y G(y ,\tilde{y})|
			\leq&  \frac{C }{|y-\tilde{y}|^{d-1}},
		\end{align*}
		where $C$ depends on $d, \mu$ and  $[A]_{C^{0,\eta}}$.

		\noindent 	As for the proof of the estimate $| \nabla_{\tilde{y}} G(y,\tilde{y})| \leq C  \frac{1 }{|y-\tilde{y}|^{d-1}}$,  we first  apply the preceding argument to the system
		\begin{equation}
			\left\{
			\begin{array}{rcll}
				-\nabla_{\tilde{y}}\cdot A^* (\tilde{y}) \: \nabla_{\tilde{y}} G^*(\tilde{y},y) +  \nabla_{\tilde{y}} \Pi^*(\tilde{y},y)&=& 0 &\mbox{in }  D_r( \tilde{y}),\\
				\nabla_{\tilde{y}} \cdot G^*(\tilde{y},y)&=& 0 \,  &\mbox{in } D_r( \tilde{y} ),\\
				G^*(\tilde{y},y) &=& 0 \,  &\mbox{on } \Gamma_r(\tilde{y})
			\end{array}
			\right.  
		\end{equation}	
		to obtain
		\begin{equation*}
			|\nabla_{\tilde{y}} G^*(\tilde{y},y)| \leq \frac{C }{|y-\tilde{y}|^{d-1}}.
		\end{equation*}
		Then the symmetry property	(\ref{sym}) yields 
		\begin{equation}\label{estimtild}
			|\nabla_{\tilde{y}} G(y,\tilde{y})| =|\nabla_{\tilde{y}} G^T\!(y,\tilde{y})|
			=|\nabla_{\tilde{y}} G^*(\tilde{y},y)| \leq \frac{C }{|y-\tilde{y}|^{d-1}}.
		\end{equation}

		\noindent \underline{Step 3 : Proof of the estimates}  $| G(y,\tilde{y})| \leq C  \frac{y\cdot n}{|y-\tilde{y}|^{d-1}}$ and  $| G(y,\tilde{y})| \leq C  \frac{ \tilde{y} \cdot n}{|y-\tilde{y}|^{d-1}}.$\\
		In this proof we set $r:=|y-\tilde{y}|.$ 
		First, we suppose $y\cdot n \geq \frac{r}{3}$.  Then by the estimate $$| G(y,\tilde{y})| \leq \frac{C}{|y-\tilde{y}|^{d-2}}$$  we have 
		\begin{align*}
			| G(y,\tilde{y})| &\leq 3\;  \frac{y\cdot n}{r}\frac{C}{|y-\tilde{y}|^{d-2}} \\
			&\leq \;C\frac{y\cdot n}{|y-\tilde{y}|^{d-1}} .
		\end{align*}
		Now we assume that $y\cdot n< \frac{r}{3}$  and let  $\bar{y}\in\partial\mathbb{H}_n(0)$ be such that $|y-\bar{y}|=y\cdot n$. 
		The mean value inequality yields 
		\begin{align*}
			| G(y,\tilde{y})- G(\bar{y},\tilde{y})| &\leq \|\nabla_y G(\cdot ,\tilde{y})\|_{L^{\infty}(D(\bar{y}, \frac{r}{3})} \; |y-\bar{y}| .
		\end{align*}
		Then by (\ref{estim_Dy_G}) and the fact that $G(\bar{y},\tilde{y})=0$,    we get 
		\begin{align}\label{G_sup_dminus1}
			| G(y,\tilde{y})  | &\leq  \sup_{z\in D(\bar{y}, \frac{r}{3})} \frac{C }{|z-\tilde{y}|^{d-1}}  \; y\cdot n .
		\end{align}
		Moreover one has  $$|z-\tilde{y}| \geq \frac{r}{3}  \mbox{ for all } z \in{D(\bar{y}, r/3)    }.$$ 
		Then estimate (\ref{G_sup_dminus1}) gives 
		\begin{align*}
			| G(y,\tilde{y})  | &\leq  C\, \frac{y\cdot n }{r^{d-1}},\\       
			&\leq  C\, \frac{y\cdot n }{ |y-\tilde{y}| ^{d-1}}.\\
		\end{align*}	
		
		\noindent As for the other estimate
		we apply the same reasoning right above by
	interchanging  $y$ and  $\tilde{y},$ and using the mean value inequality along with (\ref{estim_dertildG}).\\
		
		\noindent \underline{Step 4 : Proof of the estimates}  $| \nabla_y G(y,\tilde{y})| \leq C  \frac{\tilde{y}\cdot n}{|y-\tilde{y}|^{d}}$ and  $| \nabla_{\tilde{y}}G(y,\tilde{y})| \leq C  \frac{ y \cdot n}{|y-\tilde{y}|^{d}}.$\\
		Let  $r:=\frac{2}{3}|y-\tilde{y}|.$ 
		A rescaled version of Theorem \ref{unif_estim} applied to the system of equations 
		\begin{equation}\label{system_D_r}
			\left\{
			\begin{array}{rcll}
				-\nabla_z\cdot A (z)  \nabla_z G(z,\tilde{y}) +  \nabla_z \Pi(z,\tilde{y})&=& 0    &\mbox{in } D_r(y),\\
				\nabla_z\cdot G(\cdot,\tilde{y})&=& 0   &\mbox{in } D_r(y),\\
				G(z,\tilde{y}) &=& 0  & \mbox{on } \Gamma_r(y).
			\end{array}
			\right. 
		\end{equation}
		yields
		\begin{align}
			\|\nabla_y G(\cdot ,\tilde{y})\|_{L^{\infty}(D_{r/2}(y))} 
			&\leq \frac{C}{r} \|  G(\cdot ,\tilde{y})\|_{L^{\infty}(D_{r}(y))}.\label{estimL}
		\end{align}
		From (\ref{estim_G_y}) one draws
		\begin{align*}
			\|  G(\cdot ,\tilde{y})\|_{L^{\infty}(D_{r}(y))} =& \sup_{z\in{D_{r}(y)}} |G(z ,\tilde{y})|\nonumber \\
			\leq& \; C\sup_{z\in{D_{r}(y)}}  \frac{\tilde{y}\cdot n }{|z-\tilde{y}|^{d-1}}
		\end{align*}
		and, since $$r=\frac{2}{3}|y-\tilde{y}| \mbox{  implies  } |z-\tilde{y}| \geq r/2  \mbox{ for all } z \in{D_{r}(y)},$$ 
		we have
		\begin{align*}
			\|  G(\cdot ,\tilde{y})\|_{L^{\infty}(D_{r}(y))} 
			& \leq  2^{d-1} C  \;\frac{\tilde{y}\cdot n }{r^{d-1}}. \label{boundCrd2}
		\end{align*}
		Then by the estimate (\ref{estimL}) one obtains
		\begin{align*}
			|\nabla_y G(y,\tilde{y})|  &\leq\|\nabla_y G(\cdot ,\tilde{y})\|_{L^{\infty}(D_{r/2}(y))}\nonumber\\ 
			& \leq 2^{d-2} C \; \frac{1  }{r} \frac{\tilde{y}\cdot n  }{r^{d-1}}\nonumber\\
			& \leq C \; \frac{\tilde{y}\cdot n }{r^{d}} 
		\end{align*} 
		and one replaces $r$ by $\frac{2}{3}|y-\tilde{y}|$ to obtain
		\begin{equation}
			|\nabla_y G(y ,\tilde{y})|
			\leq C \; \frac{\tilde{y}\cdot n  }{|y-\tilde{y}|^{d}}.\\ 
		\end{equation}
		
		\noindent To prove  the other estimate, we first apply the the preceding estimate to the system 
		\begin{equation}
			\left\{
			\begin{array}{rcll}
				-\nabla_{\tilde{y}}\cdot A^* (\tilde{y}) \: \nabla_{\tilde{y}} G^*(\tilde{y},y) +  \nabla_{\tilde{y}} \Pi^*(\tilde{y},y)&=& 0   &\mbox{in }  D_r( \tilde{y} ) ,\\
				\nabla_{\tilde{y}} \cdot G^*(\tilde{y},y)&=& 0  & \mbox{in }  D_r( \tilde{y} ) ,\\
				G^*(\tilde{y},y) &=& 0 \,   &\mbox{on } \Gamma_r( \tilde{y}),
			\end{array}
			\right.  
		\end{equation}	
		to get
		\begin{equation*}
			| \nabla_{\tilde{y}} G^*(\tilde{y},y)| \leq C \; \frac{y\cdot n}{|y-\tilde{y}|^d} .
		\end{equation*}
		We then use the symmetry (\ref{sym}) to obtain
		\begin{equation*}
			| \nabla_{\tilde{y}} G(y,\tilde{y})| =  | \nabla_{\tilde{y}} G^T\!(y,\tilde{y})| = | \nabla_{\tilde{y}} G^*(\tilde{y},y)|   \leq C \; \frac{y\cdot n}{|y-\tilde{y}|^d} .\\
		\end{equation*}
		
		\noindent \underline{Step 5 : Proof of the estimate}  $| G(y,\tilde{y})| \leq C  \frac{(y\cdot n)(\tilde{y}\cdot n) }{|y-\tilde{y}|^d} .$\\
		We apply again the reasoning carried out in Step  3. 
		Let $r:=|y-\tilde{y}|.$ 
		First, we suppose $y\cdot n \geq \frac{r}{3}$.  Then by the estimate (\ref{estim_G_y})  we have 
		\begin{align*}
			| G(y,\tilde{y})| &\leq 3\;  \frac{y\cdot n}{r}\left( C\, \frac{ \tilde{y}\cdot n}{|y-\tilde{y}|^{d-1}} \right)\\
			&\leq \;C\,\frac{  (y\cdot n)  (\tilde{y}\cdot n)  }{|y-\tilde{y}|^{d}} .
		\end{align*}
		Now we assume that $y\cdot n< \frac{r}{3}$  and let  $\bar{y}\in\partial\mathbb{H}_n(0)$ be such that $|y-\bar{y}|=y\cdot n$. 
		The mean value inequality yields 
		\begin{align*}
			| G(y,\tilde{y})- G(\bar{y},\tilde{y})| &\leq \|\nabla_y G(\cdot ,\tilde{y})\|_{L^{\infty}(D(\bar{y}, \frac{r}{3})} \; |y-\bar{y}| .
		\end{align*}
		Therefore, by  (\ref{estim_Dy_G}) together with $G(\bar{y},\tilde{y})=0$,    we get 
		\begin{align*}
			| G(y,\tilde{y})  | &\leq  C \sup_{z\in D(\bar{y}, \frac{r}{3})} \frac{ \tilde{y}\cdot n}{|z-\tilde{y}|^{d}}  \; y\cdot n ,
		\end{align*}
		which leads to
		\begin{align*}
			| G(y,\tilde{y})  | &\leq  C\, \frac{(y\cdot n) (\tilde{y}\cdot n) }{r^{d}},\\       
			&\leq  C\, \frac{ (y\cdot n) (\tilde{y}\cdot n)   }{ |y-\tilde{y}| ^{d}}.\\
		\end{align*}

		\noindent \underline{Step 6 : Proof of the estimate}  $| \nabla_y \nabla_{\tilde{y}}G(y,\tilde{y})| \leq C  \frac{(y\cdot n)(\tilde{y}\cdot n) }{|y-\tilde{y}|^d} .$\\
		Let  $r:=\frac{2}{3}|y-\tilde{y}|.$ 
		First, one differentiates the system (\ref{system_D_r}) with respect to the variable  $\tilde{y}_k$, $k=1,...,d$,  to obtain 
		\begin{equation}
			\left\{
			\begin{array}{rcll}
				-\nabla_y\cdot A(y) \: \nabla_y ( \partial_{\tilde{y}_k}G(y,\tilde{y}) )  +  \nabla_y  ( \partial_{\tilde{y}_k} \Pi(y,\tilde{y}) ) &=& 0 \ &\mbox{in } D_r(y),\\
				\nabla_y\cdot  \partial_{\tilde{y}_k}G(y,\tilde{y})&=& 0 \;   & \mbox{in }  D_r(y),\\
				\partial_{\tilde{y}_k}G(y,\tilde{y}) &=& 0  & \mbox{on } \Gamma_r(y).
			\end{array}
			\right. 
		\end{equation}	
		Then Theorem \ref{unif_estim} applied to this system yields 
		\begin{align}
			\|\nabla_y  (  \partial_{\tilde{y}_k} G(\cdot,\tilde{y}) )\|_{L^{\infty}(D_{r/2}(y))} &\leq \frac{C}{r}  \left(  \dashint_{D_{r}(y)} |  \partial_{\tilde{y}_k} G(z,\tilde{y})|^2 dz \right)^{1/2}\nonumber \\
			&\leq \frac{C}{r} \| \partial_{\tilde{y}_k}   G(\cdot,\tilde{y})\|_{L^{\infty}(D_{r}(y))}. \label{Lipschitk}
		\end{align}
		Furthermore, by the estimate (\ref{estimtild}), one has
		\begin{align}
			\|  \partial_{\tilde{y}_k}  G(\cdot,\tilde{y})\|_{L^{\infty}(D_{r}(y))}&= \sup_{z\in{D_{r}}}  |  \partial_{\tilde{y}_k}  G(z,\tilde{y})| \nonumber\\
			&\leq \sup_{z\in{D_{r}}}   \frac{C }{|z-\tilde{y}|^{d-1}}\nonumber\\
			&\leq   \frac{C }{r^{d-1}}. \label{kboundr}
		\end{align}
		The inequalities  (\ref{Lipschitk}) and (\ref{kboundr}) jointly give
		\begin{align}
			\|\nabla_y (\partial_{\tilde{y}_k}  G (\cdot,\tilde{y}) )\|_{L^{\infty}(D_{r/2}(y))} 
			&\leq   \frac{C }{r^d} \nonumber
		\end{align}
		hence
		\begin{align*}
			\|\nabla_y ( \partial_{\tilde{y}_k}   G(\cdot,\tilde{y}) )\|_{L^{\infty}(D_{r/2}(y))} 
			&\leq   \frac{C }{|y-\tilde{y}|^d} .\\
		\end{align*}
		\noindent \underline{Step 7 : Proof of the estimates} \,  $ | \Pi(y,\tilde{y}) -  \Pi(z,\tilde{y})| \leq C \min \;\left\{\frac{1 }{|y-\tilde{y}|^{d-1}}  ,\frac{\tilde{y}\cdot n }{|y-\tilde{y}|^{d}} \right\} .$\\
		Let $\tilde{y}\in\mathbb{H}_n$ be fixed and let $y,z\in\mathbb{H}_n$ be such that $|y-\tilde{y}|\leq |z-\tilde{y}|$. We set $r:=|y-\tilde{y}|$ and consider the annuli $A_k:=D(\tilde{y}, 2^kr)\bs \overline{D(\tilde{y},2^{k-1}r)},\, k=1,2, ..$. For each $k$ we find a covering of $\overline{A}_k$ as follows. For instance, for $k=1$ there  is a covering $(B_{1j})_{j=1,...,N}$ of $\overline{A}_1$   where,  for all $j=1,...,N$, 
		$B_{1j}$ is a ball satisfying $\diam B_{1j}= 2 r$ and $\dist(\tilde{y},  B_{1j})\geq 2^{-1} r$. Then, through the map $\sigma_k : y\mapsto 2^{k-1} (y-\tilde{y}) + \tilde{y}$  one can see that the balls $B_{kj}= \sigma_k(B_{ 1j}), j=1,...,N$ form a covering of $\overline{A}_k$  and satisfy $\diam B_{kj}= 2^k r$ and $\dist(\tilde{y},  B_{kj})\geq 2^{k-2} r$. Moreover, the smallest finite number $N$ that can be taken depend only on $d$.
		Now one applies Theorem \ref{unif_estim} to the system  
		\begin{equation}
			\left\{
			\begin{array}{rcll}
				-\nabla_z\cdot A (z)  \nabla_z G(z,\tilde{y}) +  \nabla_z \Pi(z,\tilde{y})&=& 0   &\mbox{in } (5/4) B_{kj} ,\\
				\nabla_z\cdot G(\cdot,\tilde{y})&=& 0   &\mbox{in }   (5/4)B_{kj},\\
				G(z,\tilde{y}) &=& 0  & \mbox{on } (5/4)	B_{kj}\cap \partial\mathbb{H}_n
			\end{array}
			\right. 
		\end{equation}
		to obtain
		\begin{align}
			\underset{B_{kj}}{\textup{osc}}	|  \Pi(\cdot,\tilde{y})|
			&\leq  \frac{C}{2^kr}  \|  G(\cdot,\tilde{y})\|_{L^{\infty}((5/4) B_{kj})}\label{estim_Linf_Bkj}.
		\end{align}
		Here, $(5/4) B_{kj}$ denotes the ball that is $5/4$ times larger than $ B_{kj}$ and with the same center. The inequality $\dist(\tilde{y},  B_{kj})\geq 2^{k-2} r$ then gives $\dist(\tilde{y},  (5/4)B_{kj})\geq 2^{k-3} r$. 
		From (\ref{estim_G_y}) one draws
		\begin{align*}
			\|  G(\cdot ,\tilde{y})\|_{L^{\infty}( (5/4)B_{kj})} &= \sup_{z\in{   (5/4)B_{kj}}} |G(z ,\tilde{y})|\nonumber \\
			&\leq  C \min \left\{ \sup_{z\in{  (5/4)B_{kj}}} \frac{1 }{|z-\tilde{y}|^{d-2}};	
			\sup_{z\in{(5/4)B_{kj}}}	\frac{\tilde{y}\cdot n }{|y-\tilde{y}|^{d-1}}  \right\}
		\end{align*}
		and, since $$ \dist(\tilde{y},(5/4) B_{kj})\geq 2^{k-3} r  \mbox{  implies  } |z-\tilde{y}| \geq 2^{k-3} r  \mbox{ for all } z \in{  (5/4)B_{kj}},$$ 
		we have
		\begin{align*}
			\|  G(\cdot ,\tilde{y})\|_{L^{\infty}(  (5/4) B_{kj})} 
			& \leq  \min \left\{ \frac{ C }{2^{k(d-2)}r^{d-2}}, \frac{ \tilde{y}\cdot n }{2^{k(d-1)}r^{d-1}}  \right\} .
		\end{align*}
		Then by (\ref{estim_Linf_Bkj}) we obtain
		\begin{align*}
			\underset{B_{kj}}{\textup{osc}}	|  \Pi(\cdot,\tilde{y})| &\leq \frac{C}{2^kr}   \min \left\{  \frac{1 }{ 2^{k(d-2)}r^{d-2}  };	\frac{\tilde{y}\cdot n }{2^{k(d-1)}r^{d-1}  }\right\}\\
			&\leq C \min \left\{  \frac{1 }{ 2^{k(d-1)} }  \frac{1 }{ r^{d-1}  };	\frac{1}{2^{kd} }	\frac{ \tilde{y}\cdot n  }{r^{d}  }\right\}	.  
		\end{align*}
		From the fact
		\begin{align*}
			\mathbb{H}_n\bs D_r(\tilde{y})\subseteq \bigcup_{k\geq 1} \overline{A_k}	 \subseteq \bigcup_{k\geq 1}  \left( \bigcup^{N}_{j=1}  B_{kj}	  \right)
		\end{align*}
		one draws
		\begin{align*}
			\underset{ \mathbb{H}_n\bs D_r(\tilde{y})  }{\textup{osc}}	|  \Pi(\cdot,\tilde{y})| 
			&\leq \sum^{\infty}_{k=1} \sum^{N}_{j=1}	\underset{B_{kj}}{\textup{osc}}	|  \Pi(\cdot,\tilde{y})|\\
			&\leq  C \sum^{\infty}_{k=1} \sum^{N}_{j=1}	\min \left\{  \frac{1 }{ 2^{k(d-1)} }  \frac{1 }{ r^{d-1}  };	\frac{1}{2^{kd} }	\frac{ \tilde{y}\cdot n }{r^{d}  }\right\}\\
			&\leq  C N	\min \left\{ \sum^{\infty}_{k=1}   \left(\frac{1 }{ 2^{d-1} }\right)^k  \!\!\frac{1 }{ r^{d-1}  }, \sum^{\infty}_{k=1}  \left( \frac{1}{2^{d} }\right)^k \! \frac{ \tilde{y}\cdot n }{r^{d}  }\right\}\\
			&\leq  C N	\min \left\{  \frac{1 }{ r^{d-1}  },	\frac{ \tilde{y}\cdot n }{r^{d}  }\right\},
		\end{align*}
		that is to say
		\begin{align*}
			\underset{ \mathbb{H}_n\bs D_r(\tilde{y})  }{\textup{osc}}	|  \Pi(\cdot,\tilde{y})|
			&\leq  C 	\min \left\{  \frac{1 }{ |y-\tilde{y}|^{d-1}  };	\frac{\tilde{y}\cdot  n}{|y-\tilde{y}|^{d}  }\right\}.
		\end{align*}
		This estimate clearly implies
		\begin{align*}
			| \Pi(y,\tilde{y}) -  \Pi(z,\tilde{y})| 
			&\leq  C 	\min \left\{  \frac{1 }{ |y-\tilde{y}|^{d-1}  };	\frac{\tilde{y}\cdot n}{|y-\tilde{y}|^{d}  }\right\},\quad \mbox{ for all } z\in  \mathbb{H}_n, |z-\tilde{y}|\geq |y-\tilde{y}|.\\
		\end{align*}
		
		\noindent \underline{Proof of the estimates} \,  $ | \Pi(y,\tilde{y})| \leq C \min \;\left\{\frac{1 }{|y-\tilde{y}|^{d-1}}  ,\frac{\tilde{y}\cdot n }{|y-\tilde{y}|^{d}}   \right\} .$\\
By using arguments similar to those in  \cite{HigPrZhuge} (Appendix B) one obtains 
	$$ |\Pi(z,\tilde{y})|\xrightarrow[z\cdot n>2]{|z|\rightarrow\infty}0.$$ 
Then, by taking the limits in estimates (\ref{estim_diff_Pi_y}) as $|z|\xrightarrow[z\cdot n>2]{}\infty$ one gets
		\begin{align*}
			| \Pi(y,\tilde{y}) | 
			&\leq  C 	\min \left\{  \frac{1 }{ |y-\tilde{y}|^{d-1}  };	\frac{\tilde{y}\cdot n}{|y-\tilde{y}|^{d}  }\right\}.
		\end{align*}

		\noindent \underline{Step 8 : Proof of the estimates} \,  $	|\nabla_{\tilde{y}} \Pi(y,\tilde{y}) - \nabla_{\tilde{y}} \Pi(z,\tilde{y})|\leq \;C \min \;\left\{\frac{ 1}{|y-\tilde{y}|^{d}}; \frac{\tilde{y}\cdot n }{|y-\tilde{y}|^{d+1}}   \right\} .$\\	
		Again we set $r:=|y-\tilde{y}|$ and consider the annuli $A_k:=D(\tilde{y}, 2^kr)\bs D(\tilde{y},2^{k-1}r),\, k=1,2, ...$ 
		and the covering $(B_{kj})_{j=1,...,N}$ of $\overline{A}_k$ defined in the previous step. We apply Theorem \ref{unif_estim} to the system  
		\begin{equation}
			\left\{
			\begin{array}{rcll}
				-\nabla_y\cdot A(y) \nabla_y ( \partial_{\tilde{y}_k}G(y,\tilde{y}) )  +  \nabla_y  ( \partial_{\tilde{y}_k} \Pi(y,\tilde{y}) ) &=& 0 \   \quad \mbox{in }  (5/4) B_{kj}\\
				\nabla_y\cdot  \partial_{\tilde{y}_k}G(y,\tilde{y})&=& 0 \;   \quad  \mbox{in }  (5/4) B_{kj} ,\\
				\partial_{\tilde{y}_k}G(y,\tilde{y}) &=& 0 \; \quad  \mbox{on }  (5/4)	B_{kj}\cap \partial\mathbb{H}_n,
			\end{array}
			\right. 
		\end{equation}
		to obtain 
		\begin{align}
			\underset{B_{kj}}{\textup{osc}}	| \partial_{\tilde{y}_k} \Pi(\cdot,\tilde{y})|
			&\leq  \frac{C}{2^kr}  \|\partial_{\tilde{y}_k}G(\cdot,\tilde{y})\|_{L^{\infty}((5/4)B_{kj})}\label{estim_Linf_Bkj_der_Pi}
		\end{align}
		From estimates  (\ref{estim_dertildG}) one draws
		\begin{align*}
			\| \partial_{\tilde{y}_k} G(\cdot ,\tilde{y})\|_{L^{\infty}(  (5/4)B_{kj})} 
			&	\leq C \min \left\{ \sup_{z\in{ (5/4) B_{kj}}}  \frac{1 }{|z-\tilde{y}|^{d-1}}; \sup_{z\in{ (5/4)B_{kj}}} 
			\frac{y\cdot n }{|z-\tilde{y}|^{d}}  \right\}
		\end{align*}
		which leads to
		\begin{align*}
			\| \partial_{\tilde{y}_k} G(\cdot ,\tilde{y})\|_{L^{\infty}((5/4)  B_{kj})} 
			& \leq   \min \left\{ \frac{ C }{2^{k(d-1)}r^{d-1}}, \frac{ y\cdot n }{2^{kd}r^{d}}  \right\} .
		\end{align*}
		Then by (\ref{estim_Linf_Bkj_der_Pi}) we obtain
		\begin{align*}
			\underset{B_{kj}}{\textup{osc}}	| \partial_{\tilde{y}_k} \Pi(\cdot,\tilde{y})| 
			&\leq C \min \left\{  \frac{1 }{ 2^{kd} }  \frac{1 }{ r^{d}  };	\frac{1}{2^{k(d+1)} }	\frac{y\cdot n}{r^{d+1}  }\right\}	.  
		\end{align*}
		From the fact
		\begin{align*}
			\mathbb{H}_n\bs D_r(\tilde{y})\subset \bigcup_{k\geq 1} \overline{A}_k	 \subset \bigcup_{k\geq 1}  \left( \bigcup^{N}_{j=1}  B_{kj}	  \right)
		\end{align*}
		one draws
		\begin{align*}
			\underset{ \mathbb{H}_n\bs D_r(\tilde{y})  }{\textup{osc}}	|   \partial_{\tilde{y}_k}\Pi(\cdot,\tilde{y})| &\leq   \sum^{\infty}_{k=1} \sum^{N}_{j=1}	\underset{B_{kj}}{\textup{osc}}	| \partial_{\tilde{y}_k} \Pi(\cdot,\tilde{y})|\\
			&\leq  C \sum^{\infty}_{k=1} \sum^{N}_{j=1}	\min \left\{  \frac{1 }{ 2^{kd} }  \frac{1 }{ r^{d}  };	\frac{1}{2^{k(d+1)} }	\frac{y\cdot n}{r^{d+1}  }\right\}\\
			&\leq  C N	\min \left\{ \sum^{\infty}_{k=1}   \left(\frac{1 }{ 2^{d} }\right)^k  \frac{1 }{ r^{d}  }; \sum^{\infty}_{k=1}  \left( \frac{1}{2^{d+1} }\right)^k	\frac{y\cdot n}{r^{d+1}  }\right\}\\
			&\leq  C N	\min \left\{  \frac{1 }{ r^{d}  };	\frac{y\cdot n}{r^{d+1}  }\right\}
		\end{align*}
		that is to say
		\begin{align*}
			\underset{ \mathbb{H}_n\bs D_r(\tilde{y})  }{\textup{osc}}	|   \partial_{\tilde{y}_k}\Pi(\cdot,\tilde{y})|
			&\leq  C 	\min \left\{  \frac{1 }{ |y-\tilde{y}|^{d}  };	\frac{y\cdot n}{|y-\tilde{y}|^{d+1}  }\right\}.
		\end{align*}
		This estimate implies
		\begin{align*}
			| \partial_{\tilde{y}_k} \Pi(y,\tilde{y}) -  \partial_{\tilde{y}_k} \Pi(z,\tilde{y})| 
			&\leq  C 	\min \left\{  \frac{1 }{ |y-\tilde{y}|^{d}  };	\frac{y\cdot n}{|y-\tilde{y}|^{d+1}  }\right\},\quad \mbox{ for all } z\in  \mathbb{H}_n, |z-\tilde{y}|\geq |y-\tilde{y}|.
		\end{align*}
		Now, we take the limit as $|z|\rightarrow\infty$ and use the fact $\lim_{|z|\rightarrow\infty} | \partial_{\tilde{y}_k}\Pi(z,\tilde{y})|=0$  in the latter estimate to get
		\begin{align*}
			|  \partial_{\tilde{y}_k}\Pi(y,\tilde{y}) | 
			&\leq  C 	\min \left\{  \frac{1 }{ |y-\tilde{y}|^{d}  };	\frac{y\cdot n}{|y-\tilde{y}|^{d+1}  }\right\}.\qedhere\\
		\end{align*}
		
	\end{proof}

	\section{Uniqueness in Theorem \ref{intro_ThConverge}}\label{uniqueness}
\noindent	Here we give a proof of the uniqueness of the solution $u^s$ to problem (\ref{hsStokes}) under hypothesis (\ref{class}). Once the uniqueness property is satisfied the solution to problem such as (\ref{hsStokes}) can be represented by an integral formula e.g. through Poisson's kernel.

	\begin{proof}[Proof of the uniqueness of the solution to problem \textup{(\ref{hsStokes})}]
		If $(u^s_1, p^s_1)$ and $(u^s_2,p^s_2)$ are two solutions of (\ref{hsStokes}) then the pair $(u,p)$, $u:=u_1^s-u^s_2, \, p: =p_1^s - p_2^s,$ satisfies 
		the problem
		\begin{equation}\label{RStokes0}
			\left\{
			\begin{array}{rcl}
				-\nabla\cdot A (y) \nabla u +  \nabla p&=& 0 \;\quad \mbox{ in } \mathbb{H}_n(s)  ,\\
				\nabla\cdot u &=& 0  ,\\
				u	&=& 0  \quad \mbox{ on }  \partial\mathbb{H}_n(s).
			\end{array}
			\right. 
		\end{equation}
		We intend to show that $u=0$  by showing that for any $f\in C^\infty_c( \mathbb{H}_n(s) )$ one has 
		\begin{align}\label{targeted_result}
			\int_{  \mathbb{H}_n(s) }u(y)\cdot f(y) dy = 0.\end{align}
		
		\noindent So let $f\in C^\infty_c( \mathbb{H}_n(s)  )$ and consider the pair $(V,S)$ satisfying 
		\begin{equation}\label{RStokes_VR}
			\left\{
			\begin{array}{rcl}
				-\nabla\cdot A^* (y) \nabla V + \:  \nabla S&=& f \;\quad \mbox{ in } \mathbb{H}_n(s),\\
				\nabla\cdot V&=& 0  ,\\
				V	&=& 0  \quad \mbox{ on }  \partial \mathbb{H}_n(s).
			\end{array}
			\right. 
		\end{equation}
		The solution $V$ admits a representation formula through Green's kernel $G^*$ associated with the operator $	-\nabla\cdot A^* (y) \nabla$ and the domain $\mathbb{H}_n(s)$
		\begin{align}\label{repr_V}
			V(y)=\int_{ \mathbb{H}_n(s) } G^*(y,\tilde{y}) f(\tilde{y}) d\tilde{y}.
		\end{align}
		
		\noindent We claim that $u$, solution to problem (\ref{RStokes0}), satisfies 
		\begin{align}\label{target_result}
			\int_{ \mathbb{H}_n(s)  } f(y)\cdot u(y) dy = 0.
		\end{align}
		
		\noindent Indeed, by carrying out integrations by parts one obtains
		\begin{align*}
			\int_{ \mathbb{H}_n(s)  }f(y)\cdot u(y) dy=&	\int_{ \mathbb{H}_n(s)}   f^i(y) u^i(y) dy\\ = &	\int_{ \mathbb{H}_n(s)  }[-\partial_{y_\alpha} (A^*)^{\alpha\beta}_{ij}(y) \partial_{y_\beta} V^j(y)  + \:  \partial_{y_i}S (y)]\,  u^i(y) dy\\
			= &	\int_{ \mathbb{H}_n(s)  } (A^*)^{\alpha\beta}_{ij}(y) \partial_{y_\beta} V^j(y)\partial_{y_\alpha}u^i(y)   dy - \: \int_{ \mathbb{H}_n(s)  } S (y) \partial_{y_i} u^i(y) dy\\
			=&	\int_{ \mathbb{H}_n(s) }  A^* (y) \nabla V\cdot \nabla u \, dy + \:\int_{  \mathbb{H}_n(s) } -S \, (\nabla \cdot \ \!   u) dy.
		\end{align*}
			\noindent	The integrations by parts that we perform here are justified below.
		Now, by using the second equation in (\ref{RStokes0}) the last line reduces to 
		\begin{align*}
			\int_{ \mathbb{H}_n(s) }f(y)\cdot u(y) dy
			=&	\int_{ \mathbb{H}_n(s) } \nabla u\cdot A^* (y) \nabla V dy\\
			=&	\int_{  \mathbb{H}_n(s) } A(y)\nabla u\cdot  \nabla V dy.
		\end{align*}
		We integrate by parts again and use the first equation in (\ref{RStokes0})
		\begin{align*}
			\int_{ \mathbb{H}_n(s)  }f^i(y) u^i(y) dy=&	\int_{  \mathbb{H}_n(s) }    A^{\alpha\beta}_{ij}(y) \partial_{y_\beta} u^j(y)\, \partial_{y_\alpha}\! V^i(y) dy\\
			=& \int_{  \mathbb{H}_n(s) } (- \partial_{y_\alpha}\! A^{\alpha\beta}_{ij}(y) \partial_{y_\beta} u^j(y))\,  V^i(y)  dy\\
			=& \int_{ \mathbb{H}_n(s) } - \partial_i p(y) V^i(y) dy.
		\end{align*}
		Another integration by parts yields
		\begin{align*}
			\int_{ \mathbb{H}_n(s) } - \partial_i p(y) V^i(y) dy=& \int_{ \mathbb{H}_n(s) }   p(y)\partial_i V^i(y) dy\\
			=&\int_{  \mathbb{H}_n(s) }  p\,  (\nabla \cdot V )dy
		\end{align*}
		Then we use the second equation in (\ref{RStokes_VR}) to get
		\begin{align*}
			\int_{ \mathbb{H}_n(s) }  p \,  (\nabla \cdot  V )dy=0.
		\end{align*}
		This proves that 
		\begin{align*}
			\int_{ \mathbb{H}_n(s) }f(y)\cdot u(y) dy=0
		\end{align*}
		for arbitrary $f\in C^\infty_c( \mathbb{H}_n(s) )$. Hence $u=0$.
	\end{proof}

		\noindent	Next, we set forth a justification of  the  integrations by parts used in the above proof of the uniqueness of $u^s$.\\

		\noindent   Without loss of generality, we set $s=0$.
		
		\noindent {\bf Integration by parts of the term $ \int_{ \mathbb{H}_n  }   [-\partial_{y_\alpha} (A^*)^{\alpha\beta}_{ij}(y) \partial_{y_\beta} V^j(y)  ]\,  u^i(y) dy $}.\\
		We use the following equality
		\begin{align*}
			\int_{  \mathbb{H}_n } [-\partial_{y_\alpha} (A^*)^{\alpha\beta}_{ij}(y) \partial_{y_\beta} V^j(y)  ]\,  u^i(y)  dy = \lim_{R\rightarrow\infty} \int_{D(0,R)} [-\partial_{y_\alpha} (A^*)^{\alpha\beta}_{ij}(y) \partial_{y_\beta} V^j(y)  ]\,  u^i(y) dy
		\end{align*}
	to reduce the integration over an unbounded domain to an integration over a bounded domain. Now we carry out integrations by parts for the integral in the right-hand-side :
		\begin{align*}
			\int_{D(0,R)}& [-\partial_{y_\alpha} (A^*)^{\alpha\beta}_{ij}(y) \partial_{y_\beta} V^j(y)  ]\,  u^i(y)  dy\\ &= 	\int_{D(0,R)}  (A^*)^{\alpha\beta}_{ij}(y) \partial_{y_\beta} V^j(y)  \, \partial_{y_\alpha} u^i(y)   + 	\int_{\partial D(0,R) }  -u^i(y)e^\alpha (A^*)^{\alpha\beta}_{ij}(y) \partial_{y_\beta} V^j(y) dy
		\end{align*} where $e^\alpha$ is the $\alpha^{th}$ component of the unit exterior normal $e$ of $\partial D(0,R)$.	Then taking the limits as $R\rightarrow\infty$ we obtain
		\begin{align*}
			\lim_{R\rightarrow\infty} 	\int_{D(0,R)} (A^*)^{\alpha\beta}_{ij}(y) \partial_{y_\beta} V^j(y)  \, \partial_{y_\alpha} u^i(y) dy  = \int_{    \mathbb{H}_n}  (A^*)^{\alpha\beta}_{ij}(y) \partial_{y_\beta} V^j(y)  \, \partial_{y_\alpha} u^i(y)dy
		\end{align*}
		provided
		\begin{equation}\label{int_bound_term}
			\lim_{R\rightarrow\infty} 	\int_{\partial D(0,R) }  -  u^i(y)e^\alpha (A^*)^{\alpha\beta}_{ij}(y) \partial_{y_\beta} V^j(y) dy = 0.
		\end{equation}
	We will show that (\ref{int_bound_term}) holds below.\\
		
		\noindent {\bf Integration by parts of the term $ \int_{ \mathbb{H}_n  }  \partial_{y_i}S (y)\,  u^i(y) dy$}.\\
		Similarly, for the term involving the pressure $S$ one has 
		\begin{align*}
			\int_{  \mathbb{H}_n  }  [ \partial_{y_i}S (y)]\,  u^i(y) dy= & 	\lim_{R\rightarrow\infty}  \int_{  D(0,R) }  [\partial_{y_i}S (y)]\,  u^i(y)   dy \\
			=&		\lim_{R\rightarrow\infty} \left( \int_{ D(0,R)  } -S(y)\partial_{y_i} u^i(y)dy   + 	\int_{\partial D(0,R)  }  S(y) u^i(y) e^i dy \right) \\
			=&		\lim_{R\rightarrow\infty}  \int_{  D(0,R)  } -S(y)\partial_{y_i} u^i(y)dy   \\
			=&		 \int_{ \mathbb{H}_n}    -S(y)\partial_{y_i} u^i(y) dy
		\end{align*}
		as long as
		\begin{equation}
			\lim_{R\rightarrow\infty} 	\int_{ \partial D(0,R)  }   S(y) u^i(y) e^i  dy
			=		 0 .\label{int_bound_term_2}
		\end{equation}
		We will show that (\ref{int_bound_term_2}) holds below.\\	
		
		\noindent {\bf Integration by parts of the term $  \int_{\mathbb{H}_n}  A^{\alpha\beta}_{ij}(y) \partial_{y_\beta} u^j(y)\, \partial_{y_\alpha}\! V^i(y)  dy $}.\\ 
		On the one hand we have the limiting equality
		\begin{align*}
			\int_{  \mathbb{H}_n}& A^{\alpha\beta}_{ij}(y) \partial_{y_\beta} u^j(y)\, \partial_{y_\alpha}\! V^i(y) dy
		= 	\lim_{R\rightarrow\infty} 	\int_{D(0,R)}  A^{\alpha\beta}_{ij}(y) \partial_{y_\beta} u^j(y)\, \partial_{y_\alpha}\! V^i(y) dy
		\end{align*}and on the other hand we have the integration by parts for the integral in the right hand side
		\begin{align*}
			\int_{D(0,R)}&   A^{\alpha\beta}_{ij}(y) \partial_{y_\beta} u^j(y)\, \partial_{y_\alpha}\! V^i(y)  dy \\ 	&=\int_{ D(0,R) } (- \partial_{y_\alpha}\! A^{\alpha\beta}_{ij}(y) \partial_{y_\beta} u^j(y))\,  V^i(y)  dy  + 	\int_{\partial D(0,R) }   V^i(y)e^\alpha A^{\alpha\beta}_{ij}(y) \partial_{y_\beta} u^j(y) dy .
		\end{align*}
		Now one lets $R \rightarrow\infty$ in the preceding equality to obtain
		\begin{align*}
			\lim_{R\rightarrow\infty}	\int_{  D(0,R) } A^{\alpha\beta}_{ij}(y) \partial_{y_\beta} u^j(y)\, \partial_{y_\alpha}\! V^i(y) dy = &\lim_{R\rightarrow\infty}	\int_{  D(0,R) }   (- \partial_{y_\alpha}\! A^{\alpha\beta}_{ij}(y) \partial_{y_\beta} u^j(y))\,  V^i(y) dy  \\
			& + \lim_{R\rightarrow\infty}	\int_{\partial  D(0,R) }  V^i(y)e^\alpha A^{\alpha\beta}_{ij}(y) \partial_{y_\beta} u^j(y) dy \\
			= &	\int_{   \mathbb{H}_n }  (- \partial_{y_\alpha}\! A^{\alpha\beta}_{ij}(y) \partial_{y_\beta} u^j(y))\,  V^i(y)     dy
		\end{align*}
		as long as
		\begin{equation}\label{int_bound_term_3}
			\lim_{R\rightarrow\infty}	\int_{\partial D(0,R)}  V^i(y)e^\alpha A^{\alpha\beta}_{ij}(y) \partial_{y_\beta} u^j(y)  dy=0.\\\nonumber
		\end{equation}
			We will show that (\ref{int_bound_term_3}) holds below.\\
			
		\noindent {\bf Integration by parts of the term $ \int_{   \mathbb{H}_n }  - \partial_i p(y) V^i(y)dy$}.\\ 
		For this term involving a pressure, we have 
		\begin{align*}
			\int_{\mathbb{H}_n} - \partial_i p(y) V^i(y) dy=&	\int_{ \mathbb{H}_n }  - \nabla p \cdot V dy\\
			=&\lim_{R\rightarrow\infty}	\int_{D(0,R)}  - \nabla p \cdot V dy\\
			=& \lim_{R\rightarrow\infty} \left(\int_{D(0,R)}  p\,  (\nabla \cdot  V )dy -\int_{\partial D(0,R)  }   p (y)\,   V(y)  \cdot e \,  dy \right)\\
			=&  \int_{ \mathbb{H}_n}  p(y)\,  (\nabla \cdot  V (y))   dy  
		\end{align*}
		as long as 
		\begin{equation}\label{int_bound_term_4}
			\lim_{R\rightarrow\infty} 	\int_{\partial D(0,R)}   p(y) \,  V(y) \cdot e \, dy =0 
		\end{equation}
		where the vector $e$ is the unit exterior normal of $\partial D(0,R)$.
			We will show that (\ref{int_bound_term_4}) holds below.\\

		\noindent Now we prove that the integrals  over the boundary $ \partial D(0,R)$ in (\ref{int_bound_term})-(\ref{int_bound_term_4})   vanish as $R\rightarrow\infty$.
	\begin{lemma}
			Let 
			$f\in C^\infty_c(  \mathbb{H}_n )$ and let the pair $(V,S)$ be a solution of
			\begin{equation}\label{Stokes_VS}
				\left\{
				\begin{array}{rcl}
					-\nabla\cdot A^* (y) \nabla V + \:  \nabla S&=& f \;\quad \mbox{ in }  \mathbb{H}_n,\\
					\nabla\cdot V&=& 0  ,\\
					V	&=& 0  \quad \mbox{ on }  \partial\mathbb{H}_n ,
				\end{array}
				\right. 
			\end{equation}   
			then there exists a constant $C>0$ such that  for  $y\in\mathbb{H}_n  $ and $|y|$ large enough, we have
			\begin{align}
				|V(y)| &\leq C \frac{y\cdot n}{|y|^d}\label{estim_V}, \\
				|\nabla V(y)| &\leq C \frac{1}{|y|^d},\label{estim_nab_V} \\
				|S(y)| &\leq C \frac{1}{|y|^d}.\label{estim_on_R}
			\end{align}
		\end{lemma}

		\noindent {\bf Step 1 : proof of (\ref{int_bound_term})}\\
		Seeing that $u=0$ on $\partial\mathbb{H}_n $ we have
		\begin{align*}
			\int_{\partial D(0,R) }  - u^i(y)e^\alpha (A^*)^{\alpha\beta}_{ij}(y) \partial_{y_\beta} V^j(y) dy = C \int_{\partial D(0,R)\bs \partial\mathbb{H}_n}  - u^i(y)e^\alpha (A^*)^{\alpha\beta}_{ij}(y) \partial_{y_\beta} V^j(y) dy
		\end{align*}
		which yields
		\begin{align*}
			\left| \int_{\partial D(0,R) }  - u^i(y)e^\alpha (A^*)^{\alpha\beta}_{ij}(y) \partial_{y_\beta} V^j(y)dy  \right| &\leq C \int_{ \partial D(0,R)\bs \partial \mathbb{H}_n }  |u(y)|\cdot | \nabla V(y) |dy.
		\end{align*}
		From  $u:=u_1^s-u_2^s$ with $u_1^s, u_2^s$ belonging to $L^\infty(\mathbb{H}_n(s))$, we get
		\begin{align}
			|u(y)|&\leq |u^s_1(y)| +  |u^s_2(y)| \nonumber\\
			&\leq C , \qquad \qquad \qquad \mbox{ for all } y\in  \mathbb{H}_n,\label{boundedness_of_v}
		\end{align}
		and from
		estimate (\ref{estim_nab_V}) we draw 
		\begin{align*}
			| \nabla V(y) |
			&  \leq   \frac{C}{R^d},\qquad \qquad \mbox{ for  } y\in \partial D(0,R)\bs \partial \mathbb{H}_n ,
		\end{align*}
		where $C>0$ is independent of $R$.  Then we obtain
		\begin{align*}
			\left| \int_{\partial D(0,R) }  -  u^i(y)e^\alpha (A^*)^{\alpha\beta}_{ij}(y) \partial_{y_\beta} V^j(y)  dy  \right|
			& \leq C \,\frac{1}{R^d} \int_{\partial D(0,R) }  dy\\
			& \leq C  \,\frac{1}{R^d} C(d) R^{d-1}\\
			& \leq C \frac{1}{R} ,
		\end{align*}
		where $C$ is independent of $R.$ This implies 
		\begin{align*}
			\left| \int_{\partial D(0,R) }   -   u^i(y)e^\alpha (A^*)^{\alpha\beta}_{ij}(y) \partial_{y_\beta} V^j(y)dy \right|\xrightarrow{R\rightarrow \infty}0.
		\end{align*}

		\noindent {\bf Step 2 : proof of (\ref{int_bound_term_2})}\\
		Given that $u=0$ on $\partial\mathbb{H}_n $ one has the equality
		\begin{align*}
			\int_{\partial D(0,R)}   S(y) u^i(y) e^i  dy= & 	\int_{\partial D(0,R)\bs \partial\mathbb{H}_n }    S(y) u^i(y) e^i dy
		\end{align*}
		which implies 
		\begin{align*}
			\left|  \int_{\partial D(0,R)}   S(y) u^i(y) e^i  dy \right| \leq  & 	\int_{\partial D(0,R)\bs \partial\mathbb{H}_n    }  |S(y)| |u(y)|dy.
		\end{align*}
		Then by using estimates (\ref{estim_on_R}) and (\ref{boundedness_of_v}) one obtains
		\begin{align*}
			\left| \int_{\partial D(0,R)}   S(y) u^i(y) e^i dy\right| \leq  & 	\; C\int_{  \partial D(0,R) \bs \partial\mathbb{H}_n}    \frac{  1}{R^{d}} dy\\ 
			\leq  & 	\; C    \frac{1}{R^{d}} \cdot R^{d-1}\\
			\leq  & 	\; C   \, \frac{1}{R^{}} 
		\end{align*}
		hence
		\begin{align*}
			\lim_{R\rightarrow\infty} 	\int_{  \partial D(0,R) }    S(y) u^i(y) e^i dy
			=&		 0 .
		\end{align*}
		
		\noindent {\bf Step 3 : proof of (\ref{int_bound_term_3})}.\\		
		Because $V=0$ on  $\partial\mathbb{H}_n $   we have the equality 
		\begin{align*}
			\int_{  \partial D(0,R)  }  V^i(y)e^\alpha A^{\alpha\beta}_{ij}(y) \partial_{y_\beta} u^j(y)  dy= \int_{  \partial D(0,R) \bs \partial\mathbb{H}_n}   V^i(y)e^\alpha A^{\alpha\beta}_{ij}(y) \partial_{y_\beta} u^j(y) dy
		\end{align*}
		which implies 
			\begin{align*}
			\left|\int_{  \partial D(0,R)}   V^i(y)e^\alpha A^{\alpha\beta}_{ij}(y) \partial_{y_\beta} u^j(y)  dy \right| 
			&\leq \int_{  \partial D(0,R)  \bs \partial\mathbb{H}_n  } | V (y)|\cdot  |\nabla u(y)|  dy.
		\end{align*}
	As $u^s_1$ and $u^s_2$ satisfy (\ref{class}) one has
		\begin{align*}
			|\nabla u(y)| &\leq |\nabla u^s_1(y)| + |\nabla u^s_2(y)|,    \\
			&\leq   \frac{C}{y\cdot n}, \qquad    \qquad   \qquad \mbox{ for all } y\in\mathbb{H}_n .
		\end{align*}
		Therefore, by using estimate (\ref{estim_V}) one gets
		\begin{align*}
			\left|\int_{  \partial D(0,R)}   V^i(y)e^\alpha A^{\alpha\beta}_{ij}(y) \partial_{y_\beta} u^j(y)  dy \right| 
			&\leq C \int_{\partial D(0,R) \bs \partial\mathbb{H}_n }    \frac{1}{y\cdot n} \frac{y\cdot n}{|y|^d}dy,\\
			&\leq C \int_{ \partial D(0,R) \bs \partial\mathbb{H}_n }    \frac{1}{|y|^d}dy,\\
			&\leq C \int_{ \partial D(0,R) \bs \partial\mathbb{H}_n  }    \frac{1}{R^{d}}dy,\\
			&\leq C  \frac{1}{R},
		\end{align*}
		where $C$ is independent of $R$. Which yields 
		\begin{align*}
			\int_{  \partial D(0,R) \bs \partial\mathbb{H}_n  }    V^i(y)e^\alpha A^{\alpha\beta}_{ij}(y) \partial_{y_\beta} u^j(y)dy  \xrightarrow[]{R\rightarrow\infty} 0.\\
		\end{align*}
		
		\noindent {\bf Step 4 : proof of  (\ref{int_bound_term_4})}.\\
		The estimate (\ref{estim_Lq_Dr}) gives : for $q>1,$
		\begin{align*}
			\int_{D(0,R)} |p(y)|^q dy
			\leq \;  &   C(M_q, d,\mu) R^{d-1},\quad \mbox{ for all  } R\geq 1,
		\end{align*}	
		which  yields a sequence  $(R_k)_k$  of positive numbers satisfying  $\lim_{k\rightarrow\infty} R_k =\infty$ and
		\begin{align}\label{estim_of_q_on_partialD_R_k}
			\int_{ \partial D(0,R_k)  } |p(y)|^q dy \leq C R_k^{d-2+\alpha}, \quad \mbox{ for some } \alpha\in[0,1).
		\end{align}	
		One can replace $R$ with $R_k$ in (\ref{int_bound_term_4}). Indeed, we have
		\begin{align*}
			\int_{ \mathbb{H}_n } - \nabla p(y)  \cdot V(y) dy=&\lim_{R_k\rightarrow\infty}	\int_{D(0,R_k)}  - \nabla p(y) \cdot V(y) dy.
		\end{align*}
		As $V=0$ on $\partial\mathbb{H}_n $ one can write
		\begin{align*}
			\left|	\int_{\partial D(0,R_k)}   p(y) \, V(y) \cdot e \,  dy\right| & =	\left|\int_{\partial D(0,R_k)\bs \partial\mathbb{H}_n}   p(y) V(y) \cdot e \, dy \right|\\
			& \leq\int_{ \{ |z|=R_k\}  }	\left|  p(y) V(y)\right| dy .
		\end{align*}
		Then taking  $q,q'\in(1,\infty)$, with $\frac{1}{q}+\frac{1}{q'}=1$,  and applying the H\"older inequality one obtains 
		\begin{align*}
			\left|	\int_{\partial D(0,R_k)}   p(y)   V(y) dy\right| 	
			& \leq	\left(\int_{  \{ |z|=R_k\}  }   | p(y)|^q dy\right)^{\frac{1}{q}}\left(\int_{ \{ |z|=R_k\}   }   |  V(y)|^{q'} dy \right)^{\frac{1}{q'}}.
		\end{align*}
		Then, estimate  (\ref{estim_V}) yields
		\begin{align*}
			\left|	\int_{\partial D(0,R_k)}   p(y)   V(y) dy\right| 	
			&\leq  C\left(\int_{  \{ |y|=R_k\}  }   |p(y)|^q dy\right)^{\frac{1}{q}}\left(\int_{ \{ |y|=R_k\}   } \frac{(y\cdot n)^{q'}}{(|y|^d)^{ q'}} dy \right)^{\frac{1}{q'}}
		\end{align*}
		and using the fact $0\leq y\cdot n \leq |y|= R_k$ and estimate (\ref{estim_of_q_on_partialD_R_k}) we get
		\begin{align*}
			\left|	\int_{\partial D(0,R_k)}   p(y)  V(y) \cdot e \, dy\right| 	
			&\leq C\left(\int_{  \{ |y|=R_k\}  }   | p(y)|^q dy\right)^{\frac{1}{q}} \left(\int_{ \{ |y|=R_k\}   } \frac{R^{q'}_k}{R_k^{dq'}} dy\right)^{\frac{1}{q'}}\\
			&\leq C \left(    R_k^{d-2+\alpha } \right)^{\frac{1}{q}} \left(  R_k^{q'-dq'} R_k^{d-1} \right)^{\frac{1}{q'}}\\
			&\leq  C R_k^{\frac{d}{q}+\frac{\alpha-1}{q}-\frac{1}{q}}  R_k^{1-d} R_k^{\frac{d}{q'}-\frac{1}{q'}} \\
			&\leq  C R_k^{\frac{\alpha-1}{q}}, 
		\end{align*}
		where $\alpha<1$ and $C$ is independent of $R_k$.
		Thus, it follows from the last inequality that
		\begin{align*}
			\left|	\int_{\partial D(0,R_k)}   p(y) V(y)\cdot e \, dy\right| \xrightarrow{R_k\rightarrow\infty} 0.
		\end{align*}
\section{Quasiperiodicity of $\partial_{y_\gamma} u^s_{}$  and $p^s_{}$ on the hyperplane $\{y\cdot n-s=0\}$}\label{quasiperiod}
\noindent 	Let $\theta\in\mathbb{R}^d$ and let the pair $(u^{\theta\cdot n},p^{\theta\cdot n})$ satisfy
\begin{equation*}
	\left\{
	\begin{array}{rcll}
		-\nabla_{\!y}\cdot A(y)\nabla_{\!y}  u^{\theta\cdot n} +  \nabla_{\!y} p^{\theta\cdot n}&=& 0 & \mbox{in } \{y\cdot n- \theta\cdot n>0 \},\\
		\nabla_{\!y}	\cdot u^{\theta\cdot n}&=& 0 ,    &\\
		u^{\theta\cdot n}(y) &=& g(y)   & \mbox{on }\{y\cdot n- \theta\cdot n =0 \}.
	\end{array}
	\right. 
\end{equation*}
For arbitrary $h\in\mathbb{Z}^d$, we set
$u^h:= u^{\theta\cdot n}(\cdot-h),\; p^h:= p^{\theta\cdot n}(\cdot-h)$. Then, by the periodicity of $A=A(y)$ and $g=g(y)$,    the pair $(u^h,p^h)$ satisfies
\begin{equation*}
	\left\{
	\begin{array}{rcll}
		-\nabla_{\!y}\cdot A(y)\nabla_{\!y} u^h +  \nabla_{\!y} p^h&=& 0& \mbox{in } \{y\cdot n - \theta\cdot n- h\cdot n>0 \},\\
		\nabla_{\!y}	\cdot u^h&=& 0  ,    &\\
		u^h(y) &=& g(y)   & \mbox{on }\{y\cdot n- \theta\cdot n - h\cdot n=0 \}.
	\end{array}
	\right. 
\end{equation*}
Now one can observe that he pair $(u^h,p^h)$ coincides with the solution  $(u^s,p^s)$ to problem (\ref{hsStokes})  with $s=\theta\cdot n+h\cdot n$. Therefore, by the uniqueness of the solution we have 
\begin{align*}
	u^{(\theta+h)\cdot n}(y)&= u^h(y)\\
	&=u^{\theta\cdot n}(y-h),\qquad \mbox{ for all }\;  y\cdot n - (\theta+ h)\cdot n>0.
\end{align*}
and 
\begin{align*}
	p^{(\theta+h)\cdot n}(y)
	&=p^{\theta\cdot n}(y-h),\qquad \mbox{ for all }\;  y\cdot n - (\theta+ h)\cdot n>0
\end{align*}
In particular, for $y=\theta+h+tn, \; t>0$,
we get 
\begin{align*}
	u^{(\theta+h)\cdot n}(\theta+h+tn)&=u^{\theta\cdot n}(\theta+tn)\\
		p^{(\theta+h)\cdot n}(\theta+h+tn)&=p^{\theta\cdot n}(\theta+tn).
\end{align*}
Recalling from formula (\ref{def_Vtheta,t}) 	the latter equalities are also written 	
		 \begin{align*}
			V(\theta+h,t)&=V(\theta,t)\\
			R(\theta+h,t)&=R(\theta,t)	
		\end{align*}
		where the pair $(V,R)$ satisfies problem (\ref{prob_V}).	Since this holds for arbitrary $\theta\in\mathbb{R}^d$, $h\in\mathbb{Z}^d$
one concludes  that $V(\theta,t)$ and  $R(\theta,t)$ are periodic in  $\theta$. Then the derivatives
\begin{align*}
	\partial_{\theta_\gamma} V(\theta,t),\quad   D_\theta V(\theta,t) ,\quad \partial_t V(\theta,t)
\end{align*} 	
are also periodic in $\theta$. Now, one differentiates the  identity  
\begin{align*}
	u^s(y)&=V(y-(y\cdot n- s)n,y\cdot n- s),
\end{align*} 		
to get
\begin{align*}
	\partial_{y_\gamma} u^s(y)=&\; \partial_{\theta_\gamma} V(y-(y\cdot n- s)n,y\cdot n- s)\\
	&-n_\gamma D_\theta V(y-(y\cdot n- s)n,y\cdot n- s) n\\
	&+n_\gamma \partial_t V(y-(y\cdot n- s)n,y\cdot n- s) .
\end{align*} 
Now, it appears that when restricted to the hyperplane $\{y\cdot n- s=0\}$	
\begin{align*}
	\partial_{y_\gamma} u^s(y)=&\; \partial_{\theta_\gamma} V(y,0)-n_\gamma D_\theta V(y,0) n+n_\gamma \partial_t V(y,0) 
\end{align*}
is quasiperiodic. Similarly, the identity 
\begin{align*}
	p^s(y)&=R(y-(y\cdot n- s)n,y\cdot n- s)
\end{align*} 	
shows that $p^s(y)$ is quasiperiodic on $\{y\cdot n- s=0\}$.
		
	\end{appendices}

	\vspace{0.2in}
	
	\noindent	{\bfseries Acknowledgements}. The author would like to thank his  Ph.D. supervisor Christophe Prange for introducing  him to the homogenization theory in PDEs and for the valuable discussions he had with him. Also, the author would like to thank CY Cergy Paris University and the CNRS who  supported this research.

	\bibliographystyle{amsplain}

\end{document}